\documentclass{article}

\usepackage[a4paper]{geometry}
\usepackage[utf8]{inputenc}

\usepackage{graphicx}
\usepackage{latexsym}

\usepackage{lipsum}
\usepackage{graphicx}
\usepackage{authblk}
\usepackage[english]{babel}
\usepackage{color}
\usepackage{hyperref}
\hypersetup{
    colorlinks=true,
    linkcolor=blue,
    filecolor=blue,      
    urlcolor=blue,
    citecolor=blue,
    }
\usepackage{dsfont}
\usepackage{geometry}
\usepackage{float}
\usepackage{bbm}

\usepackage[nottoc]{tocbibind}
 \usepackage{indentfirst}
\usepackage{customcommands}
\usepackage{amsmath}

\graphicspath{{./figure/}}

\usepackage{pgfplots}
\usepackage{comment}

\usepackage{tikz}
\usepackage{caption}
\usepackage{subcaption}
\usepackage{array}    
\usepackage{booktabs}
\usepackage{multirow}

\usepackage{authblk} 
\usepackage{empheq} 
\usepackage[toc,page]{appendix} 

\usepackage[labelfont=bf]{caption}

\usepackage{amsmath,amsfonts,amssymb,amsthm}

\numberwithin{equation}{section}
\numberwithin{figure}{section}
\numberwithin{table}{section}

\usepackage{etoolbox}
\apptocmd{\thebibliography}{}{}{}

\newcommand{\R}{{\mathbb R}}

\theoremstyle{plain}
\newtheorem{theor}{Theorem}[section]

\newtheorem{propo}[theor]{Proposition}

\newtheorem{reminder*}{[theor]Reminder}

\newtheorem{details*}[theor]{Details}

\newtheorem{comm*}{Comment}
\newtheorem{defin}[theor]{Definition} 
\newtheorem{defin*}{[theor]Definition}

\newtheorem{notat*}{Notation}

\newtheorem{remar}[theor]{Remark}

\newcommand{\vertii}[1]{{\left\vert\kern-0.3ex\left\vert #1 
    \right\vert\kern-0.3ex\right\vert}}

\newcommand{\vertiii}[1]{{\left\vert\kern-0.3ex\left\vert\kern-0.3ex\left\vert #1 
    \right\vert\kern-0.3ex\right\vert\kern-0.3ex\right\vert}}

\newcommand*\circled[1]{\tikz[baseline=(char.base)]{\node[shape=circle,draw,inner sep=2pt] (char) {#1};}}

\newcommand\restr[2]{{
  \left.\kern-\nulldelimiterspace 
  #1 
  \littletaller 
  \right|_{#2} 
  }}
    
\newcommand{\littletaller}{\mathchoice{\vphantom{\big|}}{}{}{}}

\title{Four collapsing one-dimensional particles:\\a dynamical system approach of the spherical billiard reduction}

\author[1]{Roberto Castorrini}
\author[2]{Th\'eophile Dolmaire
}
\affil[1]{{\small{Classe di Scienze, Scuola Normale Superiore di Pisa,
and Dipartimento di Economia, Ingegneria, Società e Impresa (DEIM), Università della Tuscia, 01100, Viterbo, Italy}}}
\affil[2]{{\small{Dipartimento di Ingegneria e Scienze dell’Informazione e Matematica (DISIM), Università degli Studi dell’Aquila, Edificio Renato Ricamo, via Vetoio, Coppito, 67100 L’Aquila, Italy.}}}

\date{\today}

\begin{document}

\maketitle

\begin{abstract}
\noindent
We consider a system of four one-dimensional inelastic hard spheres evolving on the real line $\mathbb{R}$, and colliding according to a scattering law characterized by a fixed restitution coefficient $r$. We study the possible orders of collisions when the inelastic collapse occurs, relying on the so-called $\mathfrak{b}$-to-$\mathfrak{b}$ mapping, a two-dimensional dynamical system associated to the original particle system which encodes all the possible collision orders. We prove that the $\mathfrak{b}$-to-$\mathfrak{b}$ mapping is a piecewise projective transformation, which allows one to perform efficient numerical simulations of its orbits. We recover previously known results concerning the one-dimensional four-particle inelastic hard sphere system and we support the conjectures stated in the literature concerning particular periodic orbits. We discover three new families of periodic orbits that coexist depending on the restitution coefficient, we prove rigorously that there exist stable periodic orbits for the $\mathfrak{b}$-to-$\mathfrak{b}$ mapping for restitution coefficients larger than the upper bounds previously known, and we prove the existence of quasi-periodic orbits for this mapping.
\end{abstract}


\tableofcontents

\section{Introduction}

\noindent
In this work, we consider a system of four identical, one-dimensional inelastic hard spheres, that evolve on the real line, and that collide according to a scattering law described with a fixed restitution coefficient $r \in\ ]0,1[$. In other words, between two collisions, the particles undergo inertial motion, and the respective velocities $v$ and $v_* \in \mathbb{R}$ of two colliding particles are immediately modified into $v' = v - \frac{(1+r)}{2}(v-v_*)$ and $v'_* = v_* + \frac{(1+r)}{2}(v-v_*)$, ensuring that the momentum is conserved during the collisions, but also that a positive amount of kinetic energy is lost at each collision. Our main objective is to describe the possible orders of collisions occurring in the particle system during an inelastic collapse, that is, when infinitely many collisions take place between the particles in finite time.\\
\newline
Such a phenomenological collision model is of central importance to describe granular media, that are composed of a large number of particles that interact in a non-conservative manner. Snow, wheat or also interstellar dust can be described in terms of granular media. The reader may refer for instance to \cite{JaNB996}, \cite{Kada999} and \cite{BrPo004} for a general introduction to the physical relevance of the model of granular media, as well as an exposition of their main properties. Compared to conservative gases, that can be effectively described by the classical, elastic Boltzmann equation, in many cases the behaviour of granular media is substantially different. One of the main peculiarities of these media is their tendency to develop spontaneously spatial inhomogeneities, hence the name of \emph{granular} media. This emergence of structure in the particle system is a phenomenon that is still poorly understood, yet fundamental, since the particle systems colliding inelastically are considered as promising models to explain, for instance, the formation of large structures in the solar system, such as planetary rings (see for instance \cite{Brah975} and \cite{SpSc006}).\\
Despite its simplicity, the collision law with a fixed restitution coefficient is the paradigmatic model of inelastic collisions. In some sense, this collision law is the simplest that can be considered, and can be derived as follows. Since the contact between two colliding particles takes place on a time scale which is much smaller than the mean free flight of the particles between two consecutive collisions, the collisions are described as instantaneous events. Therefore, the deformation mechanisms of the particles, which lead to the dissipation of the kinetic energy and a complex evolution of the relative velocity between two colliding particles, is not described in the model, and only the difference between the relative velocities immediately before and after the contact is described. For particles in general dimension $d$, the collision mechanism causes a change of the normal and tangential components (with respect to the plane of contact between the two colliding particles) of the relative velocity, and the normal (respectively, tangential) restitution coefficients is defined as the ratio of the post- and pre-collisional normal (respectively, tangential) relative velocities. In general, when post-collisional velocities are computed in terms of the pre-collisional velocities by taking into account the inelastic deformation of the particles, such restitution coefficients depend on the relative velocity (see \cite{BSHP996} for the case of the normal restitution coefficient, \cite{ScBP008} for the case of the tangential coefficient, and see \cite{BrPo004} and the references therein for a general discussion on this question), and many of the physically relevant models are such that the variable normal restitution coefficients tends to $1$ as the normal component of the relative velocity tends to $0$ (the collisions with low energy are ``quasi-elastic''). Nevertheless, for the sake of simplicity, the collision laws that are considered often neglect tangential changes for the relative velocity, and the case when the normal restitution coefficient is assumed to be a constant number is considered as a reasonably good approximation of physical particle systems (see \cite{BrPo004}), and some phenomena can even be described only if this restitution coefficient is always smaller than a certain $r_0 < 1$ (see in particular \cite{SpSc006}, in the context of planetary rings' formation).\\
\newline
Considering inelastic particle systems colliding according to a law with a fixed restitution coefficient allows the description of structure formations. In particular, it was numerically observed in \cite{GoZa993} that granular gases tend to form clusters, and the numerical simulation of the dynamics of the particles in such clusters can even present singularities, due to the inelastic collapse (see for instance \cite{McYo992} and \cite{McYo994}). The onset of such a singularity yields in turn major mathematical problems, preventing a priori the rigorous study of granular media with the tools of kinetic theory. Indeed, among the clusters, the separation of scales fails, and even worse, the dynamics of the particle system is not even defined beyond the time of collapse.\\
These difficulties explain why the rigorous derivation of inelastic versions of the Boltzmann equation, as well as understanding the properties of the solutions of such equations, constitute a very challenging problem. We point however the two following recent results, providing results in these directions. On the one hand, \cite{AlLT023} provides, to the best of our knowledge, the first well-posedness result for the inelastic Boltzmann equation in a spatially inhomogeneous regime and not necessarily close to vacuum, for particles colliding according to a law with a fixed restitution coefficient. On the other hand, \cite{DoNo025} constitutes the first rigorous derivation of the inelastic linear Boltzmann equation, from a Lorentz gas of light particles that evolve deterministically among a background of heavy scatterers, distributed according to a Poisson process, and in the low density limit. Nevertheless, the two references \cite{AlLT023} and \cite{DoNo025} deal with cases in which the clusters are not expected to appear. Indeed, in \cite{AlLT023} it is assumed that $r$, despite fixed, is close enough to $1$, which corresponds to the kinetic regime described for instance in \cite{McYo994}, so that the granular gas remains spatially homogeneous in such a regime. In \cite{DoNo025}, the fact that the light particles interact only with the fixed obstacles prevents the collapse from occurring because the particles slow down due to the collisions, and so on average the time of the mean free flight increases.\\
In the case of inelastic hard spheres with fixed restitution coefficients, which is described by the fully non-linear inelastic Boltzmann equation (see \cite{BrPo004}, \cite{Vill006} and \cite{CHMR021}), the inelastic collapse was first described mathematically for one-dimensional models in \cite{ShKa989} and \cite{BeMa990}, and later it was observed numerically in \cite{McYo994} in dimension $2$ (\cite{TrBa000} investigates the case of larger dimension, and it is proved in \cite{DoVe024} that the inelastic collapse does take place in dimension $d \geq 2$ arbitrary, in a stable manner). In particular, it was observed in \cite{McYo994} that the inelastic collapse takes place among the clusters, and that the particles involved in the collisions close to the singularity time form linear, chain-like structures inside the clusters (see also \cite{PoSc005} for more details on this property of the collapse). Therefore, and even if the inelastic collapse does take place in any dimension, it appears to be an essentially one-dimensional phenomenon. As a consequence, the one-dimensional inelastic collapse is not only a toy model, but the fundamental mechanism behind the singularities that develop in dissipative hard sphere systems.\\
After the first pioneering investigations \cite{ShKa989}, \cite{BeMa990}, \cite{McYo992} on the one-dimensional inelastic collapse, an important step was completed in \cite{CoGM995}, which provides a complete understanding of the three-particle system. In particular, it is proved that the inelastic collapse of three particles cannot take place if the restitution coefficient $r$ is larger than $7-4\sqrt{3} \simeq 0.0718$. In such a system, labelling the particles from the left to the right as \footnotesize{\circled{1}}\normalsize{}, \footnotesize{\circled{2}}\normalsize{} and \footnotesize{\circled{3}}\normalsize{}, and denoting a collision of the pair \footnotesize{\circled{1}}\normalsize{}-\footnotesize{\circled{2}}\normalsize{} by $\mathfrak{a}$, and a collision of the pair \footnotesize{\circled{2}}\normalsize{}-\footnotesize{\circled{3}}\normalsize{} by $\mathfrak{b}$, it is clear that if infinitely many collisions take place, they have to take place according to the infinite repetition of the sequence of collisions $\mathfrak{ab}$. This observation allows to study one single dynamical system, representing such a pair of consecutive collisions, which, when iterated infinitely many times, provides the final state of the particle system at the time of the inelastic collapse.\\
Concerning larger systems, only few results are known. The main difficulties for systems with $N$ particles, $N \geq 4$, lies in the fact that if infinitely many collisions take place in the system, it is not clear how to determine a priori which pairs of the particles will be involved in the consecutive collisions. Adapting the notations above to the case of four particles, \cite{CDKK999} discovered that the collapse can take place, in a stable manner, according to the infinite repetition of the periods of collisions $\big(\mathfrak{ab}\big)^n \big(\mathfrak{cb}\big)^n$, with $n$ fixed but arbitrary. Nevertheless, only such periods were discovered, the largest restitution coefficient for which one of them (here, $\mathfrak{ababcbcb}$) can be realized in a stable manner was found to be $3-2\sqrt{2} \simeq 0.1716$, and no periodic pattern, even unstable was found for $r$ larger than a certain critical value $r_\text{crit} \simeq 0.1917$.\\
A recent work \cite{DoHR025} discovered a periodic pattern different from the ones discussed in \cite{CDKK999}, namely $\mathfrak{ababcb}$, but such a pattern turned out to be unstable. In addition, it is proved in \cite{DoHR025} that it is possible to associate to the four-particle system a two-dimensional dynamical system, the so-called $\mathfrak{b}$-to-$\mathfrak{b}$ mapping, which encodes all the possible collision orders that can be achieved in the original particle system. The fact that such a dynamical system is two-dimensional allowed to perform numerical simulations and representations of its orbits, but only few rigorous results were obtained in \cite{DoHR025}. In the present article, we will study in detail the $\mathfrak{b}$-to-$\mathfrak{b}$ mapping, relying on an underlying piecewise linear structure. This structure will allow us to understand the structure of the orbits of the $\mathfrak{b}$-to-$\mathfrak{b}$ mapping, and we will in particular describe new families of periodic orbits. We present extensive numerical investigations on such orbits, and we prove the existence of stable periodic orbits, for restitution coefficients larger than the critical restitution coefficients $3-2\sqrt{2}$ and $r_\text{crit}$ of \cite{CDKK999}.\\
In order to complete the review of the literature concerning the inelastic collapse, we mention \cite{BeCa999}, which provides the construction of initial configurations leading to the inelastic collapse of $N$ hard spheres in dimension $1$, as well as lower and upper bounds on the critical restitution coefficient $r(N)$, above which no collapse of $N$ particles can take place. We mention also \cite{ChKZ022}, which provides another elegant geometric construction of the collapse of $N$ inelastic particles. \cite{BrPo004}, \cite{McNa012} and the more recent \cite{Dolm025} constitute surveys about the inelastic collapse in one-dimensional particle systems. All the previous references were dealing with particles evolving on the whole real line $\mathbb{R}$. In \cite{GrMu996}, the case of three inelastic hard spheres collapsing on a ring is studied. The case of collapsing particles in higher dimensions is considered in \cite{ZhKa996}, \cite{DoVe024} and \cite{DoVe025}. Turning now to different collision models, in \cite{GSBM998} is established the absence of the inelastic collapse in the case of three one-dimensional particles colliding with a variable restitution coefficient, with quasi-elastic collisions in the low energy regime. In \cite{ScZh996}, the case of the collapse of three particles interacting via frictional collisions is studied: the rotational movement of the particles is also described, and a tangential restitution coefficient is involved. Finally, a model inspired from quantum physics, in which only a fixed amount of kinetic energy is dissipated in any collision that is energetic enough, is considered in \cite{DoVeNot}. In particular, it is shown that in such a system composed of $N$ particles ($N$ arbitrary), almost every initial configuration of the particles leads to a well-posed dynamics globally in time, so that the inelastic collapse never occurs in such globally well-posed trajectories.

\paragraph{Outline of the article.} The plan of the present article is the following. In Section \ref{SECTIModel_&_Reducti}, we introduce the model of the four one-dimensional inelastic particle system we will consider, and we recall the dimensional reduction of the system which leads to define the $\mathfrak{b}$-to-$\mathfrak{b}$ mapping introduced in \cite{DoHR025} and which encodes all the possible orders of collisions that can take place in the original particle system. In Section \ref{SECTIMappi_as_a_PiecewiseLinea}, we prove that the $\mathfrak{b}$-to-$\mathfrak{b}$ mapping can be written as a piecewise projective transformation. In Section \ref{SECTIFirstNumerSimul}, we present a first series of numerical simulations of the orbits of the $\mathfrak{b}$-to-$\mathfrak{b}$ mapping, confirming in particular a conjecture stated in \cite{CDKK999} concerning the ranges of restitution coefficients for which the periodic patterns $\big(\mathfrak{ab}\big)^n \big(\mathfrak{cb}\big)^n$ can be observed. In Section \ref{SECTISpectStudy} we study the underlying linear mappings that describe the action of the $\mathfrak{b}$-to-$\mathfrak{b}$ mapping on the different subsets of its domain. We deduce in particular the existence of quasi-periodic orbits for the $\mathfrak{b}$-to-$\mathfrak{b}$ mapping in some ranges of restitution coefficients. In Section \ref{SECTISeconNumerSimul}, we present a second series of numerical simulations of the orbits of the $\mathfrak{b}$-to-$\mathfrak{b}$ mapping, focusing this time on the search for particular periodic orbits. We describe in particular three new families of periodic orbits. In Section \ref{SECTIMatheStudyPatte} we study rigorously two of the new periodic orbits described in the previous section. We establish in particular that the periodic orbit, with the period $(\mathfrak{ab})(\mathfrak{cb})(\mathfrak{acb})(\mathfrak{cb})(\mathfrak{ab})(\mathfrak{acb})$, is stable for any restitution coefficient $r > r_{\text{crit},132^J} \simeq 0.2200$, and we prove that the periodic orbit of period $(\mathfrak{ab})(\mathfrak{cb})(\mathfrak{acb})(\mathfrak{acb})(\mathfrak{cb})(\mathfrak{ab})(\mathfrak{acb})(\mathfrak{acb})$ can never be achieved in a stable manner. Finally, in Section \ref{SECTIperspectives}, we turn to future works that can be conducted to understand better the dynamics of the $\mathfrak{b}$-to-$\mathfrak{b}$ mapping, in terms of statistical properties, relying on tools from the theory of dynamical systems.\\
\newline
The main contributions of the present work consist of the proof that the $\mathfrak{b}$-to-$\mathfrak{b}$ mapping is a piecewise projective transformation, as well as the consequences of such a result. We perform efficient numerical simulations of the orbits of the $\mathfrak{b}$-to-$\mathfrak{b}$ mapping, recovering previously known results concerning the one-dimensional four-particle inelastic hard sphere system, such as the apparent stability of the patterns $\big(\mathfrak{ab}\big)^n\big(\mathfrak{cb}\big)^n$ in the windows of stability discussed in \cite{CDKK999}. In particular, we provide further support for the conjectures stated in \cite{CDKK999}, such as the fact that $\big(\mathfrak{ab}\big)^n\big(\mathfrak{cb}\big)^n$ appears to be stable for any value of $n\geq 2$, in an interval $I_n$ of restitution coefficients whose boundaries are described conjecturally in \cite{CDKK999} and which we confirm numerically for $2 \leq n \leq 125$. Such intervals $I_n$ accumulate at the critical value of $7-4\sqrt{3}$ as $n \rightarrow +\infty$. We discover three new families of periodic orbits for the $\mathfrak{b}$-to-$\mathfrak{b}$ mapping. In addition, we find that some of these periodic orbits coexist for some restitution coefficients. We prove rigorously that there exist stable periodic orbits for the $\mathfrak{b}$-to-$\mathfrak{b}$ mapping for restitution coefficients larger than the upper bounds $3-2\sqrt{2} \simeq 0.1716$ and $r_\text{crit} \simeq 0.1917$ previously discovered in \cite{CDKK999}. Provided that $r > 3-2\sqrt{2} \simeq 0.1716$, we prove the existence of quasi-periodic orbits for this mapping, each contained in an invariant smooth manifold. These invariant manifolds are explicit curves, whose uncountable union forms a foliation of a subset of the phase space of the $\mathfrak{b}$-to-$\mathfrak{b}$ mapping with a positive Lebesgue measure.

\paragraph{Notations.} We will denote the transpose of a vector $x$ of $\mathbb{R}^d$, and more generally, of any matrix $M \in \mathcal{M}_{n\times m}(\mathbb{R})$, by:
\begin{align}
^t\hspace{-0.25mm} x \hspace{5mm} \text{and} \hspace{5mm} ^t\hspace{-0.5mm} M.
\end{align}

\section{The model and the dimensional reductions}
\label{SECTIModel_&_Reducti}

\subsection{The model}

\paragraph{The inelastic collision law with fixed restitution coefficient.} We consider a system of four identical, one-dimensional inelastic hard spheres, colliding according to the collision law defined with a \emph{fixed restitution coefficient} $r \in [0,1]$. This means that when two particles collide with respective pre-collisional velocities $v$, $v_*$, the velocities are immediately changed into $v'$, $v_*'$ defined as:
\begin{align}
\label{EQUATScattering}
\left\{
\begin{array}{ccccc}
v' &=& \displaystyle{\frac{1-r}{2} v +\frac{1+r}{2}v_*} &=& v - \displaystyle{\frac{1+r}{2}}(v-v_*),\vspace{1mm}\\
v_*' &=& \displaystyle{\frac{1+r}{2} v + \frac{1-r}{2}v_*} &=& v_* + \displaystyle{\frac{1+r}{2}}(v-v_*).
\end{array}
\right.
\end{align}
The collision law \eqref{EQUATScattering} is designed such that the momentum is conserved during collisions ($v'+v_*' = v+v_*$), and the relative velocity is (reflected and) contracted by the factor $r$ ($v'-v_*' = -r(v-v_*)$). We recover the elastic case when $r=1$ (conservation of the kinetic energy), and when $r=0$ we obtain the sticky particle regime (specific to the dimension $d=1$, the particles do not separate after the collisions).

\paragraph{A first parametrization of the evolution of the particle system.} Between collisions, we will assume that the particles have an inertial movement, that is, they move with constant velocity.\\
We will denote the respective positions and velocities of the four particles by $x_i(t) \in \mathbb{R}$ and $v_i(t) \in \mathbb{R}$, for $1 \leq i \leq 4$. Observe that in the one-dimensional case, the size of the particles plays no role (even in the case when the particles are not identical), so we can assume without loss of generality that all the particles are point particles, that is, with a zero radius.\\
Since the momentum of the whole system is conserved for all time, the dynamics of the center of mass is trivial, and we can therefore study only the relative positions and velocities. We remark also that the order of the particles on the real axis is preserved for all times by the dynamics (at least, as long as the dynamics is well-defined): if initially we have $x_1(0) \leq x_2(0) \leq x_3(0) \leq x_4(0)$, then this chain of inequalities holds true for all positive time. Therefore, we will label the particles, from the left to the right, as \footnotesize{\circled{1}}\normalsize{} to \footnotesize{\circled{4}}\normalsize{}. Following the notations in \cite{DoHR025} we introduce:
\begin{align}
p_i(t) &= x_{i+1}(t) - x_i(t),\hspace{2mm} \forall\, 1 \leq i \leq 3,\\
q_i(t) &= v_{i+1}(t) - v_i(t),\hspace{2mm} \forall\, 1 \leq i \leq 3,
\end{align}
and we define the \emph{positions vector} $p(t) \in (\mathbb{R}_+)^3$ (where $\mathbb{R}_+ = \{x \geq 0\ /\ x \in \mathbb{R}\}$ denotes the set of the non-negative real numbers) and the \emph{velocities vector} $q(t) \in \mathbb{R}^3$ as:
\begin{align}
p(t) = \left(p_1(t),p_2(t),p_3(t)\right) \hspace{3mm} \text{and} \hspace{3mm} q(t) = \left(q_1(t),q_2(t),q_3(t)\right).
\end{align}
Between two collision times $t_k$ and $t_{k+1}$ the evolution law describing the inertial movement is:
\begin{subequations}
\label{EQUATFreeTransp}
\begin{empheq}[left=\empheqlbrace]{align}
	q(t)&=q(t_k) && \text{ for }t_k\leq t<t_{k+1}, \label{EQUATFree_TrnspVariable_q}\\ 
	p(t)&=p(t_k)+(t-t_k)q(t_k) \quad && \text{ for }t_k\leq t\leq t_{k+1} \label{EQUATFree_TrnspVariable_p}.
\end{empheq}
\end{subequations}
and when a collision takes place (that is, $p_i(t_{k+1})=0$ for some $i \in \{1,2,3\}$), the collision law \eqref{EQUATScattering} is written in matricial form:
\begin{align}\label{EQUATCollision__Law_}
	q(t_{k+1}^+)=K q(t_{k+1}^-),
\end{align}
where the matrix $K\in \R^{3\times 3}$ is equal to one of the three following \emph{collision matrices}:
\begin{align}\label{EQUATCollision_Matri}
	\begin{split}
		A&= \begin{pmatrix}
			-r & 0 & 0 \\
			\alpha & 1 & 0 \\
			0& 0 & 1 \\
		\end{pmatrix}\qquad \text{if $p_1(t_{k+1}) = 0$ (collision of type \circled{1}-\circled{2}, or type } \mathfrak{a}),\\
		B &= \begin{pmatrix}
			1 & \alpha & 0 \\
			0 & -r & 0 \\
			0 & \alpha & 1 \\
		\end{pmatrix} \qquad\text{if $p_2(t_{k+1}) = 0$ (collision of type \circled{2}-\circled{3}, or type } \mathfrak{b}),\\
		C &= \begin{pmatrix}
			1 & 0 & 0 \\
			0 & 1 & \alpha \\
			0 & 0 & -r \\
		\end{pmatrix} \qquad\text{if $p_3(t_{k+1}) = 0$ (collision of type \circled{3}-\circled{4}, or type } \mathfrak{c}),
	\end{split}
	\end{align}
where
\begin{align}
\label{EQUATDefinAlpha}
\alpha = \frac{r+1}{2}\cdotp
\end{align}
Observe that in the particular case when $p_1(t_{k+1}) = p_3(t_{k+1}) = 0$, simultaneous collisions take place, involving both pairs \footnotesize{\circled{1}}\normalsize{}-\footnotesize{\circled{2}}\normalsize{} and \footnotesize{\circled{3}}\normalsize{}-\footnotesize{\circled{4}}\normalsize{}, so that $K = CA = AC$. The case when $p_1(t_{k+1}) = p_2(t_{k+1}) = 0$ or $p_2(t_{k+1}) = p_3(t_{k+1}) = 0$ is forbidden, as it would describe a triple collision occurring at time $t_{k+1}$, and we will consider that the dynamics is not defined beyond such a time $t_{k+1}$.\\
For a given initial configuration $(p_0,q_0) \in \big(\mathbb{R}_+\big)^3 \times \mathbb{R}^3$ of particles, the time-dependent function of configurations $t \mapsto \big(p(t),q(t)\big)$, with $\big(p(0),q(0)\big) = (p_0,q_0)$, and defined according to the laws \eqref{EQUATFreeTransp}, \eqref{EQUATCollision__Law_} is called the \emph{trajectory starting from $(p_0,q_0)$}, and will be denoted by $\mathcal{T}_{p_0,q_0}$.

\paragraph{Time of existence of the trajectories.} We emphasize that a trajectory $\mathcal{T}_{p_0,q_0}:t \mapsto \big(p(t),q(t)\big)$ of the particle system is properly defined as long as, and only if, only binary collisions take place. In other words, if a particle collides with only one other particle at a time (a binary collision), the dynamics can always be extended on a non-trivial time interval. In the opposite case, when three particles or more collide together at the same time, determining the post-collisional velocities in a consistent manner with the collision law \eqref{EQUATCollision__Law_} yields an ill-posed problem.\\
Observe that, in the case of the four particles we are considering, the simultaneous collisions between the particles \footnotesize{\circled{1}}\normalsize{} and \footnotesize{\circled{2}}\normalsize{} on the one hand, and between \footnotesize{\circled{3}}\normalsize{} and \footnotesize{\circled{4}}\normalsize{} on the other hand allows also to continue the dynamics. In other words, the number of simultaneous collisions does not matter, as long as such collisions are all binary.\\
Nevertheless, if a trajectory presents only binary collisions, it is not sufficient to ensure that such a trajectory will be globally defined. Indeed, the phenomenon of \emph{inelastic collapse} can take place, meaning that infinitely many collisions take place in finite time. In such a case, a group of at least three adjacent particles will dissipate all their relative velocities through collisions by reaching the time of the collapse, so that the only reasonable way to continue the dynamics beyond is to state that they remain attached and form a cluster. Nevertheless, this yields an ill-posed problem when another particle collides with the cluster, in the same way as a triple collision cannot be described via the collision law \eqref{EQUATCollision__Law_}. We refer to \cite{ShKa989} in which the dynamics of the clusters is studied nonetheless.\\
For these reasons, to any initial configuration $(p_0,q_0)$, we associate to the trajectory $\mathcal{T}_{p_0,q_0}$ the \emph{time interval of existence} of the trajectory starting from $(p_0,q_0)$, denoted by $I_{p_0,q_0}$. This time interval $I_{p_0,q_0}$, which contains $0$, is defined as the largest interval on which the trajectory is defined, only in terms of free transport \eqref{EQUATFreeTransp} and binary collisions \eqref{EQUATCollision__Law_}.

\subsection{First dimensional reductions}

\paragraph{A first reduction: the discrete evolution between consecutive collisions.} A trajectory is defined as a time-dependent function, that is, it depends on the continuous variable $t \in I_{p_0,q_0}$. Nevertheless, it is enough to describe the trajectory only through the sequence $\big(p(t_k),q(t_k)\big)_k$, where $t_k$ is the $k$-th time of the interval of existence at which a collision takes place. Indeed, the trajectory can easily be reconstructed between two consecutive collision times $t_k$ and $t_{k+1}$, because we have $\mathcal{T}_{p_0,q_0}(t) = \big(p(t_k) + (t-t_k)q(t_k),q(t_k)\big) \ \forall t \in [t_k,t_{k+1}[$.\\
Observe that we have chosen the convention that the velocity vector $t \mapsto q(t)$ is right continuous, as it can also be seen in \eqref{EQUATFree_TrnspVariable_q}.\\
We remark also that $t_k$ is not necessarily the time of the $k$-th binary collision, since two collisions might take place at the same time. More precisely, relying on the set:
\begin{align}
K = \big\{ t \in I_{p_0,q_0}\ /\ \min_{1 \leq i \leq 3}p_i(t) = 0 \big\}
\end{align}
we define $t_k$ as the time that satisfies:
\begin{align}
\# \big([0,t_k[\ \cap\, K \big) = k-1,\hspace{5mm} \# \big([0,t_k] \cap K \big) = k
\end{align}
where $\# A$ denotes the cardinal of the set $A$. The trajectory might then be described relying on the discrete evolution given by
\begin{align}
\label{EQUATDiscrDynamSyste}
\mathbb{N} \ni k \mapsto z(k) = \big(p(t_k),q(t_k)\big) \in \big(\mathbb{R}_+\big)^3 \times \mathbb{R}^3.
\end{align}
Our objective is to understand better the collapse phenomenon, in particular through the possible orders of collisions leading to the collapse. In particular, if the collapse takes place for the trajectory starting from a particular configuration $(p_0,q_0)$ (and it certainly does for some of them, even in a stable manner: see \cite{CDKK999}, \cite{HuRo023} and \cite{DoHR025}), then \eqref{EQUATDiscrDynamSyste} defines indeed a discrete dynamical system, defined for any $k \in \mathbb{N}$.\\
We will denote the collisions between the pairs of particles \footnotesize{\circled{1}}\normalsize{}-\footnotesize{\circled{2}}\normalsize{}, \footnotesize{\circled{2}}\normalsize{}-\footnotesize{\circled{3}}\normalsize{}, \footnotesize{\circled{3}}\normalsize{}-\footnotesize{\circled{4}}\normalsize{}, respectively by $\mathfrak{a}$, $\mathfrak{b}$ and $\mathfrak{c}$. Here, we aim to understand better which sequences of collisions, written with the alphabet $\big\{\mathfrak{a}, \mathfrak{b}, \mathfrak{c}\big\}$, can exist and lead to the collapse.

\paragraph{The $\mathfrak{b}$-to-$\mathfrak{b}$ mapping.} The discrete dynamical system \eqref{EQUATDiscrDynamSyste} takes values in $\mathbb{R}^6$. There is however a first reduction of dimension that can be performed in the case when an inelastic collapse takes place. Indeed, it is clear that when the collapse occurs, then necessarily infinitely many collisions of type $\mathfrak{b}$ take place. As a consequence, we can equivalently describe the trajectory only in terms of the configurations $\big(p(t_k),q(t_k)\big)$, when $p_2(t_k) = 0$. From such a configuration, which is describing a collision of type $\mathfrak{b}$ that just took place, one can easily reconstruct the dynamics until the next collision of type $\mathfrak{b}$: the next collision is either of type $\mathfrak{a}$ or $\mathfrak{c}$. In the first case, the collision that follows is either of type $\mathfrak{b}$, closing the cycle, or it is of type $\mathfrak{c}$, but then such a collision is necessarily followed by a collision type $\mathfrak{b}$ (because after $\mathfrak{a}$ and $\mathfrak{c}$, \footnotesize{\circled{1}}\normalsize{} and \footnotesize{\circled{2}}\normalsize{} are separating, as well as \footnotesize{\circled{3}}\normalsize{} and \footnotesize{\circled{4}}\normalsize{}, so that the only pair that might collide next is \footnotesize{\circled{2}}\normalsize{}-\footnotesize{\circled{3}}\normalsize{}), closing the cycle also in this case. The case when $\mathfrak{c}$ follows immediately the first $\mathfrak{b}$ collision is treated in the same way.\\
We extract from the sequence $(t_k)_k$ the subsequence $(\widetilde{t}_k)_k$ such that $p_2(\widetilde{t}_k) = 0 \ \forall k \geq 0$, so that we reduced the dimension of the dynamical system. We will then focus our attention on the following discrete dynamical system, which we call the \emph{$\mathfrak{b}$-to-$\mathfrak{b}$ mapping}:
\vspace{3mm}
\begin{align}
&\hspace{25mm} \mathbb{N} \ni k \mapsto \widetilde{z}(k) = \big(p_1(\widetilde{t}_k),p_3(\widetilde{t}_k),q_1(\widetilde{t}_k),q_2(\widetilde{t}_k),q_3(\widetilde{t}_k)\big) \in \big(\mathbb{R}_+\big)^2 \times \mathbb{R}^3\\
&\text{with} \hspace{1mm} \widetilde{K}_{\mathfrak{b}} = \big\{ t \in I_{p_0,q_0}\ /\ p_2(t) = 0 \big\} \hspace{3mm} \text{and} \hspace{3mm} \widetilde{t}_k\text{ defined as }\# \big([0,\widetilde{t}_k[\ \cap\, \widetilde{K}_{\mathfrak{b}} \big) = k-1, \hspace{3mm} \# \big( [0,\widetilde{t}_k] \cap \widetilde{K}_{\mathfrak{b}} \big) = k.\nonumber
\end{align}
\noindent
In addition to the $\mathfrak{b}$-to-$\mathfrak{b}$ reduction, we observe also that the system admits the following natural rescaling if the main information we are interested in is the order of collisions, as it was observed for instance in \cite{BeCa999}. This rescaling consists in assuming that the vectors of positions and velocities are initially normalized. Naturally, this property will not hold during the whole evolution of the trajectory, but the rescaling has no effect on the order of collisions, so that we can reduce again the dimension (by $2$) by renormalizing the position and velocity vectors at each collision time $\widetilde{t}_k$. In other words, to investigate only the order of collisions of the trajectory starting from $(p_0,q_0)$, it is enough to consider the discrete evolution:
\begin{align}
\mathbb{N} \ni k \mapsto \big(\omega(\widetilde{t}_k),\sigma(\widetilde{t}_k)\big) \in \mathbb{S}^1 \times \mathbb{S}^2,
\end{align}
where $\omega(\widetilde{t}_{k+1})$ and $\sigma(\widetilde{t}_{k+1})$ are recursively defined as the respective normalizations of the position and velocity vectors $p(\widetilde{t}_{k+1}) \in \mathbb{R}^2$ and $q(\widetilde{t}_{k+1}) \in \mathbb{R}^3$ obtained from $p(\widetilde{t}_k) \in \mathbb{S}^1$ and $q(\widetilde{t}_k) \in \mathbb{S}^2$ when the $\mathfrak{b}$-to-$\mathfrak{b}$ cycle between $\widetilde{t}_k$ and $\widetilde{t}_{k+1}$ is completed.

\subsection{The spherical reduction and its consequences}
\label{SSECTSpherReduc}

\paragraph{The dimensional reduction to the projective sphere of \cite{DoHR025}.} Finally, one of the main results of \cite{DoHR025} allows to perform one last dimensional reduction.\\
If we consider the sequence of planes $\big(\mathcal{P}(k)\big)_k = \big(\text{Span}\big[\omega(\widetilde{t}_k),\sigma(\widetilde{t}_k)\big]\big)_k$, the datum of any of the planes $\mathcal{P}(k)$ can be identified as its normal line, whose only relevant information is the orientation, and so, $\mathcal{P}(k)$ can be naturally identified as an element of the two-dimensional sphere $\mathbb{S}^2$ quotiented by the antipodal equivalence relation $\omega \simeq -\omega$. In other words, the sequence $\big(\mathcal{P}(k)\big)_k = \big(\text{Span}\big[\omega(\widetilde{t}_k),\sigma(\widetilde{t}_k)\big]\big)_k$ can be represented as a sequence in the projective plane $\mathbb{P}_2(\mathbb{R})$.\\
The result of \cite{DoHR025} (Theorem 4.1) establishes that for two trajectories $\mathcal{T}_{p_0,q_0}$ and $\mathcal{T}_{p'_0,q'_0}$ starting respectively from $(p_0,q_0)$ and $(p'_0,q'_0) \in \big(\mathbb{R}_+\big)^3\times\mathbb{R}^3$, with $p_0 \cdot e_2 = p'_0 \cdot e_2 = 0$ (where $e_2$ is the second vector of the canonical basis in $\mathbb{R}^3$) and such that $\text{Span}\big[p_0,q_0\big] = \text{Span}\big[p'_0,q'_0\big]$, these two trajectories $\big((p(\widetilde{t}_k),q(\widetilde{t}_k))\big)_k$ and $\big((p'(\widetilde{t}_k),q'(\widetilde{t}_k))\big)_k$ generate the same sequences of planes. In other words, we have that:
\begin{align}
\text{if}\hspace{3mm} \text{Span}\big[p_0,q_0\big] = \text{Span}\big[p'_0,q'_0\big], \hspace{3mm} \text{then} \hspace{3mm} \text{Span}\big[p(\widetilde{t}_k),q(\widetilde{t}_k)\big] = \text{Span}\big[p'(\widetilde{t}_k),q'(\widetilde{t}_k)\big] \hspace{3mm} \forall k \geq 0.
\end{align}
Therefore, identifying any hyperplane of $\mathbb{R}^3$ with any unit normal vector to such hyperplane, to any trajectory of four collapsing one-dimensional inelastic hard spheres can be associated the discrete dynamical system, taking values in $\mathbb{P}_2(\mathbb{R})$:
\begin{align}
\label{EQUATSpherReducDiscrMappi}
\mathbb{N} \ni k \mapsto \text{Span}\big[\omega(\widetilde{t}_k),\sigma(\widetilde{t}_k)\big] \in \mathbb{P}_2(\mathbb{R}),
\end{align}
which encompasses in particular the order of collisions, and all the possible orders of collisions in the original four particle system are encoded in the dynamical system described by \eqref{EQUATSpherReducDiscrMappi}. Rephrasing again, there exists a mapping:
\begin{align}
\label{EQUATDomai__P__Forma}
\widehat{\mathfrak{P}}:
\mathbb{P}_2(\mathbb{R}) \rightarrow \mathbb{P}_2(\mathbb{R})
\end{align}
such that the discrete dynamical system \eqref{EQUATSpherReducDiscrMappi} can be rewritten as:
\begin{align}
\label{EQUATItera__P__}
\text{Span}\big[\omega(\widetilde{t}_k),\sigma(\widetilde{t}_k)\big] = \widehat{\mathfrak{P}}^k \big( \text{Span}\big[p_0/\vert p_0 \vert,q_0/\vert q_0 \vert\big] \big) \hspace{3mm} \forall k \in \mathbb{N},
\end{align}

\noindent
where $\widehat{\mathfrak{P}}^k$ denotes the $k$-th iteration of the mapping $\widehat{\mathfrak{P}}$. This last dimensional reduction is referred to as the \emph{spherical reduction} in \cite{DoHR025} , since the only information that is recorded to describe the orbits is the sequence of hyperplanes $\text{Span}\big[\omega(\widetilde{t}_k),\sigma(\widetilde{t}_k)\big]$. For all times $t$, the trajectory of the original four particle system, described in terms of the variables $\big(p(t),q(t)\big)$, is contained in the union of the hyperplanes (provided that we consider all the hyperplanes $\text{Span}\big[\omega(\widetilde{t}_k),\sigma(\widetilde{t}_k)\big]$ associated to any collision of the trajectory (and not only to the collisions of type $\mathfrak{b}$). Considering this sequence of planes, intersected with the unit sphere $\mathbb{S}^2$, we obtain a family of arcs of circles that constitute trajectories corresponding to a billiard defined on a portion of $\mathbb{S}^2$. In such a billiard, a reflection takes place when the trajectory on the sphere reaches the boundary, that is, when an arc of circle intersects one of the planes $\{ x = 0 \}$, $\{ y = 0 \}$ or $\{ z = 0 \}$. The reflection law is then different from the usual specular reflection. Such a billiard description, with a non-standard reflection law, was already used in \cite{CoGM995} in a similar setting, to describe completely the inelastic collapse of three particles.

\paragraph{Definition of the $\mathfrak{b}$-to-$\mathfrak{b}$ mapping $\mathfrak{P}$.} The mapping $\widehat{\mathfrak{P}}$ is defined in \cite{DoHR025} as follows. First, starting from a configuration such that \footnotesize{\circled{2}}\normalsize{} and \footnotesize{\circled{3}}\normalsize{} are in contact (that is, such that a collision of type $\mathfrak{b}$ just took place), and assuming that $p_0$, $q_0$ are normalized, and that $q_0$ belongs to the normal plane to the unit sphere $\mathbb{S}^2$ at the point $p_0$ (which we can always do, according to Theorem 4.1 in \cite{DoHR025}), we have:
\begin{align}
\label{EQUATParamInitiaDataAngle}
p_0 = \begin{pmatrix}
\sin\theta \\ 0 \\ \cos\theta
\end{pmatrix},
\hspace{10mm}
q_0 = \begin{pmatrix}
\cos\varphi\cos\theta \\ \sin\varphi \\ -\cos\varphi\sin\theta
\end{pmatrix}
\end{align}
for certain $\theta \in\ ]0,\pi/2[$ and $\varphi \in\ ]0,\pi [$. The next collision is determined by finding which plane between $\{p_1=0\}$ (collision of type $\mathfrak{a}$) and $\{p_3=0\}$ (collision of type $\mathfrak{c}$) is intersected first by the arc of circle starting from $p_0$ and defined as the intersection between the unit sphere $\mathbb{S}^2$ and the plane $\mathcal{P}(0) = \text{Span}[p_0,q_0]$. $\cos\varphi > 0$ if and only if the next collision is of type $\mathfrak{c}$, in which case the configuration of the particles right after such a collision, after normalization, becomes $\big(p(t_1),q(t_1)\big)$, with:
\begin{align}
p(t_1) = \big(p_0 \wedge q_0\big) \wedge e_z \hspace{5mm} \text{and} \hspace{5mm} q(t_1) = \pi_{p(t_1)^\perp} \big[ C q_0 \big],
\end{align}
where:
\begin{itemize}
\item $e_z = (0,0,1)$ is the third vector of the canonical basis,
\item $p(t_1)$ is the vector orientating the intersection line between the planes $\text{Span}[p_0,q_0] = \big(p_0\wedge q_0\big)^\perp$ and $\{p_3 = 0\} = e_z^\perp$,
\item $q(t_1)$ is the projection on the orthogonal plane to $p(t_1)$ of the vector $C q_0$, where $C$ is the collision matrix given by \eqref{EQUATCollision_Matri} and corresponding to a collision of type $\mathfrak{c}$.
\end{itemize}
We deduce therefore that $\mathcal{P}(t_1) = \text{Span}\big[p(t_1),q(t_1)\big]$. The construction is similar if the first collision that follows the initial configuration is of type $\mathfrak{a}$, that is, when $\cos\varphi < 0$. In this case, we define:
\begin{align}
p(t_1) = \big(p_0 \wedge q_0\big) \wedge e_x \hspace{5mm} \text{and} \hspace{5mm} q(t_1) = \pi_{p(t_1)^\perp}\big[ A q_0 \big].
\end{align}
To obtain the complete expression of the $\mathfrak{b}$-to-$\mathfrak{b}$ mapping $\widehat{\mathfrak{P}}$, it remains to repeat the construction, until the first collision of type $\mathfrak{b}$ is reached, which can happen only according to the orders $\mathfrak{ab}$, $\mathfrak{acb}$, $\mathfrak{cab}$ or $\mathfrak{cb}$.\\
\newline
For the sake of completeness, we write now $\widehat{\mathfrak{P}}$ in terms of its action on the normal vector $p_0 \wedge q_0$ to the plane $\mathcal{P}(0)$. The complete definition of $\widehat{\mathfrak{P}}$ is the following, given by the following algorithm, that terminates in finite time (between $2$ and $3$ iterations, depending on the argument $u \in \mathbb{S}^2$).

\begin{defin}[Action of the $\mathfrak{b}$-to-$\mathfrak{b}$ mapping $\widehat{\mathfrak{P}}$ on $\mathbb{P}_2(\mathbb{R})$]
\label{DEFIN_Map_b_2_b__P__}
Let $u \in \mathbb{S}^2$. We define $\widehat{\mathfrak{P}}(u)=u' \in \mathbb{S}^2$, where $u'$ is computed according to the following algorithm.
\begin{enumerate}
\item If $u \cdot e_x < 0$, replace $u$ by $-u$. If not, skip this step.
\item Define the value of the variable \emph{contact} to be $\mathfrak{b}$.
\item If $\text{contact} = \mathfrak{a}$, then:\begin{enumerate}
\item Define $p$, $q$ and $v$ as:
\begin{align}
p = u \wedge e_x, \hspace{5mm} q = p \wedge u, \hspace{5mm} v = u - \big(u \cdot e_x\big)e_x.
\end{align}
\item If $q \cdot v > 0$, update the variable $\text{contact} = \mathfrak{b}$, and compute:
\begin{align}
p' = e_y \wedge u, \hspace{5mm} q' = B q.
\end{align}
If instead $q \cdot v < 0$, update the variable $\text{contact} = \mathfrak{c}$, and compute:
\begin{align}
p' = e_z \wedge u, \hspace{5mm} q' = C q.
\end{align}
\end{enumerate}
If $\text{contact} = \mathfrak{b}$, then:\begin{enumerate}
\item Define $p$, $q$ and $v$ as:
\begin{align}
p = u \wedge e_y, \hspace{5mm} q = p \wedge u, \hspace{5mm} v = u - \big(u \cdot e_y\big)e_y.
\end{align}
\item If $q \cdot v > 0$, update the variable $\text{contact} = \mathfrak{c}$, and compute:
\begin{align}
p' = e_z \wedge u, \hspace{5mm} q' = C q.
\end{align}
If instead $q \cdot v < 0$, update the variable $\text{contact} = \mathfrak{a}$, and compute:
\begin{align}
p' = e_x \wedge u, \hspace{5mm} q' = A q.
\end{align}
\end{enumerate}
If $\text{contact} = \mathfrak{c}$, then:\begin{enumerate}
\item Define $p$, $q$ and $v$ as:
\begin{align}
p = u \wedge e_z, \hspace{5mm} q = p \wedge u, \hspace{5mm} v = u - \big(u \cdot e_z\big)e_z.
\end{align}
\item If $q \cdot v > 0$, update the variable $\text{contact} = \mathfrak{a}$, and compute:
\begin{align}
p' = e_x \wedge u, \hspace{5mm} q' = A q.
\end{align}
If instead $q \cdot v < 0$, update the variable $\text{contact} = \mathfrak{b}$, and compute:
\begin{align}
p' = e_y \wedge u, \hspace{5mm} q' = B q.
\end{align}
\end{enumerate}
\item Compute $u' = q' \wedge p'$. Replace $u$ by $u'$.
\item If $\text{contact} = \mathfrak{b}$, stop the algorithm and return $u'$ as the result of the algorithm.\\
If not, repeat the algorithm starting from step $3$.
\end{enumerate}
\end{defin}
\noindent
Observe that in the algorithm above, we implictly excluded the degenerate cases $u \cdot e_x = 0$, $u \cdot e_y = 0$, $u \cdot e_z = 0$ and $q \cdot v = 0$, which all correspond to zero-measure sets in $\mathbb{S}^2$.\\
The reader may also refer to \cite{DoHR025} for another expression of the mapping $\widehat{\mathfrak{P}}$, relying on trigonometric functions.\\
\newline
In the next section, we will study in more detail the mapping $\widehat{\mathfrak{P}}$, providing in particular an explicit and simple expression.

\begin{remar}
In the end, the different dimensional reductions we presented in this section allow to study the possible orders of collisions using a $2$-dimensional dynamical system.\\
It is important to emphasize that, given any plane $\mathcal{P} = \text{Span}\big[\omega,\sigma\big]$, it is \emph{always} possible to compute its image by the mapping $\widehat{\mathfrak{P}}$. This is an important difference with respect to the original particle system, for which some configurations of particles are such that no collision will take place in the future. On the contrary, as it can be seen with the construction that we just sketched, one can always consider infinitely many iterations of the mapping $\widehat{\mathfrak{P}}$, so that the dynamical system given by \eqref{EQUATItera__P__} is properly defined.\\
It is also important to remark that we considered the $\mathfrak{b}$-to-$\mathfrak{b}$ mapping in order to define $\widehat{\mathfrak{P}}$. Although in the original particle system an infinite number of collisions implies always that there will be also infinitely many collisions of type $\mathfrak{b}$, it is less clear that such a property holds also for the iteration we constructed on the planes $\text{Span}\big[p(t_k),q(t_k)\big]$. The property holds nevertheless for this iteration (see Proposition 4.5 and Corollary 4.6 in \cite{DoHR025}): after a collision of type $\mathfrak{b}$ follows necessarily one of the consecutive sequences of collisions $\mathfrak{ab}$, $\mathfrak{cb}$, $\mathfrak{acb}$ or $\mathfrak{cab}$, so that $\widehat{\mathfrak{P}}$ is indeed properly defined.
\end{remar}

\section{The mapping $\widehat{\mathfrak{P}}$ written as a piecewise projective transformation}
\label{SECTIMappi_as_a_PiecewiseLinea}

\noindent
We recall that the mapping $\widehat{\mathfrak{P}}:\text{dom}\big(\widehat{\mathfrak{P}}\big) \subset \mathbb{P}_2(\mathbb{R}) \rightarrow \mathbb{P}_2(\mathbb{R})$ describes how evolve the planes $\mathcal{P}(k)$, defined as $\text{Span}\big[p(\widetilde{t}_k),q(\widetilde{t}_k)\big]$ during the $\mathfrak{b}$-to-$\mathfrak{b}$ transformation. In this section, we will prove that the mapping $\widehat{\mathfrak{P}}$ can be defined as the action of a certain piecewise linear mapping $\mathfrak{P}:\mathbb{R}^3 \rightarrow \mathbb{R}^3$ on a subset of $\mathbb{P}_2(\mathbb{R})$, or equivalently, as the action of $\mathfrak{P}$ on a subset of the unit sphere $\mathbb{S}^2$ which remains invariant under the action of $\mathfrak{P}$. In other words, we will prove that $\widehat{\mathfrak{P}}$ is a piecewise projective linear transformation.

\paragraph{The domain of $\widehat{\mathfrak{P}}$.} The introduction to the construction of $\widehat{\mathfrak{P}}$, presented in the previous section, was mainly formal and illustrative. We discuss now precisely the question of the domain of $\widehat{\mathfrak{P}}$, which is essential in order to prevent triple collisions, which would forbid to define the evolution of the particle system.\\ 
\noindent
Since the mapping $\widehat{\mathfrak{P}}$ describes the $\mathfrak{b}$-to-$\mathfrak{b}$ evolution of the particle system, by assumption the central pair \footnotesize{\circled{2}}\normalsize{}-\footnotesize{\circled{3}}\normalsize{} is initially in contact, in a post-collisional configuration, that is, a collision of type $\mathfrak{b}$ just took place. Therefore, the initial configuration $\big(p_0,q_0)$ has the form:
\begin{align}
p_0 = \hspace{0.5mm}^t \hspace{-0.25mm}(p_{0,1},0,p_{0,3}), \hspace{5mm} q_0 = \hspace{0.5mm}^t \hspace{-0.25mm}(q_{0,1},q_{0,2},q_{0,3}) \hspace{5mm} \text{with} \hspace{2mm} p_{0,1}, p_{0,3} > 0 \hspace{2mm} \text{and} \hspace{2mm} q_{0,2} > 0,
\end{align}
the sign of the inequality on $q_{0,2}$ being strict, because if not, the pair \footnotesize{\circled{2}}\normalsize{}-\footnotesize{\circled{3}}\normalsize{} would not separate, yielding an ill-posed trajectory when the next collision occurs (triple collision).\\
We denote by $u= \hspace{0.25mm}^t \hspace{-0.25mm} (u_1, u_2, u_3)$ any normal vector to the plane $\mathcal{P}(0) = \text{Span}\big[p_0,q_0\big]$. We will define $\mathfrak{P}$ as a mapping acting on the vector of the initial configuration $u(0)$.\\
Since in this section we will mostly work in $\mathbb{R}^3$, we will denote by $x,y,z$ the three generic coordinates of the elements of $\mathbb{R}^3$, so that the set $\{y=0\}$ has to be understood as the hyperplane of the vectors with a zero second component.\\
\newline
Since by definition $u$ is normal to $\mathcal{P}(0)$, we have that $u$ is colinear to:
\begin{align}
p_0 \wedge q_0 = \begin{pmatrix} p_{0,1} \\ 0 \\ p_{0,3} \end{pmatrix} \wedge \begin{pmatrix} q_{0,1} \\ q_{0,2} \\ q_{0,3} \end{pmatrix} = \begin{pmatrix} -p_{0,3}q_{0,2} \\ p_{0,3}q_{0,1} - p_{0,1}q_{0,3} \\ p_{0,1}q_{0,2} \end{pmatrix}.
\end{align}
Conversely, the vector $p_0$ of initial relative positions belongs to the intersection between the two planes $\mathcal{P}(0)$ and $\{y=0\}$. By definition, the first one is normal to $u$, while the second one is normal to $e_y = \hspace{0.25mm}^t \hspace{-0.25mm} (0,1,0)$. Therefore, the vectorial line corresponding to the intersection between these two planes has to be orthogonal to both $u$ and $e_y$, and so we deduce that $p_0$ is necessarily colinear to:
\begin{align}
u \wedge e_y = \begin{pmatrix} u_1 \\ u_2 \\ u_3 \end{pmatrix} \wedge \begin{pmatrix} 0 \\ 1 \\ 0 \end{pmatrix} = \begin{pmatrix} -u_3 \\ 0 \\ u_1 \end{pmatrix}.
\end{align}
We deduce therefore two constraints:
\begin{itemize}
\item to consider meaningful initial data, the additional conditions $u_1 \neq 0$, $u_3 \neq 0$ have to be considered (in the opposite case, the initial configuration describes a collision involving three or four particles),
\item since the position vector $p_0$ takes values in $\big(\mathbb{R}_+\big)^3$, we have to assume $u_1u_3 \leq 0$.
\end{itemize}

\begin{remar}
Observe that we do not need to assume $u_3 \leq 0$, $u_1 \geq 0$, because $u$ and $-u$ define the same initial plane $\mathcal{P}(0)$. Nevertheless, in order to have a better intuition on the action of the mapping and to stay close to the action on the original position and velocity variables, it is always easier to consider $u_3 \leq 0$, $u_1 \geq 0$.\\
Concerning the case of the triple collisions that we want to avoid, which induces the restriction $u_1 \neq 0$, $u_3 \neq 0$, we will see that we will be able to define anyway the mapping $\widehat{\mathfrak{P}}$ also in the ``pathological'' situations $u_1 = 0$ or $u_3 = 0$. This will have the advantage to extend the definition of $\widehat{\mathfrak{P}}$ on a compact subset of the sphere $\mathbb{S}^2$.
\end{remar}
\noindent
We define then the domain of the mapping $\widehat{\mathfrak{P}}$ as follows.

\begin{defin}[Domain of the $\mathfrak{b}$-to-$\mathfrak{b}$ mapping $\widehat{\mathfrak{P}}$]
We define the \emph{domain of the $\mathfrak{b}$-to-$\mathfrak{b}$ mapping $\widehat{\mathfrak{P}}$} as the subset, denoted by $X$, of the two-dimensional unit sphere $\mathbb{S}^2$ defined as:
\begin{align}
\label{EQUATDefinDomai_X_P_}
X = \{x^2+y^2+z^2 = 1\}\cap \Big[ \{x\geq 0, z\leq 0\} \cup \{x\leq 0, z\geq 0\} \Big].
\end{align}
\end{defin}

\paragraph{The underlying piecewise linear structure of $\widehat{\mathfrak{P}}$.} 
\noindent
We are now in a position to state the main result of this section, which is a complete description of the mapping $\widehat{\mathfrak{P}}$.

\begin{theor}
\label{THEOR__P__FukngLinea}
Let $X\subset \mathbb{R}^3$ be defined in \eqref{EQUATDefinDomai_X_P_}. Then the $\mathfrak{b}$-to-$\mathfrak{b}$ mapping $\widehat{\mathfrak{P}}:X \to \mathbb{S}^2$ introduced in Definition \ref{DEFIN_Map_b_2_b__P__}, defined on $X$, coincides almost everywhere with a piecewise projective linear transformation, whose associated piecewise linear transformation $\mathfrak{P}$ is given by:
\begin{align}
\label{EQUATSpherReduc_btobLineaVersi}
\mathfrak{P}\begin{pmatrix} x \\ y \\ z \end{pmatrix} = \left\{
\begin{array}{cl}
\hspace{-3mm}P_1 \begin{pmatrix} x \\ y \\ z \end{pmatrix} = \begin{pmatrix}
r (\alpha y - x) \\ \alpha (\alpha y - x) - ry + r\alpha z \\ r^2 z
\end{pmatrix} & \hspace{-3mm}\text{if } x > 0,\ z < 0,\ y > 0,\ \alpha y - x > 0, \\
& \\
\hspace{-1.5mm} P_2 \begin{pmatrix} x \\ y \\ z \end{pmatrix} =
\begin{pmatrix}
-r (\alpha y -x) \\ -\alpha (\alpha y - x) - \alpha (\alpha y -z) + ry \\ -r(\alpha y - z)
\end{pmatrix} & \hspace{-3mm}\text{if } x > 0,\ z < 0,\ y > 0,\ \alpha y - x < 0, \\
& \\
\hspace{-3mm}P_3 \begin{pmatrix} x \\ y \\ z \end{pmatrix} =
\begin{pmatrix}
r^2 x \\ \alpha (\alpha y - z) - ry + r \alpha x \\ r(\alpha y - z)
\end{pmatrix} & \hspace{-3mm}\text{if } x > 0,\ z < 0,\ y < 0,\ \alpha y - z < 0, \\
& \\
\hspace{-1.5mm} P_4 \begin{pmatrix} x \\ y \\ z \end{pmatrix} = \begin{pmatrix}
-r (\alpha y - x) \\ - \alpha (\alpha y - z) - \alpha (\alpha y - x) + ry \\ -r (\alpha y - z)
\end{pmatrix} & \hspace{-3mm}\text{if } x > 0,\ z < 0,\ y < 0,\ \alpha y - z > 0,
\end{array}
\right.
\end{align}
with the matrices $P_1$, $P_2$, $P_3$ and $P_4$ respectively defined as:
\begin{align}
\label{EQUATCollisionMatric_P_i_}
P_1= \begin{pmatrix}
                -r & r\alpha & 0 \\
				 -\alpha & \alpha^2- r & r\alpha \\
				 0 & 0  & r^2 \\
	 \end{pmatrix}, \hspace{2mm} 
P_2= \begin{pmatrix}
				 r & -r\alpha & 0 \\
				 \alpha & -2\alpha^2+ r & \alpha \\
				 0 & -r\alpha  & r \\
	 \end{pmatrix}, \hspace{2mm}
&P_3= \begin{pmatrix}
				 r^2 & 0 & 0 \\
				 r\alpha & \alpha^2- r & -\alpha \\
				 0 & r\alpha  & -r \\
	\end{pmatrix}, \hspace{2mm}
P_4=P_2.
\end{align}
The case when $x < 0, z > 0$ is treated by symmetry, defining $\mathfrak{P}(u) = -\mathfrak{P}\big(-u\big)$, where $\mathfrak{P}\big(-u\big)$ is defined in \eqref{EQUATCollisionMatric_P_i_}.\\
In other words, we have:
\begin{align}
\widehat{\mathfrak{P}}(u) = \frac{\mathfrak{P}(u)}{ \big\vert \mathfrak{P}(u) \big\vert }\hspace{3mm} \forall u = \hspace{1mm}^t\hspace{-0.25mm}(x,y,z) \in X \hspace{2mm} \text{such that} \hspace{2mm} x \neq 0,\, z \neq 0,\, \alpha y - x \neq 0 \text{  and  } \alpha y - z \neq 0.
\end{align}
\end{theor}

\noindent
In other words, Theorem \ref{THEOR__P__FukngLinea} establishes that there exists a piecewise linear mapping $\mathfrak{P}$ defined almost everywhere on $\Big[ \{x\geq 0, z\leq 0\} \cup \{x\leq 0, z\geq 0\} \Big] \subset \mathbb{R}^3$, which induces a mapping on $X \subset \mathbb{S}^2$ that coincides almost everywhere with the $\mathfrak{b}$-to-$\mathfrak{b}$ mapping $\widehat{\mathfrak{P}}$. In other words, seen as a mapping acting on an invariant subset of $\mathbb{P}_2(\mathbb{R})$, the $\mathfrak{b}$-to-$\mathfrak{b}$ mapping $\widehat{\mathfrak{P}}$ coincides almost everywhere with a piecewise projective linear mapping.

\begin{proof}
Given the symmetries of the system, we focus on the configurations for which $x > 0$, $z < 0$, and $y > 0$. The case $y < 0$ is deduced by considering the transformation $\mathbb{R}^3 \supset X \ni \hspace{0.5mm}^t\hspace{-0.25mm} (x,y,z) \rightarrow \hspace{0.5mm}^t\hspace{-0.25mm} (-z,-y,-x)$.\\
\newline
We assume first that $y>0$ and $\alpha y-x>0$, where $\alpha$ is defined in \eqref{EQUATDefinAlpha}. The meaning of the condition on $\alpha y -x$ will become clear after the computation of the first collision. We denote by $u(0) = \hspace{0.5mm}^t\hspace{-0.25mm} (x,y,z) \in X$ the vector normal to the plane $\mathcal{P}(0)$ describing the initial configuration of the system. After the initial configuration, corresponding to the case when a collision of type $\mathfrak{b}$ just took place, the condition $y > 0$ implies that the next collision is of type $\mathfrak{a}$. Indeed, the vector of the initial positions $p(0)$ is colinear to:
\begin{align}
u(0)\wedge e_y=\begin{pmatrix} -z\\ 0\\x\end{pmatrix}
\end{align}
(observe that the assumptions $x,-z > 0$ are such that the entries of $u(0)\wedge e_y$ are positive, in agreement with the fact that this vector represents relative distances between the particles). The vector $q(0)$ of the initial relative velocities normal to $u(0)$ is then:
\begin{align}
q(0)= p(0) \wedge u(0) = \begin{pmatrix} -xy \\ x^2+z^2 \\ -yz \end{pmatrix}
\end{align} 
(observe that the second entry of $q(0)$ is always positive, implying that the pair \footnotesize{\circled{2}}\normalsize{}-\footnotesize{\circled{3}}\normalsize{} is separating). We prove now the following statement. When evolving on the intersection between the unit sphere and the plane $\mathcal{P}(0)$ normal to $u(0)$, the first plane that is intersected between $\{x = 0\}$ and $\{z = 0\}$ is determined by the scalar product between $q(0)$ and the vector $v(0) = u(0) - \big(u(0)\cdot e_y\big)e_y =  \hspace{0.5mm}^t\hspace{-0.25mm} (x,0,z)$.\\
To see that the next collision is determined by the sign of $q(0)\cdot v(0)$, for $\lambda > 0$ we consider the line $t \mapsto p(0) + \big(q(0) - \lambda p(0)\big) t$. The term $-\lambda p(0)$ ensures that the velocity reaches at least one of the two planes $\{x = 0\}$ or $\{z = 0\}$ in finite positive time (and the choice of $\lambda$ plays no role in determining the next collisions that take place, according to Theorem 4.1 in \cite{DoHR025}). The line intersects the planes $\{x = 0\}$ and $\{z = 0\}$, respectively at times $t_x$ and $t_z$ given by:
\begin{align}
t_x = \frac{-z}{xy + \lambda(-z)}, \hspace{5mm} t_z = \frac{x}{yz+\lambda x},
\end{align}
so that in the case when $y > 0$, $t_x$ is always positive, while $t_z$ might have both signs. More precisely, if $\lambda < y(-z)/x$, $t_z < 0$, and $t_z > 0$ if $\lambda > y(-z)/x$. For $\lambda > 0$ small enough, relying on Theorem 4.1 in \cite{DoHR025}, we clearly see that only collision takes place in the future, which has to be of type $\mathfrak{a}$. To see that the first collision to take place is necessarily of this type, without relying on the result of \cite{DoHR025}, we compute:
\begin{align}
t_x - t_z = \frac{y (x^2 + z^2)}{- x(-z) \lambda^2 + \big[yz^2-x^2y\big] \lambda + xy^2(-z)} \cdotp
\end{align}
The denominator is a quadratic polynomial in $\lambda$, whose discriminant is the perfect square $y^2\big(x^2+z^2\big)^2$. Observing in addition that this polynomial vanishes at $\lambda = y(-z)/x$ and that its derivative is negative at this point, we deduce that the polynomial has no root larger than $y(-z)/x$, proving that $t_x < t_z$ in the case when $t_z$ is positive. In summary, if $y > 0$, then $0 < t_x < t_z$ or $t_z < 0 < t_x$, and conversely if $y < 0$, then $0 < t_z < t_x$ or $t_x < 0 < t_z$. In other words, the next collision to take place is of type $\mathfrak{a}$ when $y > 0$, and of type $\mathfrak{c}$ when $y < 0$.\\
Besides, since we have:
\begin{align}
q(0) \cdot v(0) = -y(x^2+z^2),
\end{align}
we deduce that a collision of type $\mathfrak{c}$ follows if $q(0)\cdot v(0) > 0$, and $\mathfrak{a}$ follows if $q(0)\cdot v(0) < 0$.\\
\newline
We compute then the normal vector $u_\mathfrak{a}(1)$ to the plane $\mathcal{P}_\mathfrak{a}(1) = \text{Span}\big[p_\mathfrak{a}(1),q_\mathfrak{a}(1)\big]$ when the first collision takes place. It is colinear to:
\[
q_{\mathfrak{a}}(1)\wedge p_{\mathfrak{a}}(1)= \underbrace{Aq(0)}_{=q_{\mathfrak{a}}(1)} \wedge \underbrace{[e_x \wedge u(0)]}_{=p_{\mathfrak{a}}(1)},
\] 
where $A$ is the first matrix given in \eqref{EQUATCollision_Matri}. We have:
\begin{align}
p_\mathfrak{a}(1) = e_x \wedge u(0) =\begin{pmatrix} 0 \\-z\\y \end{pmatrix}
\hspace{5mm} \text{and} \hspace{5mm}
q_\mathfrak{a}(1) = Aq(0)= \begin{pmatrix}
			-r & 0 & 0 \\
			\alpha & 1 & 0 \\
			0& 0 & 1 \\
		\end{pmatrix}   \begin{pmatrix} -xy \\ x^2+z^2 \\ -yz \end{pmatrix}= \begin{pmatrix} rxy \\ -\alpha xy + x^2 + z^2 \\ -yz \end{pmatrix}.
\end{align}
Hence,
\begin{align}
q_{\mathfrak{a}}(1)\wedge p_{\mathfrak{a}}(1) = \begin{pmatrix} rxy \\ -\alpha xy + x^2 + z^2 \\ -yz \end{pmatrix} \wedge \begin{pmatrix} 0 \\-z\\y \end{pmatrix} = \begin{pmatrix} -\alpha xy^2+x^2y \\ -rxy^2 \\ -rxyz\end{pmatrix}
\end{align}
which is colinear to the following vector, that we denote by $u_\mathfrak{a}(1)$:
\begin{equation}
\label{EQUATLinea__P__u_a__}
u_\mathfrak{a}(1) = \begin{pmatrix} x-\alpha y \\ -ry \\ -rz \end{pmatrix}.
\end{equation}
Turning now to the second collision, we rely on the same arguments we used for the first collision, and we determine its type by computing the scalar product $q(1)\cdot v(1) = \big(p(1)\wedge u_\mathfrak{a}(1)\big) \cdot v(1)$, with $q(1) = p(1) \wedge u_\mathfrak{a}(1)$, $p(1) = u_\mathfrak{a}(1) \wedge e_x$ and $v(1) = u_\mathfrak{a}(1) - \big( u_\mathfrak{a}(1) \cdot e_x \big)e_x$. If we have $q(1) \cdot v(1) > 0$, then the second collision is of type $\mathfrak{b}$, and if $q(1) \cdot v(1) < 0$, the second collision is of type $\mathfrak{c}$. We have:
\begin{align}
\label{EQUATCompu_q(1)}
q(1)\cdot v(1) = \big(p(1)\wedge u_\mathfrak{a}(1)\big) \cdot v(1) = \Big( \begin{pmatrix} 0 \\ -rz \\ ry \end{pmatrix} \wedge \begin{pmatrix} x-\alpha y \\ -ry \\ -rz \end{pmatrix} \Big) \cdot \begin{pmatrix} 0 \\ -ry \\ -rz \end{pmatrix} = r^2 (y^2+z^2) (\alpha y - x).
\end{align}
Therefore, the assumption $\alpha y - x > 0$ implies that the second collision is of type $\mathfrak{b}$.\\
We compute now the normal vector $u_\mathfrak{ab}(2)$ to the plane $\mathcal{P}_\mathfrak{ab}(2) = \text{Span}\big[ p_\mathfrak{ab}(2),q_\mathfrak{ab}(2) \big]$ associated to the configuration of the system after this second collision. $u_\mathfrak{ab}(2)$ is colinear, and has the same orientation as $q_\mathfrak{ab}(2) \wedge p_\mathfrak{ab}(2)$, where:
\begin{align}
p_\mathfrak{ab}(2) = e_y \wedge u_\mathfrak{a}(1) = \begin{pmatrix} 0 \\ 1 \\ 0 \end{pmatrix} \wedge \begin{pmatrix} x-\alpha y \\ -ry \\ -rz \end{pmatrix} = \begin{pmatrix} -rz \\ 0 \\ \alpha y - x \end{pmatrix}
\end{align}
and
\begin{align}
q_\mathfrak{ab}(2) = B q(1) = \begin{pmatrix} 1 & \alpha & 0 \\ 0 & -r & 0 \\ 0 & \alpha & 1 \end{pmatrix} \begin{pmatrix} r^2(y^2+z^2) \\ ry(x - \alpha y) \\ rz (x- \alpha y) \end{pmatrix} = \begin{pmatrix} r^2(y^2+z^2) + r\alpha y(x - \alpha y) \\ -r^2 y (x -\alpha y) \\ r(\alpha y + z)(x - \alpha y) \end{pmatrix}.
\end{align}
We find then:
\begin{align}
q_\mathfrak{ab}(2) \wedge p_\mathfrak{ab}(2) &= \begin{pmatrix} r^2(y^2+z^2) + r\alpha y(x - \alpha y) \\ -r^2 y (x -\alpha y) \\ r(\alpha y + z)(x - \alpha y) \end{pmatrix} \wedge \begin{pmatrix} -rz \\ 0 \\ \alpha y - x \end{pmatrix} \nonumber\\
&= \begin{pmatrix}
r^2 y (\alpha y - x)^2 \\
-r^2 z (\alpha y + z)(x-\alpha y) - r^2(y^2+z^2)(\alpha y - x) + r\alpha y (\alpha y - x)^2 \\
-r^3 y z (x - \alpha y)
\end{pmatrix},
\end{align}
which is a vector colinear and with the same orientation as:
\begin{align}
\begin{pmatrix}
ry(\alpha y - x) \\
rz(\alpha y + z) - r(y^2+z^2) + \alpha y(\alpha y - x) \\
r^2 yz
\end{pmatrix}.
\end{align}
Observing that $rz(\alpha y + z) - r(y^2 + z^2) = r\alpha yz - ry^2 = ry(\alpha z - y)$, we can simplify even more and obtain that $q_\mathfrak{ab}(2) \wedge p_\mathfrak{ab}(2)$ is colinear and has the same orientation as the following vector, that we denote by $u_\frak{ab}(2)$:
\begin{align}
\label{EQUATLinea__P__u_ab_}
u_\frak{ab}(2) =
\begin{pmatrix}
r(\alpha y - x) \\
r(\alpha z - y) + \alpha (\alpha y - x) \\
r^2 z
\end{pmatrix}.
\end{align}
The expression of \eqref{EQUATLinea__P__u_ab_}, which describes the configuration of the particle system after the first collision of type $\mathfrak{b}$, corresponds then to the expression of $\mathfrak{P}\big( \hspace{0.5mm}^t\hspace{-0.5mm}(x,y,z)\big)$, on the subset $y > 0$, $\alpha y - x > 0$. Therefore $\widehat{\mathfrak{P}}$, restricted on this subset, coincides with a projective transformation.\\
\newline
We turn now to the case when $y > 0$, $\alpha y - x < 0$. The first assumption implies that the first collision is of type $\mathfrak{a}$, already studied in the previous case, so that we restart from the normal vector $u_\mathfrak{a}(1)$ to the plane $\mathcal{P}_\mathfrak{a} = \text{Span}\big[p_\mathfrak{a}(1),q_\mathfrak{a}(1)\big]$, given by the expression \eqref{EQUATLinea__P__u_a__}. The assumption $\alpha y - x < 0$ implies this time that the second collision is of type $\mathfrak{c}$. The configuration of the system right after this second collision is given by the following position and velocity vectors:
\begin{align}
p_\mathfrak{ac}(2) = e_z \wedge u_\mathfrak{a}(1) = \begin{pmatrix} 0 \\ 0 \\ 1 \end{pmatrix} \wedge \begin{pmatrix} x-\alpha y \\ -ry \\ -rz \end{pmatrix} = \begin{pmatrix} ry \\ x-\alpha y \\ 0 \end{pmatrix}
\end{align}
and $q_\mathfrak{ac}(2) = Cq(1)$ where $q(1) = p(1) \wedge u_\mathfrak{a}(1)$ was computed in \eqref{EQUATCompu_q(1)}, so that:
\begin{align}
q_\mathfrak{ac}(2) =
\begin{pmatrix}
1 & 0 & 0 \\ 0 & 1 & \alpha \\ 0 & 0 & -r
\end{pmatrix}
\begin{pmatrix}
r^2(y^2+z^2) \\ ry(x-\alpha y) \\ rz(x-\alpha y)
\end{pmatrix}
= \begin{pmatrix}
r^2(y^2+z^2) \\ ry(x-\alpha y) + r\alpha z(x-\alpha y) \\ -r^2 z(x-\alpha y)
\end{pmatrix}
\end{align}
so that:
\begin{align}
q_\mathfrak{ac}(2) \wedge p_\mathfrak{ac}(2) &= \begin{pmatrix}
r^2(y^2+z^2) \\ ry(x-\alpha y) + r\alpha z(x-\alpha y) \\ -r^2 z(x-\alpha y)
\end{pmatrix}
\wedge
\begin{pmatrix} ry \\ x-\alpha y \\ 0 \end{pmatrix} \nonumber\\
&= r\begin{pmatrix}
rz(x-\alpha y)^2 \\ -r^2yz(x-\alpha y) \\ r(y^2+z^2)(x-\alpha y) - ry^2(x-\alpha y) - r\alpha y z(x-\alpha y)
\end{pmatrix},
\end{align}
which is colinear, and has the same direction (keeping in mind that $z < 0$) as the following vector, that we denote by $u_\mathfrak{ac}(2)$:
\begin{align}
u_\mathfrak{ac}(2) = \begin{pmatrix} \alpha y - x \\ r y \\ \alpha y - z \end{pmatrix}.
\end{align}
We compute now the third collision. We pose:
\begin{align}
p(2) = u_\mathfrak{ac}(2) \wedge e_z = \begin{pmatrix} ry \\ x-\alpha y \\ 0 \end{pmatrix}, \hspace{3mm} q(2) = p(2) \wedge u_\mathfrak{ac}(2) = \begin{pmatrix} (x-\alpha y)(\alpha y - z) \\ -ry(\alpha y - z) \\ r^2y^2 + (x-\alpha y)^2 \end{pmatrix},
\end{align}
and we deduce that the third collision, following $\mathfrak{a}$ and $\mathfrak{c}$, is necessarily of type $\mathfrak{b}$ by computing $q(2) \cdot v(2) < 0$ with $v(2) = u_\mathfrak{ac}(2) - \big(u_\mathfrak{ac} \cdot e_z \big)e_z$. We recover here the result of Proposition 4.5 in \cite{DoHR025}. The configuration of the system after this third collision is given by the two vectors:
\begin{align}
q_\mathfrak{acb}(3) = Bq(2)
&= \begin{pmatrix}
1 & \alpha & 0 \\
0 & -r & 0 \\
0 & \alpha & 1 \\
\end{pmatrix}
\begin{pmatrix} (x-\alpha y)(\alpha y - z) \\ -ry(\alpha y - z) \\ r^2y^2 + (x-\alpha y)^2 \end{pmatrix} \nonumber\\
&= \begin{pmatrix}
\big(x-(r+1)\alpha y\big)(\alpha y -z) \\
r^2 y (\alpha y - z) \\
-r\alpha y (\alpha y - z) + r^2y^2 + (x - \alpha y)^2
\end{pmatrix}
\end{align}
and
\begin{align} 
p_{acb}(3) = e_y \wedge u_{\mathfrak{ac}}(2) = \begin{pmatrix}
\alpha y - z \\ 0 \\ x - \alpha y
\end{pmatrix}.
\end{align}
Therefore, the plane spanned by $p_\mathfrak{acb}(3)$ and $q_\mathfrak{acb}(3)$ is normal to:
\begin{align}
q_\mathfrak{acb}(3) \wedge p_\mathfrak{acb}(3) &= \begin{pmatrix}
\big(x-(r+1)\alpha y\big)(\alpha y -z) \\
r^2 y (\alpha y - z) \\
-r\alpha y (\alpha y - z) + r^2y^2 + (x - \alpha y)^2
\end{pmatrix}
\wedge
\begin{pmatrix}
\alpha y - z \\ 0 \\ x - \alpha y
\end{pmatrix} \nonumber\\
&= \begin{pmatrix}
r^2 y (x-\alpha y)(\alpha y - z) \\
-r\alpha y(\alpha y - z)^2 + r^2y^2(\alpha y - z) + r\alpha y(\alpha y - z)(x - \alpha y) \\
-r^2 y(\alpha y - z)^2
\end{pmatrix}
\end{align}
which is colinear and has the same orientation of the following vector, that we denote by $u_\mathfrak{acb}(3)$:
\begin{align}
u_{\mathfrak{acb}}(3)= \begin{pmatrix} r(x-\alpha y) \\ -\alpha (\alpha y - z) + ry + \alpha (x - \alpha y) \\ -r(\alpha y- z) \end{pmatrix}.
\end{align}
The configuration described by $u_{\mathfrak{acb}}(3)$ represents the state of the system right after the first collision of type $\mathfrak{b}$ that follows the initial configuration, in the case when $y > 0$ and $\alpha y - x < 0$. We have therefore $\widehat{\mathfrak{P}}\big(\hspace{0.25mm}^t\hspace{-0.25mm}(x,y,z)\big) = u_\mathfrak{acb}(3) / \vert u_\mathfrak{acb}(3) \vert$ on this subset, which proves that the restriction of $\widehat{\mathfrak{P}}$ to $y > 0$, $\alpha y - x < 0$ is also a projective transformation.\\
The cases when $y < 0$ being treated by symmetry, the proof of Theorem \ref{THEOR__P__FukngLinea} is complete.
\end{proof}

\begin{remar}
The mapping $\mathfrak{P}$ is a piecewise linear endomorphism of $\mathbb{R}^3$, restricted to the quadrant $x \geq 0$, $z \leq 0$. This linear structure is consistent with what this mapping describes: the argument $\hspace{0.5mm}^t\hspace{-0.5mm}(x,y,z)$ has to be seen as a normal vector to the plane $\mathcal{P}(0)$ spanned by the position and velocity vectors of the configuration of the particle system, and $\widehat{\mathfrak{P}}$ describes how the orientation of the plane $\mathcal{P}(0)$ evolves. Therefore, only the orientation of the normal vector matters, so that any non zero scalar (positive or negative) multiplying the normal vector provides the same plane.
\end{remar}


\begin{remar}
We remark that the region $x \geq 0$, $z \leq 0$ is invariant under the action of $\mathfrak{P}$. Indeed, if we consider the first expression in \eqref{EQUATSpherReduc_btobLineaVersi}, $\alpha y - x$ is assumed to be positive in this region, so that the first component of $\mathfrak{P}\big( \hspace{0.25mm}^t\hspace{-0.25mm}(x,y,z) \big)$ is positive, while the third component is negative as is $z$ by assumption.\\
In the case of the second expression in \eqref{EQUATSpherReduc_btobLineaVersi}, the conclusion concerning the first component is direct. As for the third component, $-r(\alpha y - z)$ is the sum of two negative numbers. The third and the fourth expressions have the same property.
\end{remar}

\begin{remar}
The mapping $\mathfrak{P}$ is defined almost everywhere on the quadrant $\{x \geq 0, z \leq 0\}$, namely, for every $^t \hspace{-0.25mm}(x,y,z) \in \{x \geq 0, z \leq 0\}$ except if $x = 0$, or if  $z = 0$, or if $\alpha y - x = 0$, or if $\alpha y - z = 0$. Nevertheless, we can extend the definition of $\mathfrak{P}$ to the whole quadrant $\{x \geq 0, z \leq 0\}$, for instance by assuming that $\mathfrak{P}(u) = P_2u$ if $\alpha y - x = 0$ or $\alpha y - z = 0$. This extension of the domain to the whole quadrant is arbitrary, and there is no such extension which would correspond to a continuous extension of $\mathfrak{P}$ on the whole quadrant.
\end{remar}

\section{First numerical simulations of the orbits of $\widehat{\mathfrak{P}}_r$}
\label{SECTIFirstNumerSimul}

\noindent
In order to emphasize the dependence of the $\mathfrak{b}$-to-$\mathfrak{b}$ mapping $\widehat{\mathfrak{P}}$ on the restitution coefficient $r$, we will denote the mapping also by $\widehat{\mathfrak{P}}_r$. The piecewise linear expression \eqref{EQUATSpherReduc_btobLineaVersi} of $\mathfrak{P}$ allows to perform efficient numerical simulations. In this section, we discuss a first collection of plots of orbits of the dynamical system $\widehat{\mathfrak{P}}_r$.
The code that is used to construct the orbits is provided in Appendix \ref{SSECTAppenAlgoLinea} for the sake of completeness.

\paragraph{The initial configurations and the representation of the orbits.} For each of the $5000$ values of the restitution coefficient $r$, evenly spaced between two limiting values $r_\text{min}$ and $r_\text{max}$, we compute $32$ different trajectories $\big(X_i(\widetilde{t}_k,r)\big)_k = \big(\widehat{\mathfrak{P}}_r^k(X_{i,0}) \big)_k$ of the mapping $\widehat{\mathfrak{P}}_r$ (that depends on $r$), obtained from $32$ different initial configurations $X_{i,0}$ ($1 \leq i \leq 32$), of the form:
\begin{align}
X_{i,0} = \begin{pmatrix} 1 \\ -\frac{1+8}{2+g} \\ -0.1-h \end{pmatrix}
\end{align}
with $1 \leq g \leq 8$ and $1 \leq h \leq 4$. The grid is chosen such that the initial data satisfy $\big( X_{i,0} \big)_y < 0$, ensuring that the first sequence of collisions to follow are either $\mathfrak{cb}$ or $\mathfrak{cab}$. By symmetry, there is indeed no need to consider initial data such that $\big( X_{i,0} \big)_y > 0$. For each of these trajectories, we compute $5000$ iterations ($1 \leq k \leq 5000$), and we plot the configurations associated to the last $100$ iterations. Each iteration of each trajectory is represented with a single real number $\theta_k = \theta_k(r,i)$, which corresponds to the angle parametrizing the normalized position $p(\widetilde{t}_k) = \big(\sin(\theta_k),0,\cos(\theta_k)\big)$ (we recall that the collision times $\widetilde{t}_k$ are associated to collisions of type $\mathfrak{b}$, when the pair of particles \footnotesize{\circled{2}}\normalsize{}-\footnotesize{\circled{3}}\normalsize{} collides), so that:
\begin{align}
\tan(\theta_k) = \frac{p_1(\widetilde{t}_k)}{p_3(\widetilde{t}_k)} = - \frac{z_i(\widetilde{t}_k,r)}{x_i(\widetilde{t}_k,r)}
\end{align}
where we denoted $X_i(\widetilde{t}_k,r) = \hspace{0.5mm}^t\hspace{-0.25mm}\big(x_i(\widetilde{t}_k,r),y_i(\widetilde{t}_k,r),z_i(\widetilde{t}_k,r)\big)$ and where we used that $p(\widetilde{t}_k)$ is colinear to the intersection between $X_i(\widetilde{t}_k,r)^\perp$ and $\{y = 0\}$.\\
\paragraph{First simulation, with $0.0717 \leq r \leq 0.1717$.} In the case when $r_\text{min} = 0.0717$ and $r_\text{max} = 0.1717$ (so that the interval between two consecutive restitution coefficients that are plotted is $\Delta_r = 2 \times 10^{-5}$), we obtain Figure \ref{FIGURPlTy10.0717_____<r<_____0.1717}. The choice of the boundary values $r_\text{min} = 0.0717$, $r_\text{max} = 0.1717$ is motivated by the facts that, on the one hand, the largest restitution coefficient for which the collapse can take place in systems of three particles is $7 - 4 \sqrt{3} \simeq 0.0717\ 9676\ 9724$ \cite{CoGM995}, and on the other hand, the largest restitution coefficient for which a periodic sequence of collisions leading to the collapse, in a stable manner, that was found so far is $3-2\sqrt{2} \simeq 0.1715\ 7287\ 5253$ \cite{CDKK999}.

\begin{figure}[h!]
\centering
\includegraphics[trim = 0cm 0cm 0cm 0cm, width=1\linewidth]{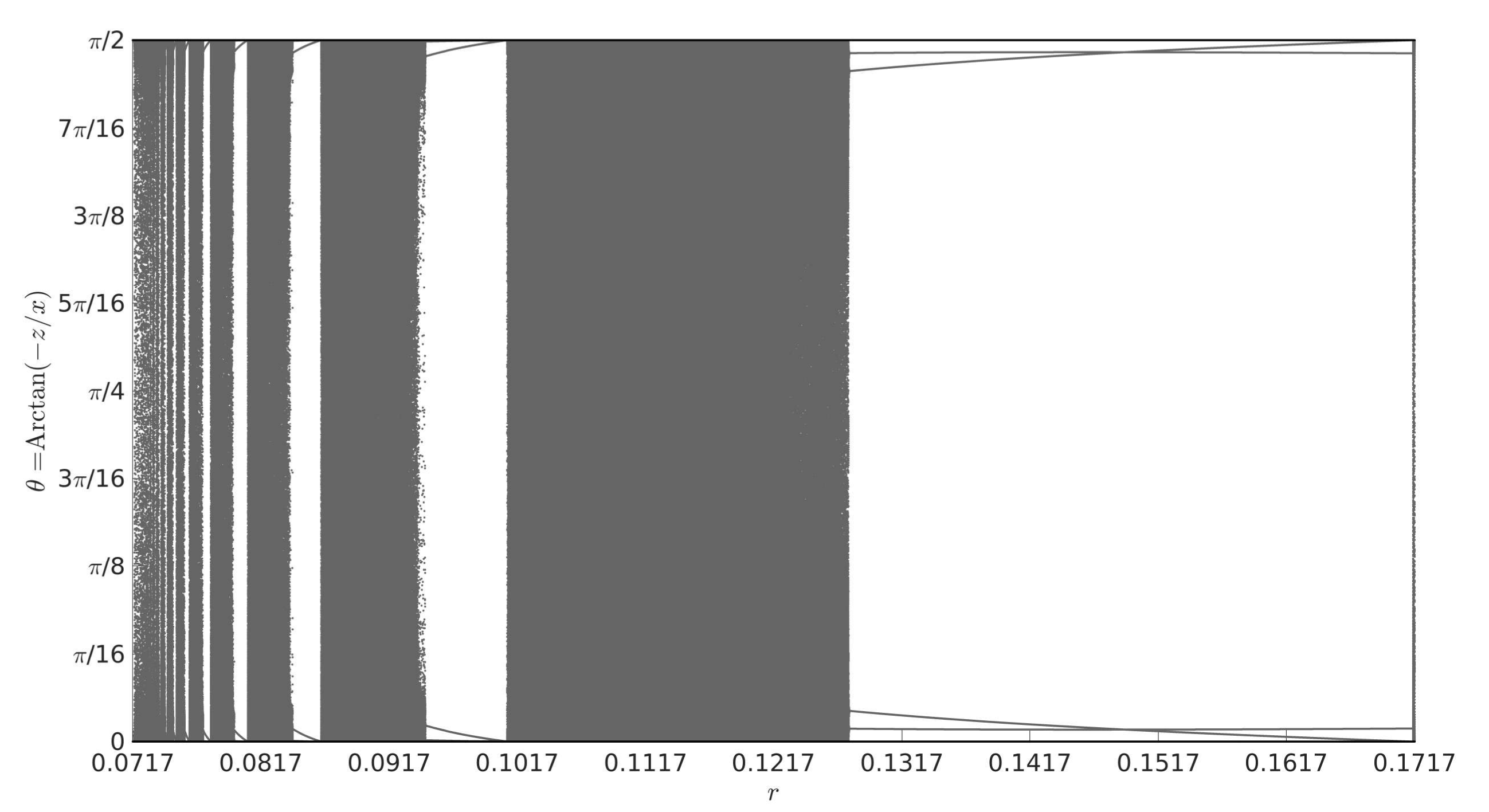} 
\caption{Plot of the tails (last $100$ iterations) of the $32$ orbits of $\widehat{\mathfrak{P}}_r$ for $0.0717 \leq r \leq 0.1717$ after $5000$ iterations, obtained with the algorithm of Appendix \ref{SSECTAppenAlgoLinea}.}
\label{FIGURPlTy10.0717_____<r<_____0.1717}
\end{figure}

\noindent
On the right of Figure \ref{FIGURPlTy10.0717_____<r<_____0.1717}, we observe an interval, approximately located at $0.12175 \leq r \leq 0.1717$ for which all the orbits accumulate on four particular values at $r$ fixed. This is consistent with the simulations already presented in \cite{DoHR025} (see for instance Figures 6 and 7 in this reference). This accumulation corresponds to the interval of existence and stability of the collapse that takes place according to the periodic sequence of collisions $\mathfrak{ababcbcb}$, which presents four collisions of type $\mathfrak{b}$ per period, hence the four accumulation points.\\
Observe that no other accumulation point appears for such initial conditions, which suggests that there is no other stable pattern that coexist in this interval.\\
The lower bound of the interval coincides with the only root in $[0,1]$ of the polynomial $5r^6 - 50r^5 + 107r^4 - 188r^3 + 107r^2 - 50r + 5$, approximately equal to $0.1275\ 4409\ 7592$, identified in \cite{CDKK999} (see also \cite{HuRo023} for more details) as the minimal value of $r$ providing the stability of the pattern $\mathfrak{ababcbcb}$, this minimal value corresponding to a loss of stability in the position variable.\\
Below, approximately for $0.1017 \leq r \leq 0.12175$, where the lower bound corresponds to $5-2\sqrt{6} \simeq 0.1010205144$ (see \cite{CDKK999}), we observe an interval where the tails of the orbits fill entirely the interval $[0,\pi/2]$ of possible values for the angle $\theta = \text{Arctan}(-z/x)$. In this range, we cannot discern orbits that accumulate around finite number of points. In other words, we observe an apparent chaos in this range. This apparent chaos was already described in \cite{CDKK999}, and observed later in \cite{DoHR025}.\\
For $r$ below $0.1017$, the pattern is repeated, alternating between windows of stability and chaos, that we can clearly discern on Figure \ref{FIGURPlTy10.0717_____<r<_____0.1717}, at least until the $8$-th window of stability. These windows of stability tend to accumulate to the value of $7-4\sqrt{3} \simeq 0.0718$. The other windows of stability are associated to other periodic patterns of the form $(\mathfrak{ab})^n(\mathfrak{cb})^n$, as we shall see with the next figure.

\paragraph{Simulation for $0.0717 \leq r \leq 0.1717$ in logarithmic scale.} In the windows of stability, the orbits accumulate around the extreme values ($\theta = 0$ and $\theta = \pi/2$). The results of the simulations obtained in Figure \ref{FIGURPlTy10.0717_____<r<_____0.1717} are symmetric with respect to the line $\theta = \pi/4$ (which is a consequence of the fact that by reverting the labels of the particles, the orbits are reflected with respect to the line $\theta = \pi/4$). We will therefore focus on the value $\theta = 0$, and we will replot the results presented in Figure \ref{FIGURPlTy10.0717_____<r<_____0.1717} in logarithmic scale for the $\theta$-variable. We obtain Figure \ref{FIGURPlTy10.0717_____<r<_____0.1717Log__}.

\begin{figure}[h!]
\centering
\includegraphics[trim = 0cm 0cm 0cm 0cm, width=1\linewidth]{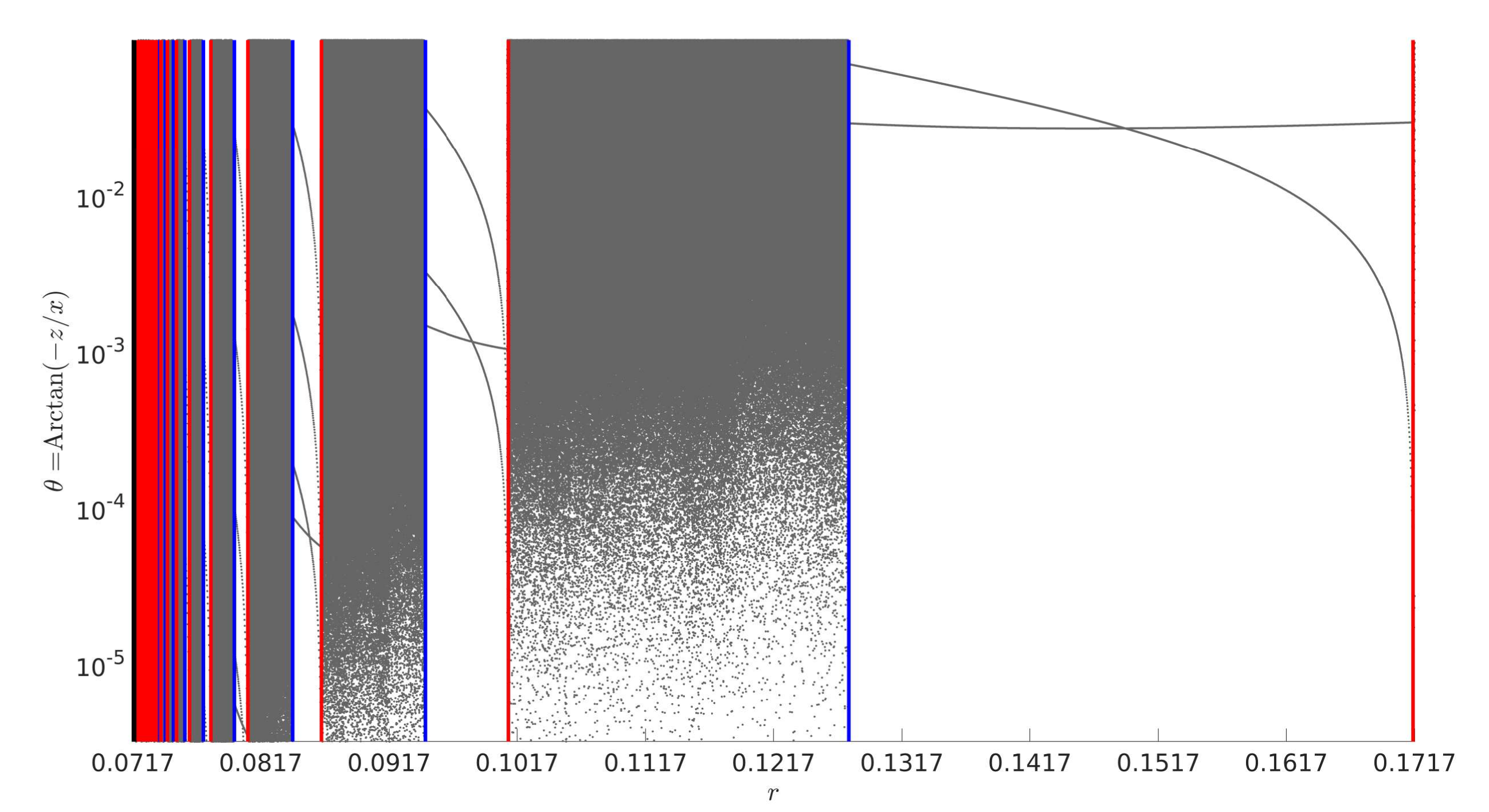} 
\caption{Plot of the tails (last $100$ iterations) of the $32$ orbits of $\widehat{\mathfrak{P}}_r$ for $0.0717 \leq r \leq 0.1717$ after $5000$ iterations, in logarithmic scale.}
\label{FIGURPlTy10.0717_____<r<_____0.1717Log__}
\end{figure}

\noindent
On Figure \ref{FIGURPlTy10.0717_____<r<_____0.1717Log__} we clearly observe that the orbits accumulate on two points in the first window of stability, then on three in the second window, and so on, up to the fourth window, where the orbits accumulates on five points. Therefore, the simulation strongly indicates that these windows of stabilities correspond indeed to the ranges of restitution coefficients associated to the stable patterns $(\mathfrak{ab})^n(\mathfrak{cb})^n$. In \cite{CDKK999}, the longest pattern that was reported corresponds to $n = 6$, that is, $(\mathfrak{ab})^6(\mathfrak{cb})^6$. Such a pattern is hard to identify on Figure \ref{FIGURPlTy10.0717_____<r<_____0.1717Log__}, but we will present below finer simulations to exhibit the $6$ accumulation points close to $\theta = 0$ associated to this pattern.\\
\paragraph{Lower and upper bounds of the windows of stability.} On Figure \ref{FIGURPlTy10.0717_____<r<_____0.1717Log__} are also represented the limiting values that enclose the different windows of stability. More precisely, it is rigorously proved in \cite{CDKK999} that if the collapse takes place in a stable manner for a certain interval of restitution coefficients (the window of stability) according to the pattern $(\mathfrak{ab})^n(\mathfrak{cb})^n$, and if for some restitution coefficient $r_\text{crit,1}$ in this interval the associated collision matrix $(BC)^n(BA)^n$ has a double eigenvalue, then necessarily $r_\text{crit,1}$ is one of the two extremities of such a window of stability. In other words, beyond this critical value $r_\text{crit,1}$, the collision pattern $(\mathfrak{ab})^n(\mathfrak{cb})^n$ does not take place in a stable manner, and so, it shouldn't be observed via numerical simulations. It is also found in \cite{CDKK999} that the critical restitution coefficients for which the respective characteristic polynomials of the collision matrices have double roots correspond to the lower bounds of the intervals of stability of the patterns $(\mathfrak{ab})^n(\mathfrak{cb})^n$. In addition, these critical values are roots of the polynomials (see \cite{CDKK999}):
\begin{align}
\label{EQUATPolynLowerBound}
Q_n(r) = r^{2n}P_n(r)P_n(1/r) - r^{2n}, 
\end{align}
where $P_n$ is the trace of the matrix $J(BA)^n$, with:
\begin{align}
\label{EQUATDefin__J__}
J =
\begin{pmatrix}
0 & 0 & 1 \\ 0 & 1 & 0 \\ 1 & 0 & 0
\end{pmatrix},
\end{align}
$P_n$ being a polynomial of degree $2n$. We provide in Appendix \ref{SSECTAppenPolynLowerBoundStabi} the expressions of the first polynomials $Q_n$ (of respective degrees $4n$), for $2 \leq n \leq 10$. The lower bounds of the windows of stability are represented in blue on Figure \ref{FIGURPlTy10.0717_____<r<_____0.1717Log__}, and for the sake of completeness we provide in Appendix \ref{SSECTAppenLowerUpperBoundWindoStabi} the list of the numerical approximations of the lower bounds of the first $100$ windows of stability (see Table \ref{TABLEDecimals_LowerUpperBounds}). To obtain these numerical approximations, the computations were performed using $512$ digits (due to the exponential growth in $n$ of the coefficients of the polynomials $Q_n$, requiring a very high precision arithmetic).\\
We emphasize that it is an open question to identify precisely the polynomials $Q_n$ and their respective roots, in terms of remarkable polynomials.\\
As for the upper bounds of the different windows of stability, their characterization remains an open question. In \cite{CDKK999}, it is conjectured that such an upper bound coincides with the critical restitution coefficient $r_\text{crit,2}$ for which the second entry of the eigenvectors of the reduced collision matrix $J(BA)^n$ is zero.\\
Here, we did not specify precisely the eigenvector whose second entry has to be zero, for the following reason. For any eigenvector of $J(BA)^n$, if the second entry of this eigenvector is zero, then the entry $(2,1)$ (second line, first column) of $J(BA)^n$ has also to be zero. Since it is shown in \cite{CDKK999} (Section 7) that, for any restitution coefficient $r$, such an entry is proportional to the $(2n-1)$-th Chebyshev polynomial of the second kind evaluated at $-\frac{1+r}{4\sqrt{r}}$, we deduce that $-\frac{1+r_\text{crit,2}}{4\sqrt{r_\text{crit,2}}}$ has to be a root of the $(2n-1)$-th Chebyshev polynomial of the second kind. Considering the smallest of these roots, which writes explicitly:
\begin{align}
\cos\Big( \frac{\pi}{2n} \Big),
\end{align}
we consider then the following expression for the critical restitution coefficient $r_\text{crit,2}$:
\begin{align}
\label{EQUATCritiRestiCoeffUpperWindo}
r_\text{crit,2}(n) = \Big( 2\cos\Big(\frac{\pi}{2n}\Big) - \sqrt{4\cos^2\Big(\frac{\pi}{2n}\Big) - 1} \ \Big)^2,
\end{align}
which constitutes one candidate to be the upper bound of the $n$-th window of stability (there are indeed several candidates, as the Chebyshev polynomials have several roots). In \cite{CDKK999}, it was conjectured that the upper bounds of the $n$-th windows of stability are given by the numbers $r_\text{crit,2}(n)$, and such a conjecture was supported by numerical simulations, for $2 \leq n \leq 6$ (we recall that the pattern $\mathfrak{abcb}$, obtained when $n=1$, is never stable, as it was proved in \cite{CDKK999}, Section 5). Figure \ref{FIGURPlTy10.0717_____<r<_____0.1717Log__} is in agreement with this conjecture and the simulations reported in \cite{CDKK999}. We provide in Appendix \ref{SSECTAppenLowerUpperBoundWindoStabi} (Table \ref{TABLEDecimals_LowerUpperBounds}) the first $12$ decimals of the numerical approximations of the values of $r_{\text{crit},2}(n)$, for $2 \leq n \leq 100$.\\
We will now see that the piecewise linear expression of $\mathfrak{P}$ allows us to support the conjecture, for $n$ much larger than what was established previously in the literature.

\paragraph{Simulations for $0.0717 \leq r \leq 0.0817$, $0.0717 \leq r \leq 0.0727$, $0.07179 \leq r \leq 0.07189$ and $0.07183 \leq r \leq 0.07184$ in logarithmic scale.} On Figure \ref{FIGURPlTy10.0717_____<r<_____0.1717Log__}, the fourth window, associated to $n = 5$, is already hard to see. On Figure \ref{FIGURPlTy10.0717_____<r<_____0.0817Log__} below, we consider a smaller range of restitution coefficients, namely $0.0717 \leq r \leq 0.0817$, with an interval $\Delta_r$ equal to $2 \times 10^{-6}$ between two consecutive restitution coefficients.

\begin{figure}[h!]
\centering
\includegraphics[trim = 0cm 0cm 0cm 0cm, width=1\linewidth]{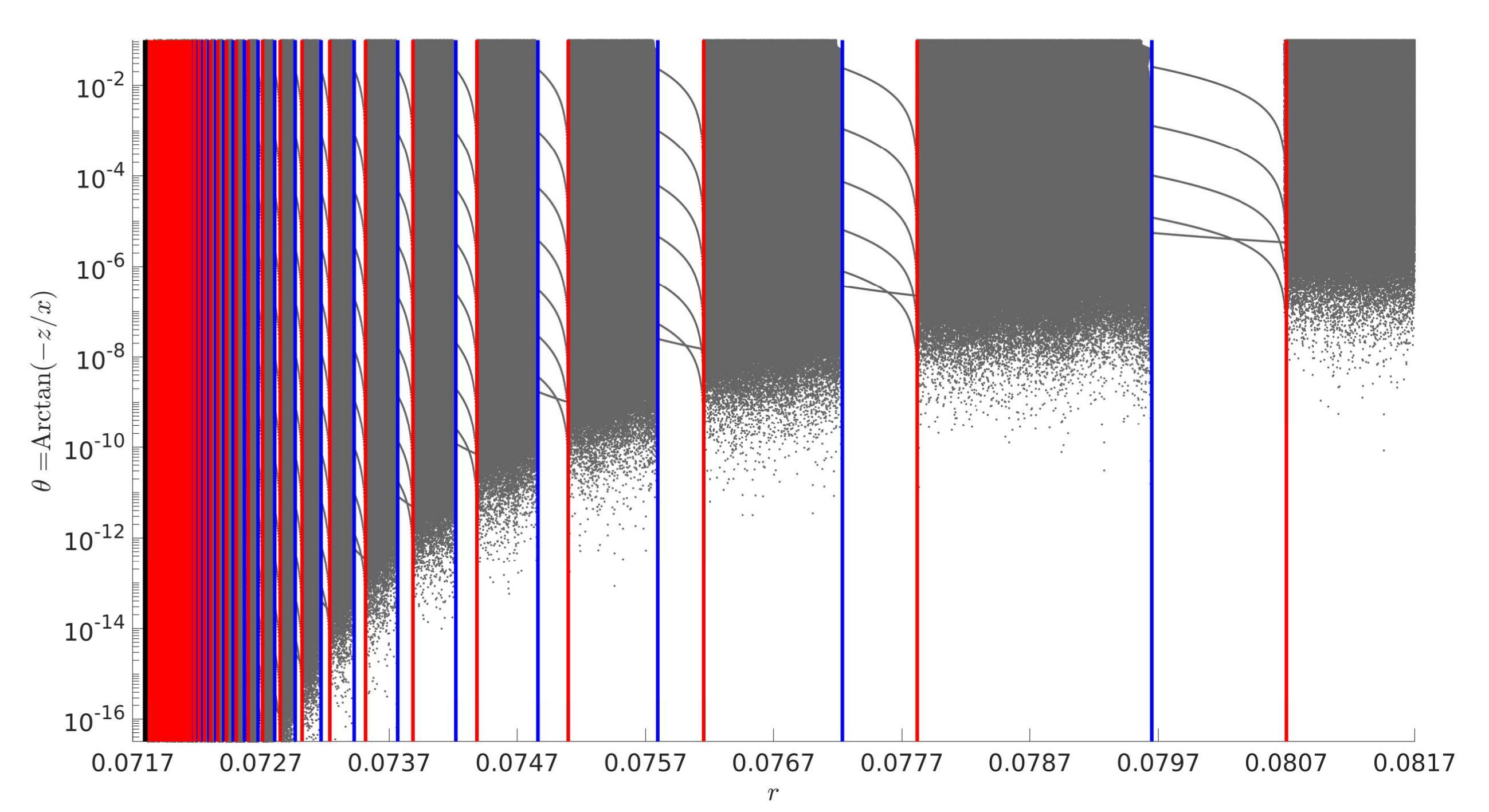} 
\caption{Plot of the tails of the orbits of $\widehat{\mathfrak{P}}_r$ for $0.0717 \leq r \leq 0.0817$ after $5000$ iterations, in logarithmic scale.}
\label{FIGURPlTy10.0717_____<r<_____0.0817Log__}
\end{figure}

\noindent
On Figure \ref{FIGURPlTy10.0717_____<r<_____0.0817Log__}, we clearly observe now the windows of stability associated to $n$, from $n = 5$ (the first from the right) to, at least, $n = 14$ (around $ r = 0.0727$). The piecewise linear expression of $\mathfrak{P}$ allows already to reach much larger values than $n = 6$. We observe also that, on the one hand, the lower bounds of the different windows of stability that we can see match perfectly with the theoretical bound given in \cite{CDKK999} as the roots of the polynomials \eqref{EQUATPolynLowerBound}, and on the other hand, the conjectured upper bound given by the expression \eqref{EQUATCritiRestiCoeffUpperWindo} is also an excellent match.\\

\begin{figure}[h!]
\centering
\includegraphics[trim = 0cm 0cm 0cm 0cm, width=1\linewidth]{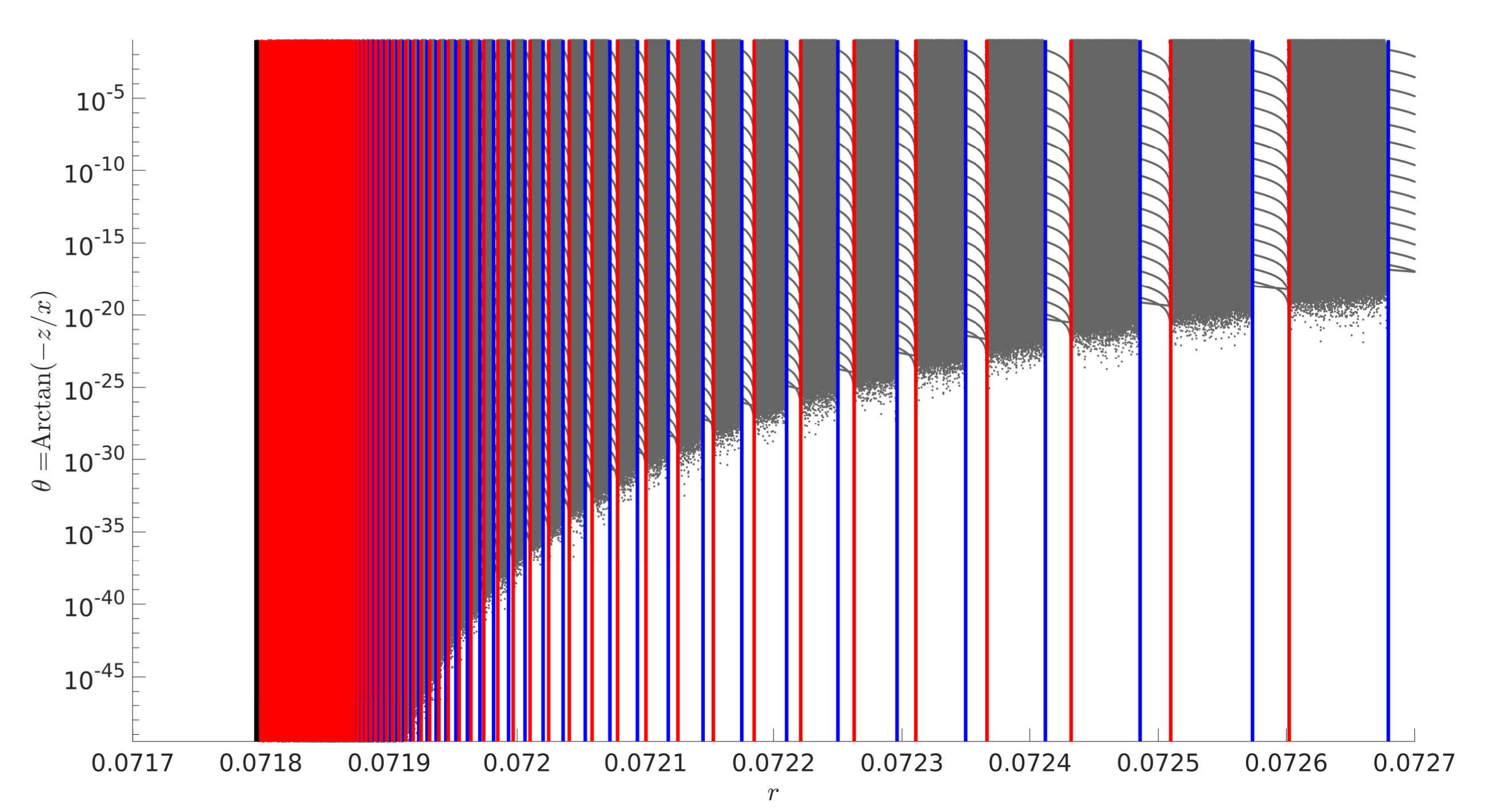} 
\caption{Plot of the tails of the orbits of $\widehat{\mathfrak{P}}_r$ for $0.0717 \leq r \leq 0.0727$ after $5000$ iterations, in logarithmic scale.}
\label{FIGURPlTy10.0717_____<r<_____0.0727Log__}
\end{figure}

\noindent
Repeating the numerical simulations between $0.0717$ and $ 0.0727$ ($\Delta_r = 2 \times 10^{-7}$), we obtain Figure \ref{FIGURPlTy10.0717_____<r<_____0.0727Log__}. On this Figure, the rightmost window of stability corresponds to $n = 15$, and we can observe the windows up to, at least, $n = 32$ (the first window below $r = 0.072$). We mentioned above that the computations were performed with $512$ digits in order to obtain the theoretical lower bounds of the windows of stability. If one computes them with a usual precision of $32$ digits, then beyond $n > 24$ the algorithm returns values that are not sharp enough. Once again, we observe on Figure \ref{FIGURPlTy10.0717_____<r<_____0.0727Log__} an excellent match with, on the one hand, the theoretical lower bounds of the windows represented in blue, and on the other hand with the upper bounds conjectured in \cite{CDKK999}.\\
\newline
Surprisingly, we can push the precision even more around the critical value $7 - 4\sqrt{3} \simeq 0.0718$. On Figures \ref{FIGURPlTy10.07179____<r<____0.07189Log__} and \ref{FIGURPlTy10.07183____<r<____0.07184Log__}, we take respectively $r_\text{min} = 0.07179$, $r_\text{max} = 0.07189$ ($\Delta_r = 2 \times 10^{-8}$) and $r_\text{min} = 0.07183$, $r_\text{max} = 0.07184$ ($\Delta_r = 2 \times 10^{-9}$).

\begin{figure}[h!]
\centering
\includegraphics[trim = 0cm 0cm 0cm 0cm, width=1\linewidth]{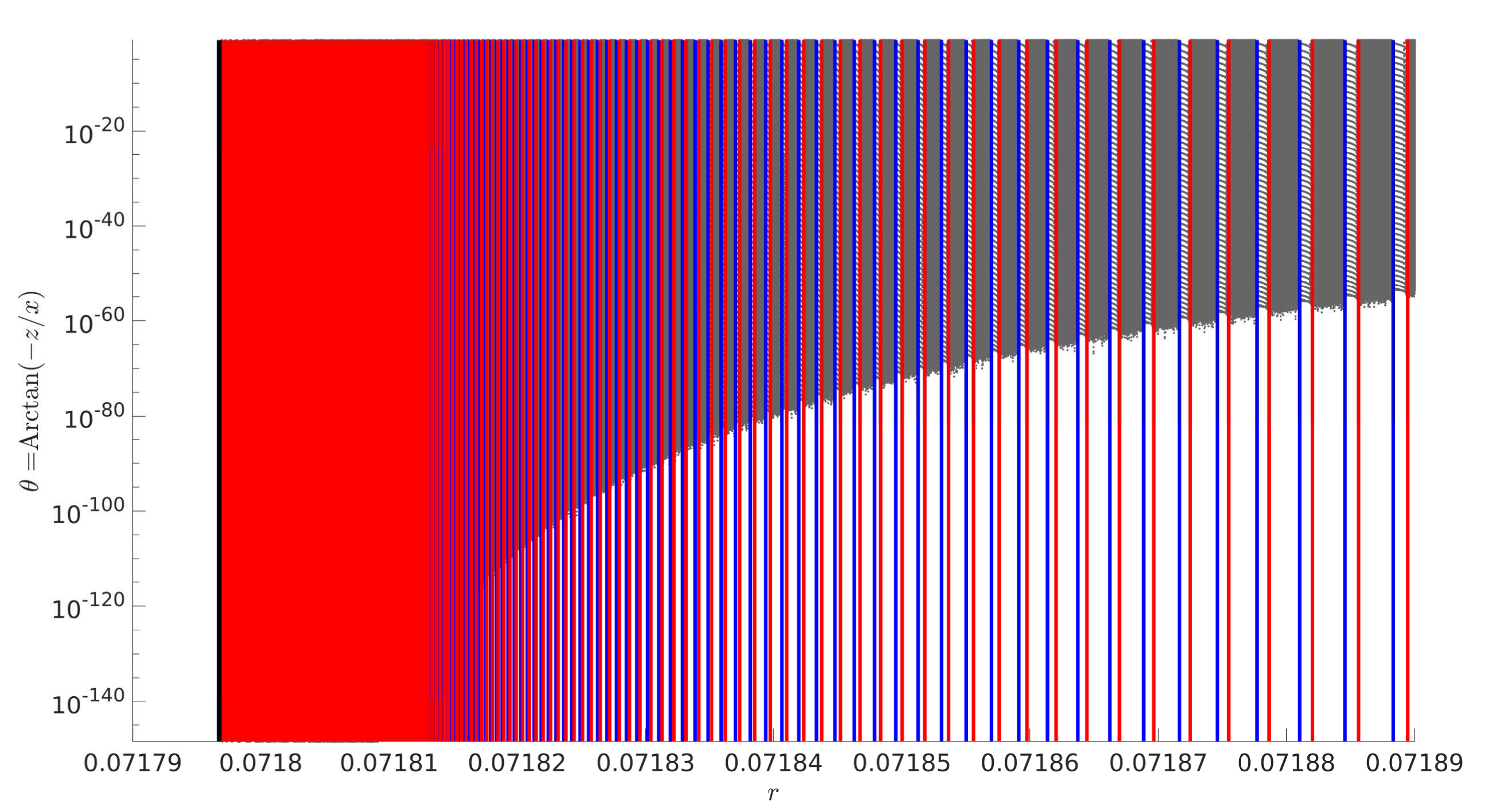} 
\caption{Plot of the tails of the orbits of $\widehat{\mathfrak{P}}_r$ for $0.07179 \leq r \leq 0.07189$ after $5000$ iterations, in logarithmic scale.}
\label{FIGURPlTy10.07179____<r<____0.07189Log__}
\end{figure}

\begin{figure}[h!]
\centering
\includegraphics[trim = 0cm 0cm 0cm 0cm, width=1\linewidth]{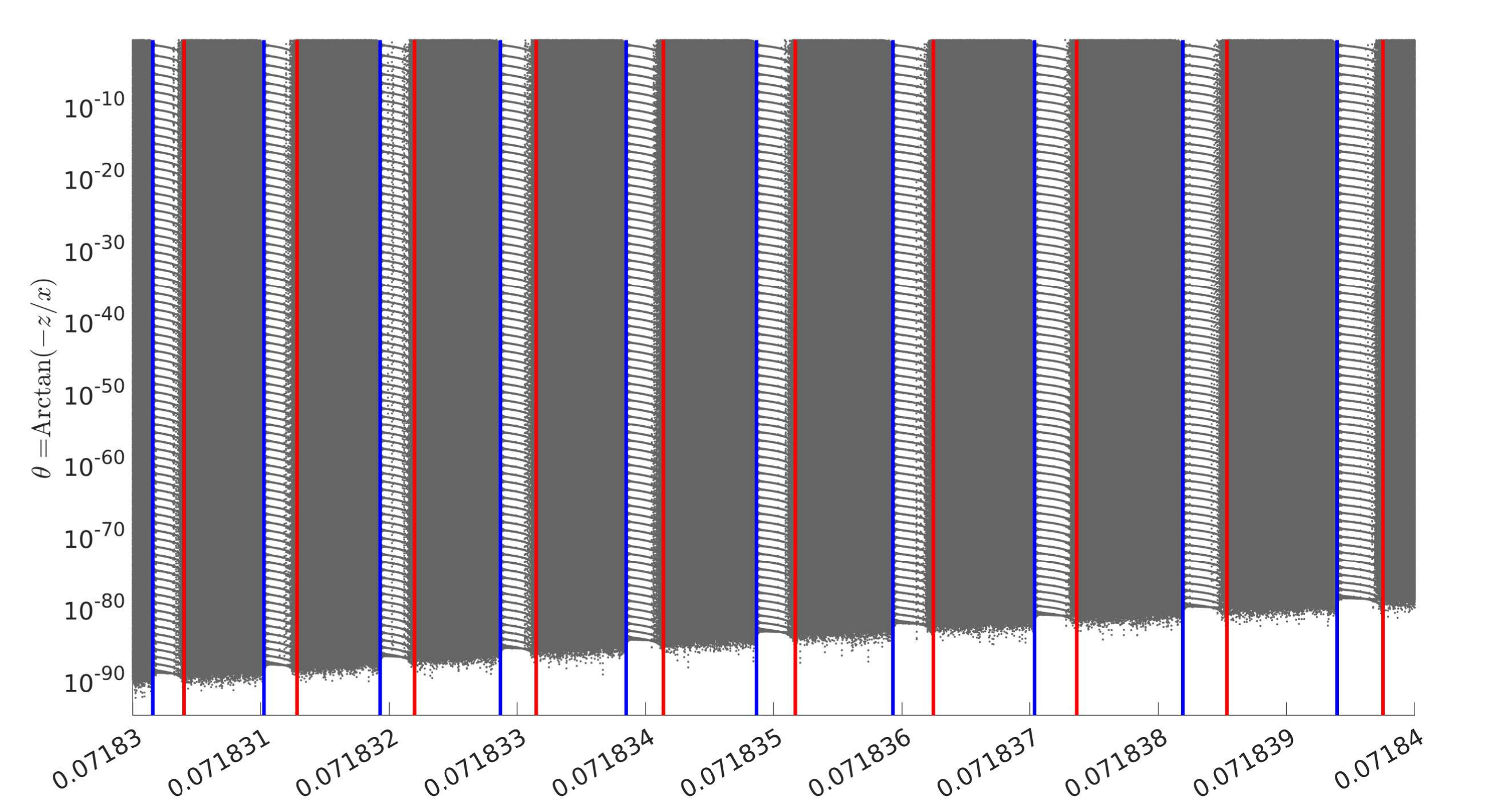} 
\caption{Plot of the tails of the orbits of $\widehat{\mathfrak{P}}_r$ for $0.07183 \leq r \leq 0.07184$ after $5000$ iterations, in logarithmic scale.}
\label{FIGURPlTy10.07183____<r<____0.07184Log__}
\end{figure}

\noindent
In particular, on Figure \ref{FIGURPlTy10.07179____<r<____0.07189Log__}, the rightmost window of stability corresponds to $n = 47$, which is in agreement with the numerical computations of the lower and upper bounds provided in Table \ref{TABLEDecimals_LowerUpperBounds}. All the windows of stability that can be seen (at least until $n = 75$ on Figure \ref{FIGURPlTy10.07179____<r<____0.07189Log__}) seem to be correctly delimited by the boundaries plotted in blue and red. Figure \ref{FIGURPlTy10.07183____<r<____0.07184Log__} is a magnification of Figure \ref{FIGURPlTy10.07179____<r<____0.07189Log__}, highlighting the $10$ windows of stability obtained for $69 \leq n \leq 78$, as it can be confirmed relying on Table \ref{TABLEDecimals_LowerUpperBounds}. With the same scale of precision, choosing $r_\text{min} = 0.07181$ and $r_\text{max} = 0.07182$, we can observe the $n$-th windows of stability with $94 \leq n \leq 124$, correctly delimited by the theoretical previsions.\\
In summary, the numerical simulations based on the piecewise linear expression of $\mathfrak{P}$ exhibits that the windows of stability are delimited from below by the roots of the polynomials $Q_n$, in agreement with the theoretical result of \cite{CDKK999}, at least up to $n = 125$. This observation supports that the numerical simulations based on the algorithm presented in Section \ref{SSECTAppenAlgoLinea} are reliable. In turn, our numerical simulations support the conjecture of \cite{CDKK999} concerning the upper bound of the windows of stability, at least up to $n = 125$. In comparison, the numerical simulations conducted before in the literature (\cite{CDKK999} and \cite{DoHR025}) supported the conjecture only up to $n = 6$.\\
Finally, the orbits of the spherical reduction mapping $\widehat{\mathfrak{P}}_r$, computed with the algorithm presented in \ref{SSECTAppenAlgoTrigo}, are represented on Figure \ref{FIGURPlot__Old_Algor}. The interval of the restitution coefficients that are considered is as in Figure \ref{FIGURPlTy10.0717_____<r<_____0.1717}. This algorithm was used in \cite{DoHR025}.\\

\begin{figure}[h!]
\centering
\includegraphics[trim = 0cm 0cm 0cm 0cm, width=0.65\linewidth]{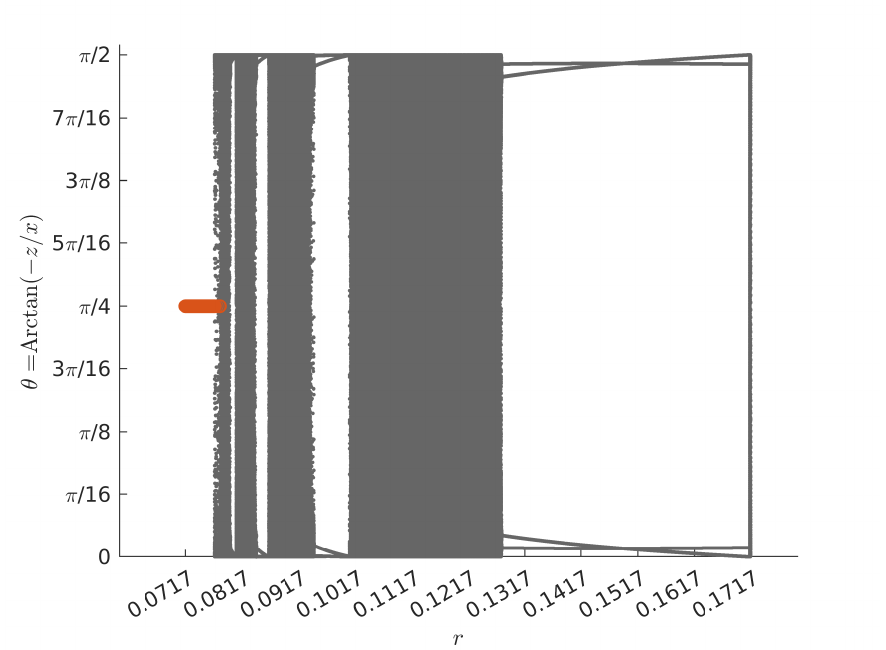} 
\caption{Plot of the tails of the orbits of $\widehat{\mathfrak{P}}_r$ for $0.0717 \leq r \leq 0.1717$ after $5000$ iterations, relying on the trigonometric algorithm of Section \ref{SSECTAppenAlgoTrigo}.}
\label{FIGURPlot__Old_Algor}
\end{figure}

\noindent
The limitations of the algorithm appear clearly on Figure \ref{FIGURPlot__Old_Algor}, on which the first windows of stability can be observed for $2 \leq n \leq 5$, but not beyond. The red dots that are plotted correspond to restitution coefficients for which the program was unable to compute enough iterations of $\widehat{\mathfrak{P}}_r$, due to the numerical singularities. In particular, along the orbits the angle $\theta$ is closer and closer to $0$ or $\pi/2$ as the restitution coefficient gets close to the critical value $7-4\sqrt{3}$. Therefore, the repeated use of the trigonometric functions and their inverses in the algorithm of Section \ref{SSECTAppenAlgoTrigo} leads to computations for which the numerical errors accumulate  (in particular, because the derivative of the arccosine function is singular at $\theta = 0$) and end up being the dominating observable effect, providing unusable results. This behaviour explains the apparent chaos observed in Figure \ref{FIGURPlot__Old_Algor} for $r \leq 0.08$.\\
The fact that the orbits concentrate around $\theta = 0$ and $\theta = \pi/2$ for $r$ close to $7-4\sqrt{3}$ is a simple consequence of the fact that the collapse of three particles, despite not yet possible as $r$ is still too large, attracts more and more the orbits of the system, leading to repeated sequence of collisions involving only three particles of the system, while the fourth particle remains spectator. When finally the sub-system of three particles separate, another sub-system of three particles starts another long sequence of collisions, until it separates, and so on, until the complete collapse of the whole system of four particles. The periodic patterns observed in the windows of stability, of the form $(\mathfrak{ab})^n(\mathfrak{cb})^n$ with $n$ larger and larger as $r$ is closer and closer to $7-4\sqrt{3}$, correspond precisely to this scenario.

\section{Spectral study and periodic orbits of the mapping $\widehat{\mathfrak{P}}_r$}
\label{SECTISpectStudy}

\noindent
In this section, we will investigate the general properties of the $\mathfrak{b}$-to-$\mathfrak{b}$ mapping $\widehat{\mathfrak{P}}_r$. In particular, we will start with a careful study of the matrices $P_i$, $1 \leq i \leq 4$, and their respective spectra, where the matrices $P_i$ are given in \eqref{EQUATCollisionMatric_P_i_}. Such a study will provide a qualitative understanding of the behaviour of the orbits of the dynamical system $(\widehat{\mathfrak{P}}_r^n(u))_n$. In a second time, another set of numerical simulations combined with the spectral study of the collision matrices $P_i$ will allow us to determine new stable periodic orbits of the dynamical system $(\widehat{\mathfrak{P}}_r^n(u))_n$. We emphasize that the present study exhibits the first stable periodic orbits associated to the four-particle collapsing system, different from the familiar pattern $(\mathfrak{ab})^n(\mathfrak{cb})^n$ of \cite{CDKK999}.

\subsection{Spectral study of the matrices $P_i$}
\label{SSECTSpectStudyMatri_P_i_}

\noindent
We start with the matrix $P_1$, that reads:
\begin{align}
P_1 = \begin{pmatrix} -r & r\alpha & 0 \\
-\alpha & \alpha^2 - r & r\alpha \\
0 & 0 & r^2 \end{pmatrix},
\end{align}
recalling that $\alpha = \frac{r+1}{2}$. We summarize the results on the eigenelements of $P_1$ in the following Proposition.

\begin{propo}[Eigenelements of $P_1$]
\label{PROPOEigen_P_1_}
If $r \in\ ]0,7-4\sqrt{3}]$, the three eigenvalues $\lambda_{P_1,i}$, $1 \leq i \leq 3$, of the matrix $P_1$ (defined in \eqref{EQUATCollisionMatric_P_i_}) are real and given by the expressions:
\begin{gather}
\lambda_{P_1,1} = r^2,\nonumber\\
\begin{align}
\lambda_{P_1,2} = \frac{r^2-6r+1 - (r+1)\sqrt{r^2-14r+1}}{8},\hspace{3mm} \lambda_{P_1,3} = \frac{r^2-6r+1 + (r+1)\sqrt{r^2-14r+1}}{8}.
\end{align}
\end{gather}
We have:
\begin{align}
0 < \lambda_{P_1,1} < \lambda_{P_1,2} < \lambda_{P_1,3}.
\end{align}
If $r \in\ ]7-4\sqrt{3},1]$, only one eigenvalue $\lambda_{P_1,1}$ of $P_1$ is real, while the two other eigenvalues $\lambda_{P_1,i}$, $i =2,3$ are complex conjugated. The eigenvalues are given by the expressions:
\begin{gather}
\lambda_{P_1,1} = r^2,\nonumber\\
\begin{align}
\lambda_{P_1,2} = \frac{r^2-6r+1 - i(r+1)\sqrt{ \vert r^2-14r+1 \vert }}{8},\hspace{3mm} \lambda_{P_1,3} = \frac{r^2-6r+1 + i(r+1)\sqrt{ \vert r^2-14r+1 \vert }}{8},
\end{align}
\end{gather}
and we have:
\begin{align}
0 < \lambda_{P_1,1} < \vert \lambda_{P_1,2} \vert = \vert \lambda_{P_1,3} \vert = r.
\end{align}
For any $r \in\ ]0,1]$, the eigenspace associated to $\lambda_{P_1,1}$ is $\text{Vect}\{(\alpha,r+1,3\alpha)\}$.\\
If $r \in\ ]0,7-4\sqrt{3}]$, the eigenspaces associated respectively to $\lambda_{P_1,i}$, $i = 2,3$, are subsets of the hyperplane $\{(x,y,z) \in \mathbb{R}^3\ /\ z=0 \}$, which is invariant under the action of $P_1$, and complementary in $\mathbb{R}^3$ with the eigenspace associated to $\lambda_{P_1,1}$. The eigenspaces respectively associated to $\lambda_{P_1,2}$ and $\lambda_{P_1,3}$ are:
\begin{align}
\label{EQUATEigenspace_P_1_Domin}
\text{Vect}\{(2r,\alpha \mp \sqrt{\alpha^2-4r},0)\} \hspace{5mm} \text{with } \mp = - \text{ if } i=2, \pm = + \text{ if } i=3.
\end{align}
If $r \in\ ]7-4\sqrt{3},1]$, the hyperplane $\{(x,y,z) \in \mathbb{R}^3\ /\ z=0 \}$ is invariant under the action of $P_1$, and complementary in $\mathbb{R}^3$ with the eigenspace associated to $\lambda_{P_1,1}$, and $P_1$ restricted to this hyperplane writes:
\begin{align}
\restr{P_1}{\{z=0\}} = \begin{pmatrix} \frac{r^2-6r+1}{8} & \frac{(r+1)\sqrt{ \vert r^2-14r+1 \vert }}{8} \\
-\frac{(r+1)\sqrt{ \vert r^2-14r+1 \vert }}{8} & \frac{r^2-6r+1}{8} \end{pmatrix}
\end{align}
in the basis $\big(\mathfrak{Re}(u_2),\mathfrak{Im}(u_2)\big)$, where $u_2$ is an eigenvector associated to the eigenvalue $\lambda_{P_1,2}$.
\end{propo}

\begin{remar}
The critical restitution coefficient for which all the eigenvalues of $P_1$ are real coincides with the critical upper bound on the restitution coefficient for which the collapse of three particles is possible and stable (\cite{CoGM995}, \cite{CDKK999}). This is consistent with the fact that the matrix $P_1$ describes the consecutive collisions $\mathfrak{ab}$ in the particle system. Since the dominating eigenvalue is real only if $r \leq 7-4\sqrt{3}$, we deduce that this condition is equivalent to the existence of a stable fixed point of the mapping $\varphi_1$ corresponding to the action of $P_1$ on the unit sphere $\mathbb{S}^2$.
\end{remar}

\noindent
More generally, Proposition \ref{PROPOEigen_P_1_} provides a complete understanding of the dynamical system obtained by iterating $P_1$. The great circle $\{ z = 0 \} \cap \mathbb{S}^2$ is an invariant manifold, which is attracting all the orbits (as it is associated with the two dominating eigenvalues). In the case when $r < 7-4\sqrt{3}$, the orbits on this great circle converge towards the fixed points corresponding to the intersection between the eigenspace of $\lambda_{P_1,3}$ and $\mathbb{S}^2$.\\
In the case when $r \geq 7-4\sqrt{3}$, the matrix of $P_1$ restricted to $\{ z=0 \}$ is a similarity in a basis that is not necessarily orthogonal, therefore the iterates of $\varphi_1$ restricted to the associated great circle are obtained by iterating a certain rotation, and then by composing with the linear change of variables corresponding to the change of basis from the canonical basis to the basis of the real and imaginary parts of the eigenvector associated to the dominating eigenvalue. In other words, the orbits in the plane $z = 0$ rotate along ellipses or twisted spirals, and to deduce the orbits on the sphere it remains to renormalize the iterations $\big(P_1^n(u)\big)_n$. For any value of $r$, the intersection between $\mathbb{S}^2$ and $\text{Vect}\{(\alpha,r+1,3\alpha)\}$ (the eigenspace associated to $\lambda_{P_1,1} = r^2$) is a fixed point of $\varphi_1$, which is never stable. In addition, this fixed point is never crossed by a stable manifold, that is, a perturbation from this fixed point in any direction will lead to an orbit that will diverge from this point, and that will eventually be attracted by the great circle $\{ z= 0 \} \cap \mathbb{S}^2$.

\begin{proof}[Proof of Proposition \ref{PROPOEigen_P_1_}]
The characteristic polynomial of the matrix $P_1$ is $\chi_{P_1}(\lambda) = (\lambda - r^2) \big[ \lambda^2 + (2r-\alpha^2) \lambda + r^2 \big]$, and its determinant is $\det(P_1) = r^4$.\\
Therefore, $\lambda_{P_1,1} = r^2$ is an eigenvalue of $P_1$, associated to the one-dimensional eigenspace $\text{Vect}\{(\alpha,r+1,3\alpha)\}$ given by the intersection of the two hyperplanes $E_1 = \{(x,y,z) \in \mathbb{R}^3\ /\ (r+1)x = \alpha y\}$ and $E_2 = \{(x,y,z) \in \mathbb{R}^3\ /\ \alpha x + (r^2+r-\alpha^2) y - r \alpha z = 0\}$. Besides, the hyperplane $E_3 = \{(x,y,z) \in \mathbb{R}^3\ /\  z=0 \}$ is invariant under the action of $P_1$. Since $r \neq 0$, the intersection between the three planes $E_1 \cap E_2 \cap E_3$ is reduced to $\{0\}$, and so the action of $P_1$ on $\mathbb{R}^3$ can be decomposed into its action, on the one hand, on the invariant line $E_1 \cap E_2$ (associated to the eigenvalue $\lambda_{P_1,1} = r^2$), and on the other hand on invariant hyperplane $E_3$ (associated to the two other eigenvalues).\\
The three eigenvalues of $P_1$ are real if and only if the discriminant $\Delta_{P_1}$ of $\lambda^2 + (2r-\alpha^2) \lambda + r^2$ is non-negative. Since we have:
\begin{align*}
\Delta_{P_1} = \alpha^2(\alpha^2 -4r) = \frac{\alpha^2}{4}(r^2-14r+1),
\end{align*}
the two roots of $r^2 - 14r + 1$ are $7 \pm 4\sqrt{3}$, and only $7 - 4\sqrt{3} \simeq 0.0718$ belongs to $[0,1]$. The three eigenvalues of $P_1$ are then real if and only if $r \leq 7-4\sqrt{3}$. In this case, the eigenvalues different from $\lambda_{P_1,1} = r^2$ are:
\begin{align}
\lambda_{P_1,i} = \frac{r^2-6r+1 \mp(r+1)\sqrt{r^2-14r+1}}{8} \hspace{5mm} \text{with}\ i=2\ \text{or}\ 3.
\end{align}
We will denote by $\lambda_{P_1,3}$ the eigenvalue obtained when $\mp = +$. In the case when $\lambda_{P_1,i}$, $i =2,3$, are real, we have $0 \leq r^2 \leq \lambda_{P_1,i}$ for any $r \in [-1,7-4\sqrt{3}]$, with equality only if $r=0$ or $r=-1$. Therefore:
\begin{align}
0 < \lambda_{P_1,1} < \lambda_{P_1,i} \hspace{5mm} \forall r \in\ ]0,7-4\sqrt{3}],
\end{align}
so that $\lambda_{P_1,3}$ is the dominating eigenvalue for any $r \in\ ]0,7-4\sqrt{3}]$.\\
In the case when the eigenvalues $\lambda_{P_1,i}$ ($i = 2,3$) are complex, relying on the determinant of $P_1$ we have $\vert \lambda_{P_1,i} \vert = r$, so that $\lambda_{P_1,i}$, $i =2,3$, are the dominating eigenvalues in the case when $r \in\ ]7-4\sqrt{3},1]$.
\end{proof}

\noindent
It is important to observe that we can deduce the eigenelements of $P_3$ from the eigenelements of $P_1$, since we have:
\begin{align}
P_3 = \begin{pmatrix} r^2 & 0 & 0 \\
\alpha r & \alpha^2-r & -\alpha \\
0 & r\alpha & -r \end{pmatrix} = \begin{pmatrix} 0 & 0 & 1 \\
0 & 1 & 0 \\
1 & 0 & 0 \end{pmatrix}
\begin{pmatrix} -r & r\alpha & 0 \\
-\alpha & \alpha^2 - r & \alpha r \\
0 & 0 & r^2 \end{pmatrix}
\begin{pmatrix} 0 & 0 & 1 \\
0 & 1 & 0 \\
1 & 0 & 0 \end{pmatrix} = J P_1 J,
\end{align}
where $J$ is defined in \eqref{EQUATDefin__J__}, so that $P_3 J u = \lambda J u$ if and only if $P_1 u = \lambda u$ for any vector $u \in \mathbb{R}^3$ and complex number $\lambda$. Therefore, $P_1$ and $P_3$ have the same spectrum, and if $S$ denotes the set of the eigenvectors  of $P_1$, $JS = \{Ju\ /\ u \in S\}$ is the set of eigenvectors of $P_3$.\\
\newline
\noindent
We turn now to $P_2$, that reads:
\begin{align}
P_2 = \begin{pmatrix} r & -r\alpha & 0 \\
\alpha & -2\alpha^2 + r & \alpha \\
0 & -r\alpha & r \end{pmatrix}.
\end{align}

\noindent
We obtain the following result.

\begin{propo}[Eigenelements of $P_2$]
\label{PROPOEigen_P_2_}
If $r \in\ ]0,3-2\sqrt{2}]$, the three eigenvalues $\lambda_{P_2,i}$, $1 \leq i \leq 3$, of the matrix $P_2$ (defined in \eqref{EQUATCollisionMatric_P_i_}) are real and given by the expressions:
\begin{gather}
\lambda_{P_2,1} = r,\nonumber\\
\begin{align}
\lambda_{P_2,2} = r - \alpha^2 + \alpha \sqrt{\alpha^2 - 2r},\hspace{3mm} \lambda_{P_2,3} = r - \alpha^2 - \alpha \sqrt{\alpha^2 - 2r}.
\end{align}
\end{gather}
We have:
\begin{align}
\lambda_{P_2,3} < \lambda_{P_2,2} < 0 \hspace{5mm} \text{and} \hspace{5mm} \vert \lambda_{P_2,2} \vert < \lambda_{P_2,1} < \vert \lambda_{P_2,3} \vert.
\end{align}
If $r \in\ ]3-2\sqrt{2},1]$, only one eigenvalue $\lambda_{P_2,1}$ of $P_2$ is real, while the two other eigenvalues $\lambda_{P_2,i}$, $i =2,3$ are complex conjugated. The eigenvalues are given by the expressions:
\begin{gather}
\lambda_{P_2,1} = r,\nonumber\\
\begin{align}
\lambda_{P_2,2} = r - \alpha^2 + i \alpha \sqrt{\alpha^2 - 2r},\hspace{3mm} \lambda_{P_2,3} = r - \alpha^2 - i \alpha \sqrt{\alpha^2 - 2r},
\end{align}
\end{gather}
and we have:
\begin{align}
\lambda_{P_2,1} = \vert \lambda_{P_2,2} \vert = \vert \lambda_{P_2,3} \vert = r.
\end{align}
For any $r \in\ ]0,1]$, the eigenspace associated to $\lambda_{P_2,1}$ is $\text{Vect}\{(1,0,-1)\}$.\\
If $r \in\ ]0,3-2\sqrt{2}]$, the eigenspaces associated respectively to $\lambda_{P_2,i}$, $i = 2,3$, are subsets of the hyperplane $\{(x,y,z) \in \mathbb{R}^3\ /\ x = z \}$, which is invariant under the action of $P_2$, and orthogonal in $\mathbb{R}^3$ with the eigenspace associated to $\lambda_{P_2,1}$. The eigenspaces respectively associated to $\lambda_{P_2,2}$ and $\lambda_{P_2,3}$ are:
\begin{align}
\label{EQUATEigenspace_P_1_Domin}
\text{Vect}\{(r,\alpha \mp \sqrt{\alpha^2-2r},r)\} \hspace{5mm} \text{with } \mp = - \text{ if } i=2, \pm = + \text{ if } i=3.
\end{align}
If $r \in\ ]3-2\sqrt{2},1]$, the hyperplane $\{(x,y,z) \in \mathbb{R}^3\ /\ x = z \}$ is invariant under the action of $P_2$, and orthogonal in $\mathbb{R}^3$ with the eigenspace associated to $\lambda_{P_2,1}$, and $P_2$ restricted to this hyperplane writes:
\begin{align}
\restr{P_2}{\{x=z\}} = \begin{pmatrix} r-\alpha^2 & -\alpha \sqrt{ \vert \alpha^2-2r \vert } \\
\alpha \sqrt{ \vert \alpha^2-2r \vert } & r-\alpha^2 \end{pmatrix}
\end{align}
in the basis $\big(\mathfrak{Re}(u_2),\mathfrak{Im}(u_2)\big)$, where $u_2$ is an eigenvector associated to the eigenvalue $\lambda_{P_2,2}$.
\end{propo}

\begin{remar}
It is interesting to observe that the critical restitution coefficient for which all the eigenvalues of $P_2$ are real coincides with the upper bound for the restitution coefficients associated to the stable periodic pattern of collision $\mathfrak{ababcbcb}$, which is also the upper bound of the restitution coefficients associated to any of the stable periodic patterns that were observed in \cite{CDKK999}.\end{remar}

\noindent
As for $P_1$, Proposition \ref{PROPOEigen_P_2_} allows to understand completely the dynamical system on the unit sphere $\mathbb{S}^2$ obtained by iterating $P_2$. In the case when $r < 3-2\sqrt{2}$, there exists two unique stable fixed points, obtained as the intersection between the eigenspace associated to the dominating eigenvalue $\lambda_{P_2,3}$ and $\mathbb{S}^2$, and these two fixed points attract almost every orbit (the only orbits that are not attracted by these fixed points are those that start exactly from the invariant manifold obtained as the intersection between $\{ x = z \}$ and $\mathbb{S}^2$). More precisely, since $\lambda_{P_2,3} < 0$, along the convergence the orbits are oscillating between the two points of intersection between the eigenspace of $\lambda_{P_2,3}$ and $\mathbb{S}^2$. Nevertheless, when considering the action of $P_2$ on the projective space $\mathbb{P}_2(\mathbb{R})$, almost every orbit will converge towards the unique stable fixed point associated to the one-dimensional eigenspace of $\lambda_{P_2,3}$.\\
In the case when $r > 3-2\sqrt{2}$, we encounter a new situation: all the eigenvalues have the same modulus, and therefore no invariant manifold attracts the orbits. More precisely, the fixed points lying at the intersection between the line $\text{Vect}\{(1,0-1)\}$ with $\mathbb{S}^2$ are stable in the sense that an orbit close to this point will remain close but will never converge towards the fixed point (except of course if the orbits is reduced to the fixed point only), and similarly, an orbit close to the invariant manifold obtained as the intersection between $\{ x = z \}$ and $\mathbb{S}^2$ will remain close to this invariant manifold. In addition, if we consider the hyperplanes orthogonal to the eigenspace $\text{Vect}\{(1,0-1)\}$ associated to $\lambda_{P_2,1}$, the intersection between any of these hyperplanes and the unit sphere $\mathbb{S}^2$ constitutes an invariant manifold, so that we obtain a partition of $\mathbb{S}^2$ into invariant manifolds. The evolution of the orbits on any of these invariant manifolds is deduced from the fact that $P_2$ reduced to $\{ x = z \}$ is a similarity in a basis which is not necessarily orthogonal. The situation is therefore the same as for $P_1$ when this matrix has complex eigenvalues: the orbits can be written in terms of iterated rotations in a twisted basis.

\begin{proof}[Proof of Proposition \ref{PROPOEigen_P_2_}]
The characteristic polynomial of the matrix $P_2$ is $\chi_{P_2}(\lambda) = (\lambda - r) \big[ \lambda^2 + (2\alpha^2 - 2r) \lambda + r^2 \big]$, and its determinant is $\det(P_2) = r^3$. $\lambda_{P_2,1} = r$ is an eigenvalue of $P_2$, associated to the one-dimensional eigenspace $E_4 = \text{Vect}\{(1,0,-1)\}$. In addition, the hyperplane $E_5 = \{(x,y,z) \in \mathbb{R}^3\ /\ x=z \}$ is invariant under the action of $P_2$. In addition, $E_4$ and $E_5$ are orthogonal.\\
The eigenvalues of $P_2$ are all real if and only if the discriminant $\Delta_{P_2}$ of $\lambda^2 + (2\alpha^2 - 2r)\lambda + r^2$ is non-negative. We have here:
\begin{align}
\Delta_{P_2} = 4 \alpha^4 - 8 r \alpha^2 = 4 \alpha^2 (\alpha^2 - 2r) = \alpha^2 ( r^2 - 6r + 1),
\end{align}
the two roots of $r^2 - 6r + 1$ are $3 \pm 2\sqrt{2}$, and only $3 - 2\sqrt{2} \simeq 0.1716$ belongs to $[0,1]$. Therefore, the three eigenvalues of $P_2$ are real if and only if $r \leq 3 - 2\sqrt{2}$. In this case, the eigenvalues different from $\lambda_{P_2,1} = r$ are:
\begin{align}
\lambda_{P_2,i} = r - \alpha^2 \pm \alpha \sqrt{\alpha^2 - 2r} \hspace{5mm} \text{with}\ i=2\ \text{or}\ 3.
\end{align}
We will denote by $\lambda_{P_2,2}$ the eigenvalue obtained when $\pm = +$, and by $\lambda_{P_2,3}$ the eigenvalue obtained when $\pm = -$. Since $r \neq 0$, $r-\alpha^2 = -\frac{1}{4}(1-r)^2 \leq 0$ and $(r - \alpha^2)^2 > \alpha^2(\alpha^2 - 2r)$, we deduce that $\lambda_{P_2,i} < 0$ for $i = 2,3$. Writing in addition the eigenvalues as:
\begin{align}
\lambda_{P_2,i} = -r - \sqrt{\alpha^2-2r} \big[ \sqrt{\alpha^2 - 2r} \mp \alpha \big],
\end{align}
we deduce $\vert \lambda_{P_2,2} \vert \leq \lambda_{P_2,2} \leq \vert \lambda_{P_2,3} \vert$.\\
In the case when $r > 3 - 2\sqrt{2}$, since $\lambda_{P_2,1} = r$, the determinant of $P_2$ allows to deduce that the three eigenvalues have the same modulus, equal to $r$.
\end{proof}

\subsection{Application to the study of the mapping $\widehat{\mathfrak{P}}_r$}
\label{SSECTAppli_to_P}

\noindent
We saw in Section \ref{SSECTSpectStudyMatri_P_i_} that determining the eigenelements of the matrices $P_i$ allows to understand completely the dynamical systems $\big( P_i^n(u) /\vert P_i^n(u) \vert \big)_n$ induced on the unit sphere. Nevertheless, to apply these results to the dynamics obtained with iterating $\widehat{\mathfrak{P}}_r$, we need to take into account the different domains of the quadrant $X = \{x^2+y^2+z^2 = 1\}\cap \{x\geq 0, z\leq 0\}$ on which the piecewise linear mapping $\mathfrak{P}$ coincides with the different matrices $P_i$. We will discuss this question in the present section.

\paragraph{The domain where $\mathfrak{P}$ coincides with $P_1$ or $P_3$.} In the case when $r \leq 7-4\sqrt{3}$, that is, when all the eigenvalues of $P_1$ and $P_3$ are real, any orbit of the iteration $(P_i^n(u)/\vert P_i^n(u) \vert)_n$ converges towards one of the stable fixed points of the mappings induced by $P_i$, for $i = 1,3$ and for almost every $u \in \mathbb{R}^3$. Since the quadrant on which $\mathfrak{P}$ is defined is $\{x \geq 0, z \leq 0\}$, the eigenspaces associated to the dominating eigenvalues of $P_1$ and $P_3$ belong to the boundary of the quadrant, according to \eqref{EQUATEigenspace_P_1_Domin}. In the case of $P_1$, the domain on which $\mathfrak{P}$ coincides with this matrix is defined by the conditions $x > 0$, $z < 0$, $y > 0$ and $\alpha y - x > 0$. Since $(2r,\alpha + \sqrt{\alpha^2-4r},0)$ is an eigenvector associated to the dominating eigenvalue of $P_1$, and since we have $\alpha (\alpha + \sqrt{\alpha^2-4r}) - 2r > 0$, the stable fixed point that attracts all the orbits of the dynamical system induced by $P_1$ belongs to the boundary of the domain in which $\mathfrak{P}$ coincides with $P_1$. We recover a result that is consistent with the existence and stability of the collapse of three particles (that is, obtained with the infinite repetition of the period $\mathfrak{ab}$) described in \cite{CoGM995}. The fact that the attracting fixed point belongs to $\{z=0\}$ is also consistent with the collapse of three particles, for which two pairs of particles are eventually in contact, while a fourth particle (here \footnotesize{\circled{4}}\normalsize{}) remains at a positive distance from the three-particle cluster \footnotesize{\circled{1}}\normalsize{}-\footnotesize{\circled{2}}\normalsize{}-\footnotesize{\circled{3}}\normalsize{}.\\
In the case when $r > 7-4\sqrt{3}$, considering $P_1$ restricted to $\{ z=0 \}$ it is possible to show that the orbits leave the domain given by the conditions $x > 0$, $z < 0$, $y > 0$ and $\alpha y - x > 0$ in finite time, since the condition $y > 0$ will eventually be violated. In the case when $r > 3-2\sqrt{2}$, a direct computation shows that the condition $\alpha y - x > 0$ is violated after a single iteration of $P_1$. After a certain number of iterations of $P_1$, the orbits eventually enter the domain, either on which $\mathfrak{P}$ coincides with $P_2$, or with $P_3$.\\
By symmetry, the same conclusions hold for the domain of $P_3$.

\paragraph{The domain where $\mathfrak{P}$ coincides with $P_2$.} In the case when $r < 3 - 2\sqrt{2}$, all the eigenvalues of $P_2$ are real, and $\lambda_{P_2,3} = r-\alpha^2 - \alpha\sqrt{\alpha^2 - 2r}$ is dominating. The associated eigenvector reads $(r,\alpha+\sqrt{\alpha^2-2r},r)$, and does not belong to the quadrant $x \geq 0$, $z \leq 0$ on which $\mathfrak{P}$ is defined. In particular, since the dominating eigenvector of $P_2$ belongs to the interior of the complement of the domain $\{ x > 0, z < 0, y > 0, \alpha y - x < 0 \} \cup \{ x > 0, z < 0, y < 0, \alpha y - z > 0 \}$ where $\mathfrak{P}$ coincides with $P_2$, we deduce that, starting from this domain, any orbit leaves the domain after finitely many iterations of the matrix $P_2$, except for the orbit that is reduced to the single fixed point in the domain, given by the intersection between $\text{Vect}\{(1,0,-1)\}$ and the unit sphere $\mathbb{S}^2$. Since in addition the quadrant $x \geq 0$, $z \leq 0$ is invariant under the action of $\mathfrak{P}$, we deduce that almost every orbit reaches, either the domain on which $\mathfrak{P}$ coincides with $P_1$, or the domain on which $\mathfrak{P}$ coincides with $P_3$, after finitely many iterations. The case when an orbit reaches in finite time the boundary between the respective domains of $P_2$ and $P_1$, or between the respective domains of $P_2$ and $P_3$, corresponds indeed to a Lebesgue-negligible set of initial configurations.\\
Let us mention that the fixed point associated to $\text{Vect}\{(1,0,-1)\}$, which is associated to the eigenvalue $r$, is not stable, but lies at the crossing between two invariant manifolds, the first constituting the stable direction of the fixed point, the other the unstable direction. On the unstable direction, the orbits of $\widehat{\mathfrak{P}}_r$ converge towards the stable fixed point, given by the intersection between $\text{Vect}\{(1,0,-1)\}$ and the unit sphere. On the stable direction, the orbits of $\widehat{\mathfrak{P}}_r$ converge from the intersection between $\text{Vect}\{(r,\alpha-\sqrt{\alpha^2-2r},r)\}$ and the unit sphere, towards the intersection between $\text{Vect}\{(1,0,-1)\}$ and the unit sphere.\\
In the case when $r > 3-2\sqrt{2}$, $P_2$ has two complex eigenvalues, and the three eigenvalues have the same modulus. In this case, the fixed point at the intersection between $\text{Vect}\{(1,0,-1)\}$ and the unit sphere $\mathbb{S}^2$ belongs to the domain $\{ x > 0, z < 0, y > 0, \alpha y - x < 0 \} \cup \{ x > 0, z < 0, y < 0, \alpha y - z > 0 \}$. This fixed point is stable, in the sense that orbits close to it remain close to it, but without converging. The orbits rotate around this fixed point. Therefore, there are only two possible scenarios. In the first case, the orbits remain on curves on the sphere that are close enough to the fixed point, so that they remain in the domain, in which case such orbits constitute quasi-periodic orbits of $\mathfrak{P}$. By quasi-periodic, we mean that for any positive number $\varepsilon > 0$ and any point $\mathfrak{P}^n(x_0)$ of an orbit on such invariant curves, there exists $k = k(n,\varepsilon)$ such that $\vert \mathfrak{P}^n(x_0) - \mathfrak{P}^{n+k}(x_0) \vert \leq \varepsilon$. Such orbits can even be periodic, provided that the restitution coefficient $r$ is chosen in an appropriate manner. In the second case, the orbits of $P_2$ are on curves on the sphere that are not contained in the domain $\{ x > 0, z < 0, y > 0, \alpha y - x < 0 \} \cup \{ x > 0, z < 0, y < 0, \alpha y - z > 0 \}$, in which case such orbits can leave this domain in finite time, and enter the domain in which $\mathfrak{P}$ coincides either with $P_1$, or with $P_3$.

\section{Second numerical simulations: search for periodic orbits}
\label{SECTISeconNumerSimul}

\noindent
We mentioned already that the only stable periodic patterns leading to the collapse of $4$ particles, up to now, were discovered in \cite{CDKK999}, and are all of the form $(\mathfrak{ab})^n(\mathfrak{cb})^n$. It seems that for any $n \geq 2$, one can obtain an open interval of restitution coefficients, and an open set of initial configurations of the four particle system leading to the collapse, as it was discussed in \cite{CDKK999}. Our numerical simulations presented in Section \ref{SECTIFirstNumerSimul} support also this conjecture. The largest restitution coefficient for which a periodic pattern of the form $(\mathfrak{ab})^n(\mathfrak{cb})^n$ is stable is $r = 3-2\sqrt{2} \simeq 0.1716$, for the particular choice of $n = 2$.\\
Besides, the feasible periodic pattern (feasible in the sense that there exist initial configurations leading to a periodic collapse according to such a pattern) associated to the largest known restitution coefficient is $\mathfrak{abcb}$, feasible until $r$ approximately equal to $0.1917$ (see \cite{CDKK999}, and also \cite{HuRo023} for more details). Nevertheless, this pattern is not stable. Another periodic pattern which does not belong to the family of the patterns of the form $(\mathfrak{ab})^n(\mathfrak{cb})^n$, and that is not stable, is found in \cite{DoHR025}, namely $\mathfrak{ababcb}$, feasible until $r = 5-2\sqrt{6} \simeq 0.1010$.

\paragraph{Description of the second numerical simulations.} Our objective in the present section is to find periodic patterns beyond the maximal restitution coefficients we discussed: we aim to find stable periodic patterns for $r$ larger than $3-2\sqrt{2}$, and feasible periodic patterns for $r$ larger than $0.1917$. To do so, we start with plotting bifurcation diagrams. This time, we will represent the angle $\varphi$, orientating the velocity vector right after a collision of type $\mathfrak{b}$, instead of representing the angle $\theta$, as we did on Figures \ref{FIGURPlTy10.0717_____<r<_____0.1717}-\ref{FIGURPlot__Old_Algor}. The reason is that the periodic orbits do not accumulate on the boundary of the intervals in the $\varphi$ variable, contrary to what happens when projecting the orbits on the $\theta$ variable. For a normal vector $u =\hspace{0.25mm} ^t \hspace{-0.25mm} (x,y,z)$ to a plane $\mathcal{P}$ describing the post-collisional configuration of the system right after a collision of type $\mathfrak{b}$, the angle $\varphi \in [0,\pi]$ is given by the following formula (see \cite{DoHR025}):
\begin{align}
\cos\varphi = - \frac{y \sqrt{x^2+z^2}}{\vert q \vert} = -\frac{y}{\vert u \vert},
\end{align}
where $\sqrt{x^2 + z^2}$ is the norm of the position vector $p = \hspace{0.25mm} ^t \hspace{-0.25mm} (x,0,z)$, and $\vert q \vert$ is the norm of the velocity vector $q = \hspace{0.25mm} ^t \hspace{-0.25mm} (-xy,x^2+z^2,-yz)$.

\paragraph{First bifurcation diagram: indication of a chaotic behaviour.} In order to confirm the behaviour described in \cite{CDKK999} concerning the loss of stability for the patterns $(\mathfrak{ab})^n(\mathfrak{cb})^n$, and in order to detect periodic patterns in the range of restitution coefficients already discussed in the literature, we start with plotting the velocity angle $\varphi$ obtained along the $p$-th iterations, for $4900 \leq p \leq 5000$, of the mapping $\widehat{\mathfrak{P}}_r$, for $0.1 \leq r \leq 0.102$ and with an interval of $\Delta_r = 5\times 10^{-7}$ between two consecutive restitution coefficients that are considered. For each restitution coefficient, we repeat the procedure for $32$ different initial configurations, all taken from a fixed $8\times 4$ grid of the phase space. We obtain Figure \ref{FIGUR_r_=__0.1_0.102}.

\begin{figure}[h!]
\centering
\begin{subfigure}{0.48\textwidth}
    \includegraphics[trim = 0cm 0cm 0cm 0cm, width=\linewidth]{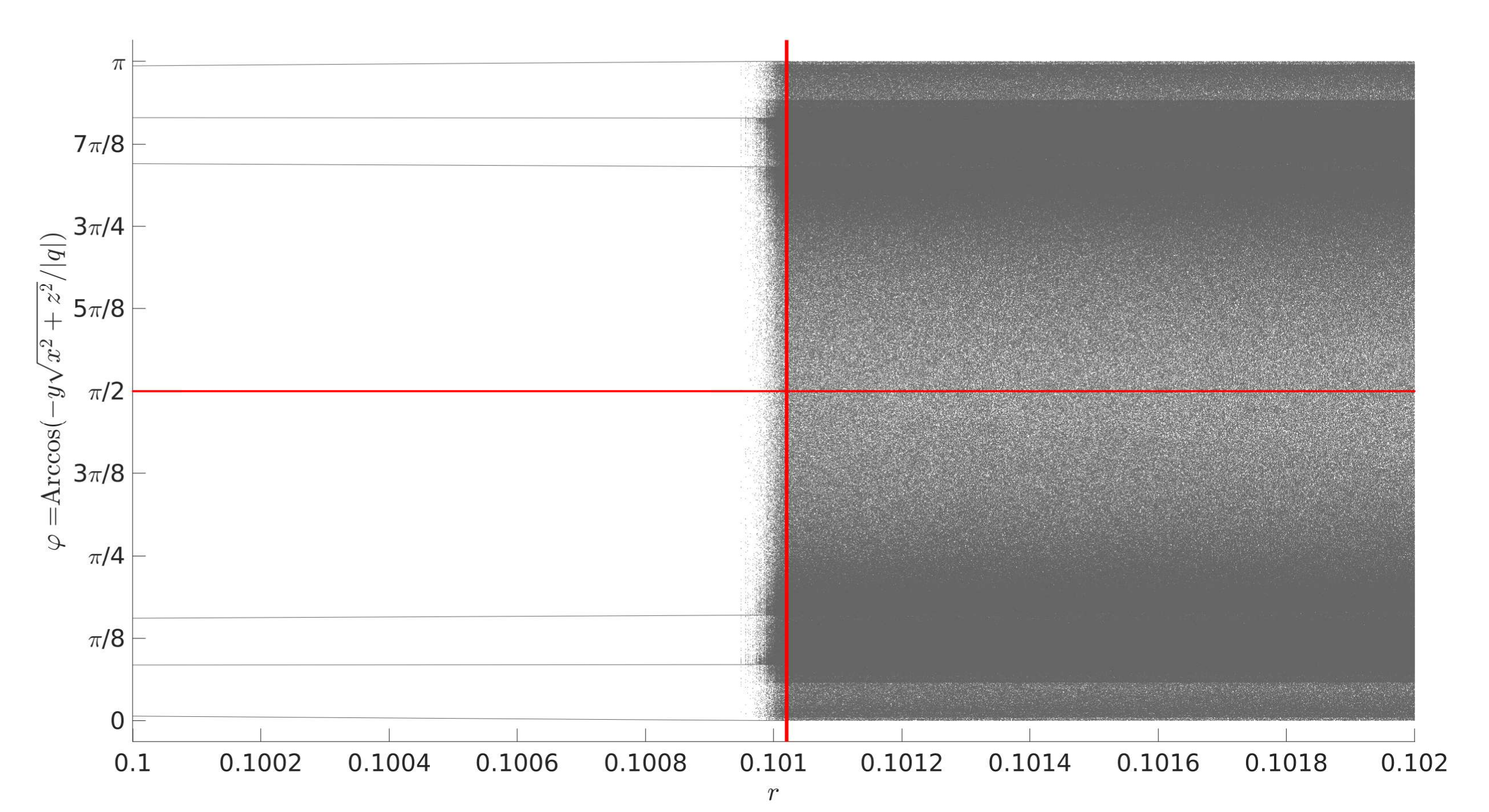}
    \caption{$p$-th iterations of $\widehat{\mathfrak{P}}_r$ for $4900 \leq p \leq 5000$.}
    \label{FIGUR_r_=__0.1_0.102}
\end{subfigure}
\hfill
\begin{subfigure}{0.48\textwidth}
    \includegraphics[trim = 0cm 0cm 0cm 0cm, width=\linewidth]{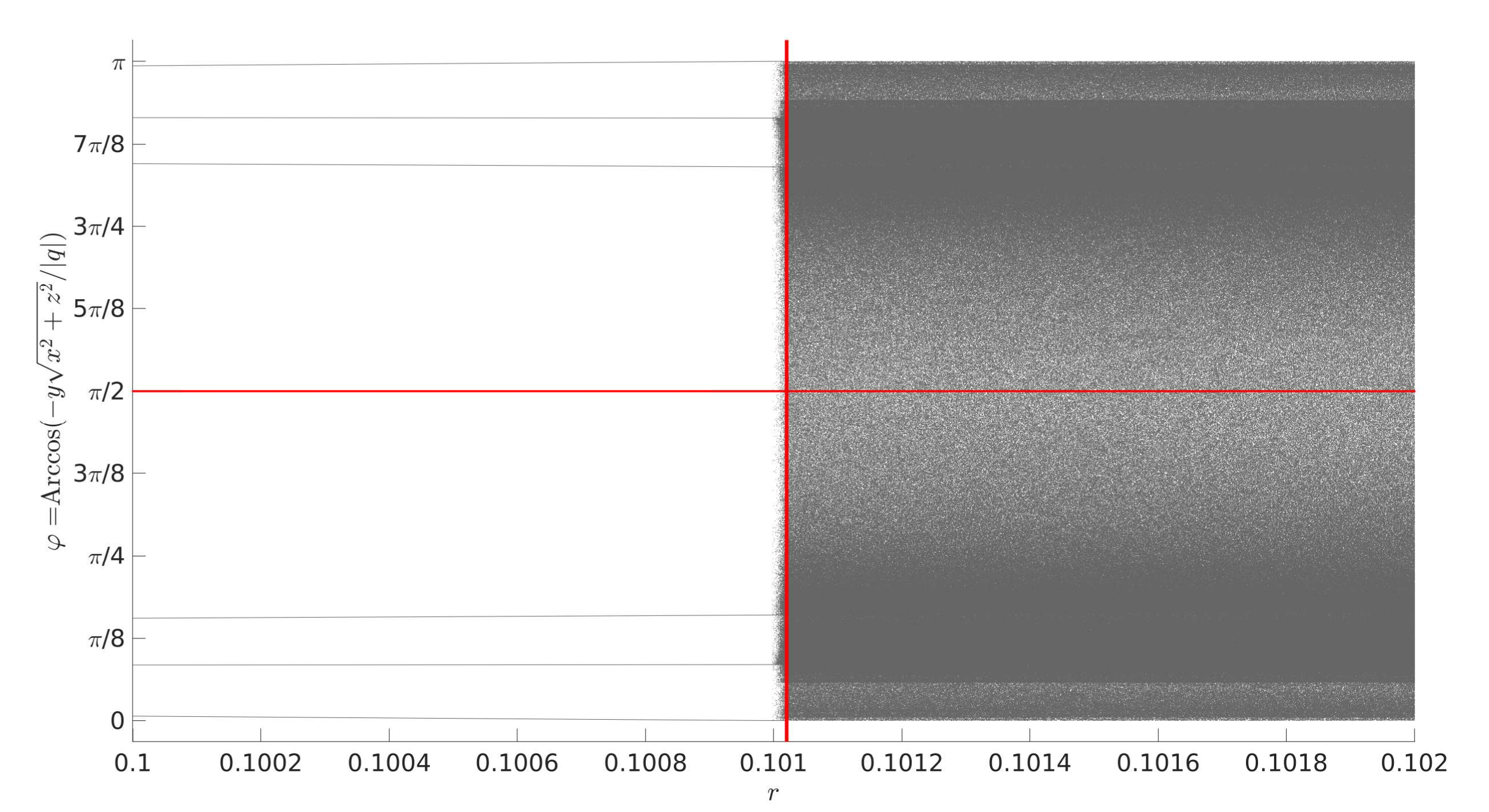}
    \caption{$p$-th iterations of $\widehat{\mathfrak{P}}_r$ for $24900 \leq p \leq 25000$.}
    \label{FIGUR_r_=__0.1_0.102_25000}
\end{subfigure}
\caption{Plot of the tails (last $100$ iterations) of the orbits of $\widehat{\mathfrak{P}}_r$ for $0.1 \leq r \leq 0.102$.}
\label{FIGUR_ab^3_cb^3}
\end{figure}

\noindent
On Figure \ref{FIGUR_r_=__0.1_0.102}, we observe in particular on the left hand side, below the critical value $5-2\sqrt{6} \simeq 0.1010\ 2051$ (represented by the vertical red line on Figure \ref{FIGUR_r_=__0.1_0.102}), the accumulation of the orbits on six particular values of the angle $\varphi$. This corresponds to the six different collisions of type $\mathfrak{b}$ per period in the collision pattern $(\mathfrak{ab})^3(\mathfrak{cb})^3$. This time, in terms of the angle $\varphi$, the fact that three accumulation points lie below the line $\varphi = \pi/2$, and three above, is a clear consequence of the fact that if $\cos\varphi > 0$, then the next collision to follow is of type $\mathfrak{c}$, while if $\cos\varphi < 0$, the next collision is of type $\mathfrak{a}$.\\
Interestingly, for restitution coefficients $r$ slightly below $5-2\sqrt{6}$, the orbits do not concentrate completely on the six accumulation points. The theoretical study conducted in \cite{CDKK999} ensures that configurations leading to the self-similar collapse $(\mathfrak{ab})^3(\mathfrak{cb})^3$ are stable, for any $0.0946 < r < 5-2\sqrt{6}$. Nevertheless, the problem of characterizing the basin of attraction of such self-similar configurations remains an open question. Since we do not observe another attracting orbit in this range of restitution coefficients, this lack of convergence might indicate the presence of another pattern, which is unstable, but which prevents the orbits to converge fast towards the stable self-similar configuration. Indeed, if we repeat the numerical simulations, considering this time the $p$-th iterates, with $24900 \leq p \leq 25000$ (see Figure \ref{FIGUR_r_=__0.1_0.102_25000}), less restitution coefficients than in Figure \ref{FIGUR_r_=__0.1_0.102} are associated to orbits that do not clearly concentrate on the six accumulation points. We observe also that the basin of attraction of the self-similar configuration associated to $(\mathfrak{ab})^3(\mathfrak{cb})^3$ seems to be very large: as it can be seen by comparing Figures \ref{FIGUR_r_=__0.1_0.102} and \ref{FIGUR_r_=__0.1_0.102_25000}, for $r$ slightly below $5-2\sqrt{6}$, a large proportion of the interval $[0,\pi]$ on which the orbits are projected seems to not converge before the $5000$-th iteration, but seems also to converge eventually towards the six accumulation points.\\
For $r \geq 5-2\sqrt{6}$, an apparent chaotic regime takes place, in which the orbits seem to cover completely the phase space, in a manner that looks uniform, or at least without any apparent order.

\paragraph{Bifurcation diagrams beyond $r = 3-2\sqrt{2} \simeq 0.1716$.} We turn now to larger restitution coefficients. In order to have a broader understanding of the dynamics of $\widehat{\mathfrak{P}}_r$, we consider now restitution coefficients such that $0.16 \leq r \leq 0.56$, with $\Delta_r = 10^{-4}$, and computing still the $5000$ first iterations, plotting only the last $100$ of these. We proceed in a similar way as before, considering orbits obtained from the same $32$ initial configurations. We obtain Figure \ref{FIGUR_r_=__0.16_0.56}.

\begin{figure}[h!]
\centering
\includegraphics[trim = 0cm 0cm 0cm 0cm, width=1\linewidth]{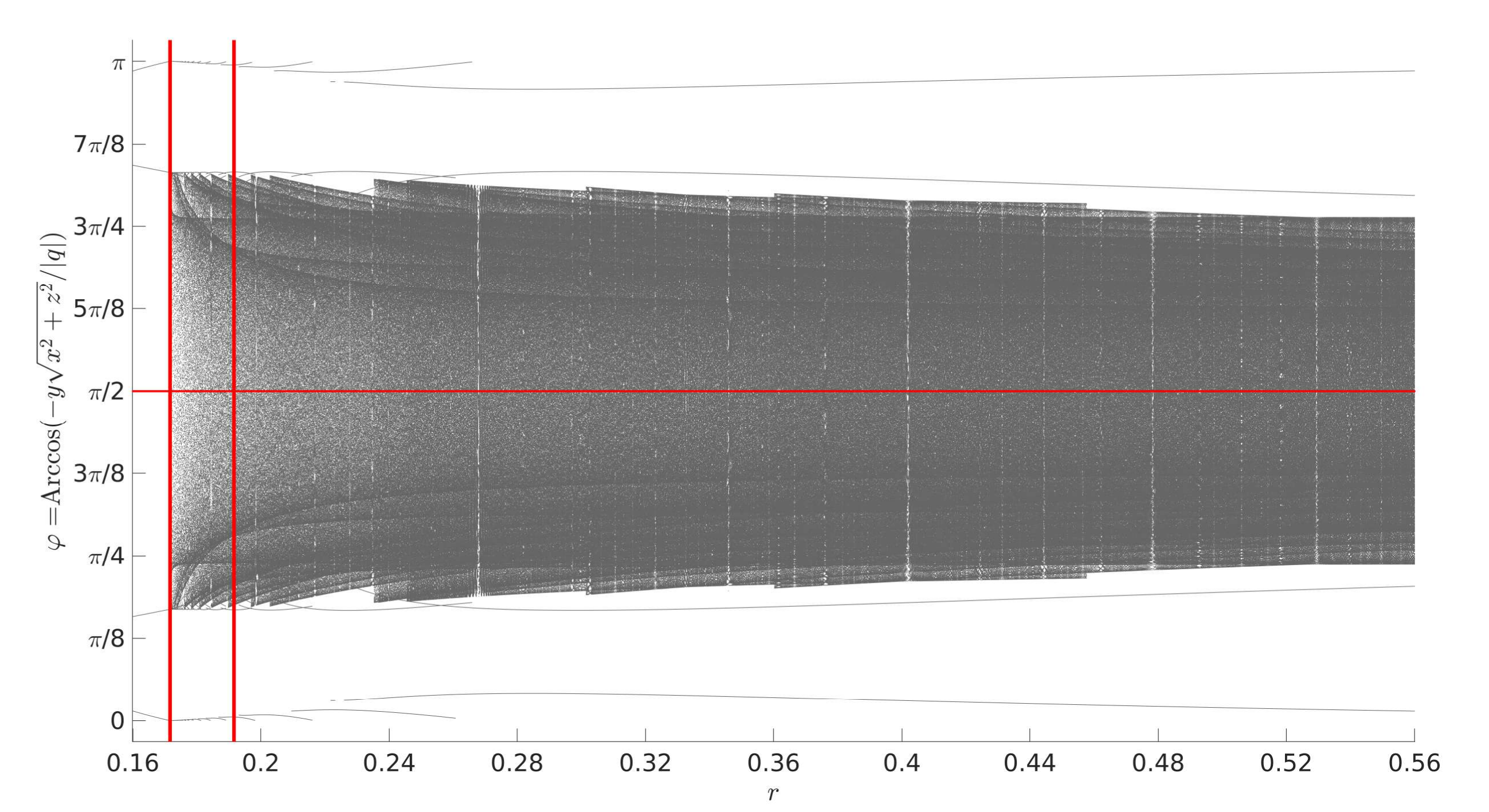} 
\caption{Plot of the tails (last $100$ iterations) of the $32$ orbits of $\widehat{\mathfrak{P}}_r$ for $0.16 \leq r \leq 0.56$ after $5000$ iterations.}
\label{FIGUR_r_=__0.16_0.56}
\end{figure}

\noindent
On Figure \ref{FIGUR_r_=__0.16_0.56}, except for $r < 3-2\sqrt{2} \simeq 0.1716$, which is associated to the range of existence and stability of the pattern $\mathfrak{ababcbcb}$, it seems at first glance that the whole range of restitution coefficients is exhibiting a chaotic behaviour. Nevertheless, above $3-2\sqrt{2}$, as already reported in \cite{DoHR025}, we observe around few peculiar values of $r$ an accumulation of the trajectories on subparts of the interval $[0,\pi]$: for instance, for $r \simeq 0.27$ or $r \simeq 0.405$. We called these regions ``thin stripes of stability'' in \cite{DoHR025}. A closer look at Figure \ref{FIGUR_r_=__0.16_0.56} shows that there are many of these stripes, on the whole interval $3-2\sqrt{2} \leq r \leq 0.56$ of restitution coefficients. If we consider now magnifications of Figure \ref{FIGUR_r_=__0.16_0.56} at different places of the interval $[0.16,0.56]$, we obtain Figures \ref{FIGUR_r_=0.191_0.195} ($0.191 \leq r \leq 0.195$), Figure \ref{FIGUR_r_=___0.5__0.6} ($0.5 \leq r \leq 0.6$) and Figure \ref{FIGUR_r_=___0.1916____0.1917__} ($0.1916 \leq r \leq 0.1917$). On Figures \ref{FIGUR_r_=0.191_0.195} and \ref{FIGUR_r_=___0.5__0.6}, we observe that the thin stripes of stability occur at different scales, and on Figure \ref{FIGUR_r_=___0.1916____0.1917__} we observe a very regular structure. More precisely, to each iteration that is plotted is associated a curve parametrized by the restitution coefficient, which is consistent with the fact that the collision matrices depend continuously on $r$, and that the initial data are chosen on a fixed grid which does not depend on $r$. The interesting observation is that the curves associated to different iterations are similar, so that the bifurcation diagrams obtained on Figures \ref{FIGUR_r_=0.191_0.195}-\ref{FIGUR_r_=___0.1916____0.1917__} appear to be superposition of such curves, as it is particularly clear on Figure \ref{FIGUR_r_=___0.1916____0.1917__}.

\begin{figure}[h!]
\centering
\includegraphics[trim = 0cm 0cm 0cm 0cm, width=1\linewidth]{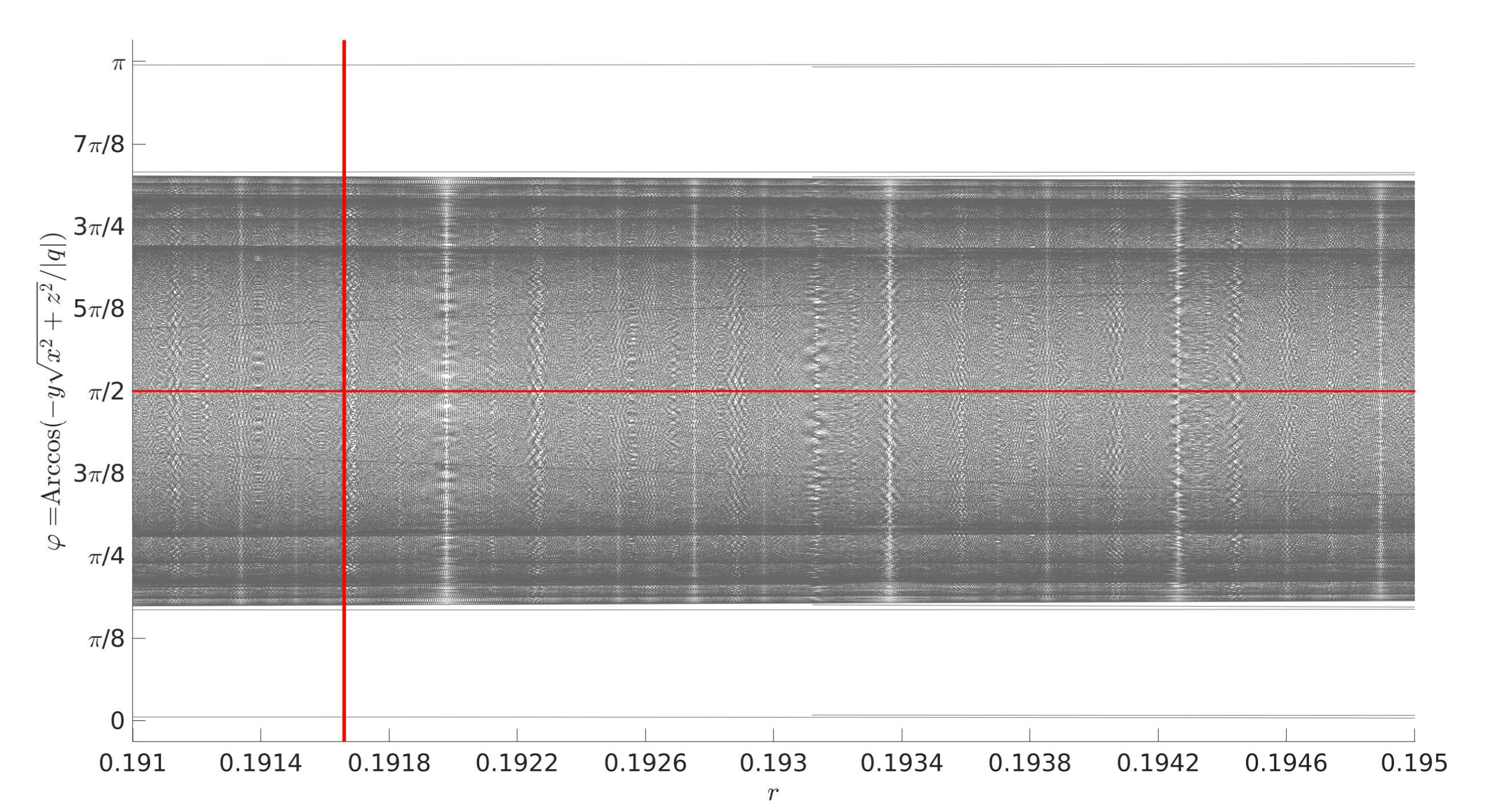} 
\caption{Plot of the tails (last $100$ iterations) of the orbits of $\widehat{\mathfrak{P}}_r$ for $0.191 \leq r \leq 0.195$ after $5000$ iterations.}
\label{FIGUR_r_=0.191_0.195}
\end{figure}

\begin{figure}[h!]
\centering
\includegraphics[trim = 0cm 0cm 0cm 0cm, width=1\linewidth]{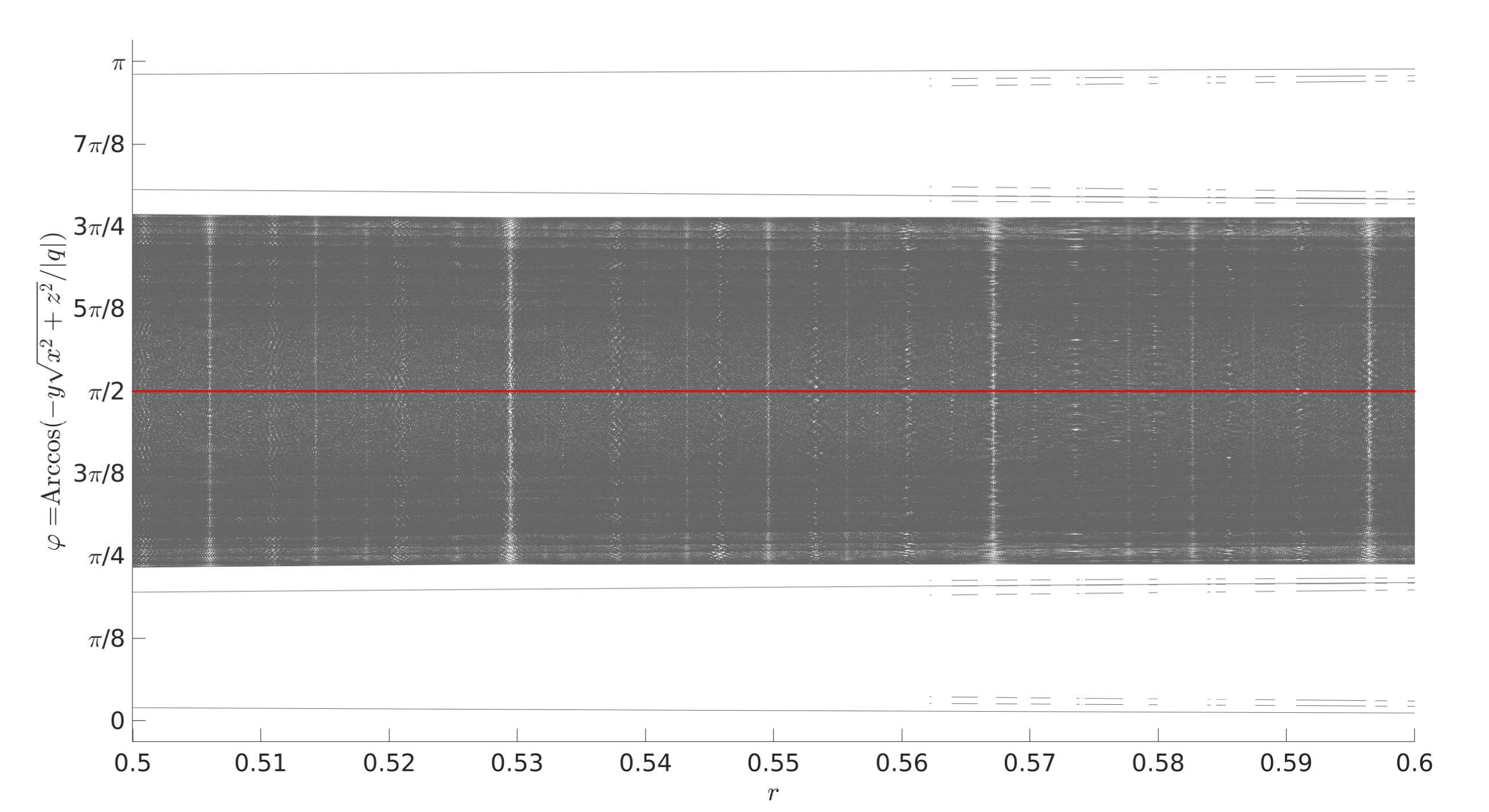} 
\caption{Plot of the tails (last $100$ iterations) of the orbits of $\widehat{\mathfrak{P}}_r$ for $0.5 \leq r \leq 0.6$ after $5000$ iterations.}
\label{FIGUR_r_=___0.5__0.6}
\end{figure}

\begin{figure}[h!]
\centering
\includegraphics[trim = 0cm 0cm 0cm 0cm, width=1\linewidth]{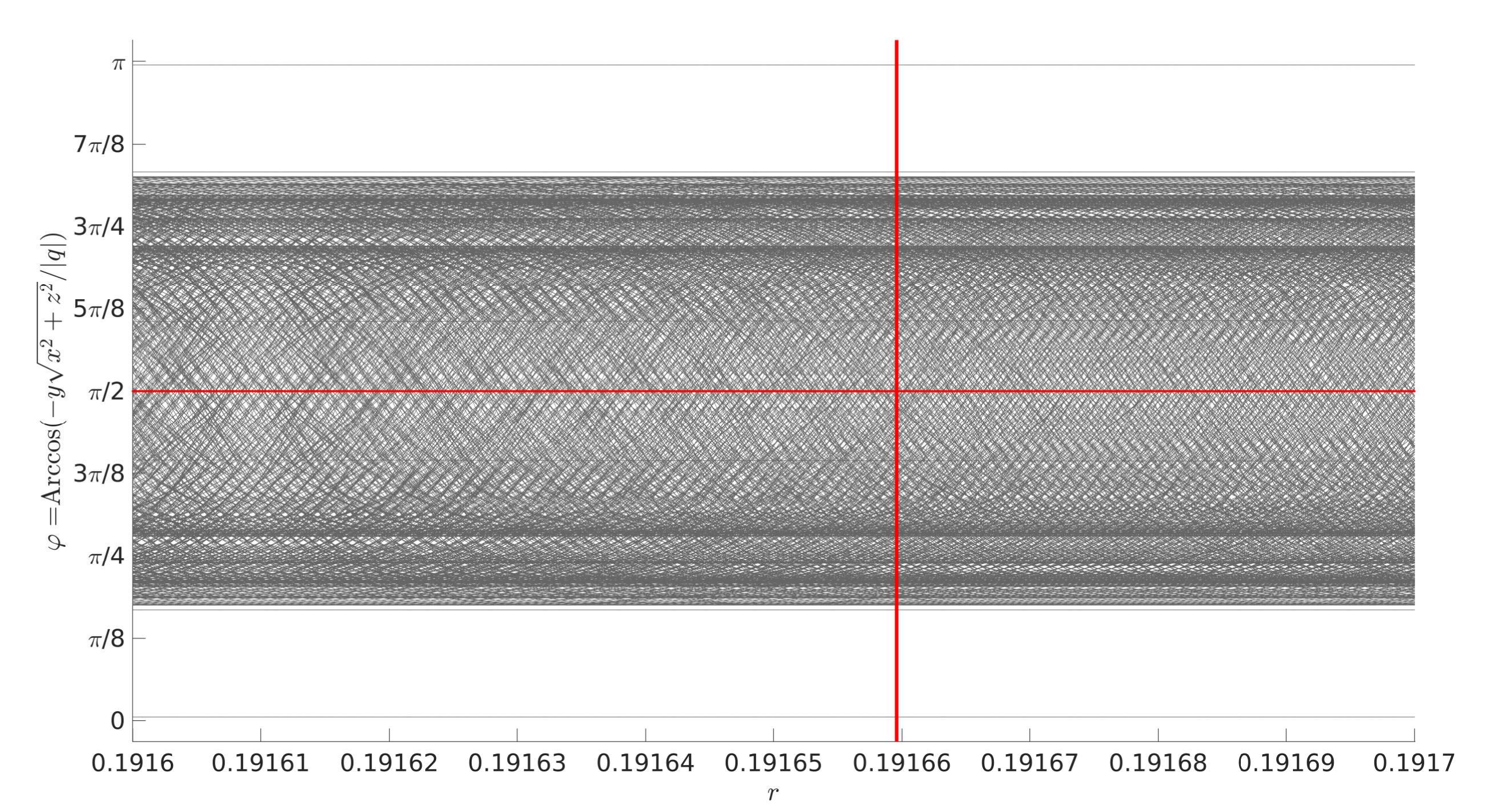} 
\caption{Plot of the tails (last $100$ iterations) of the orbits of $\widehat{\mathfrak{P}}_r$ for $0.1916 \leq r \leq 0.1917$ after $5000$ iterations.}
\label{FIGUR_r_=___0.1916____0.1917__}
\end{figure}

\noindent
In order to observe the very regular structure of the orbits, we plot on Figure \ref{FIGUR_r_=__0.19_0.21} the $p$-th iterations, with $4900 \leq p \leq 5000$, of one single orbit, obtained from a single initial configuration, chosen independently from the restitution coefficient $r$. On Figure \ref{FIGUR_r_=__0.19_0.21}, we observe together the two phenomena we described above: the thin stripes of stability appear at different scales, and the whole bifurcation diagram exhibits many symmetries, and appears to be the union of entangled curves of similar nature.

\begin{figure}[h!]
\centering
\includegraphics[trim = 0cm 0cm 0cm 0cm, width=1\linewidth]{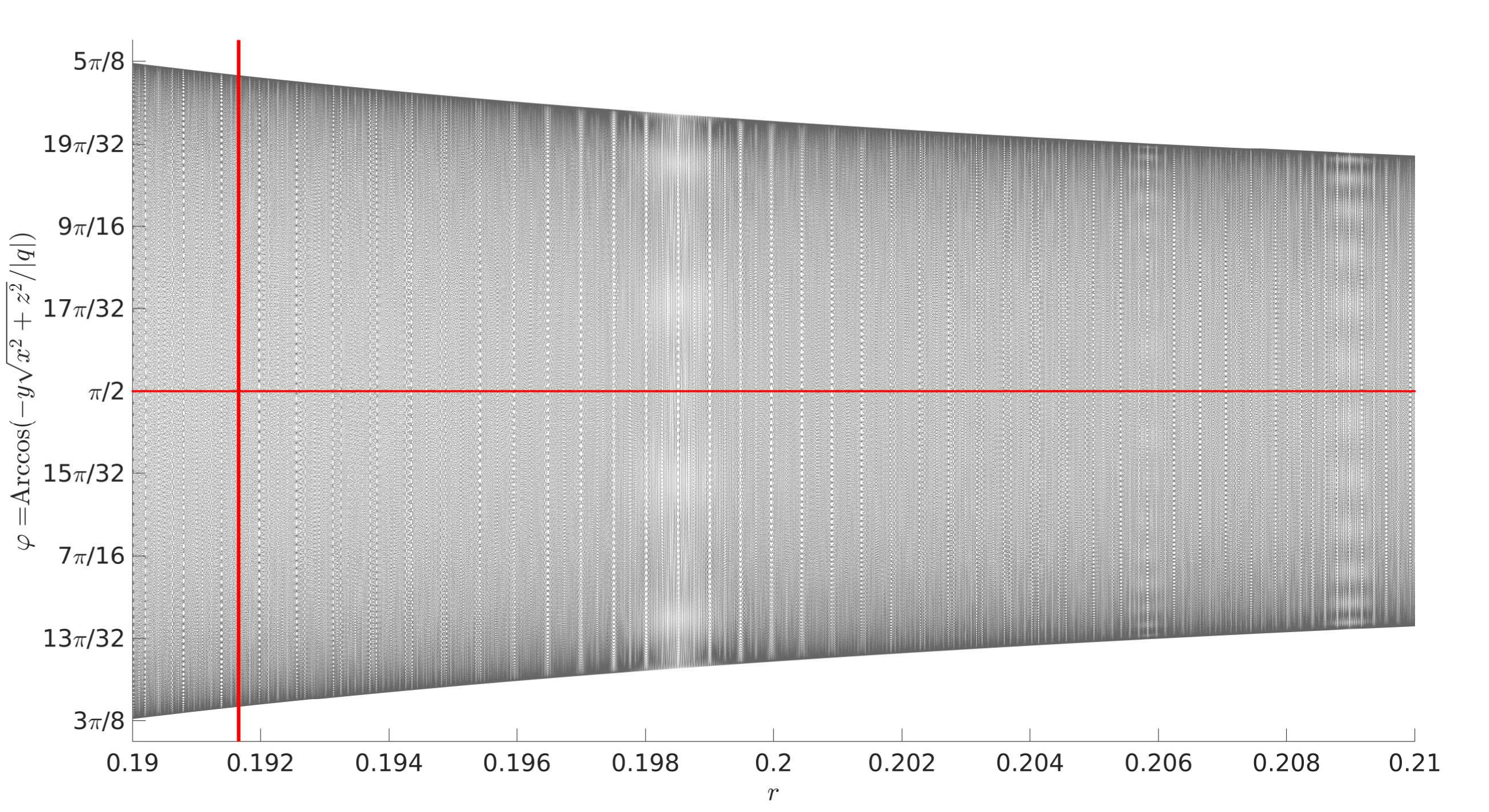} 
\caption{Plot of the tails (last $100$ iterations) of the orbits of $\widehat{\mathfrak{P}}_r$ for $0.19 \leq r \leq 0.21$ after $5000$ iterations.}
\label{FIGUR_r_=__0.19_0.21}
\end{figure}

\paragraph{Indication of existence of periodic orbits.} Finally, besides the union of the similar curves, we observe also the concentration of certains orbits (or of some particular iterations of these orbits) on few accumulation points that seem to be independent from $r$ (the four horizontal lines we can see outside the interval on which the curves accumulate on Figures \ref{FIGUR_r_=0.191_0.195}, \ref{FIGUR_r_=___0.5__0.6} and \ref{FIGUR_r_=___0.1916____0.1917__}, supplemented with the four/eight intermittent lines that appear or disappear depending on the value of $r$). If these accumulation points are not independent from $r$, at least, some orbits seem to partially concentrate on accumulation points that vary slowly with respect to $r$ (the four unions of curved segments outside the accumulation interval on Figure \ref{FIGUR_r_=__0.16_0.56}). In comparison, the similar curves composing the main accumulation interval appear to oscillate very fast as $r$ is varying, as already observed in \cite{DoHR025}.\\
If now we plot the bifurcation diagram slightly above the critical restitution coefficient $r = 3-2\sqrt{2}$ (Figure \ref{FIGUR_r=_0.171_0.175}), we observe such accumulation points, which lie on a grid that becomes thinner and thinner as $r$ converges from above to $3-2\sqrt{2}$ (the vertical red line on Figure \ref{FIGUR_r=_0.171_0.175}). We recover this way the observation already formulated in \cite{DoHR025} (see in particular Figure $9$ and the associated discussion). Such accumulation points seem to be associated to periodic orbits, with a non-trivial basin of attraction, since we observe such an accumulation on the numerical simulations of the bifurcation diagrams. Nevertheless, and contrary to the case when $r < 3-2\sqrt{2}$, it seems that above the critical restitution coefficients already discussed in the literature, stable periodic orbits of the dynamical system $(\mathfrak{P}^n)_n$ coexist with orbits presenting a more complicated dynamics.

\begin{figure}[h!]
\centering
\includegraphics[trim = 0cm 0cm 0cm 0cm, width=1\linewidth]{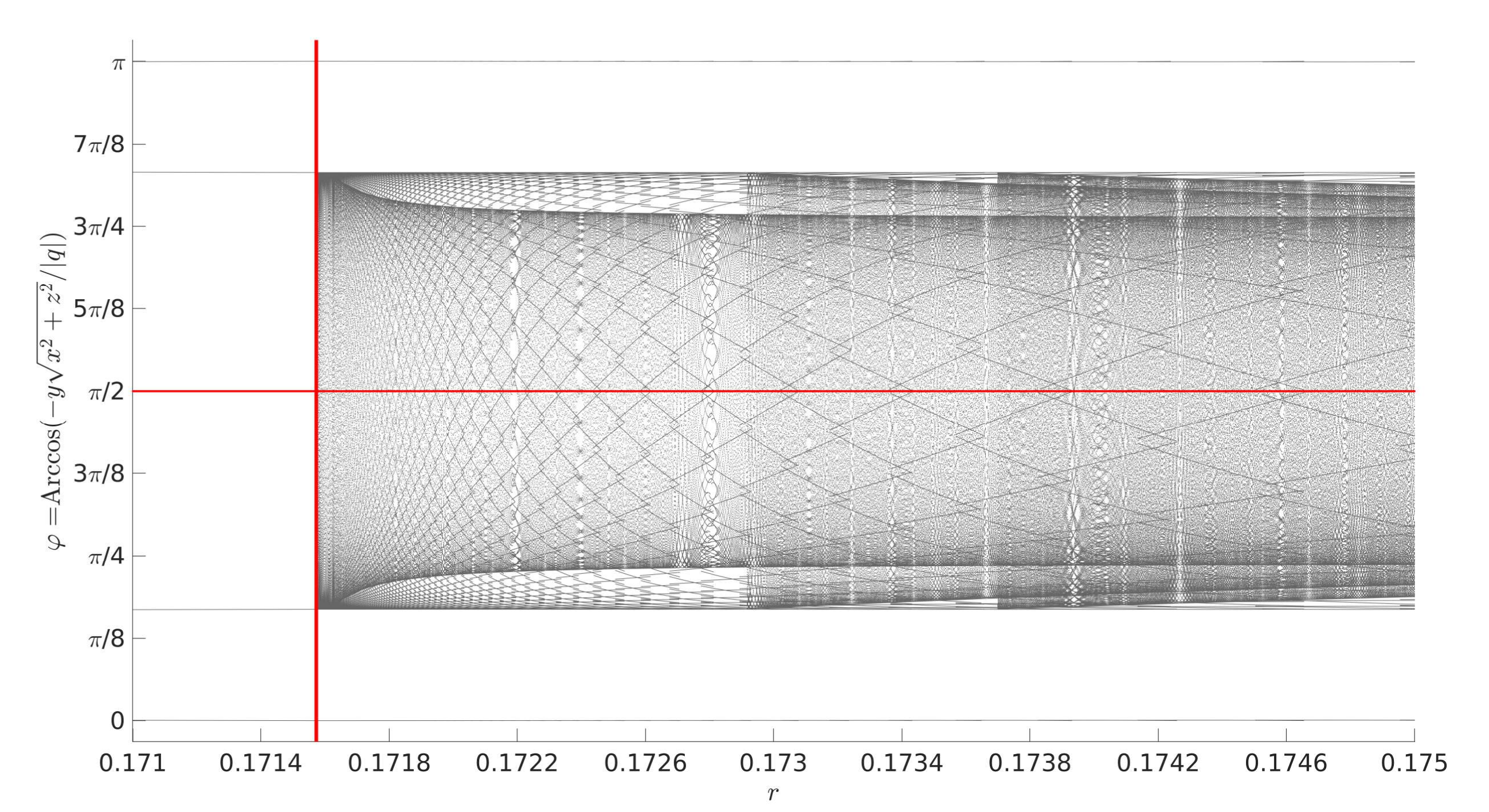} 
\caption{Plot of the tails (last $100$ iterations) of the orbits of $\widehat{\mathfrak{P}}_r$ for $0.171 \leq r \leq 0.175$ after $5000$ iterations.}
\label{FIGUR_r=_0.171_0.175}
\end{figure}

\subsection{Interpretation of the bifurcation diagrams}

\noindent
The fact that the bifurcation diagrams appear to be unions of similar curves (together with the possible periodic orbits) is consistent with the fact that the collision matrix $P_2$ induces a rotation (in appropriate coordinates) on the phase space, which then suggests that some orbits of the dynamical system $\big( \widehat{\mathfrak{P}}_r^n \big)_n$ can be quasi-periodic. This observation, in turn, provides an explanation concerning the onset of the ``thin stripes of stability'' already observed in \cite{DoHR025}. Indeed, for particular choices of $r$, the matrix $P_2$ may induce a rotation with an angle of the form $2\pi l/m$, with $l \in \mathbb{Z}, m \in \mathbb{N}^*$, so that periodic orbits of $\big( \widehat{\mathfrak{P}}^n \big)_n$ would exist for such particular restitution coefficients, provided that these orbits remain inside the domain on which $\mathfrak{P}$ coincides with $P_2$. To be more precise, according to Proposition \ref{PROPOEigen_P_2_}, when $r > 3 - 2\sqrt{2}$ the collision matrix $P_2$ restricted to $\{x = z\}$ can be written as:
\begin{align}
\begin{pmatrix}
r \cos \beta & - r \sin \beta \\
r \sin \beta & r \cos \beta
\end{pmatrix},
\end{align}
where $\cos \beta = - \frac{(1-r)^2}{4r}$. Therefore, choosing $r$ as:
\begin{align}
\label{EQUATThinStripe__r__}
r(l/m) = 1 + 2\xi - 2\sqrt{\xi + \xi^2} \hspace{3mm} \text{with} \hspace{3mm} \xi = \cos\Big( \pi - 2\pi \frac{l}{m} \Big),
\end{align}
with $\frac{1}{4} \leq \frac{l}{m} \leq \frac{3}{4}$, the angle $\beta$ is by construction of the form $2\pi l/m$. It turns out that this observation allows to predict the restitution coefficients associated to thin stripes of stability, as it can be seen on Figure \ref{FIGUR_r=_0.191_0.195_ThinStrip}, in which we reproduce the bifurcation diagram presented in Figure \ref{FIGUR_r_=0.191_0.195}, highlighting in addition all the restitution coefficients of the form \eqref{EQUATThinStripe__r__}, for $1 \leq m \leq 120$ (represented as vertical blue lines).\\
Therefore, due to this very likely quasi-periodic behaviour, it seems that chaos does not take place for $r \geq 3-2\sqrt{2}$ (since quasi-periodic orbits with different initial data would not diverge).

\begin{figure}[h!]
\centering
\includegraphics[trim = 0cm 0cm 0cm 0cm, width=1\linewidth]{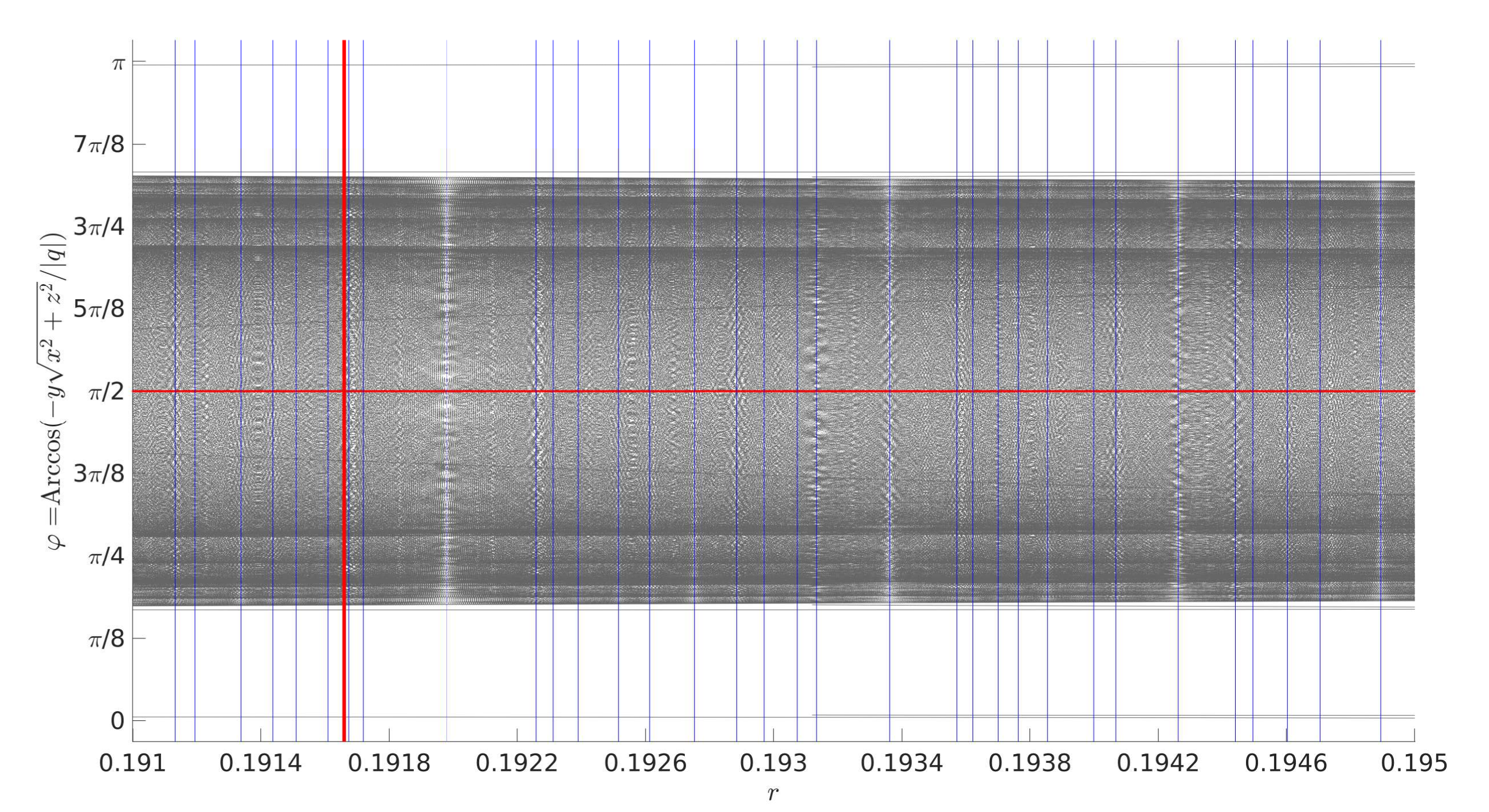} 
\caption{Plot of the tails (last $100$ iterations) of the orbits of $\widehat{\mathfrak{P}}_r$ for $0.191 \leq r \leq 0.195$ after $5000$ iterations, together with the restitution coefficients of the form \eqref{EQUATThinStripe__r__}, with $1 \leq m \leq 120$.}
\label{FIGUR_r=_0.191_0.195_ThinStrip}
\end{figure}

\subsection{Individual simulations of orbits, at $r$ fixed}
\label{SSECTSimulIndivOrbit}

\noindent
The bifurcation diagrams obtained in Figures \ref{FIGUR_r_=__0.16_0.56}-\ref{FIGUR_r=_0.171_0.175} suggest that the dynamical system $\big( \widehat{\mathfrak{P}}_r^n \big)_n$ admits stable periodic orbits, although their respective basins of attraction seem to be in general strict subsets of the phase space. In this section, we will represent, at $r$ fixed, individual orbits of the dynamical system $\big( \widehat{\mathfrak{P}}_r^n \big)_n$. Contrary to the bifurcation diagrams, we will not represent the orbits via one-dimensional projections, but we will represent the orbits completely. This procedure will allow us to detect and describe precisely periodic trajectories.

\paragraph{Two-dimensional representation of the orbits.} The fact that the dynamical system $\widehat{\mathfrak{P}}_r$ is acting on the projective plane $\mathbb{P}_2(\mathbb{R})$ allows a complete representation of its orbits: only two real variables are necessary to represent an element of $\mathbb{P}_2(\mathbb{R})$. We will proceed as follows. Denoting by $u(n) = \mathfrak{P}^n\big(u(0)\big) \in \mathbb{R}^3$ the $n$-th iteration of the mapping $\mathfrak{P}$, acting on an initial datum $u(0) \in \mathbb{R}^3$, and whose only relevant information is the orientation, we will represent the orbits of $\big( \widehat{\mathfrak{P}}_r^n \big)_n$ via the action of $\mathfrak{P}$ on the plane $\{x-z = 1\}$. More precisely, if we write the $n$-th iteration $u(n) = \big(u_x(n),u_y(n),u_z(n)\big)$ of a given orbit in coordinates, we will consider the intersection of $\text{Span}\big( u(n) \big)$ with $\{x-z = 1\}$, that is represented in the two-dimensional strip $]-\infty,+\infty[\, \times [-1,1]$ as:
\begin{align}
\label{EQUATFormuRepre_2_d_}
\frac{1}{u_x(n)-u_z(n)} \big( u_y(n), u_x(n)+u_z(n) \big).
\end{align}
Up to a rescaling that we will describe, the two-dimensional strip corresponds to the intersection between the plane $\{x-z = 1\}$ and the domain $x \geq 0, z \leq 0$ of the mapping $\mathfrak{P}$. We will denote by $(w_1,w_2)$ the coordinates of a generic point in the strip. \eqref{EQUATFormuRepre_2_d_} is obtained as follows: we will represent the plane $\{x-z=1\}$ (subset of $\mathbb{R}^3$) such that the three-dimensional axis along the $y$ direction is the first component in the planar representation (that is, in the strip). The intersection between $\text{Span}\big( u(n) \big)$ and $\{x-z = 1\}$ is the vector:
\begin{align}
\frac{1}{u_x(n)-u_z(n)} u(n),
\end{align}
providing directly the first component of \eqref{EQUATFormuRepre_2_d_}. The second component is the signed distance between $\frac{1}{u_x(n)-u_z(n)} \big( u_x(n),u_y(n),u_z(n) \big)$ and $\big(1/2,\frac{u_y(n)}{u_x(n)-u_z(n)},-1/2\big)$, given by $\frac{\sqrt{2}}{2}\big(u_x(n) + u_z(n)\big)$. The intersection between the domain $x \geq 0, z \leq 0$ and the plane $\{x-z = 1\}$ is a strip of width $\sqrt{2}$, for the sake of simplicity, we will rescale the second component in the strip so that it ranges between $-1$ and $1$, providing the second coordinate in \eqref{EQUATFormuRepre_2_d_}.\\
In the strip $]-\infty,+\infty[\, \times [-1,1]$, representing the plane $\{ x-z = 1 \}$, there are three remarkable segments:
\begin{itemize}
\item $\{w_1 = 0\}$, corresponding to the intersection of the three-dimensional planes $\{x-z = 1\}$ and $\{ y = 0 \}$ in $\mathbb{R}^3$: if $w_1 > 0$, that is, if $y > 0$, the first collision to take place from the initial configuration $^t (x,y,z)$ is of type $\mathfrak{a}$, and of type $\mathfrak{c}$ if $w_1 < 0$.
\item $\{ w_2 = 2\alpha w_1 - 1\}$, corresponding to the intersection of the three dimensional planes $\{x-z = 1\}$ and $\{ \alpha y - x = 0\}$: if $w_2 < 2\alpha w_1 - 1$, the two first collisions to take place after the initial configuration are $\mathfrak{ab}$ in this order, and if $w_2 > 2\alpha w_1 - 1$ and if $w_1 > 0$, the three first collisions to take place are $\mathfrak{acb}$ in this order,
\item $\{ w_2 = 2\alpha w_1 + 1\}$, corresponding to the intersection of the three dimensional planes $\{x-z = 1\}$ and  $\{ \alpha y - z = 0\}$: if $w_2 > 2\alpha w_1 + 1$, the two first collisions to take place after the initial configuration are $\mathfrak{cb}$ in this order, and if $w_2 < 2\alpha w_1 + 1$ and if $w_1 < 0$, the three first collisions to take place are $\mathfrak{cab}$ in this order.
\end{itemize}
In other words, the domain on which $\mathfrak{P}$ coincides with $P_2$ (see Proposition \ref{THEOR__P__FukngLinea}) is a parallelogram centered in the strip, contained between the two segments $\{ w_2 = 2\alpha w_1 \mp 1\}$. In the parallelogram, an iteration of $\mathfrak{P}$ describes either the three consecutive collisions $\mathfrak{acb}$ (if $w_1 > 0$) or $\mathfrak{cab}$ (if $w_1 < 0$). On the right hand side of the parallelogram, $\mathfrak{P}$ coincides with $P_1$ (one iteration describing the two consecutive collisions $\mathfrak{ab}$), and with $P_3$ on the left hand side (one iteration describing the collisions $\mathfrak{cb}$).

\paragraph{Notation $1,2,3$ for the collision sequences.} In order to discuss in an efficient manner the collision sequences, we introduce the following notation: $1$ will denote the two consecutive collisions $\mathfrak{ab}$ ($\mathfrak{P}$ coincides with $P_1$), $3$ the two consecutive collisions $\mathfrak{cb}$ ($\mathfrak{P}$ coincides with $P_3$), and $2$, indistinctly, the three consecutive collisions $\mathfrak{acb}$ or $\mathfrak{cab}$ ($\mathfrak{P}$ coincides with $P_2$).\\
For example, the sequence $1133$ stands for the collision sequence $(\mathfrak{ab})(\mathfrak{ab})(\mathfrak{cb})(\mathfrak{cb})$.

\paragraph{First orbits, for $r = 0.2$.} We start with the representation of ten orbits, in the particular case when $r = 0.2$, that is, slightly above any restitution coefficient associated to a feasible pattern discussed in the literature. To produce each of these orbits, we choose randomly the initial configuration in the strip, we compute $5000$ iterations of $\mathfrak{P}$, and we represent the last $2000$ iterations. We obtain Figure \ref{FIGUR10Orb_r=.2}.

\begin{figure}[h!]
\centering
\begin{subfigure}{0.45\textwidth}
    \includegraphics[trim = 0cm 0cm 0cm 0cm, width=\linewidth]{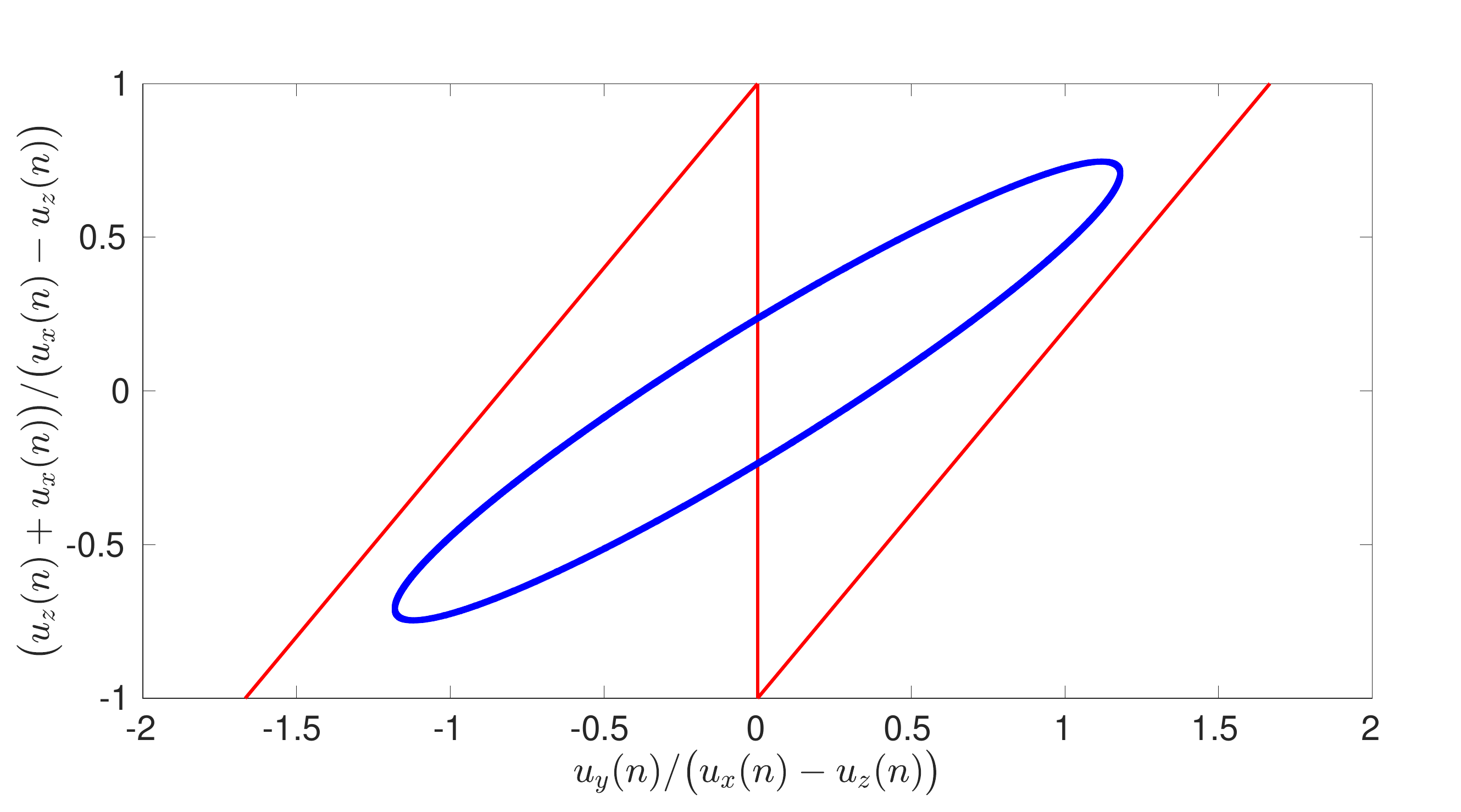}
    \caption{First simulation, quasi-periodic orbit.}
    \label{FIGURSngl1}
\end{subfigure}
\hfill
\begin{subfigure}{0.45\textwidth}
    \includegraphics[trim = 0cm 0cm 0cm 0cm, width=\linewidth]{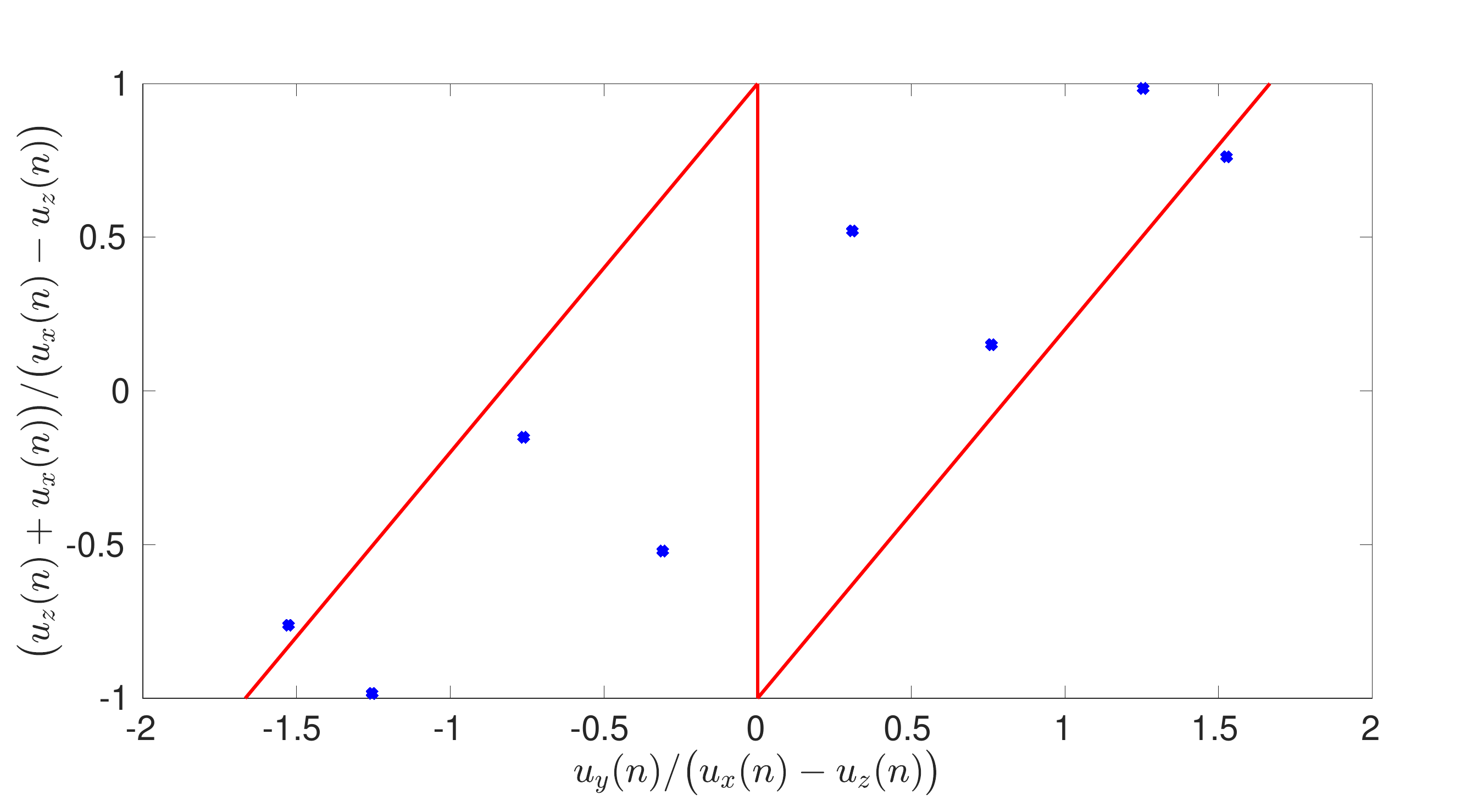}
    \caption{Second simulation, periodic orbit.}
    \label{FIGURSngl2}
\end{subfigure}

\centering
\begin{subfigure}{0.45\textwidth}
    \includegraphics[trim = 0cm 0cm 0cm 0cm, width=\linewidth]{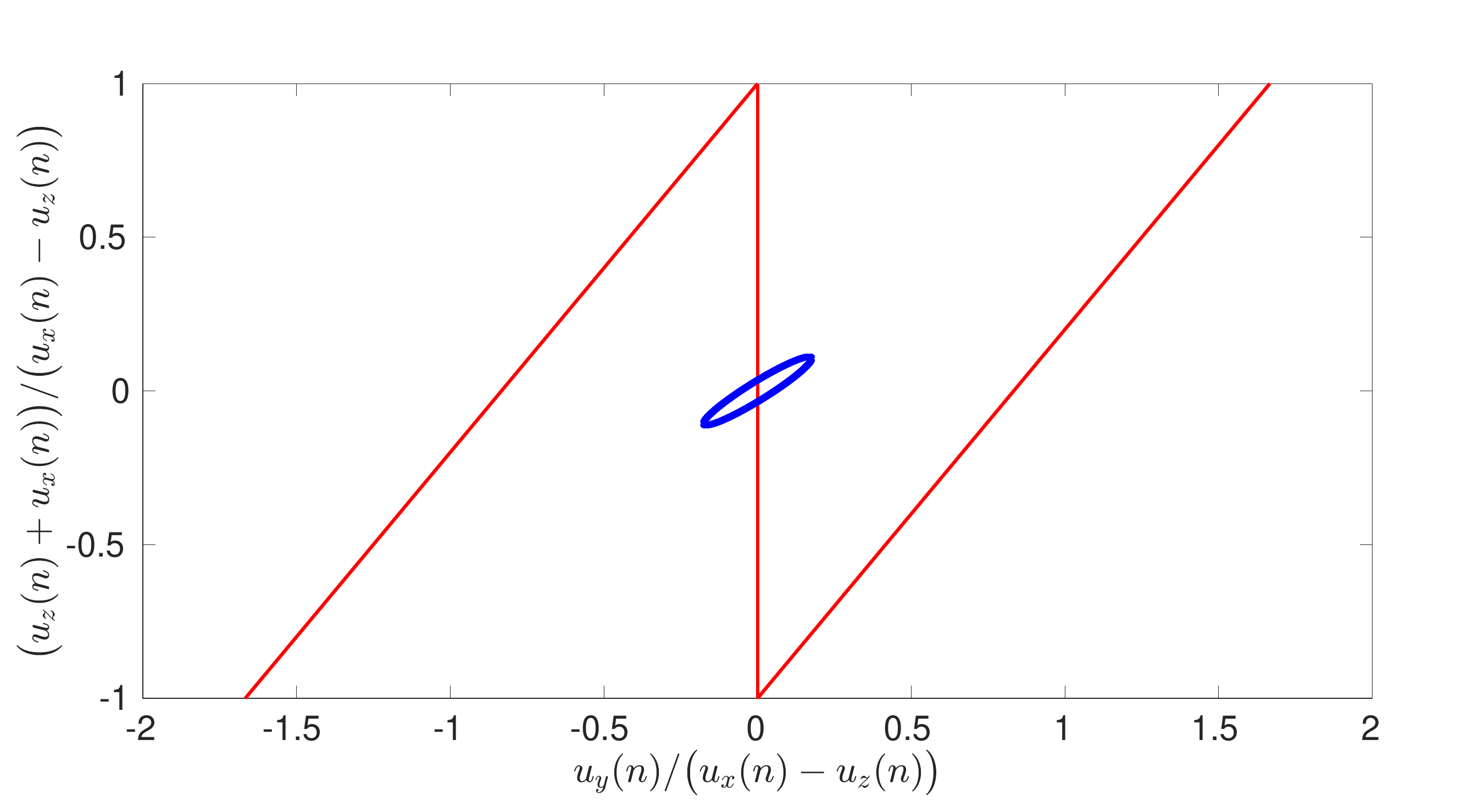}
    \caption{Third simulation, quasi-periodic orbit.}
    \label{FIGURSngl3}
\end{subfigure}
\hfill
\begin{subfigure}{0.45\textwidth}
    \includegraphics[trim = 0cm 0cm 0cm 0cm, width=\linewidth]{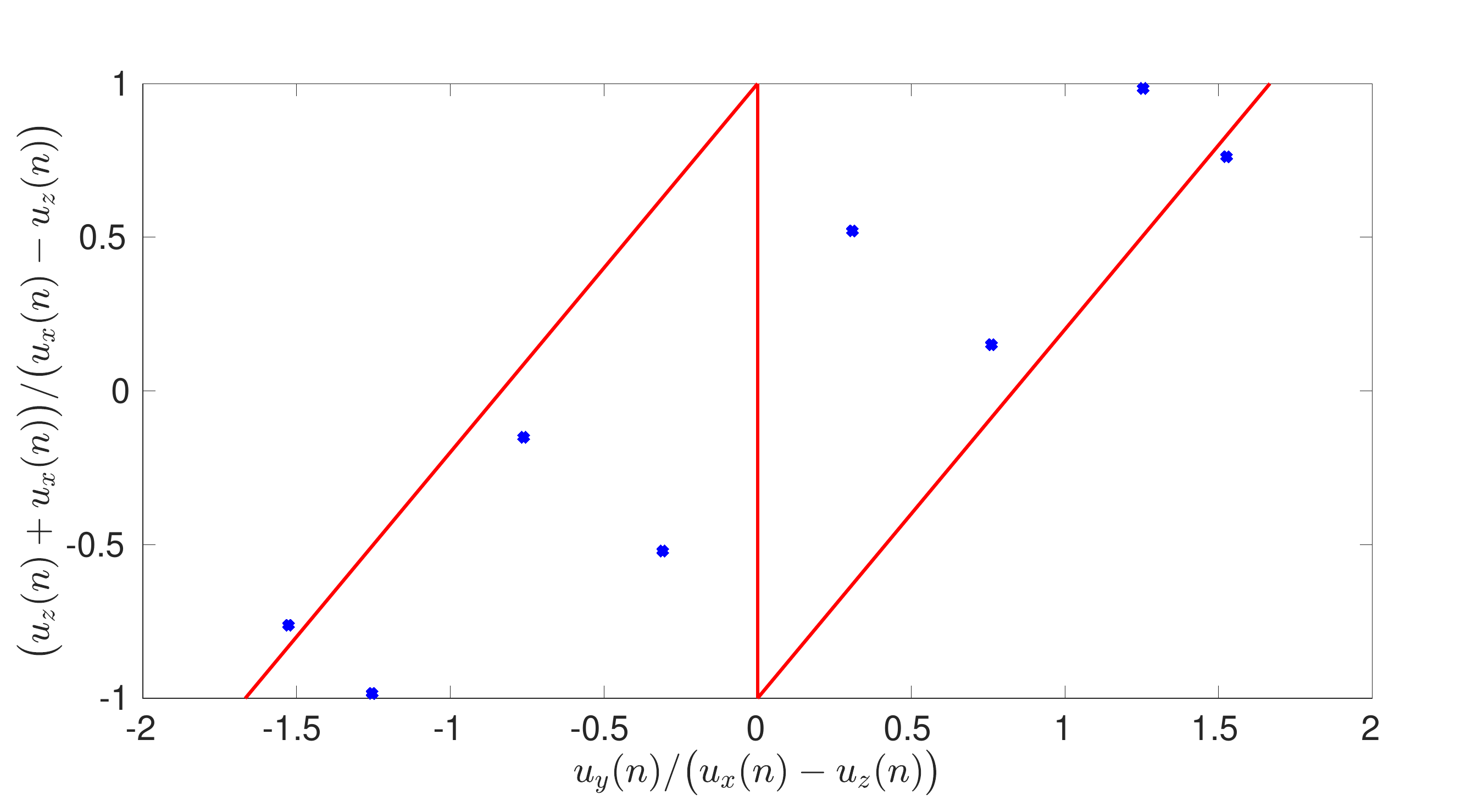}
    \caption{Fourth simulation, periodic orbit.}
    \label{FIGURSngl4}
\end{subfigure}

\centering
\begin{subfigure}{0.45\textwidth}
    \includegraphics[trim = 0cm 0cm 0cm 0cm, width=\linewidth]{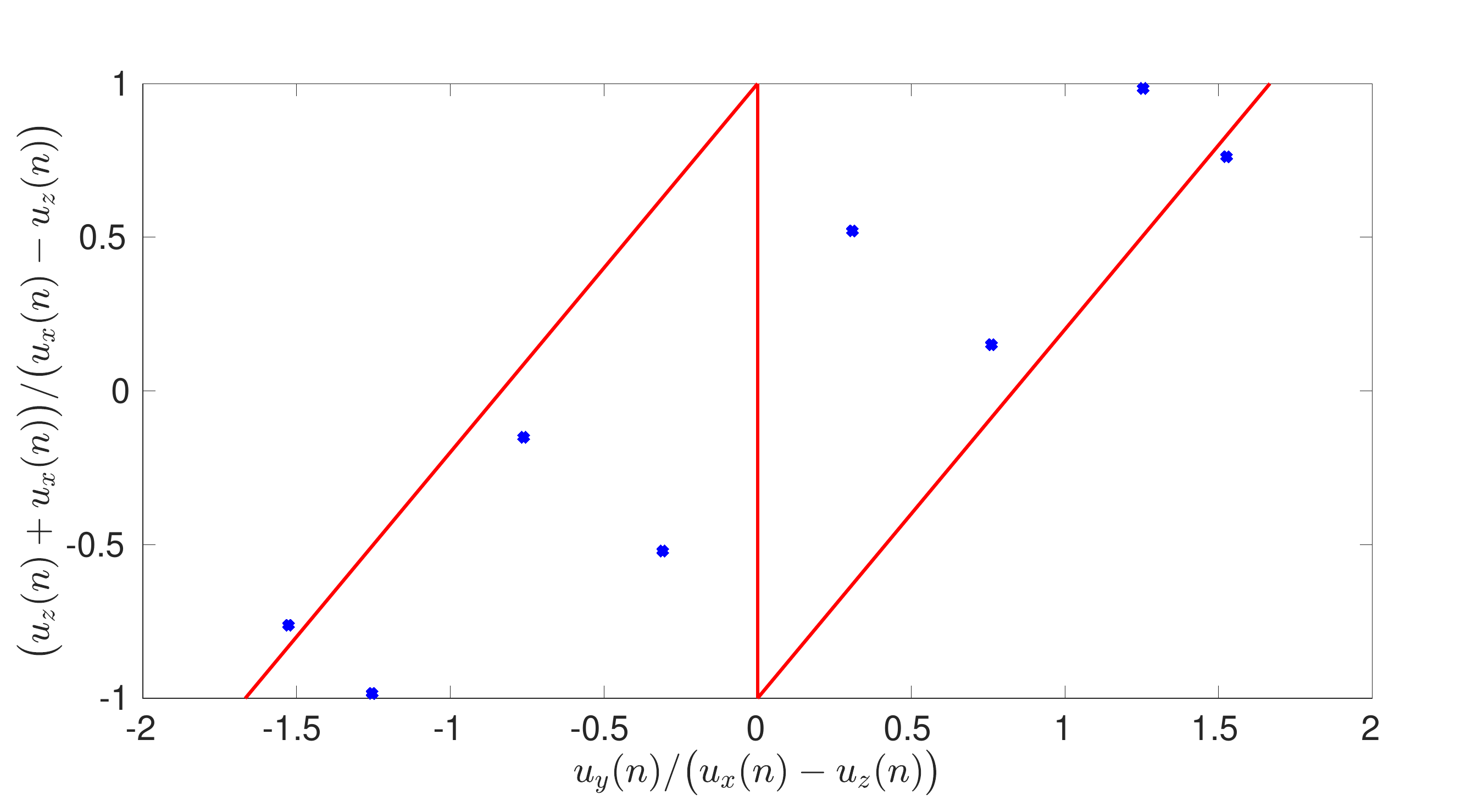}
    \caption{Fifth simulation, periodic orbit.}
    \label{FIGURSngl5}
\end{subfigure}
\hfill
\begin{subfigure}{0.45\textwidth}
    \includegraphics[trim = 0cm 0cm 0cm 0cm, width=\linewidth]{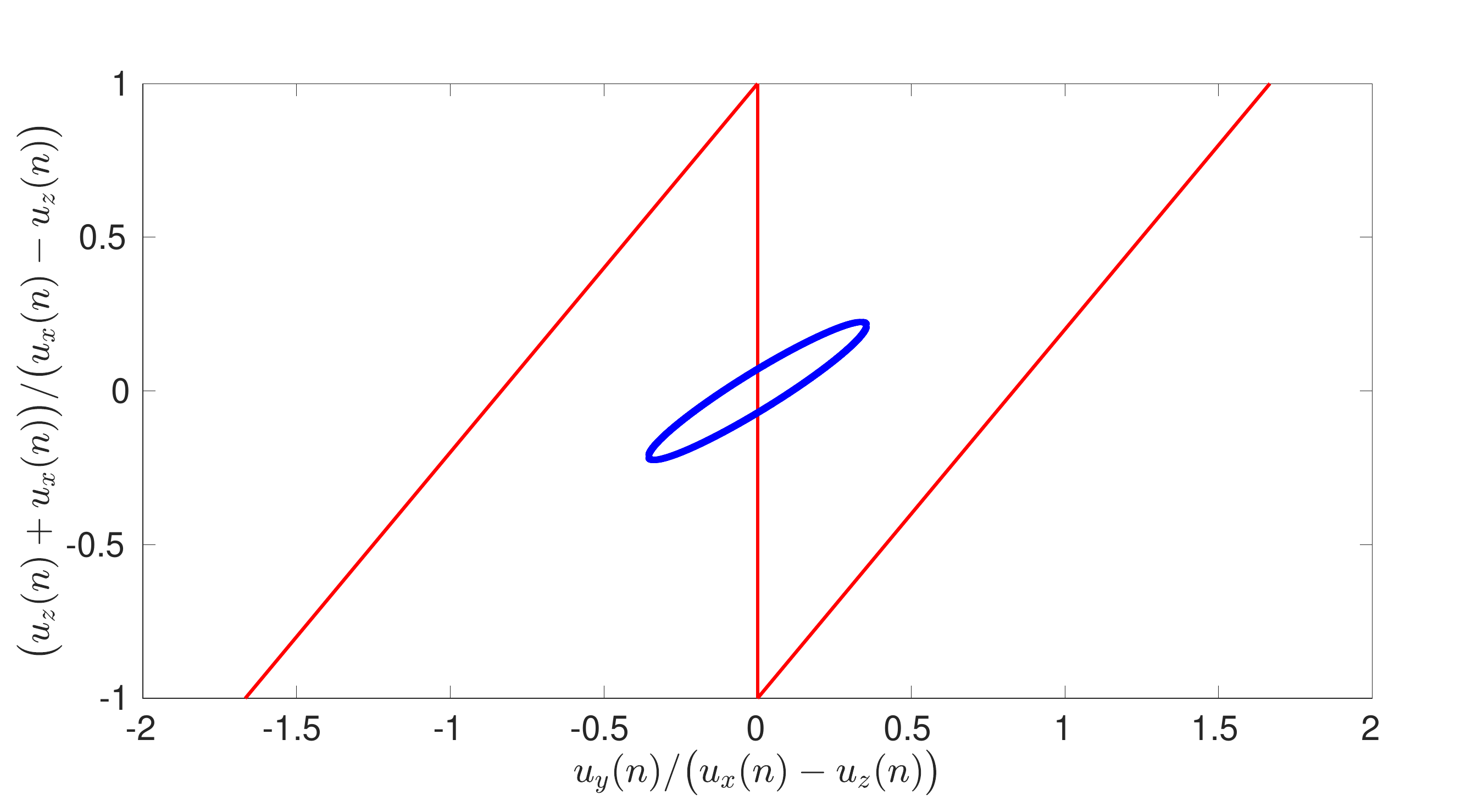}
    \caption{Sixth simulation, quasi-periodic orbit.}
    \label{FIGURSngl6}
\end{subfigure}

\centering
\begin{subfigure}{0.45\textwidth}
    \includegraphics[trim = 0cm 0cm 0cm 0cm, width=\linewidth]{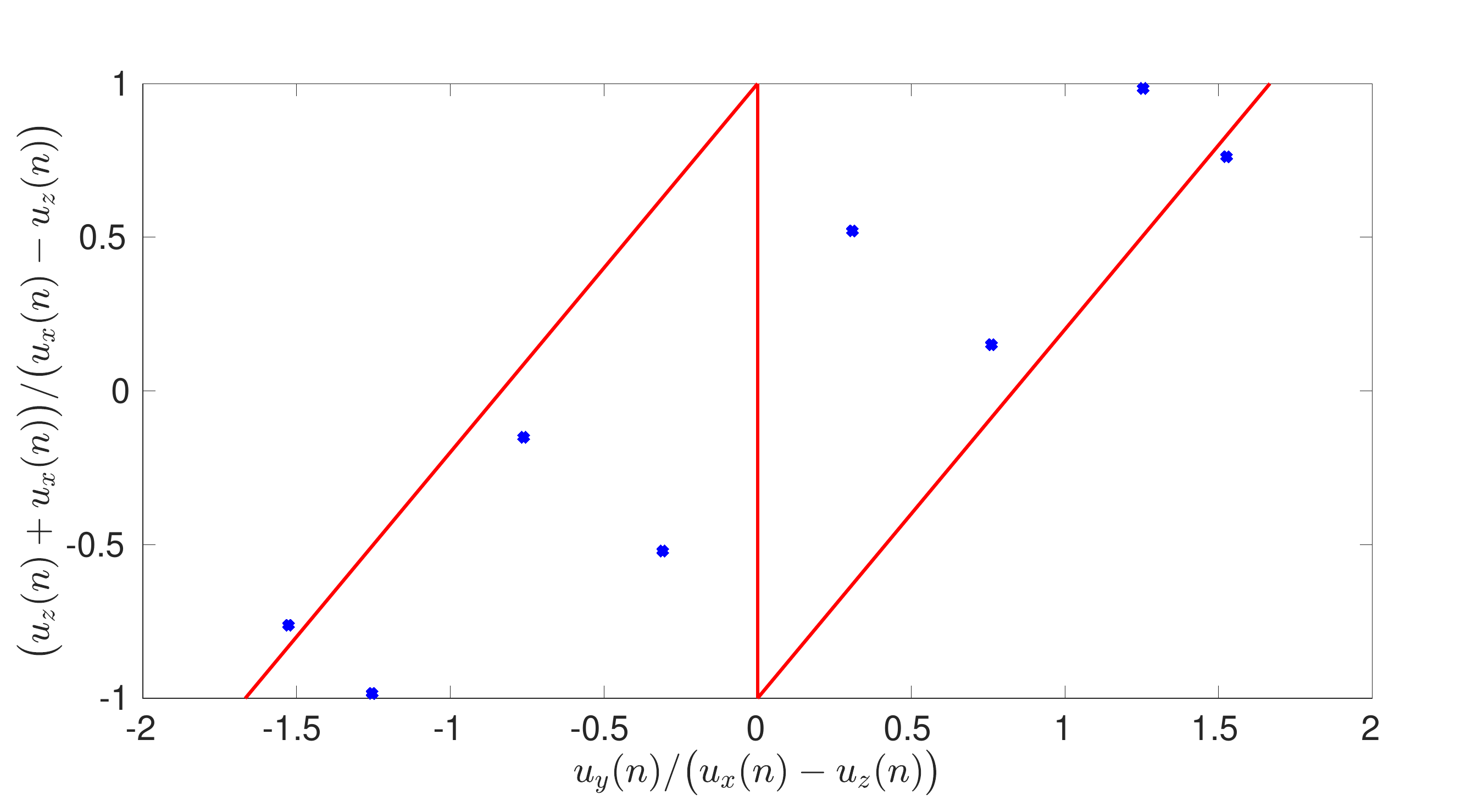}
    \caption{Seventh simulation, periodic orbit.}
    \label{FIGURSngl7}
\end{subfigure}
\hfill
\begin{subfigure}{0.45\textwidth}
    \includegraphics[trim = 0cm 0cm 0cm 0cm, width=\linewidth]{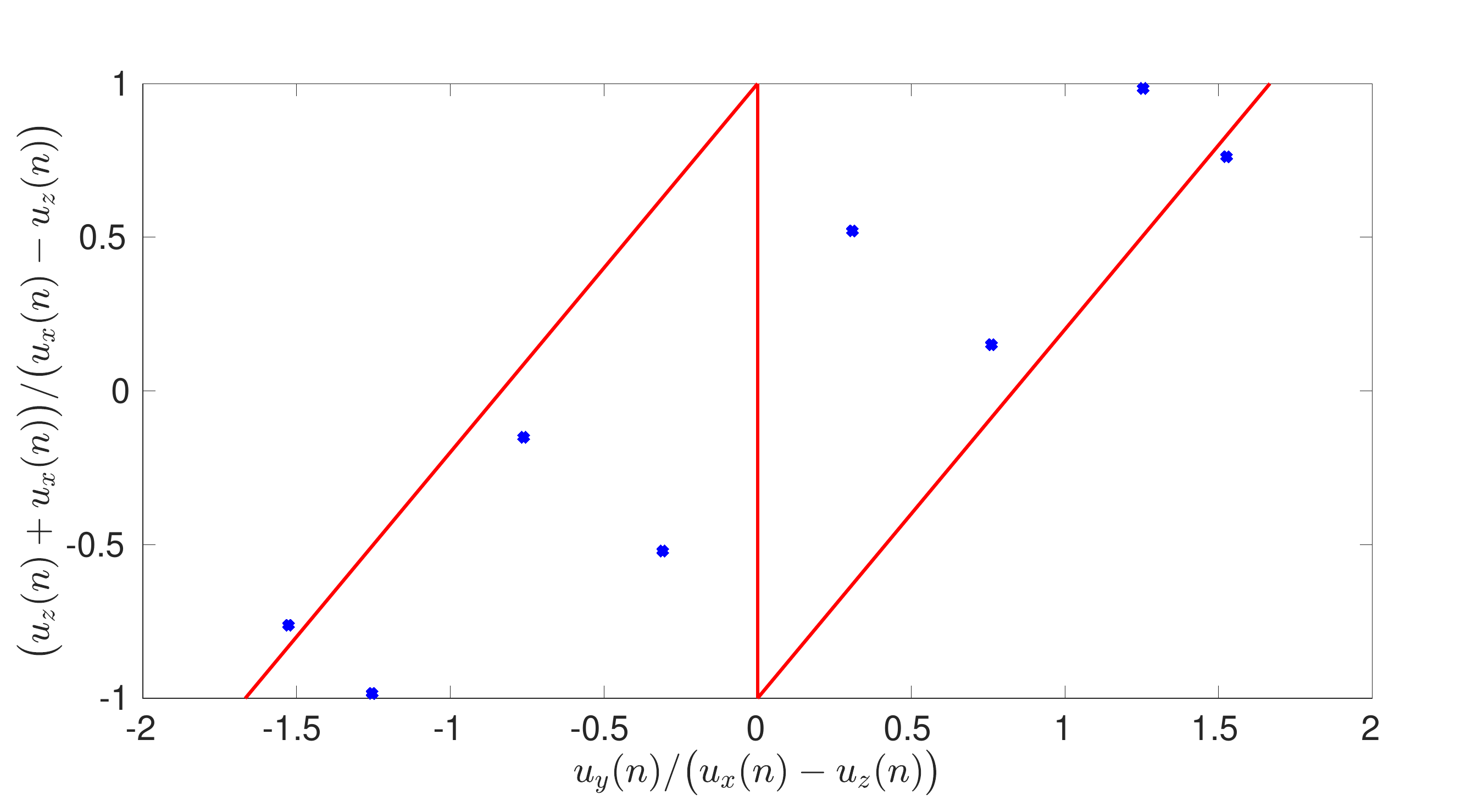}
    \caption{Eighth simulation, periodic orbit.}
    \label{FIGURSngl8}
\end{subfigure}

\centering
\begin{subfigure}{0.45\textwidth}
    \includegraphics[trim = 0cm 0cm 0cm 0cm, width=\linewidth]{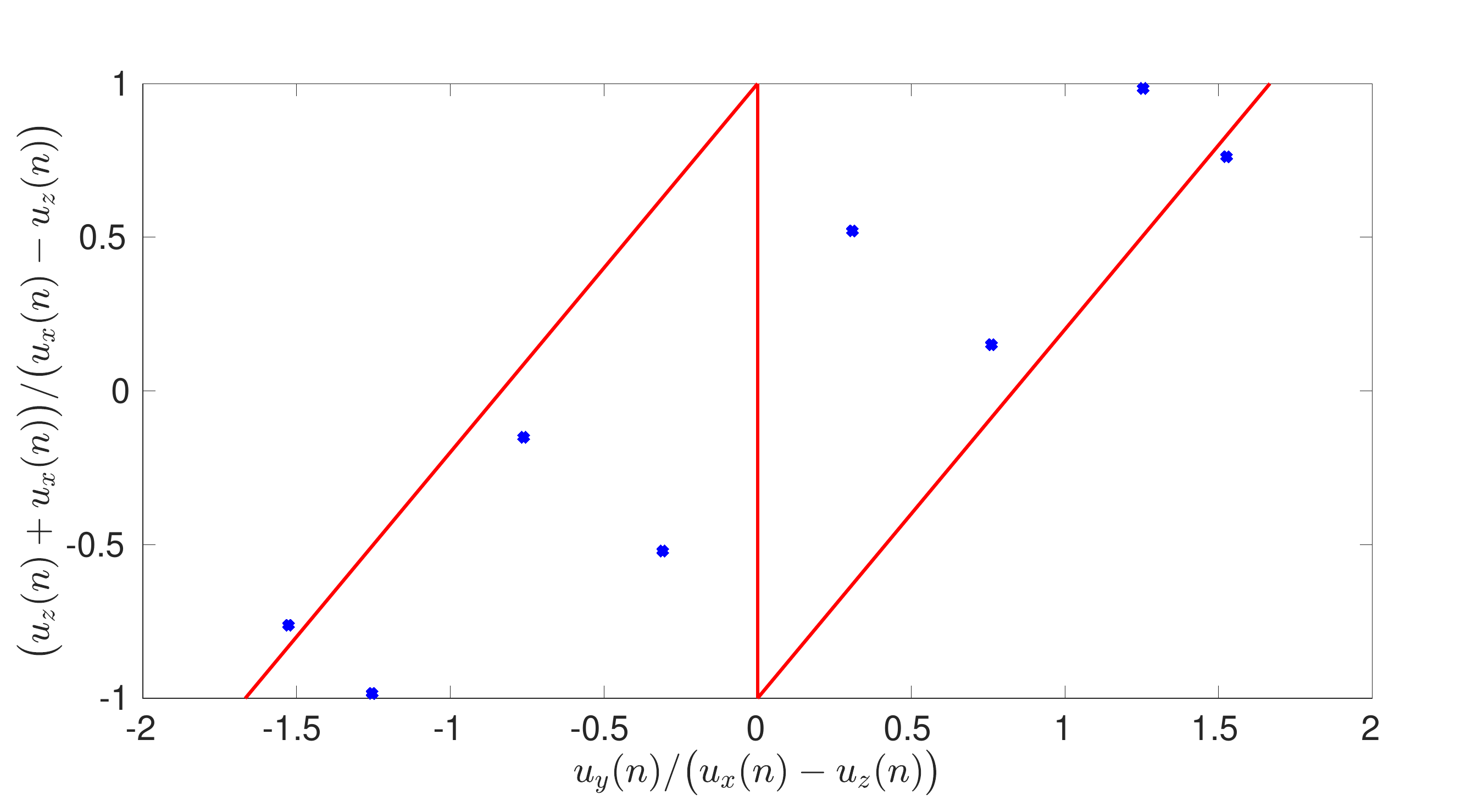}
    \caption{Ninth simulation, periodic orbit.}
    \label{FIGURSngl9}
\end{subfigure}
\hfill
\begin{subfigure}{0.45\textwidth}
    \includegraphics[trim = 0cm 0cm 0cm 0cm, width=\linewidth]{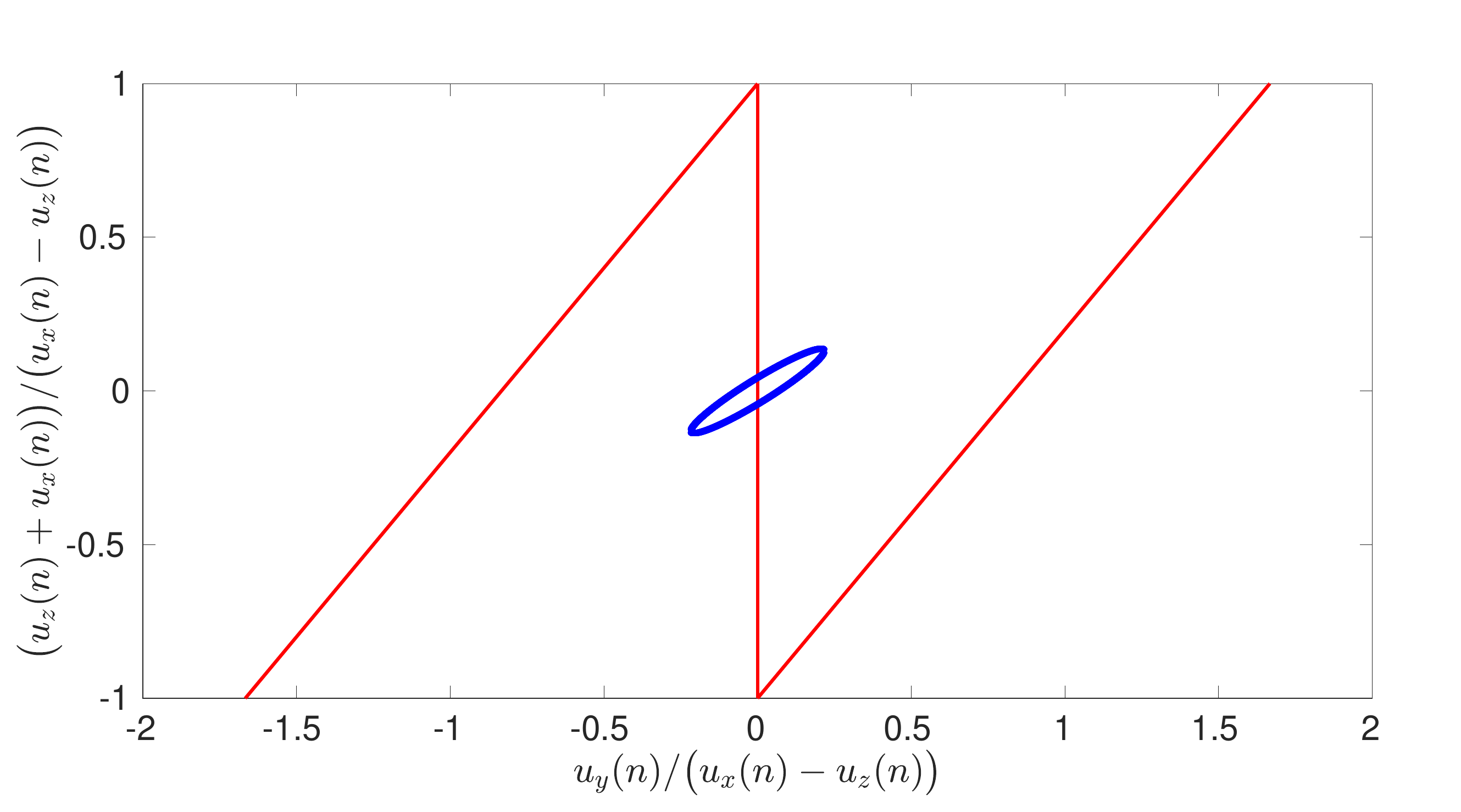}
    \caption{Tenth simulation, quasi-periodic orbit.}
    \label{FIGURSngl10}
\end{subfigure}
\caption{Plot of the tails (last $2000$ iterations) of the orbits of $\widehat{\mathfrak{P}}_r$ after $5000$ iterations, for $ r = 0.2$.}
\label{FIGUR10Orb_r=.2}
\end{figure}

\noindent
On the ten simulations, six of them exhibit orbits that seem to accumulate on a finite number of points, this set of points being independent from the initial datum (Figures \ref{FIGURSngl2}, \ref{FIGURSngl4}, \ref{FIGURSngl5}, \ref{FIGURSngl7}, \ref{FIGURSngl8} and \ref{FIGURSngl9}). In the case of these figures, the accumulation set seems to be composed of $8$ points. Nevertheless, when considering the trajectory, we discern the following period:
\begin{align}
1322231222.
\end{align}
In particular, at each period, the domains $\{w_2 < 2\alpha w_1 - 1\}$ and $\{w_2 > 2\alpha w_1 + 1\}$ are both visited twice by the seemingly periodic orbits. It is indeed the case: one has to notice that the orbits present two points per period with large $w_1$ coordinate (in absolute value) with respect to the other points, as it can be seen on Figure \ref{FIGURSngl2Resca}, which is another plot of the orbit obtained in \ref{FIGURSngl2}, with $-30 \leq w_1 \leq 30$, instead of $-2 \leq w_1 \leq 2$. 

\begin{figure}[h!]
\centering
\includegraphics[trim = 0cm 0cm 0cm 0cm, width=0.5\linewidth]{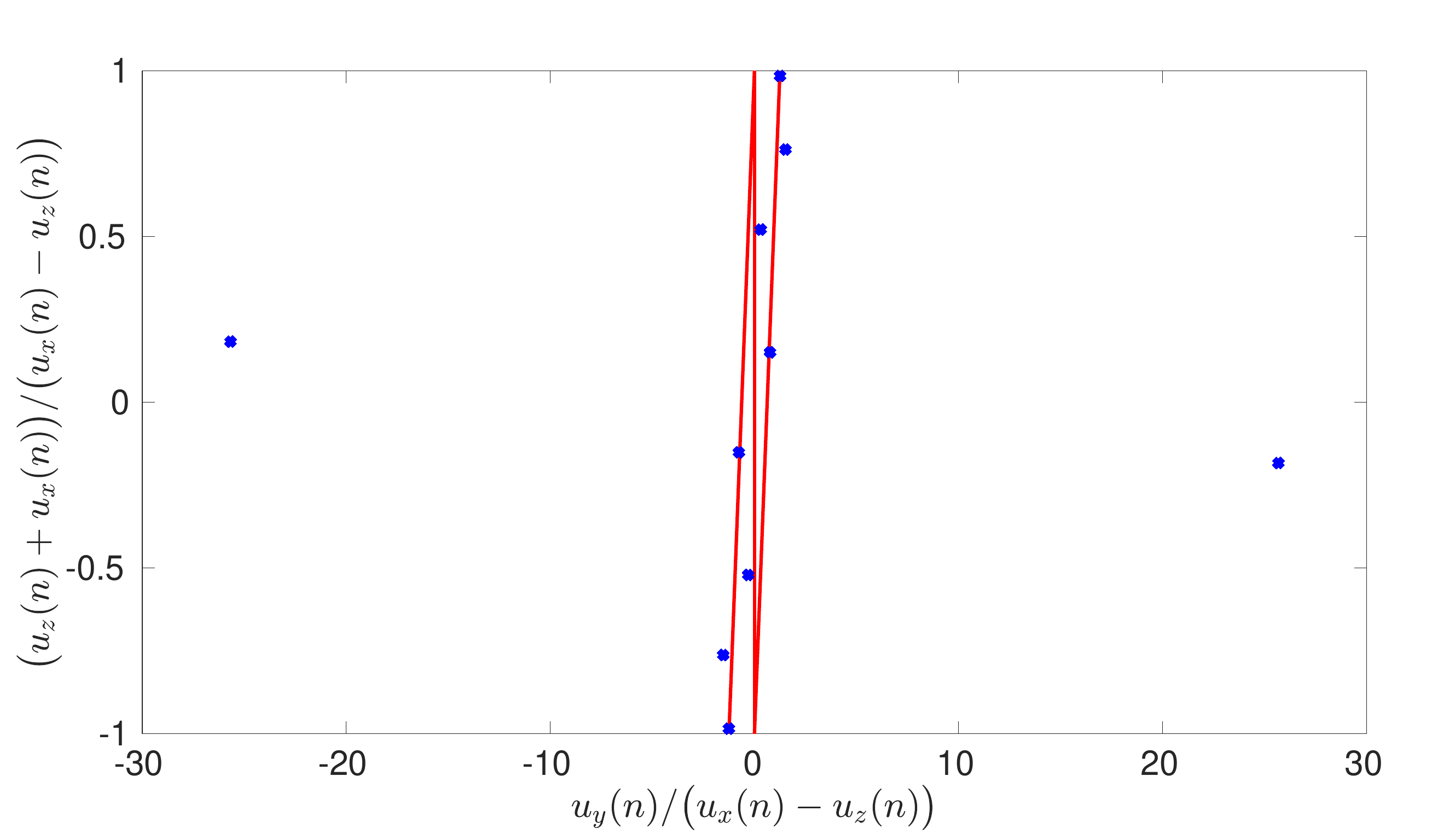} 
\caption{Rescaled plot of the orbit represented on Figure \ref{FIGURSngl2}.}
\label{FIGURSngl2Resca}
\end{figure}

\noindent
We detected therefore a candidate for a new periodic orbit, which was never discussed in the literature to the best of our knowledge, given by the collision sequence:

\begin{align}
\label{EQUATColli13(2^3)31(2^3)}
(\mathfrak{ab})(\mathfrak{cb})(\mathfrak{acb})(\mathfrak{acb})(\mathfrak{acb})(\mathfrak{cb})(\mathfrak{ab})(\mathfrak{acb})(\mathfrak{acb})(\mathfrak{acb}).
\end{align}

\noindent
Observe that this collision sequence is composed with $26$ consecutive collisions. The length of the collision sequence made it very hard to detect by brute force methods: indeed, without knowing a priori the length and the composition of a candidate periodic collision sequence, there are too many possible cases to test.\\
We observe also that the writing \eqref{EQUATColli13(2^3)31(2^3)} of the collision sequence might be inaccurate and is presented only for illustrative purposes: we did not differentiate between the collision sequences $\mathfrak{acb}$ and $\mathfrak{cab}$ in the recording of the trajectories, as they both correspond to an iteration of $\mathfrak{P}$ in the form of the matrix $P_2$. We remark that this non-differentiation occurs also at the level of the original particle system, at least, concerning the evolution of the velocity vector along the different collisions, since the two collision matrices $A$ and $C$ commute.

\paragraph{Individual orbits, for different restitution coefficients.} We repeated the procedure we described in the previous paragraph for several values of the restitution coefficient, focusing on the ``critical values'', that is, around $3-2\sqrt{2} \simeq 0.1716$ ($0.172 \leq r \leq 0.18$) and $0.1917$ ($0.19 \leq r \leq 0.25$), but not only: the interval $[0.3,0.34]$ presents interestingly pronounced thin stripes of stability, and the interval $[0.5,0.6]$ contains the value $2^{2/3}-1 \simeq 0.5874\ 0105\ 1968$, which is known to be an upper bound on the restitution coefficient, above which no collapse can take place in the original one-dimensional four particle system (see \cite{BeCa999} for a proof of this result).\\
Here again, at $r$ fixed we simulated ten different orbits, all generated from initial data randomly chosen in the strip. For each of the trajectories, we computed $5000$ iterations of $\widehat{\mathfrak{P}}_r$, and represented the last $2000$. When the orbits seemed to accumulate on only a finite number of points, we considered that a periodic orbit was detected. The complete description of our investigations is summarized in Table \ref{TABLEPerioPatteDetec}.\\

\begin{table}[h!]
\vspace{1cm}
\hspace{-3cm}
\begin{center}
  \begin{tabular}{c|c|c|c} 
    \toprule
    $r$ & Periodic orbits (\%) & Patterns detected & Length of the patterns \\
    \toprule
    0.1 & 100\% & $111333$ & 12\\
    0.101 & 50\% & $111333$ & 12 \\
    0.102 & 0\% & none & -- \\
    0.103 & 0\% & none & -- \\
    \midrule
    0.15 & 100\% & $1133$ & 8 \\
    0.16 & 100\% & $1133$ & 8 \\
    0.17 & 100\% & $1133$ & 8 \\
    \midrule
    0.172 & 90\% & $132^{36}$,\ $312^{36}$ & 112 \\
    0.174 & 70\% & $132^{14}$,\ $312^{14}$ & 46 \\
    0.176 & 70\% & $132^{10}$,\ $312^{10}$ & 34 \\
    0.178 & 70\% & $132^8$,\ $312^8$ & 28 \\
    0.18 & 70\% & $132^7312^7$ & 50 \\
    \midrule
    0.19 & 70\% & $132^4$,\ $312^4$ & 16 \\
    0.2 & 60\% & $132^3312^3$ & 26 \\
    0.21 & 50\% & $132^2$,\ $312^2$,\ $132^3312^3$ & 10,\ 26 \\
    0.22 & 50\% & $132^2$,\ $312^2$ & 10 \\
    0.23 & 40\% & $132^2$,\ $312^2$ & 10 \\
    0.24 & 50\% & $132^2$,\ $312^2$,\ $132312$ & 10,\ 14 \\
    0.25 & 40\% & $312^2$,\ $132312$ & 10,\ 14 \\
    \midrule
    0.3 & 30\% & $132312$ & 14 \\
    0.31 & 40\% & $132312$ & 14 \\
    0.32 & 20\% & $132312$ & 14 \\
    0.33 & 40\% & $132312$ & 14 \\
    0.34 & 30\% & $132312$ & 14 \\
    \midrule
    0.5 & 30\% & $132312$ & 14 \\
    0.51 & 30\% & $132312$ & 14 \\
    0.52 & 10\% & $132312$ & 14 \\
    0.53 & 30\% & $132312$ & 14 \\
    0.54 & 10\% & $132312$ & 14 \\
    \midrule
    0.57 & 10\% & $132312$ & 14 \\
    0.58 & 30\% & $132312$, $131312^3313132^3$ & 14,\ 38 \\
    0.59 & 30\% & $132312$ & 14 \\
    0.6 & 10\% & $132312$ & 14 \\
    \bottomrule
  \end{tabular}
   \caption{Periodic patterns detected along the two-dimensional simulations of individual orbits of $\big( \widehat{\mathfrak{P}}_r^n \big)_n$.}
  \label{TABLEPerioPatteDetec}
\end{center}
\end{table}

\noindent
The orbits obtained for $0.1 \leq r \leq 0.103$ and $0.15 \leq r \leq 0.17$ are consistent with the results of \cite{CDKK999} and \cite{HuRo023}: we identified the familiar periodic patterns $(\mathfrak{ab})^2(\mathfrak{cb})^2$ and $(\mathfrak{ab})^3(\mathfrak{cb})^3$. For these ranges of restitution coefficients, we mention also that in the case when these patterns were not detected, we were unable to discover any regularity in the order of the collisions. This last observation supports once again the conjecture of the existence of a chaotic behaviour of the dynamical system between the windows of stability of the patterns $(\mathfrak{ab})^n(\mathfrak{cb})^n$.\\
\newline
In addition to the new pattern $132^3312^3$ already discussed in the previous paragraph concerning $r = 0.2$, our numerical investigations provided other new patterns, that we can regroup in three distinct families:

\begin{align}
\begin{split}
13&2^{n_1} \hspace{3mm} \text{(and its ``symmetry'' } 312^{n_1}\text{)}, \hspace{10mm} n_1 \in \mathbb{N}^*, \label{EQUATPattr_AntiPalin}
\end{split}\\
\begin{split}
132^{n_2}&312^{n_2}, \hspace{52.5mm} n_2 \in \mathbb{N}^*, \label{EQUATPattrPalinSimpl}
\end{split}\\
\begin{split}
131312^{n_3}&313132^{n_3}, \hspace{47.35mm} n_3 \in \mathbb{N}^* \label{EQUATPattrPalinCmplx}.
\end{split}
\end{align}

\noindent
More precisely, we found the pattern \eqref{EQUATPattr_AntiPalin} for the particular values of $n_1 = 2$, $4$, $8$, $10$, $14$ and $36$. We found the pattern \eqref{EQUATPattrPalinSimpl} for the particular values of $n_2 = 1$, $3$ and $7$. Finally, we detected the more complex pattern \eqref{EQUATPattrPalinCmplx} only for the particular value of $n_3 = 3$.\\
In order to discuss the question of the larger values of the exponents $n_i$, it is interesting to consider the case when $r$ is taken close to the critical value $3-2\sqrt{2}$. Taking in particular $r = 0.1717$, we obtain the orbit represented on Figure \ref{FIGURSingl_r=0.1717_}.

\begin{figure}[h!]
\centering
\includegraphics[trim = 0cm 0cm 0cm 0cm, width=0.75\linewidth]{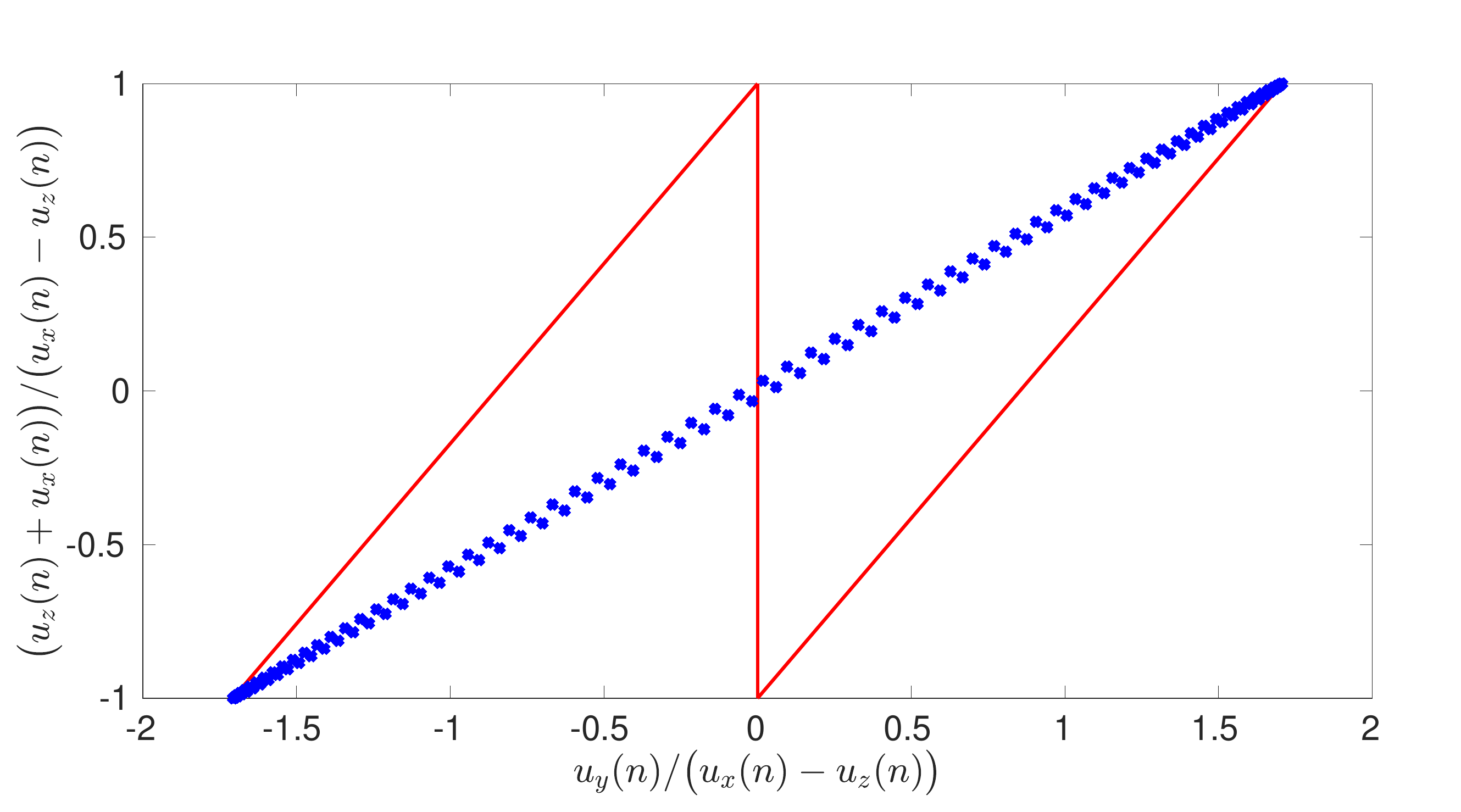} 
\caption{Plot of the tails (last $2000$ iterations) of the orbits of $\widehat{\mathfrak{P}}_r$ after $5000$ iterations, for $r = 0.1717$.}
\label{FIGURSingl_r=0.1717_}
\end{figure}

\noindent
Despite the large number of accumulation points, a careful inspection of Figure \ref{FIGURSingl_r=0.1717_} shows that the orbit appears to be periodic. When considering the trajectory, it is indeed possible to see that the collision pattern $132^{67}312^{67}$ takes place repeatedly. Such a collision pattern presents $138$ collisions of type $\mathfrak{b}$, and $410$ collisions in total, considering the types $\mathfrak{a}$, $\mathfrak{b}$ and $\mathfrak{c}$ together. At least for the family \eqref{EQUATPattrPalinSimpl} of patterns, it seems therefore that we can find configurations leading to collapsing sequences with arbitrarily large values of $n_2$.

\paragraph{Comments and conjectures.} Since we observed the patterns recorded in Table \ref{TABLEPerioPatteDetec} in numerical simulations, they should be associated with respective basins of attraction with non empty interiors. It is remarkable that different (stable) collision patterns can coexist for a fixed restitution coefficient, answering an open question formulated in \cite{CDKK999}. The fact that $132^{n_1}$ and $312^{n_1}$ coexist is not surprising, since if one exists, the other should also exist by symmetry. It is more surprising that patterns of two different families can coexist, for instance $132^2$ and $132^3312^3$, in the case when $r = 0.21$. This coexistence exhibits the rich behaviour of the dynamical system $\big( \widehat{\mathfrak{P}}_r^n \big)_n$.\\
We observe also that all the patterns regrouped in Table \ref{TABLEPerioPatteDetec} are symmetric, in the sense introduced in \cite{DoHR025}, meaning that along a period, the collisions of type $\mathfrak{a}$ and $\mathfrak{c}$ take place the same number of times. However, we discovered patterns that are not palindromic, in the sense that reversing the collision order provides a different period. This is the case with the family $132^{n_1}$. It is possible nevertheless that periodic collision patterns that are not symmetric, and that are different from $113$ (discussed in \cite{DoHR025}) exist. But since $113$ was not stable, it is likely that if other asymmetric collision patterns exist, they would also be unstable. This general, albeit speculative, property of the asymmetric collision patterns would explain why none of them were detected in our numerical simulations.\\
Finally, the length of the detected collision patterns is larger and larger as $r$ gets closer to the critical restitution coefficient $3-2\sqrt{2}$. This is perfectly consistent with the intricate structure of the apparent periodic orbits observed for similar restitution coefficients, as pictured on Figure \ref{FIGUR_r=_0.171_0.175}.\\
Based on the patterns discovered in the simulations and summarized in Table \ref{TABLEPerioPatteDetec}, we conjecture that each of the three families of patterns \eqref{EQUATPattr_AntiPalin}, \eqref{EQUATPattrPalinSimpl} and \eqref{EQUATPattrPalinCmplx} should contain infinitely many different patterns. Nevertheless, we will see in Section \ref{SSECTPatter_13223122} that some choices of the exponents $n_i \in \mathbb{N}^*$ ($1 \leq i \leq 3$) correspond to patterns that can never be realized by a periodic orbit in a stable manner.

\section{Mathematical study of some fixed points of $\widehat{\mathfrak{P}}_r$}
\label{SECTIMatheStudyPatte}

\noindent
In this section, based on the periodic patterns detected by the numerical investigations presented in Section \ref{SSECTSimulIndivOrbit}, we will study rigorously such periodic patterns, corresponding to fixed points of the mapping $\big(\widehat{\mathfrak{P}}_r\big)^n$, for a certain period $n \in \mathbb{N}^*$. More precisely, we will investigate the respective ranges of existence and stability of two different patterns, the first being $132312$, observed for many restitution coefficients stretching from $0.24$ and $0.6$, the second being $132^2312^2$, which was not observed in the simulations of individual orbits.\\
The objective of the present section is then twofold: first, confirm that the numerical observations reflects indeed the existence of stable periodic orbits for $\big(\widehat{\mathfrak{P}}_r\big)^n$, and second, to understand better the candidate periodic patterns, belonging to the countable families \eqref{EQUATPattr_AntiPalin}-\eqref{EQUATPattrPalinCmplx}, that were not observed in numerical simulations. In both cases, we will rely on a rigorous mathematical study of particular orbits of $\widehat{\mathfrak{P}}_r$.\\
We observe that we will adapt the tools developed in \cite{CDKK999} to study the self-similar collapses. In our case, the study turns out to be simpler, as the stability of the periodic orbits of $\widehat{\mathfrak{P}}_r$ can be performed by studying only the piecewise linear mapping $\mathfrak{P}$. In other words, in the present case we have only to study the properties of the products of matrices $P_i$ (which corresponds essentially to the study of the matrices $A$, $B$ and $C$ in \cite{CDKK999}), whereas in \cite{CDKK999} this study has to be supplemented with a careful investigation of the evolution of the relative positions between the physical particles.

\subsection{The pattern $132312$}
\label{SSECTPatte__132312__}

\noindent
The periodic pattern $132312$ is associated to the product of matrices $\widetilde{P}_{132312} = P_2P_1P_3P_2P_3P_1$. This periodic pattern for the mapping $\widehat{\mathfrak{P}}_r$ is realized if there exists a fixed point of the mapping:
\begin{align}
u \in \mathbb{S}^2 \cap X \mapsto \frac{\widetilde{P}_{132312}u}{\vert \widetilde{P}_{132312}u \vert} \in \mathbb{S}^2 \cap X.
\end{align}
Such a fixed point is then necessarily an eigenvector of the collision matrix $\widetilde{P}_{132312}$. Therefore, we will focus our attention on the eigenvectors of such a matrix. Nevertheless, not all the eigenvectors correspond to a fixed point, because such an eigenvector needs first to belong to the domain $X$ of the mapping $\mathfrak{P}_r$ in order to correspond to a meaningful initial configuration of the particle system. In addition, $X$ is partitioned into the respective domains of the three collisions matrices $P_i$, $1 \leq i \leq 3$, and so, depending on where the eigenvector $u \in \mathbb{S}^2$ lies in $X$, $\widehat{\mathfrak{P}}_r(u)$ is necessarily equal to $\frac{P_i u}{\vert P_i u \vert}$, for a certain $i$ depending on $u$. This first collision matrix needs necessarily to be $P_1$ in order to have a fixed point of $\widetilde{P}_{132312}$, which provides another necessary condition to be fulfilled by the eigenvector $u$. Similarly, the image of $\frac{P_1 u}{\vert P_1 u \vert}$ by $\widehat{\mathfrak{P}}_r$ has to match with $\frac{P_3 P_1 u}{\vert P_3 P_1 u \vert}$, that is, $\frac{P_1 u}{\vert P_1 u \vert}$ has to belong to the domain of the collision matrix $P_3$, and so on. As a consequence, an eigenvector of $\widetilde{P}_{132312}$ will generate the periodic orbit of period $132312$ only if $6$ conditions are fulfilled, one for each iteration of the mapping $\widehat{\mathfrak{P}}_r$. These conditions are given by the inequalities defining the respective domains of the matrices $P_i$, $1 \leq i \leq 3$. We will call such conditions the \emph{feasibility inequalities}.

\paragraph{Reduction of the feasibility inequalities for palindromic patterns.} In the case of the period $132312$, as well as any period of the two families of patterns \eqref{EQUATPattrPalinSimpl} and \eqref{EQUATPattrPalinCmplx}, we can simplify the study of the fixed points. We will detail the argument only in the case of the period $132312$, but the argument can be applied directly also to any palindromic pattern such that its collision matrix $P$ satisfies $P = \big( J P_\text{reduc} \big)^2$ for a certain matrix $P_\text{reduc}$ corresponding to half of the collisions of the period, and where $J$ is defined in \eqref{EQUATDefin__J__}. In the case of $132312$, the simplification is as follows. We have:
\begin{align}
P_2 P_1 P_3 P_2 P_3 P_1 = \big( J P_2 P_3 P_1 \big)^2.
\end{align}
We will therefore study the spectrum of the matrix $J P_2 P_3 P_1$, instead of $\widetilde{P}_{132312}$, as an eigenvector of $J P_2 P_3 P_1$ is also an eigenvector of $\widetilde{P}_{132312}$. The advantage of this approach is that there are less constraints to be checked in order to verify that an eigenvector of $J P_2 P_3 P_1$ generates a periodic orbit of $\widehat{\mathfrak{P}}_r$ of period $132312$. Indeed, if we prove that such an eigenvector $u$ is such that it belongs to the domain of $P_1$, such that $P_1 u$ belongs to the domain of $P_3$, and such that $P_3 P_1 u$ belongs to the domain of $P_2$, then:
\begin{align*}
\big(\widehat{\mathfrak{P}}_r\big)^3(u) = \frac{P_2 P_3 P_1 u}{\vert P_2 P_3 P_1 u \vert}.
\end{align*}
If in addition $u = (u_x,u_y,u_z)$ (assuming without loss of generality that $u_x > 0$) is an eigenvector of $J P_2 P_3 P_1$, associated to a certain eigenvalue $\lambda_u$, we have:
\begin{align}
P_2 P_3 P_1 u = \lambda_u J u = \lambda_u \begin{pmatrix} u_z \\ u_y \\ u_x \end{pmatrix}.
\end{align}
Assuming that $u$ belongs to the domain of $P_1$ implies in particular that we have $u_z > 0$, $u_y > 0$, and $\alpha u_y - u_x > 0$. Therefore, in the case when $\lambda_u < 0$, $\lambda_u Ju$ satisfies $\lambda_u (Ju)_y \leq 0$, and $\alpha \lambda_u (Ju)_y - \lambda_u (Ju)_z = \lambda_u \big( \alpha u_y - u_x \big) < 0$, and therefore $\lambda_u J u$ belongs to the domain of the collision matrix $P_3$. Considering $P_3 P_2 P_3 P_1 u = \lambda_u P_3 J u = \lambda_u J P_1 u$, the same arguments provides that if $P_1 u$ belongs to the domain of $P_3$, then $P_3 P_2 P_3 P_1 u$ belongs to the domain of the matrix $P_1$, and similarly $P_1 P_3 P_2 P_3 P_1 u$ belongs to the domain of the collision matrix $P_2$ if $P_3 P_1 u$ belongs to the domain of $P_2$. In other words, if an eigenvector $u$ of the matrix $J P_2 P_3 P_1$, associated to a negative eigenvalue, satisfies the three first feasibility inequalities of the period $132312$, then $u$ satisfies also the last three feasibility inequalities, ensuring therefore that $u$ generates the periodic orbit $132312$ when applying repeatedly the mapping $\widehat{\mathfrak{P}}_r$.\\
Concerning the stability of the periodic orbits, choosing the eigenvector of $\widetilde{P}_{132312}$ associated to the dominating eigenvalue (if such a dominating eigenvalue is real), we obtain a fixed point to $\big(\widehat{\mathfrak{P}}_r\big)^6$ which is locally stable, in the sense that any orbit starting close to the fixed point remains close to this fixed point. If in addition there exists a unique dominating eigenvalue, the fixed point is locally attracting. Therefore, we will study only the eigenvector of $J P_2 P_3 P_1$ associated to the dominating eigenvalue of this matrix.\\
We obtain then the following result.

\begin{theor}[Stability of the periodic orbit $132312$]
\label{THEORStabi____132312}
Let $\lambda_{132^J}^\text{dom}$ be the dominating eigenvalue of the matrix $J P_2 P_3 P_1$, always real for any $0 \leq r \leq 1$, and let $u_{\lambda_{132^J}^\text{dom}}$ be its associated eigenvector
. Let $r_{\text{crit},132^J}$ be the only root in $]0,1[$ of the equation:
\begin{align}
\alpha \big( u_{\lambda_{132^J}^\text{dom}} \big)_y - \big( u_{\lambda_{132^J}^\text{dom}} \big)_x = 0,
\end{align}
so that we have:
\begin{align}
\label{EQUATTheorr_cr2_132J}
r_{\text{crit},132^J} \simeq 0.2200\ 6978\ 6146\ 3104\ 7521.
\end{align}
Then, for any $r \in\ ]r_{\text{crit},132^J},1[$, there exists a single point in the domain of $\widehat{\mathfrak{P}}_r$ which generates the periodic orbit $132312$ and this orbit is locally stable. This point is given by $u_{\lambda_{132^J}^\text{dom}} / \vert u_{\lambda_{132^J}^\text{dom}} \vert$.
\end{theor}

\begin{remar}
Below $r_{\text{crit},132^J}$, approximately given by \eqref{EQUATTheorr_cr2_132J}, the periodic orbit $132312$ cannot be stable. The existence of such an unstable orbit can be studied by investigating the feasibility inequalities of the eigenvectors that are associated with the eigenvalues of $J P_2 P_3 P_1$ that are not dominating.\\
The result of Theorem \ref{THEORStabi____132312} is consistent with the numerical simulations recorded in Table \ref{TABLEPerioPatteDetec}. Indeed, the periodic pattern $132312$ is not observed for $r$ below $0.24$, but is observed for any $r \geq 0.24$. The fact that $132312$ was not observed for $r = 0.23$ in Table \ref{THEORStabi____132312} can be explained by the fact that the basin of attraction of the periodic orbit might be relatively small compared to the other orbits that were detected for this particular restitution coefficient.
\end{remar}

\begin{proof}[Proof of Theorem \ref{THEORStabi____132312}]
The characteristic polynomial $\chi_{\widetilde{P}_{132^J}}$ of the reduced collision matrix $\widetilde{P}_{132^J} = J P_2 P_3 P_1$ (with $J$ defined in \eqref{EQUATDefin__J__}) is given by the expression:
\begin{align}
\chi_{\widetilde{P}_{132^J}}(\lambda) &= \lambda^3 + \frac{\big(7r^6-24r^5-29r^4+21r^2-8r+1\big)}{32} \lambda^2 \nonumber\\
&\hspace{50mm}+ \frac{\big(r^{11}-8r^{10}+21r^9-29r^7-24r^6+7r^5\big)}{32} \lambda + r^{11}.
\end{align}
The eigenvalues of the collision matrix $\widetilde{P}_{132^J}$ are plotted on Figure \ref{FIGUREigenvalue_132J} when they are real, as well as their respective moduli on Figure \ref{FIGURModulEigen_132J} in the case when they are real or complex.

\begin{figure}[h!]
\centering
\begin{subfigure}{1\textwidth}
    \includegraphics[trim = 0cm 0cm 0cm 0cm, width=\linewidth]{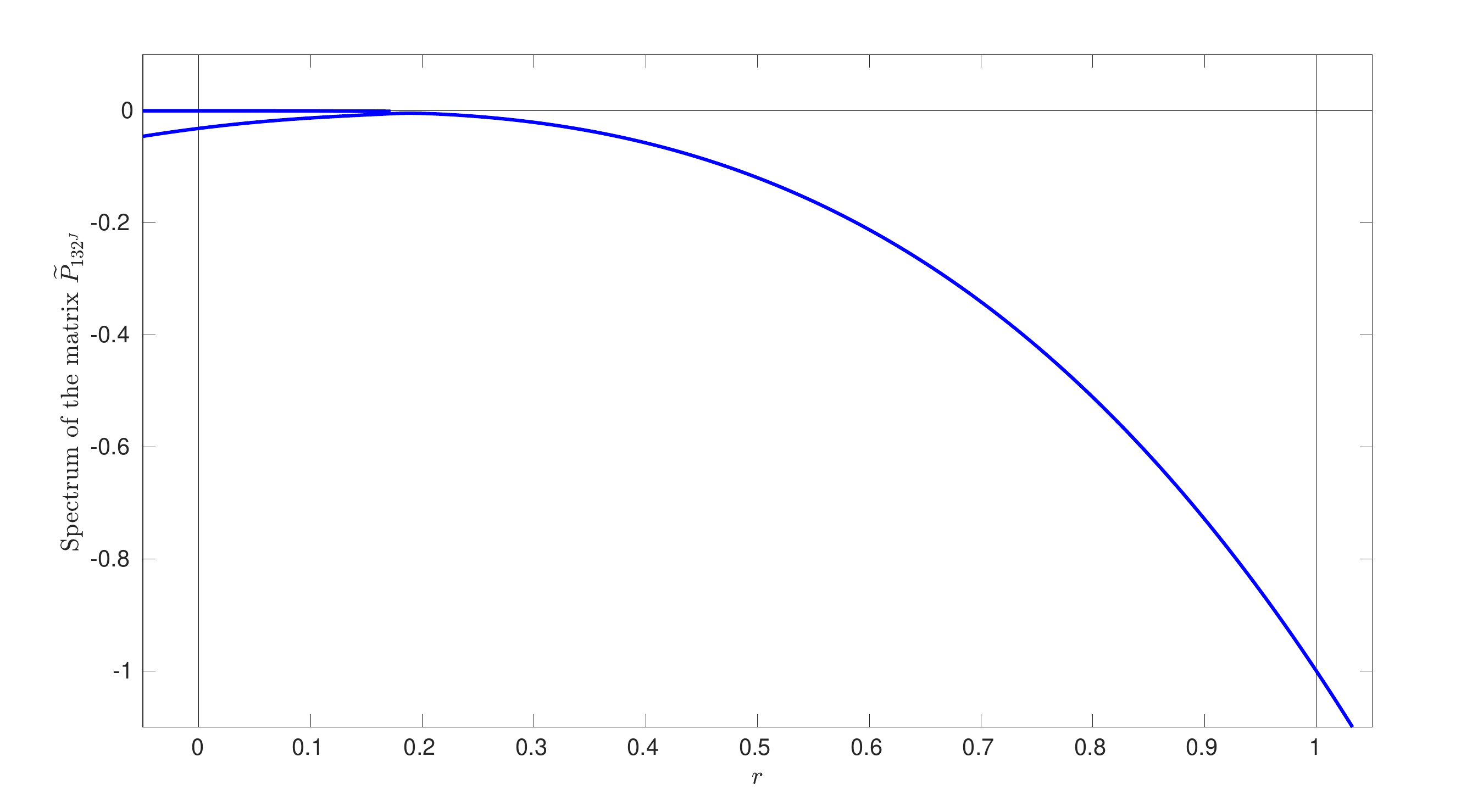}
    \caption{Spectrum (when real) for $0 \leq r \leq 1$.}
    \label{FIGUREigenvalue_132J_Glob}
\end{subfigure}
\begin{subfigure}{1\textwidth}
    \includegraphics[trim = 0cm 0cm 0cm 0cm, width=\linewidth]{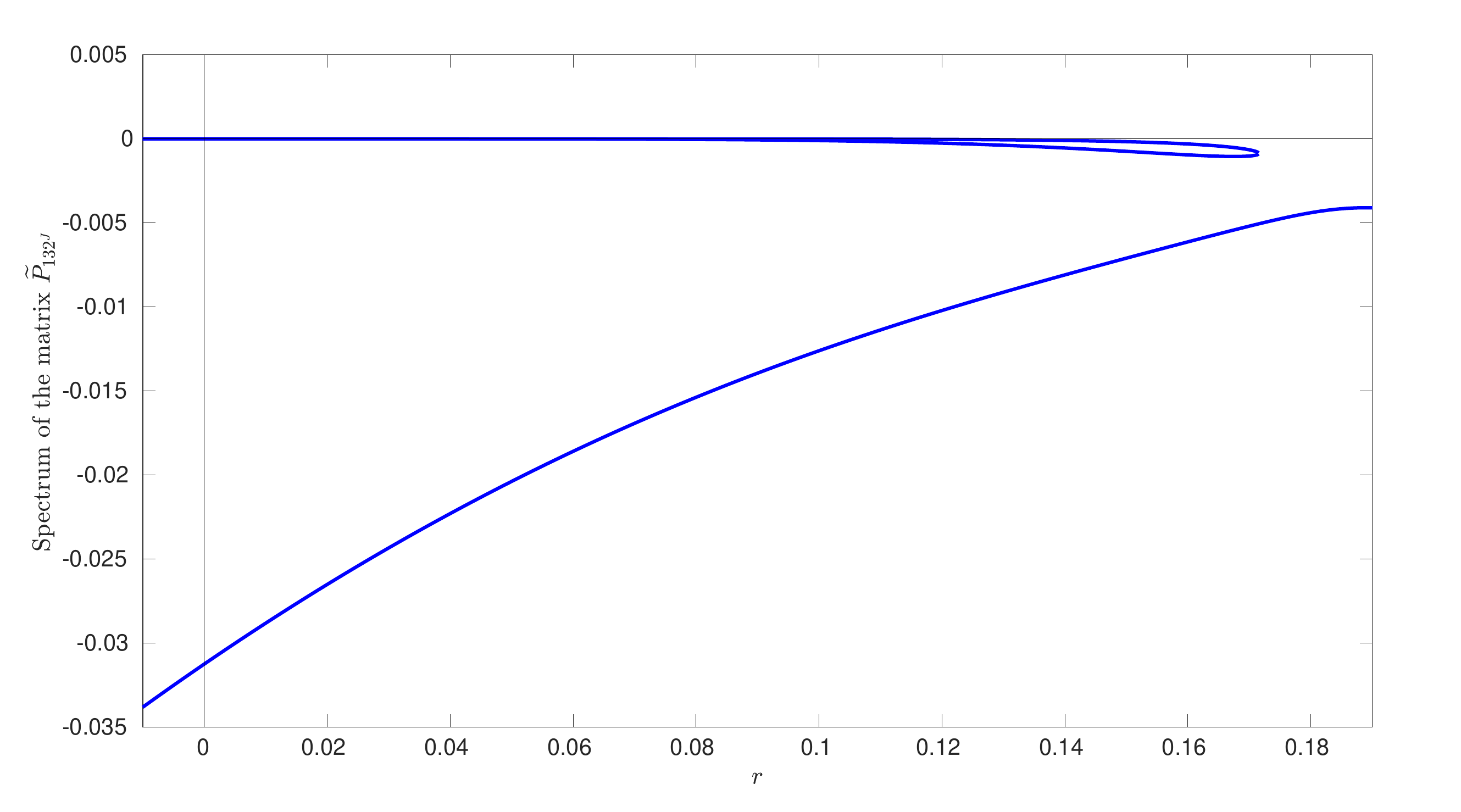}
    \caption{Magnification of the plot of the spectrum (when real), for $0 \leq r \leq 0.19$.}
    \label{FIGUREigenvalue_132J_Zoom}
\end{subfigure}
\caption{Eigenvalues of the collision matrix $\widetilde{P}_{132^J}$, as functions of the restitution coefficient $r$, in the case when these eigenvalues are real.}
\label{FIGUREigenvalue_132J}
\end{figure}

\begin{figure}[h!]
\centering
\begin{subfigure}{1\textwidth}
    \includegraphics[trim = 0cm 0cm 0cm 0cm, width=\linewidth]{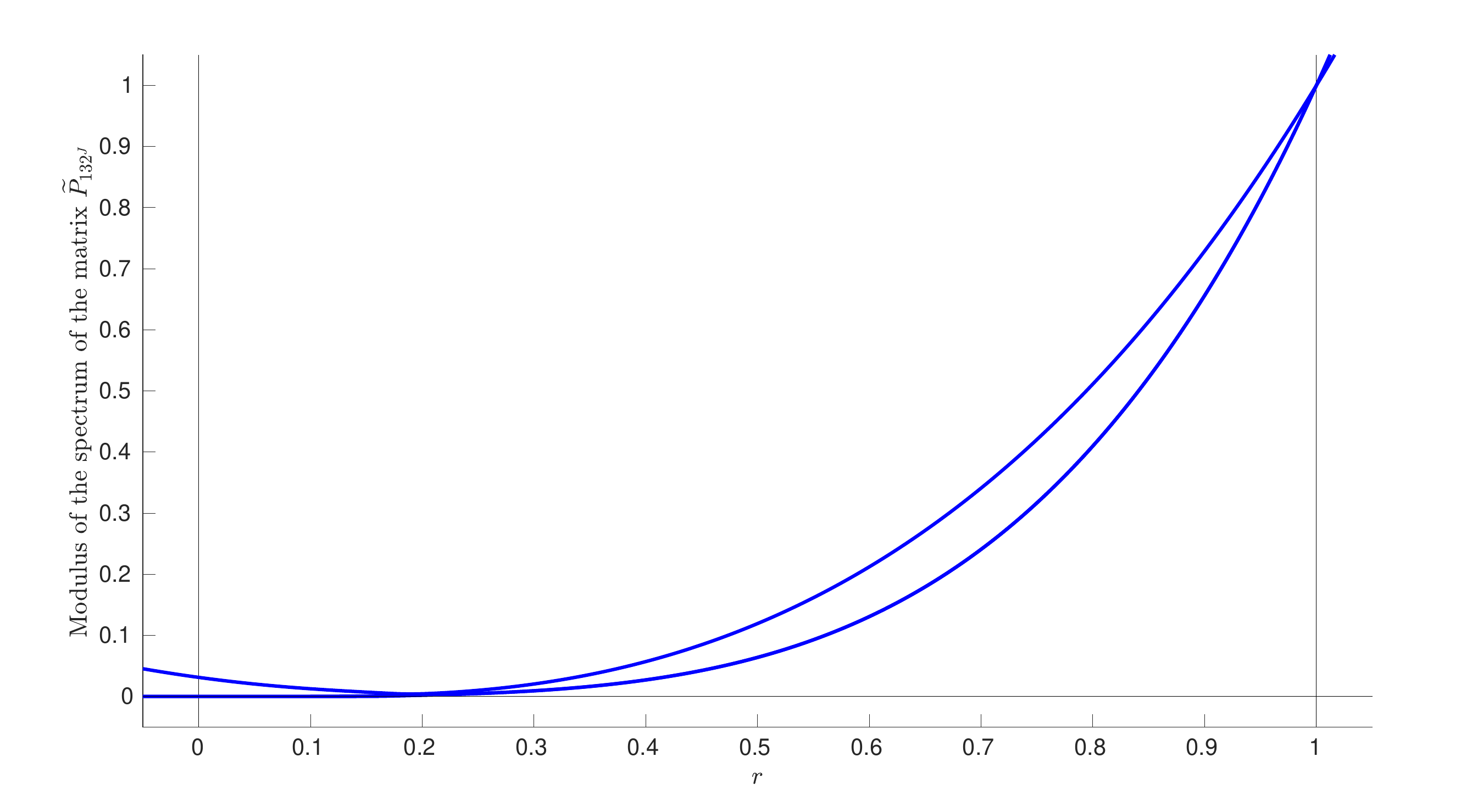}
    \caption{Moduli of the eigenvalues for $0 \leq r \leq 1$.}
    \label{FIGURModulEigen_132J_Glob}
\end{subfigure}
\begin{subfigure}{1\textwidth}
    \includegraphics[trim = 0cm 0cm 0cm 0cm, width=\linewidth]{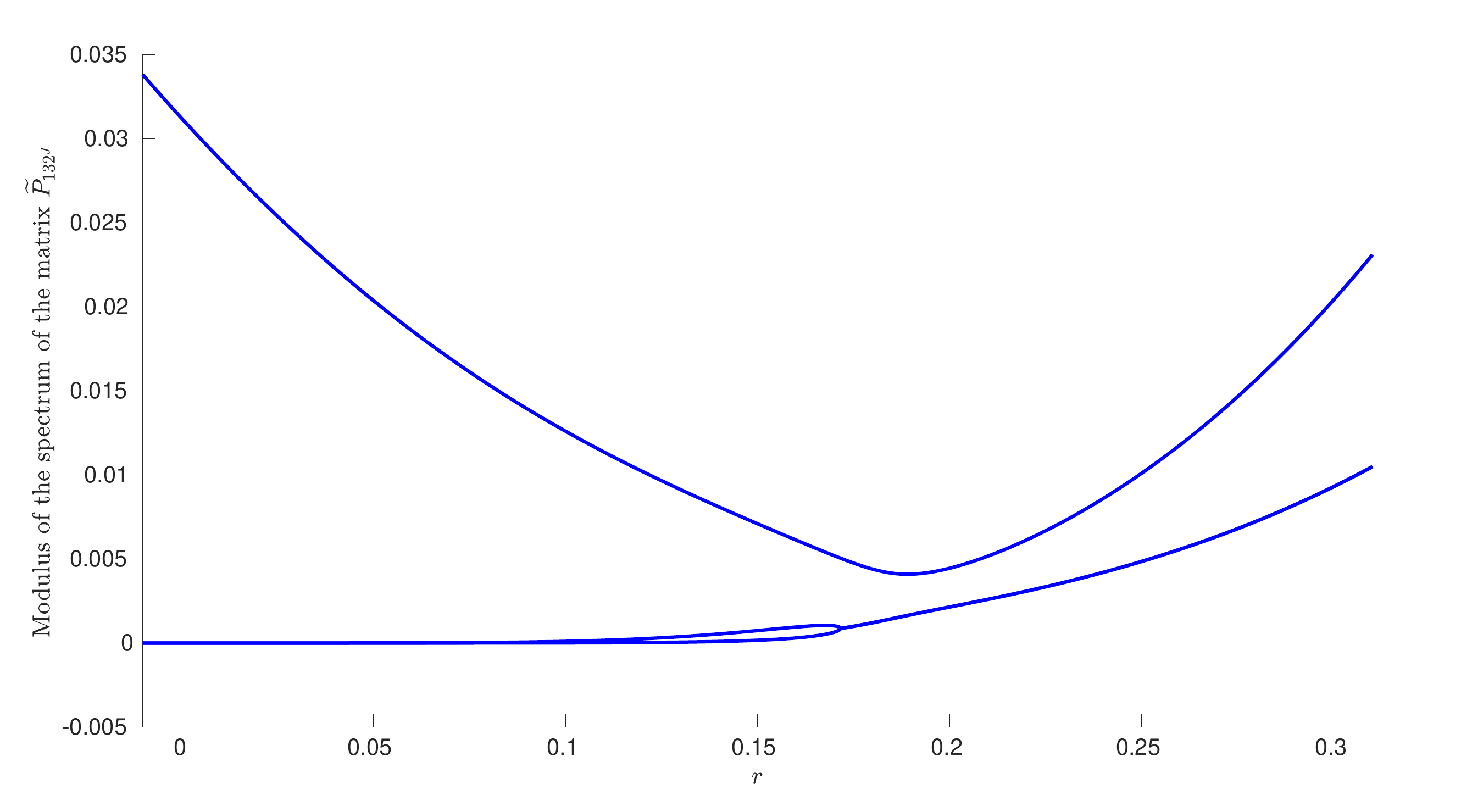}
    \caption{Magnification of the plot of the moduli of the eigenvalues, for $0 \leq r \leq 0.3$.}
    \label{FIGURModulEigen_132J_Zoom}
\end{subfigure}
\caption{Moduli of the eigenvalues of the collision matrix $\widetilde{P}_{132^J}$, as functions of the restitution coefficient $r$.}
\label{FIGURModulEigen_132J}
\end{figure}

\noindent
The discriminant of the characteristic polynomial, denoted by $\Delta_{\chi_{\widetilde{P}_{132^J}}}$, is itself a polynomial in the restitution coefficient $r$, and is given by the expression:
\begin{align}
\Delta_{\chi_{\widetilde{P}_{132^J}}}(r) &= -\frac{r^{10}(r - 1)^4(r + 1)^2}{113246208} Q_{\Delta_{132^J}}(r),\nonumber\\
\\
\text{with} \hspace{5mm} Q_{\Delta_{132312}}(r) &= 49 r^{18} - 1150 r^{17} + 11561 r^{16} - 64272 r^{15} + 210772 r^{14} - 445384 r^{13} + 907556 r^{12} \nonumber\\
&\hspace{8mm} - 1807728 r^{11} + 1753646 r^{10} - 3227252 r^9 + 1753646 r^8 - 1807728 r^7 + 907556 r^6  \nonumber\\
&\hspace{20mm} - 445384 r^5 + 210772 r^4 - 64272 r^3 + 11561 r^2 - 1150 r + 49.\nonumber
\end{align}
The discriminant $\Delta_{\chi_{\widetilde{P}_{132^J}}}$ cancels clearly for $r = -1$, $0$ and $1$, so that for such values the characteristic polynomial $\chi_{\widetilde{P}_{132^J}}$ has a root with a multiplicity at least $2$. In addition, the discriminant has one single other real root $r_{\Delta,132^J}$ in $]0,1[$, that can be approximated as:
\begin{align}
r_{\Delta,132^J} \simeq 0.1715\ 7287\ 5253\ 8099\ 0240.
\end{align}
For $0 < r < r_{\Delta,132^J}$, we have $\Delta_{\chi_{\widetilde{P}_{132^J}}}(r) < 0$, indicating that the characteristic polynomial $\chi_{\widetilde{P}_{132^J}}$ has three distinct real roots, while $\Delta_{\chi_{\widetilde{P}_{132^J}}}(r) > 0$ for $r_{\Delta,132^J} < r <1$, so that in this case $\chi_{\widetilde{P}_{132^J}}$ has one single real root, and two complex conjugated roots. For $0 < r < 1$, all the roots of the characteristic polynomial, when real, are negative. In addition, the only eigenvalue that remains real for any $0 < r < 1$, that we denote by $\lambda_{132^J}^\text{dom}$, is dominating the two other eigenvalues, both in the cases when they are real, or complex (see Figure \ref{FIGURModulEigen_132J_Zoom}). We observe that it is possible to give an explicit expression of $\lambda_{132^J}^\text{dom}$ relying on Cardano's cubic formula.\\
\newline
We choose now the eigenvector $u_{\lambda_{132^J}^\text{dom}}$ associated to the dominating eigenvalue $\lambda_{132^J}^\text{dom}$ such that $\big(u_{\lambda_{132^J}^\text{dom}}\big)_x =~1$, which can always be done except for $r_{\text{crit},132^J,1}$, with:
\begin{align}
\label{EQUATr_cr1_132J}
r_{\text{crit},132^J,1} \simeq 0.1826\ 1545\ 3792\ 6080\ 6055,
\end{align}
in which case the first component of the eigenvector is zero. On Figures \ref{FIGURFeasiIneq_132J___1__}-\ref{FIGURFeasiIneq_132J___3__}, we always highlight the critical value $r_{\text{crit},132^J,1}$ of the restitution coefficient, represented as a vertical red line. Under the assumption that $\big(u_{\lambda_{132^J}^\text{dom}}\big)_x =~1$, the third component $\big(u_{\lambda_{132^J}^\text{dom}}\big)_z$ of the eigenvector $u_{\lambda_{132^J}^\text{dom}}$ is plotted for $r \in [0,1]$ on Figure \ref{FIGURFeasiIneq_132J_SFig1}, and is negative on $[0,1]$ if and only if $r > r_{\text{crit},132^J,1}$.\\

\begin{figure}[h!]
\centering

\begin{subfigure}{1\textwidth}
\centering
    \includegraphics[trim = 0cm 0cm 0cm 0cm, width=\linewidth]{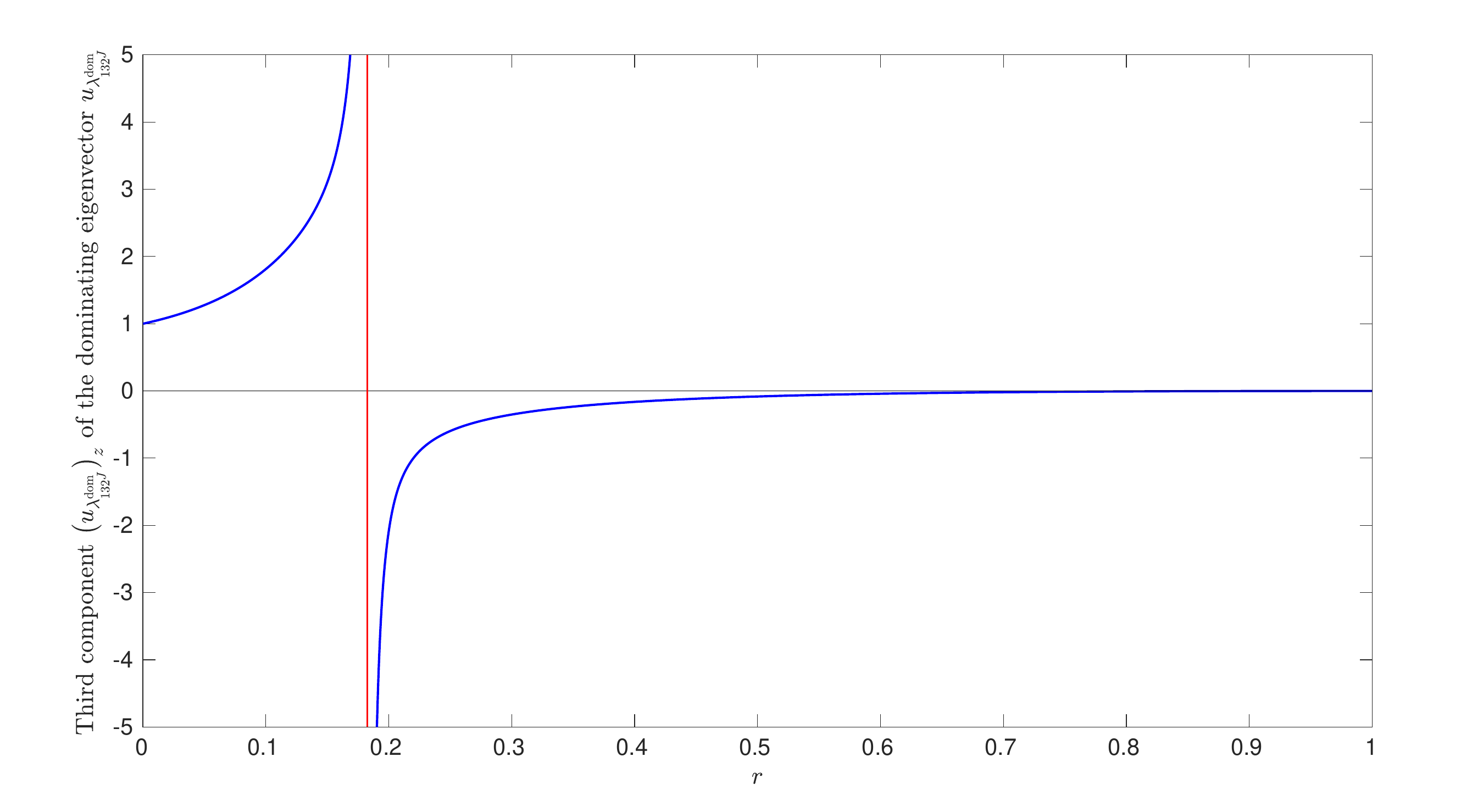}
    \caption{First feasibility inequality for 
    $132312$: plot of $\big(u_{\lambda_{132^J}^\text{dom}}\big)_z$ as a function of $r$.}
    \label{FIGURFeasiIneq_132J_SFig1}
\end{subfigure}

\begin{subfigure}{1\textwidth}
\centering
    \includegraphics[trim = 0cm 0cm 0cm 0cm, width=\linewidth]{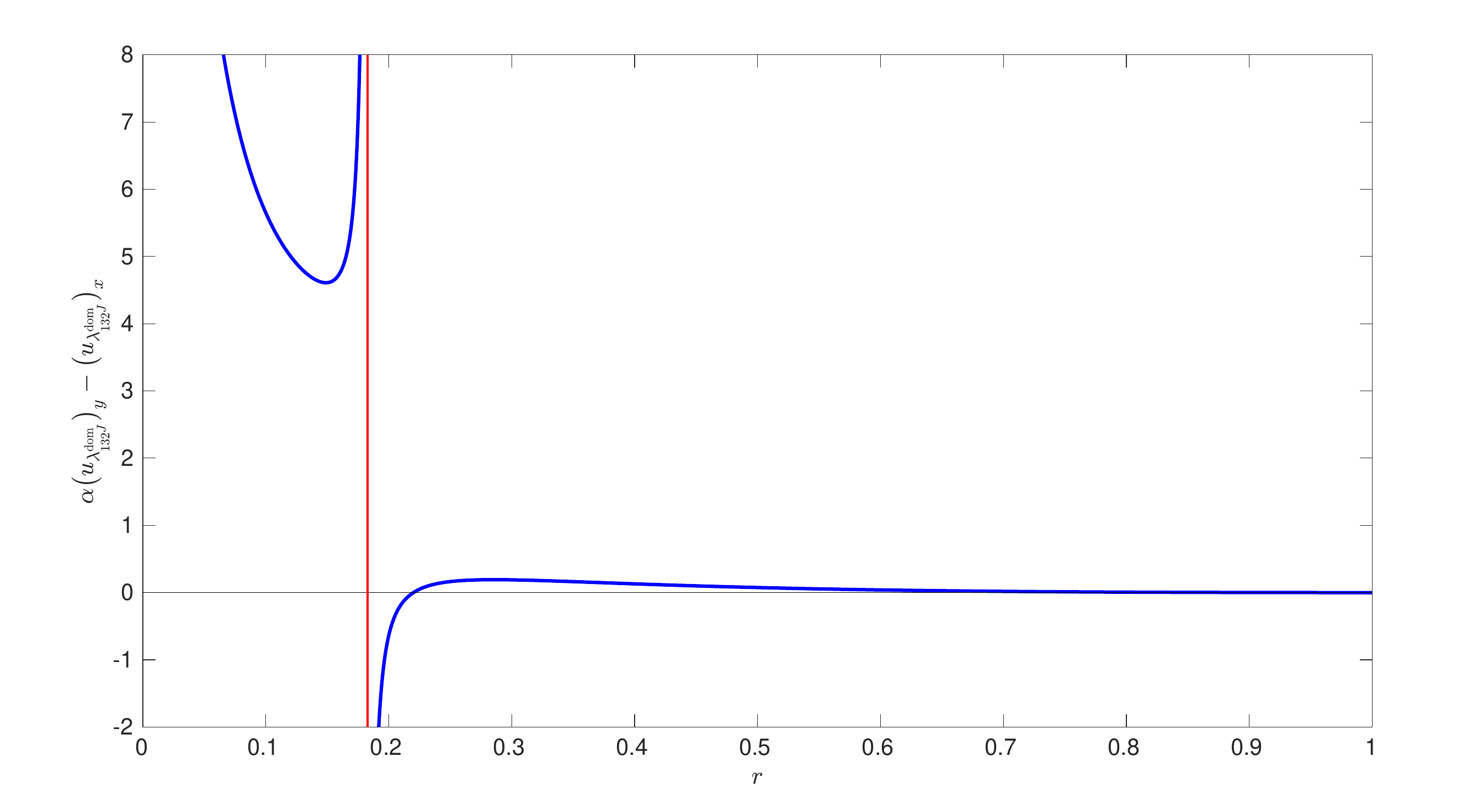}
    \caption{Second feasibility inequality for 
    $132312$: plot of $\alpha \big(u_{\lambda_{132^J}^\text{dom}}\big)_y - \big(u_{\lambda_{132^J}^\text{dom}}\big)_x$ as a function of $r$.}
    \label{FIGURFeasiIneq_132J_SFig2}
\end{subfigure}

\caption{The two first feasibility inequalities associated to the pattern $132312$.}
\label{FIGURFeasiIneq_132J___1__}
\end{figure}

\begin{figure}[h!]
\centering

\begin{subfigure}{1\textwidth}
\centering
    \includegraphics[trim = 0cm 0cm 0cm 0cm, width=\linewidth]{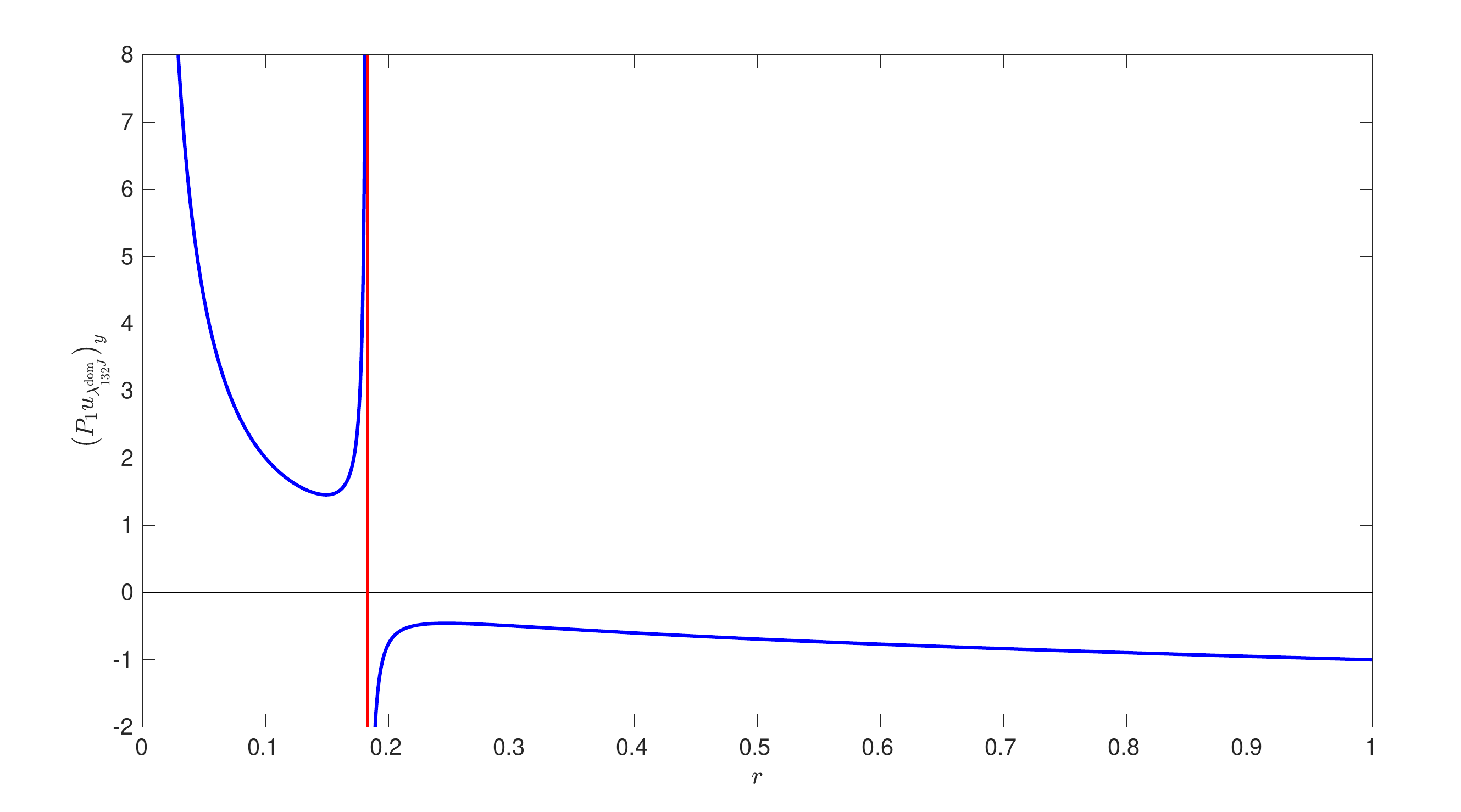}
    \caption{Third feasibility inequality for 
    $132312$: plot of $\big( P_1 u_{\lambda_{132^J}^\text{dom}}\big)_y$ as a function of $r$.}
    \label{FIGURFeasiIneq_132J_SFig3}
\end{subfigure}

\begin{subfigure}{1\textwidth}
\centering
    \includegraphics[trim = 0cm 0cm 0cm 0cm, width=\linewidth]{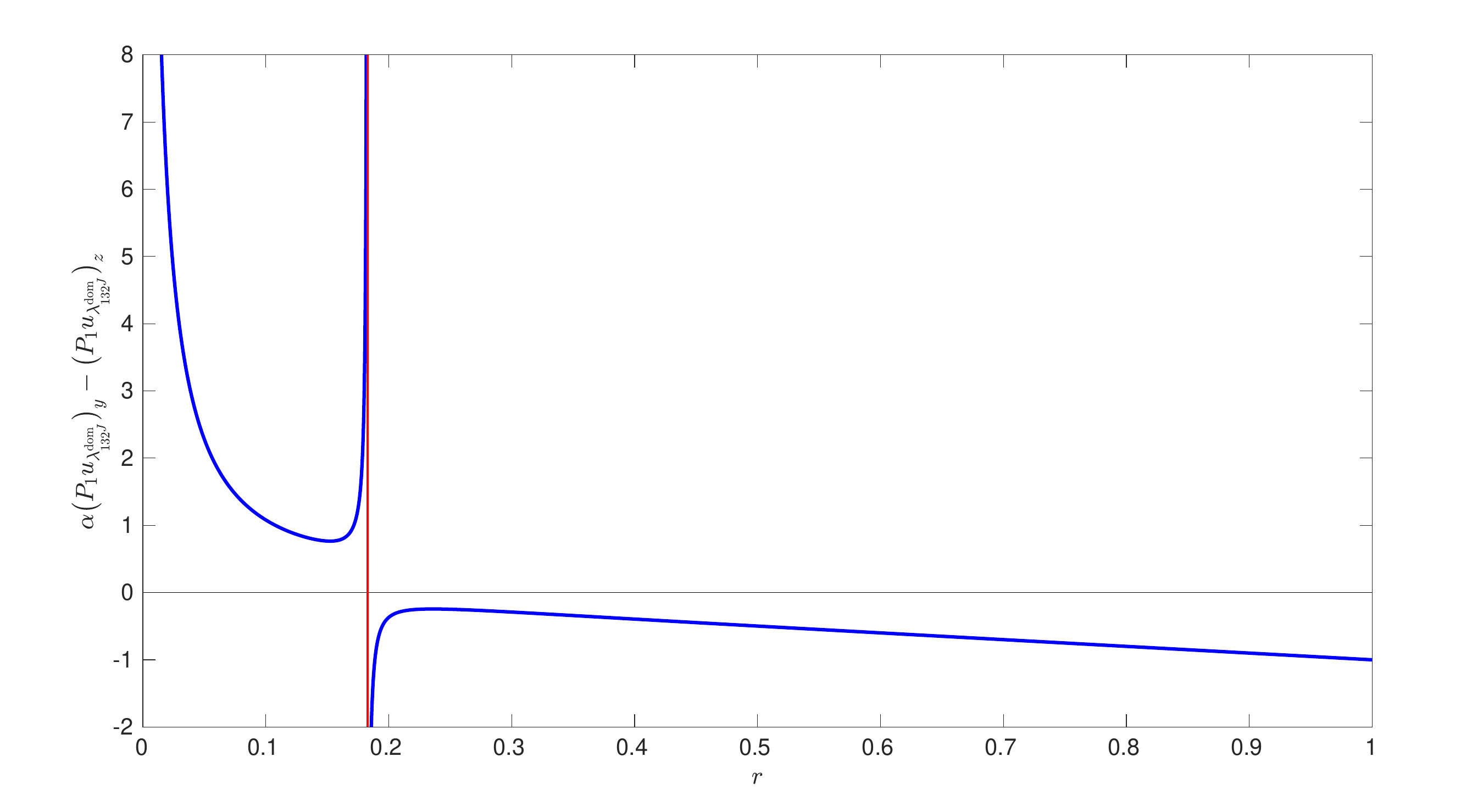}
    \caption{Fourth feasibility inequality for 
    $132312$: plot of $\alpha \big( P_1 u_{\lambda_{132^J}^\text{dom}}\big)_y - \big( P_1 u_{\lambda_{132^J}^\text{dom}}\big)_z$ as a function of $r$.}
    \label{FIGURFeasiIneq_132J_SFig4}
\end{subfigure}

\caption{Third and fourth feasibility inequalities associated to the pattern $132312$.}
\label{FIGURFeasiIneq_132J___2__}
\end{figure}

\begin{figure}
\centering

\begin{subfigure}{1\textwidth}
\centering
    \includegraphics[trim = 0cm 0cm 0cm 0cm, width=\linewidth]{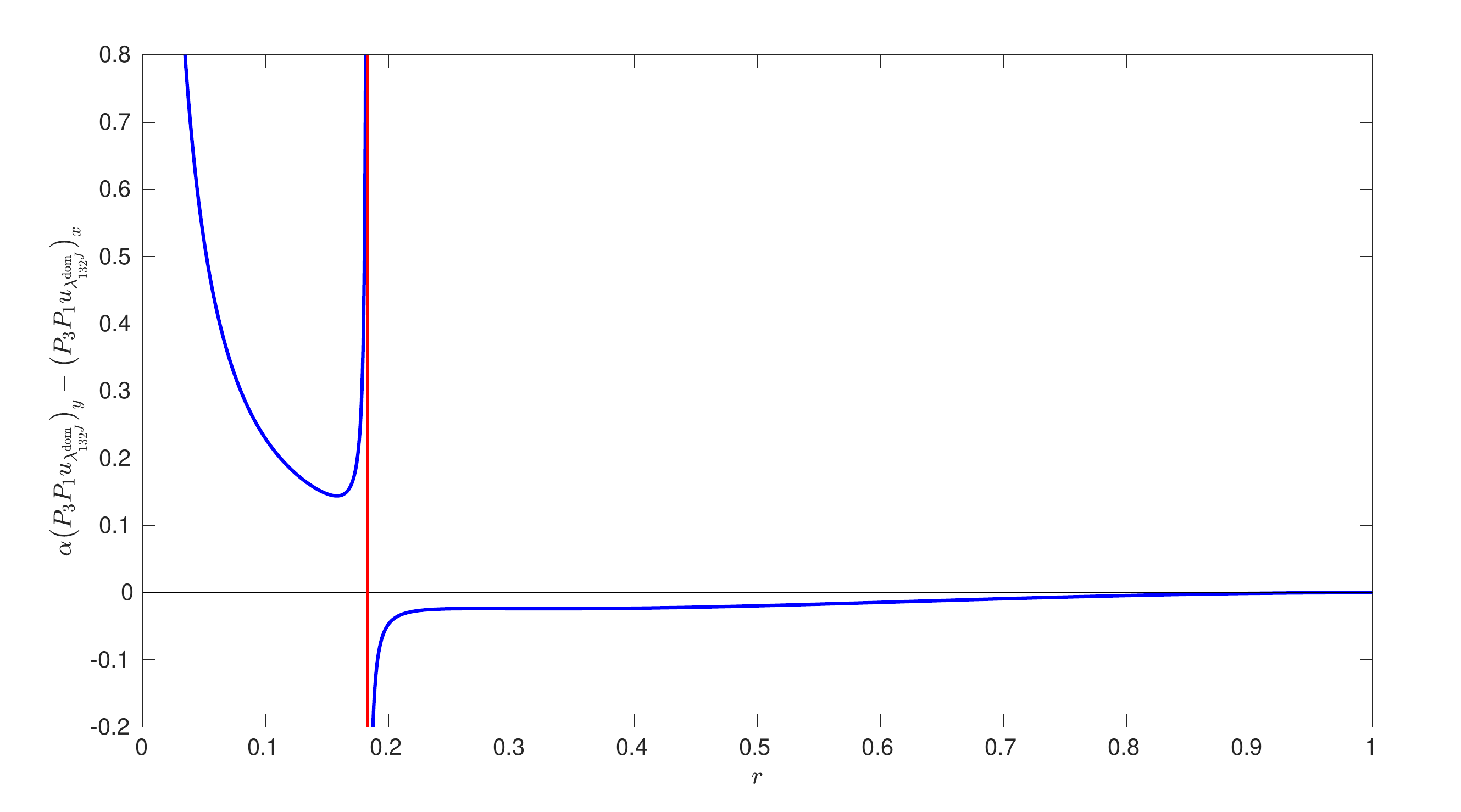}
    \caption{Fifth feasibility inequality for 
    $132312$: plot of $\alpha \big( P_3 P_1 u_{\lambda_{132^J}^\text{dom}}\big)_y - \big( P_3 P_1 u_{\lambda_{132^J}^\text{dom}}\big)_x$ as a function of $r$.}
    \label{FIGURFeasiIneq_132J_SFig5}
\end{subfigure}

\begin{subfigure}{1\textwidth}
\centering
    \includegraphics[trim = 0cm 0cm 0cm 0cm, width=\linewidth]{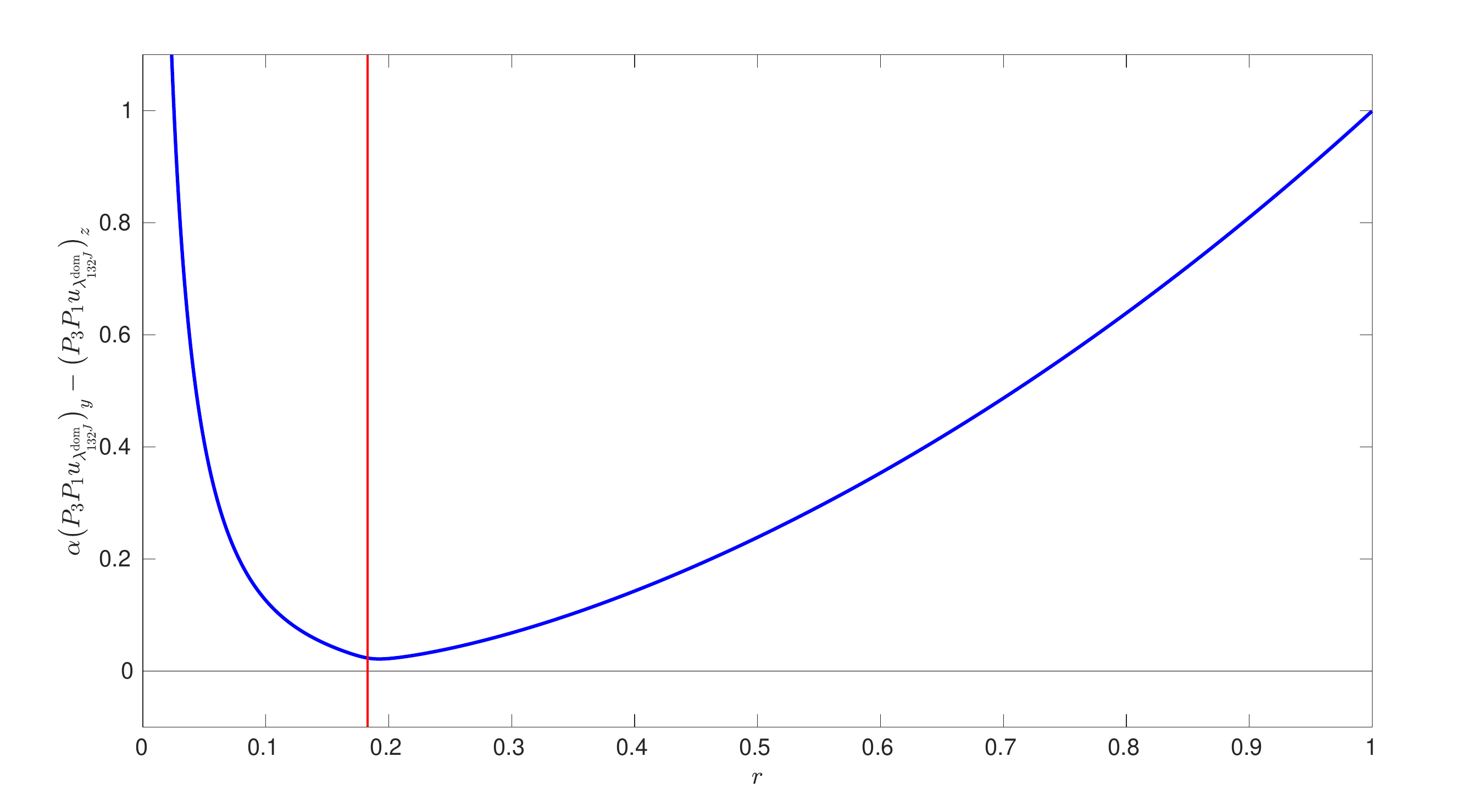}
    \caption{Sixth feasibility inequality for 
    $132312$: plot of $\alpha \big( P_3 P_1 u_{\lambda_{132^J}^\text{dom}}\big)_y - \big( P_3 P_1 u_{\lambda_{132^J}^\text{dom}}\big)_z$ as a function of $r$.}
    \label{FIGURFeasiIneq_132J_SFig6}
\end{subfigure}
\caption{The two last feasibility inequalities associated to the pattern $132312$.}
\label{FIGURFeasiIneq_132J___3__}
\end{figure}

\noindent
Starting from the initial configuration $u_{\lambda_{132^J}^\text{dom}}$, in order to detect if the first block of collisions is $\mathfrak{ab}$, that is, if $\mathfrak{P}\big( u_{\lambda_{132^J}^\text{dom}} \big) = P_1 u_{\lambda_{132^J}^\text{dom}}$, we plot on Figure \ref{FIGURFeasiIneq_132J_SFig2} the quantity $\alpha \big( u_{\lambda_{132^J}^\text{dom}} \big)_y - \big( u_{\lambda_{132^J}^\text{dom}} \big)_x$. We observe in particular that $\alpha \big( u_{\lambda_{132^J}^\text{dom}} \big)_y - \big( u_{\lambda_{132^J}^\text{dom}} \big)_x$ is positive on $]r_{\text{crit},132^J,1},1]$ if and only if $r > r_{\text{crit},132^J,2}$, with $r_{\text{crit},132^J,2}$ defined such that:
\begin{align}
\label{EQUATr_cr2_132J}
\alpha \big( u_{\lambda_{132^J}^\text{dom}} \big)_y - \big( u_{\lambda_{132^J}^\text{dom}} \big)_x = 0
\end{align}
(where the quantities $\alpha$ and $u_{\lambda_{132^J}^\text{dom}}$ both depend on $r$), and that we can numerically evaluate as:
\begin{align}
r_{\text{crit},132^J,2} \simeq 0.2200\ 6978\ 6146\ 3104\ 7521.
\end{align}
Since $r_{\text{crit},132^J,2} > r_{\text{crit},132^J,1}$, for any $r \in\ ]r_{\text{crit},132^J,2},1]$ the dominating eigenvector $u_{\lambda_{132^J}^\text{dom}}$ belongs to the domain of $\mathfrak{P}$, and even to the domain in which $\mathfrak{P}$ coincides with $P_1$.\\
We turn now to the feasibility inequalities concerning $P_1 u_{\lambda_{132^J}^\text{dom}}$. As it can be observed on Figure \ref{FIGURFeasiIneq_132J_SFig3}, the second component $\big( P_1 u_{\lambda_{132^J}^\text{dom}} \big)_y$ is always negative for any $r \in\ ]r_{\text{crit},132^J,1},1]$, ensuring that $\mathfrak{P}$ coincides either with $P_2$ or $P_3$ when evaluated on $P_1 u_{\lambda_{132^J}^\text{dom}}$. In order to distinguish between these two cases, we represent on Figure \ref{FIGURFeasiIneq_132J_SFig4} the quantity $\alpha \big( P_1 u_{\lambda_{132^J}^\text{dom}} \big)_y - \big( P_1 u_{\lambda_{132^J}^\text{dom}} \big)_z$. Since this quantity is negative for any $r \in\ ]r_{\text{crit},132^J,1},1]$, we have: $\mathfrak{P}\big( P_1 u_{\lambda_{132^J}^\text{dom}} \big) = P_3 P_1 u_{\lambda_{132^J}^\text{dom}}$.\\
We turn now to the last feasibility inequalities. In order to have $\mathfrak{P}\big( P_3 P_1 u_{\lambda_{132^J}^\text{dom}} \big) = P_2 P_3 P_1 u_{\lambda_{132^J}^\text{dom}}$, we need to have $\alpha \big( P_3 P_1 u_{\lambda_{132^J}^\text{dom}} \big)_y - \big( P_3 P_1 u_{\lambda_{132^J}^\text{dom}} \big)_x < 0$ and $\alpha \big( P_3 P_1 u_{\lambda_{132^J}^\text{dom}} \big)_y - \big( P_3 P_1 u_{\lambda_{132^J}^\text{dom}} \big)_z > 0$. We plot respectively on Figures \ref{FIGURFeasiIneq_132J_SFig5} and \ref{FIGURFeasiIneq_132J_SFig6} the two quantities $\alpha \big( P_3 P_1 u_{\lambda_{132^J}^\text{dom}} \big)_y - \big( P_3 P_1 u_{\lambda_{132^J}^\text{dom}} \big)_x$ and $\alpha \big( P_3 P_1 u_{\lambda_{132^J}^\text{dom}} \big)_y - \big( P_3 P_1 u_{\lambda_{132^J}^\text{dom}} \big)_z$. We see on Figure \ref{FIGURFeasiIneq_132J_SFig5} that  $\alpha \big( P_3 P_1 u_{\lambda_{132^J}^\text{dom}} \big)_y - \big( P_3 P_1 u_{\lambda_{132^J}^\text{dom}} \big)_x < 0$ for any $r \in\ ]r_{\text{crit},132^J,1},1]$. On the other hand, on Figure \ref{FIGURFeasiIneq_132J_SFig6} we see also that $\alpha \big( P_3 P_1 u_{\lambda_{132^J}^\text{dom}} \big)_y - \big( P_3 P_1 u_{\lambda_{132^J}^\text{dom}} \big)_z > 0$ for any $r \in\ ]r_{\text{crit},132^J,1},1]$.\\
In the end, for any $r \in\ ]r_{\text{crit},132^J,2},1]$, where $r_{\text{crit},132^J,2}$ is defined in \eqref{EQUATr_cr2_132J}, all the feasibility inequalities are fulfilled for the dominating eigenvector $u_{\lambda_{132^J}^\text{dom}}$, so that:
\begin{align}
\mathfrak{P}\big( u_{\lambda_{132^J}^\text{dom}} \big) = P_1 u_{\lambda_{132^J}^\text{dom}},\hspace{3mm} \mathfrak{P}^2\big( u_{\lambda_{132^J}^\text{dom}} \big) = P_3 P_1 u_{\lambda_{132^J}^\text{dom}}, \hspace{3mm} \mathfrak{P}^3\big( u_{\lambda_{132^J}^\text{dom}} \big) = P_2 P_3 P_1 u_{\lambda_{132^J}^\text{dom}},
\end{align}
and since by definition $u_{\lambda_{132^J}^\text{dom}}$ is an eigenvector of $J P_2 P_3 P_1$, the feasibility inequalities for the next three iterations of $\mathfrak{P}$ are fulfilled and we have:
\begin{align}
\mathfrak{P}^4\big( u_{\lambda_{132^J}^\text{dom}} \big) &= P_3 P_2 P_3 P_1 u_{\lambda_{132^J}^\text{dom}},\hspace{3mm} \mathfrak{P}^5\big( u_{\lambda_{132^J}^\text{dom}} \big) = P_1 P_3 P_2 P_3 P_1 u_{\lambda_{132^J}^\text{dom}}, \nonumber\\
&\hspace{-5mm} \mathfrak{P}^6\big( u_{\lambda_{132^J}^\text{dom}} \big) = P_2 P_1 P_3 P_2 P_3 P_1 u_{\lambda_{132^J}^\text{dom}} = \big(\lambda_{132^J}^\text{dom}\big)^2 u_{\lambda_{132^J}^\text{dom}}.
\end{align}
Since $u_{\lambda_{132^J}^\text{dom}}$ is the eigenvector associated to the dominating eigenvalue of $J P_2 P_3 P_1$, it is also locally stable under the repeated iterations of $\mathfrak{P}$. The proof of Theorem \ref{THEORStabi____132312} is complete.
\end{proof}

\subsection{The pattern $13223122$}
\label{SSECTPatter_13223122}

\noindent
In this section, we will investigate another palindromic pattern, which is the period $13223122$. After the study of $132312$ in the previous section, this pattern constitutes the second element of the family of periods \eqref{EQUATPattrPalinSimpl} of the form $132^{n_2}312^{n_2}$. We will actually show that the period $13223122$ can never be observed.

\begin{theor}[Unstability of the periodic orbit $13223122$]
\label{THEORUnstabi13223122}
Let $r \in\ ]0,1[$ be any restitution coefficient.\\
Then, there exists no point in the domain of $\widehat{\mathfrak{P}}_r$ which generates the periodic orbit $13223122$ in a stable manner.
\end{theor}

\begin{remar}
The result of Theorem \ref{THEORUnstabi13223122} is a negative result. The fact that the periodic pattern $13223122$ can never be realized in a stable manner is consistent with the results of the numerical simulations presented in Section \ref{SSECTSimulIndivOrbit}, since this pattern was never observed. Indeed, a periodic orbit that has a basin of attraction of zero Lebesgue measure cannot be observed in numerical simulations.\\
In order to determine if the period $13223122$ can be realized in an unstable manner, it would be necessary to investigate for $r \in\ ]0,1[$ the feasibility inequalities related to the eigenvector of $\widetilde{P}_{1322^J} = J P_2 P_2 P_3 P_1$ associated to the eigenvalue that is real for any value of $r$. Indeed, the other eigenvalues being either positive or complex, they cannot be associated with periodic orbits (either the associated eigenvector is not real-valued, or it cannot fulfill the feasibility inequalities because of the signs of such eigenvalues).
\end{remar}

\begin{proof}[Proof of Theorem \ref{THEORUnstabi13223122}]
In the case of the collision pattern $13223122$, which is also palindromic, we study the eigenvalues and eigenvectors of the matrix $\widetilde{P}_{1322^J} = J P_2 P_2 P_3 P_1$ (with $J$ defined in \eqref{EQUATDefin__J__}). The characteristic polynomial $\chi_{\widetilde{P}_{1322^J}}$ of $\widetilde{P}_{1322^J}$ is:
\begin{align}
\chi_{\widetilde{P}_{1322^J}}(\lambda) &= \lambda^3 + \frac{\big( -7r^8 + 42 r^7 - 18 r^6 - 58 r^5 + 24 r^4 + 38 r^3 - 30 r^2 + 10 r - 1\big)}{64} \lambda^2 \nonumber\\
&\hspace{8mm}+ \frac{\big( -r^{14} + 10 r^{13} - 30 r^{12} + 38 r^{11} + 24 r^{10} - 58 r^9 - 18 r^8 + 42 r^7 - 7 r^6 \big)}{64} \lambda + r^{14}.
\end{align}
We plot the eigenvalues of $\widetilde{P}_{1322^J}$ for $r \in [0,1]$ in the case when they are real on Figures \ref{FIGUREigenvalue1322J_Glob} and \ref{FIGUREigenvalue1322J_Zoom}. We also plot the moduli of the three eigenvalues on Figures \ref{FIGURModulEigen1322J_Glob} and \ref{FIGURModulEigen1322J_Zoom}.\\

\begin{figure}[h!]
\centering
\begin{subfigure}{1\textwidth}
    \includegraphics[trim = 0cm 0cm 0cm 0cm, width=\linewidth]{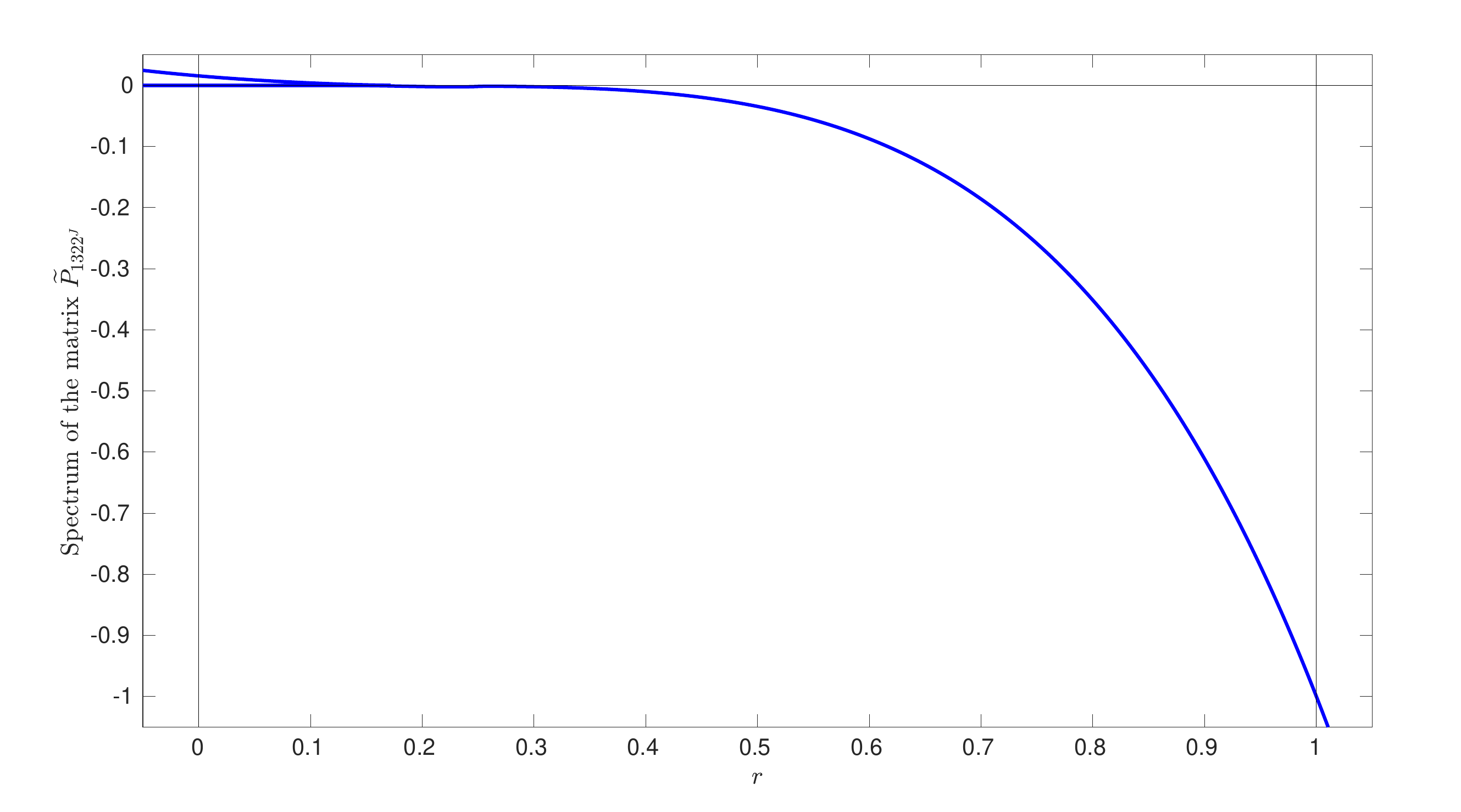}
    \caption{Spectrum (when real) for $0 \leq r \leq 1$.}
    \label{FIGUREigenvalue1322J_Glob}
\end{subfigure}
\begin{subfigure}{1\textwidth}
    \includegraphics[trim = 0cm 0cm 0cm 0cm, width=\linewidth]{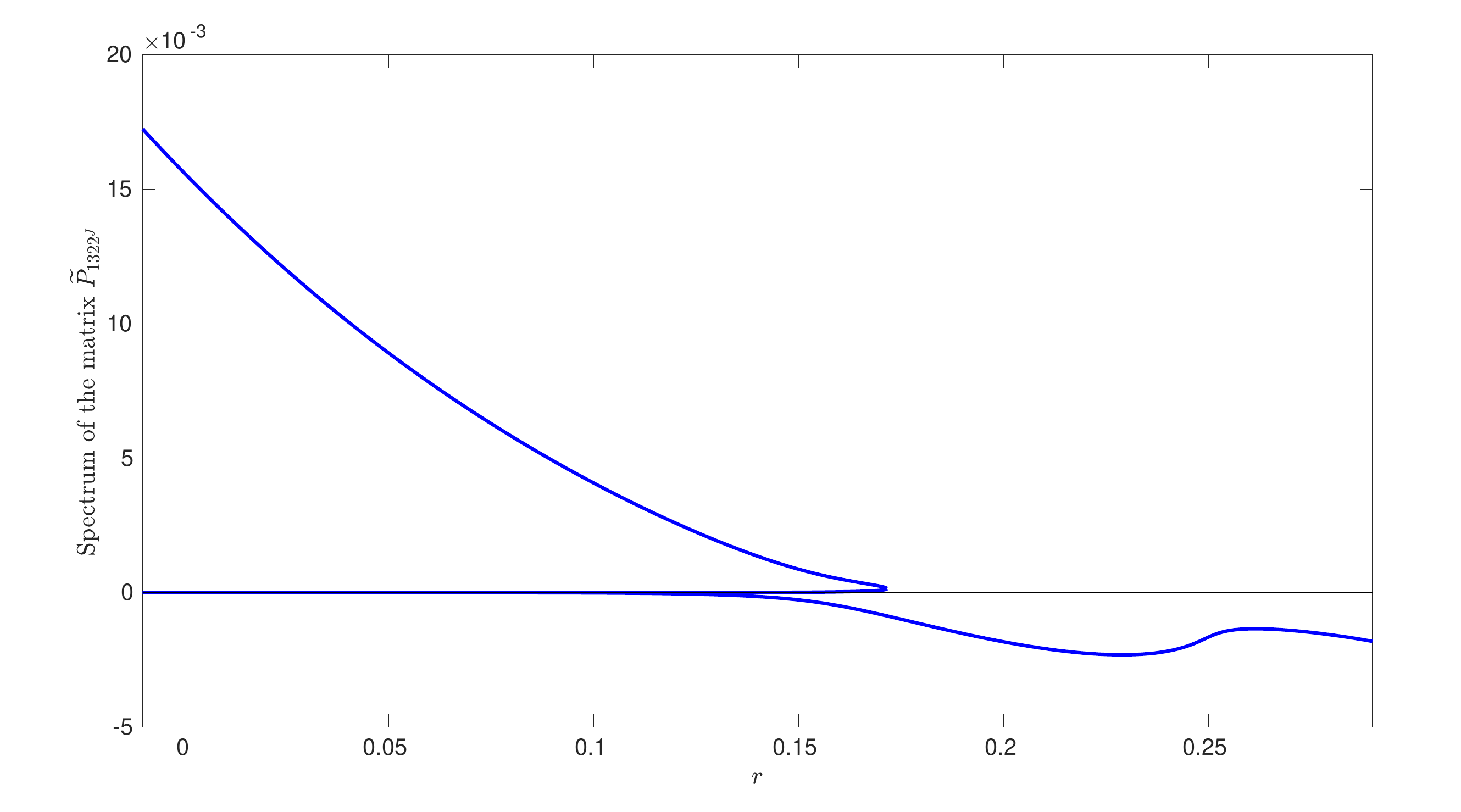}
    \caption{Magnification of the plot of the spectrum (when real), for $0 \leq r \leq 0.3$.}
    \label{FIGUREigenvalue1322J_Zoom}
\end{subfigure}
\caption{Eigenvalues of the collision matrix $\widetilde{P}_{1322^J}$, as functions of the restitution coefficient $r$, in the case when these eigenvalues are real.}
\label{FIGUREigenvalue1322J}
\end{figure}

\begin{figure}[h!]
\centering
\begin{subfigure}{1\textwidth}
    \includegraphics[trim = 0cm 0cm 0cm 0cm, width=\linewidth]{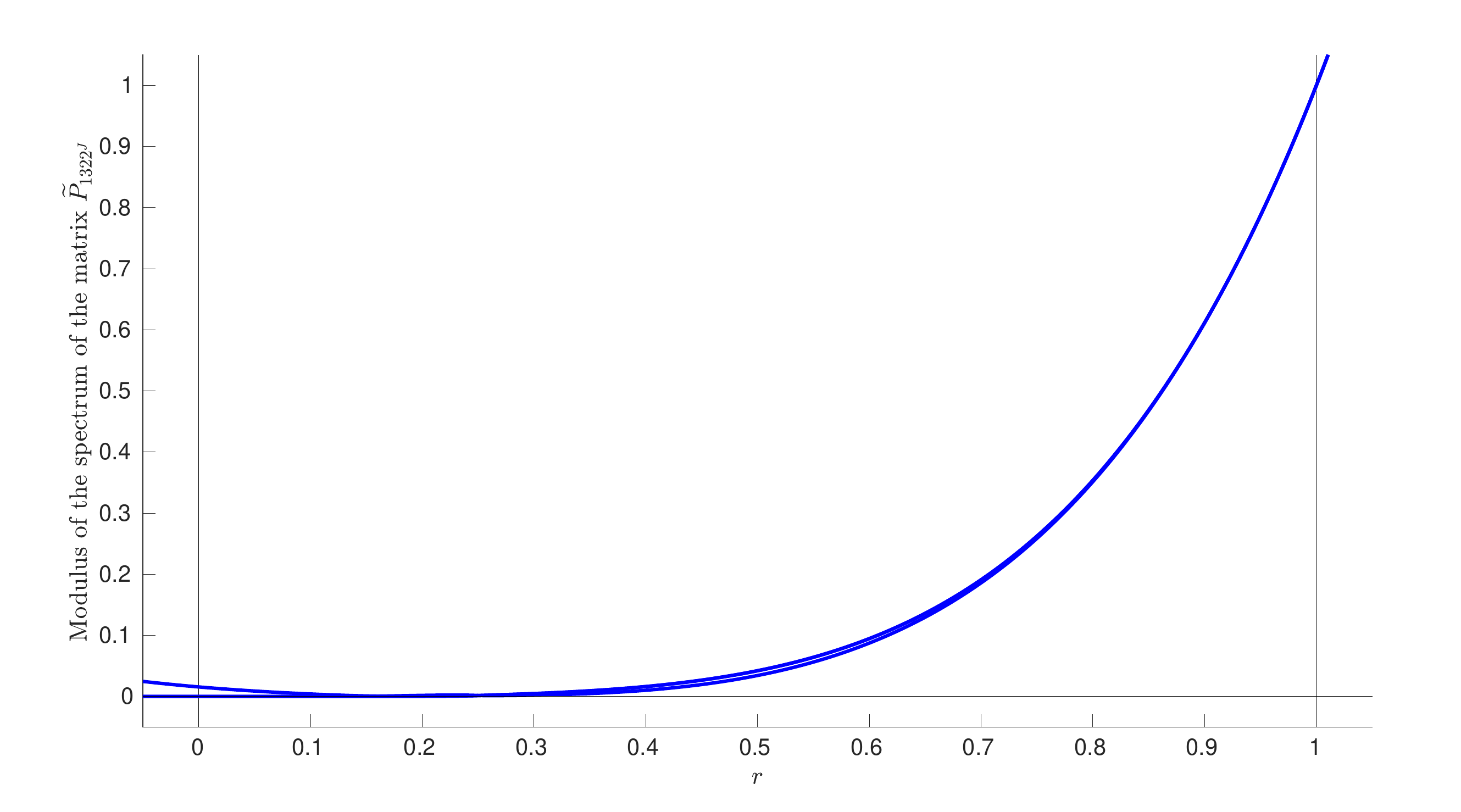}
    \caption{Moduli of the eigenvalues for $0 \leq r \leq 1$.}
    \label{FIGURModulEigen1322J_Glob}
\end{subfigure}
\begin{subfigure}{1\textwidth}
    \includegraphics[trim = 0cm 0cm 0cm 0cm, width=\linewidth]{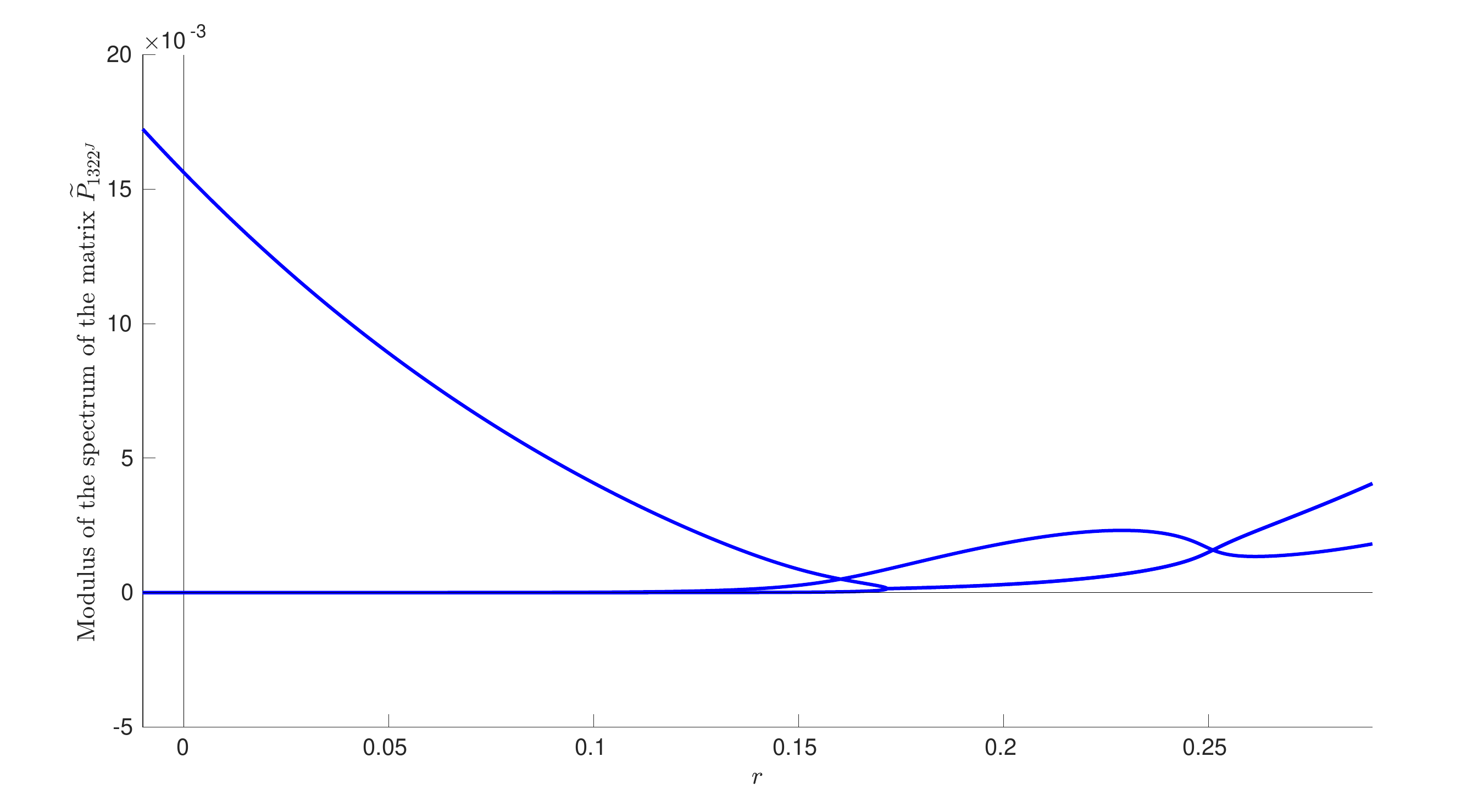}
    \caption{Magnification of the plot of the moduli of the eigenvalues, for $0 \leq r \leq 0.3$.}
    \label{FIGURModulEigen1322J_Zoom}
\end{subfigure}
\caption{Moduli of the eigenvalues of the collision matrix $\widetilde{P}_{1322^J}$, as functions of the restitution coefficient $r$.}
\label{FIGURModulEigen1322J}
\end{figure}

\noindent
The discriminant $\Delta_{\chi_{\widetilde{P}_{1322^J}}}$ of the characteristic polynomial $\chi_{\widetilde{P}_{1322^J}}$ is:
\begin{align}
\Delta_{\chi_{\widetilde{P}_{1322^J}}}(r) = - \frac{r^{12} \big( r^3 + r^2 + r + 1 \big)^2}{1811939328} Q_{\Delta_{1322^J}}(r)
\end{align}
with
\begin{align}
Q_{\Delta_{1322^J}}(r) &= 49 r^{26} - 1666 r^{25} + 25057 r^{24} - 221216 r^{23} + 1281770 r^{22} - 5145140 r^{21} + 14717738 r^{20} \nonumber\\
&\hspace{5mm} - 31143328 r^{19} + 54462799 r^{18} - 92352670 r^{17} + 144006079 r^{16} - 174418496 r^{15} \nonumber\\
&\hspace{6mm}+ 197859020 r^{14} - 246451544 r^{13} + 197859020 r^{12} - 174418496 r^{11} + 144006079 r^{10} \\
&\hspace{7mm} - 92352670 r^9 + 54462799 r^8 - 31143328 r^7 + 14717738 r^6 - 5145140 r^5 + 1281770 r^4 \nonumber\\
&\hspace{40mm} - 221216 r^3 + 25057 r^2 - 1666 r + 49. \nonumber
\end{align}
$Q_{\Delta_{1322^J}}$ has one single real root $r_{\Delta,1322^J}$ in $]0,1[$, which satisfies:
\begin{align}
r_{\Delta,1322^J} \simeq 0.1715\ 7287\ 5253\ 8099\ 0240.
\end{align}
Since $\Delta_{\chi_{\widetilde{P}_{1322^J}}}(r) < 0$ for any $r \in\ ]0,r_{\Delta,1322^J}[$ and $\Delta_{\chi_{\widetilde{P}_{1322^J}}}(r) > 0$ for any $r \in\ ]r_{\Delta,1322^J},1[$, the characteristic polynomial $\chi_{\widetilde{P}_{1322^J}}$ has three distinct real roots if $r < r_{\Delta,1322^J}$, and one single real root and two complex conjugated roots if $r > r_{\Delta,1322^J}$.\\
According to the discussion at the beginning of Section \ref{SSECTPatte__132312__} concerning the feasibility inequalities for palindromic collision sequences, a stable periodic orbit, with period $13223122$, corresponds necessarily to a dominating eigenvector of the reduced collision matrix $\widetilde{P}_{1322^J}$, and such an eigenvector has necessarily to be associated to a negative eigenvalue. Considering Figures \ref{FIGUREigenvalue1322J_Glob} and \ref{FIGUREigenvalue1322J_Zoom}, we denote by $\lambda^\text{real}_{1322^J}$ the only root of the characteristic polynomial $\chi_{\widetilde{P}_{1322^J}}$ which is real for any value of $r$. $\lambda^\text{real}_{1322^J}$  is always negative, but contrary to the case of the pattern $132312$, such an eigenvalue is not always dominating (see Figures \ref{FIGURModulEigen1322J_Glob} and \ref{FIGURModulEigen1322J_Zoom}). More precisely, since the two other roots $\lambda^j_{1322^J}$, $2 \leq j \leq 3$, are positive for any $r \in\ ]0,r_{\Delta,1322^J}[$, we can determine the restitution coefficient for which $\lambda^\text{real}_{1322^J} = - \lambda^j_{1322^J}$ by determining when the characteristic polynomial of $\big( \widetilde{P}_{1322^J} \big)^2$ has a double root (since the eigenvalues of $\big( \widetilde{P}_{1322^J} \big)^2$ are the squares of the eigenvalues of $\widetilde{P}_{1322^J}$). There is only one real root of the discriminant of the characteristic polynomial of $\big( \widetilde{P}_{1322^J} \big)^2$ in $]0,r_{\Delta,1322^J}[$, that we denote by $r_{\text{crit},1322^J,1}$, and which is approximately equal to:
\begin{align}
r_{\text{crit},1322^J,1} \simeq 0.1602\ 8989\ 1518\ 0595\ 0668.
\end{align}
The determinant of $\widetilde{P}_{1322^J}$ is equal to $-r^{14}$. Therefore, in the case when the eigenvalues $\lambda^j_{1322^J}$ are complex conjugated, they have the same modulus, and so the real eigenvalue $\lambda^\text{real}_{1322^J}$ has the same modulus as the complex eigenvalues when $r$ solves the equation:
\begin{align}
\label{EQUAT1322J_r_crit_2_}
\lambda^\text{real}_{1322^J}(r) = - r^{14/3}.
\end{align}
Equation \eqref{EQUAT1322J_r_crit_2_} has one single real root $r_{\text{crit},1322^J,2} \in\ ]r_{\text{crit},1322^J,1},1[$, such that:
\begin{align}
r_{\text{crit},1322^J,2} \simeq 0.2511\ 2749\ 5442\ 8146\ 2755.
\end{align}
Therefore, for restitution coefficients $r \in\ ]0,1[$, the dominating eigenvalue of $\widetilde{P}_{1322^J}$ is real and negative if and only if $r_{\text{crit},1322^J,1} < r < r_{\text{crit},1322^J,2}$. We will emphasize the two critical values $r_{\text{crit},1322^J,1}$ and $r_{\text{crit},1322^J,2}$ on Figures \ref{FIGURFeasiIneq1322J_SFig1} and \ref{FIGURFeasiIneq1322J_SFig2} when dealing with the feasibility inequalities associated to the period $13223122$. They will be represented as dashed red vertical lines.\\
\newline
Turning to these feasibility inequalities, we consider the eigenvector $u_{\lambda^\text{real}_{1322^J}}$ associated to the eigenvalue $\lambda^\text{real}_{1322^J}$ (the dominating eigenvalue if $r \in\ ]r_{\text{crit},1322^J,1},r_{\text{crit},1322^J,2}[$) such that $\big( u_{\lambda^\text{real}_{1322^J}} \big)_x = 1$, which is always possible except for $r = r_{\text{crit},1322^J,3}$, which is such that:
\begin{align}
r_{\text{crit},1322^J,3} \simeq 0.2363\ 6180\ 6783\ 6584\ 8104,
\end{align}
in which case we have necessarily $\big( u_{\lambda^\text{real}_{1322^J}} \big)_x = 0$. This critical value will be represented on Figures \ref{FIGURFeasiIneq1322J_SFig1}-\ref{FIGURFeasiIneq1322J_SFig2} as a solid red vertical line.\\

\begin{figure}
\centering

\begin{subfigure}{1\textwidth}
\centering
    \includegraphics[trim = 0cm 0cm 0cm 0cm, width=\linewidth]{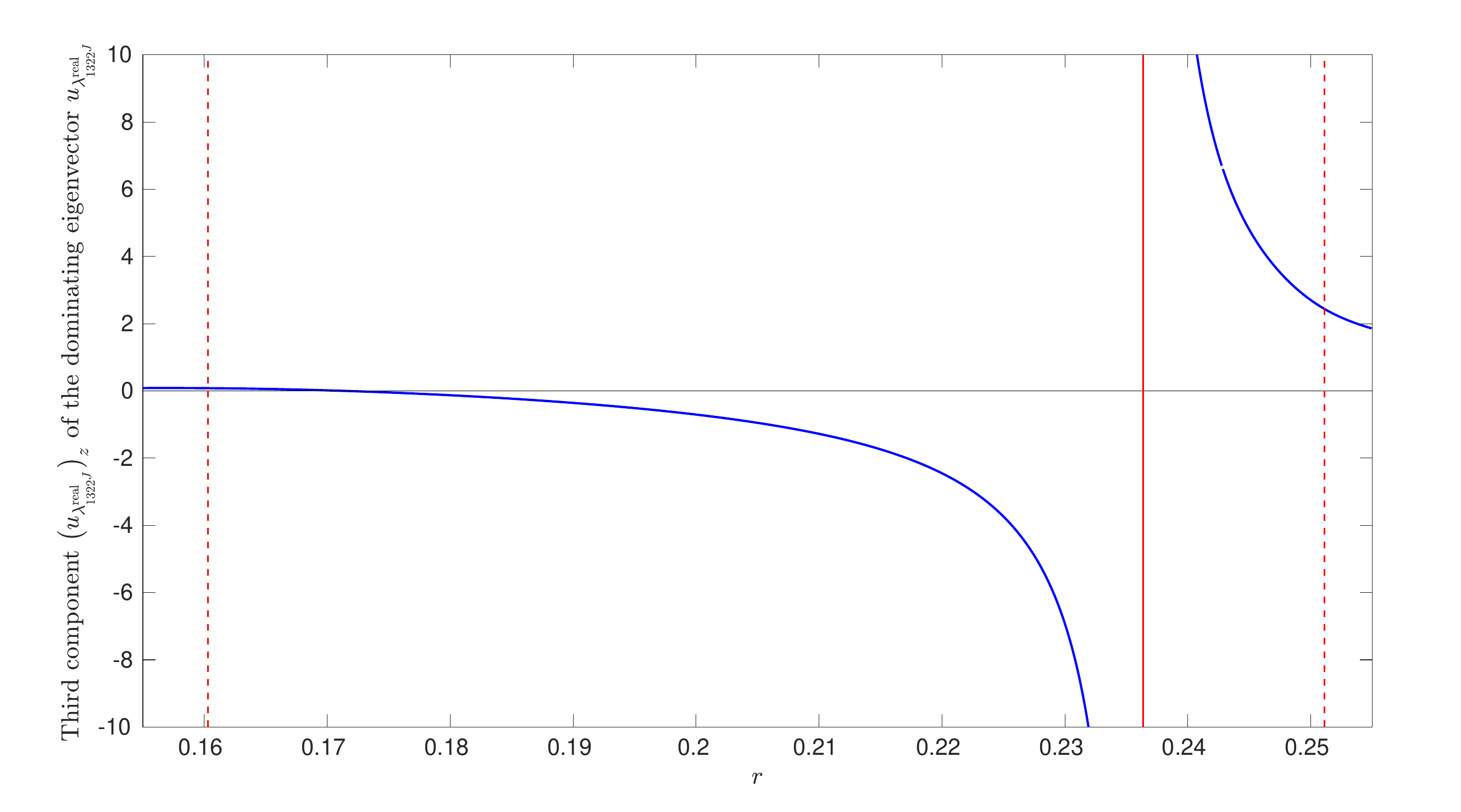}
    \caption{First feasibility inequality for 
    $13223122$: plot of $\big( u_{\lambda_{1322^J}^\text{dom}}\big)_z$ as a function of $r$.}
    \label{FIGURFeasiIneq1322J_SFig1}
\end{subfigure}

\begin{subfigure}{1\textwidth}
\centering
    \includegraphics[trim = 0cm 0cm 0cm 0cm, width=\linewidth]{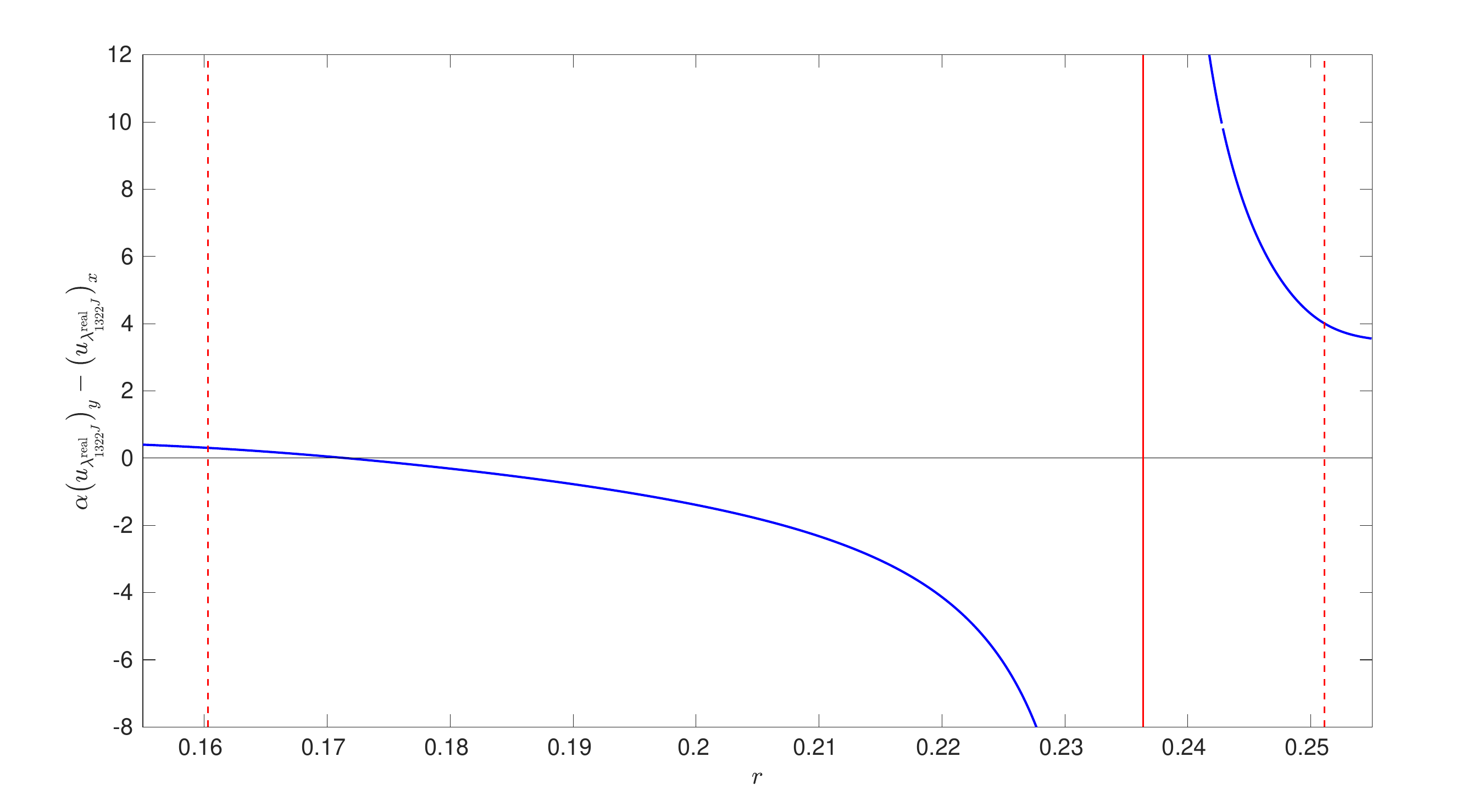}
    \caption{Second feasibility inequality for 
    $13223122$: plot of $\alpha \big( u_{\lambda_{1322^J}^\text{dom}}\big)_y - \big( u_{\lambda_{1322^J}^\text{dom}}\big)_x$ as a function of $r$.}
    \label{FIGURFeasiIneq1322J_SFig2}
\end{subfigure}
\caption{The two first feasibility inequalities associated to the pattern $13223122$.}
\label{FIGURFeasiIneq1322J___1__}
\end{figure}

\noindent
On Figure \ref{FIGURFeasiIneq1322J_SFig1}, we start with plotting the third component $\big( u_{\lambda^\text{real}_{1322^J}} \big)_z$, which has to be negative in order to have that $u_{\lambda^\text{real}_{1322^J}}$ belongs to the quadrant $\{x \geq 0, z \leq 0\}$, which is the domain of $\mathfrak{P}$. On $]r_{\text{crit},1322^J,1},r_{\text{crit},1322^J,2}[$, the third component vanishes only once, at $r_{\text{crit},1322^J,4}$ that is approximately given by:
\begin{align}
r_{\text{crit},1322^J,4} \simeq 0.1715\ 7287\ 5253\ 8099\ 0240.
\end{align}
Therefore, on $]r_{\text{crit},1322^J,1},r_{\text{crit},1322^J,2}[$, $\big( u_{\lambda^\text{real}_{1322^J}} \big)_z < 0$ if and only if $r \in\ ]r_{\text{crit},1322^J,4},r_{\text{crit},1322^J,2}[$. We plot now the quantity $\alpha \big( u_{\lambda^\text{real}_{1322^J}} \big)_y - \big( u_{\lambda^\text{real}_{1322^J}} \big)_x$ on Figure \ref{FIGURFeasiIneq1322J_SFig2}, which, when positive, ensures that $\mathfrak{P}\big( u_{\lambda^\text{real}_{1322^J}} \big) = P_1 u_{\lambda^\text{real}_{1322^J}}$. This quantity vanishes only once on $]r_{\text{crit},1322^J,1},r_{\text{crit},1322^J,2}[$, at $r_{\text{crit},1322^J,5}$, approximately given by:
\begin{align}
r_{\text{crit},1322^J,5} \simeq 0.1715\ 7287\ 5253\ 8099\ 0240.
\end{align}
We deduce then that $\alpha \big( u_{\lambda^\text{real}_{1322^J}} \big)_y - \big( u_{\lambda^\text{real}_{1322^J}} \big)_x$ is positive on $]r_{\text{crit},1322^J,1},r_{\text{crit},1322^J,2}[$ if and only if $r < r_{\text{crit},1322^J,5}$.\\
\newline
We now show that $r_{\text{crit},1322^J,4} = r_{\text{crit},1322^J,5}$, which implies that the periodic pattern $13223122$ cannot take place starting from the eigenvector $u_{\lambda^\text{real}_{1322^J}}$ associated to the negative eigenvalue $\lambda^\text{real}_{1322^J}$, at least when such an eigenvalue is dominating. To do so, we will determine two conditions for which the vector $u_\alpha = \hspace{0.5mm}^t\hspace{-0.25mm} (\alpha,1,0)$ is an eigenvector of the reduced collision matrix $\widetilde{P}_{1322^J}$.\\
On the one hand, $u_\alpha$ is an eigenvector of the reduced collision matrix $\widetilde{P}_{1322^J}$ only if $\alpha \big( \widetilde{P}_{1322^J} u_\alpha \big)_y = \big( \widetilde{P}_{1322^J} u_\alpha \big)_x$. This condition can be written in terms of the following polynomial equation on the restitution coefficient~$r$:
\begin{align}
Q_{1322^J,1}(r) = 0 \hspace{5mm} \text{with} \hspace{5mm} Q_{1322^J,1}(r) &= - \frac{r(r - 1)^2}{32} \big(- r^5 + 5r^4 + 4r^3 + 4r^2 + 5r - 1 \big) \\
&= \frac{r(r - 1)^2}{32} \big(r^3 + r^2 + r + 1 \big)\big( r^2 - 6r + 1 \big). \nonumber
\end{align}
On the other hand, $u_\alpha$ is an eigenvector of the reduced collision matrix $\widetilde{P}_{1322^J}$ only if $\big( \widetilde{P}_{1322^J} u_\alpha \big)_z = 0$. This condition can also be written as a polynomial equation on $r$, which reads:
\begin{align}
Q_{1322^J,2}(r) = 0 \hspace{5mm} \text{with} \hspace{5mm} Q_{1322^J,2}(r) &= -4r \big( r^5 - 5r^4 - 4r^3 - 4r^2 - 5r + 1 \big) \\
&= - 4r \big(r^3 + r^2 + r + 1 \big)\big( r^2 - 6r + 1 \big). \nonumber
\end{align}
As a side comment, we observe that the polynomial $r^2 - 6r + 1$ is also a factor of $Q_{\Delta_{1322^J}}$, involved in the discriminant of the characteristic polynomial of $\widetilde{P}_{1322^J}$.\\
The two conditions $Q_{1322^J,1}(r) = 0$ and $Q_{1322^J,2}(r) = 0$ are necessary and sufficient to ensure that $u_\alpha$ is an eigenvector of $\widetilde{P}_{1322^J}$. On $]0,1[$, each of these two conditions is fulfilled for the same restitution coefficient, which is $3-2\sqrt{2} \simeq 0.1716$, and there is no other solution on $]0,1[$. Therefore, for this particular choice of $r$, $u_\alpha$ is indeed an eigenvector of $\widetilde{P}_{1322^J}$, and a direct computation shows that $\big( \widetilde{P}_{1322^J} u_\alpha \big)_x < 0$, so that $u_\alpha$ is actually an eigenvector associated to the only eigenvalue of $\widetilde{P}_{1322^J}$ which is real for any $r$. $u_\alpha$ is therefore an eigenvector associated to the dominating eigenvalue. Since for $r = 3-2\sqrt{2}$, $u_\alpha$ is such that $\big( u_\alpha \big)_z = 0$ on the one hand, and such that $\alpha \big( u_\alpha \big)_y - \big( u_\alpha \big)_x = 0$ on the other hand, by uniqueness of the roots of the feasibility inequalities we deduce that:
\begin{align}
r_{\text{crit},1322^J,4} = r_{\text{crit},1322^J,5} = 3 - 2\sqrt{2}.
\end{align}
In the end, on the one hand, we need to have $r > r_{\text{crit},1322^J,4} = 3 - 2\sqrt{2}$ in order to fulfill the first feasibility inequality $\big( u_{\lambda^\text{real}_{1322^J}} \big)_z < 0$. On the other hand, we need to have $r < r_{\text{crit},1322^J,5} = 3 - 2\sqrt{2}$ in order to fulfill the second feasibility inequality $\alpha \big( u_{\lambda^\text{real}_{1322^J}} \big)_y - \big( u_{\lambda^\text{real}_{1322^J}} \big)_x > 0$. We deduce that there exists no value of the restitution coefficient $r$ for which the periodic orbit $13223122$ can be realized in a stable manner for the dynamical system $\big(\widehat{\mathfrak{P}}_r^n \big)_n$. The proof of Theorem \ref{THEORUnstabi13223122} is complete.
\end{proof}

\subsection{Concluding remarks}

\noindent
The writing of the dynamical system $\big( \widehat{\mathfrak{P}}_r^n \big)_n$ as the iteration of a piecewise projective linear transformation, together with a careful study of the three different matrices with which $\mathfrak{P}$ coincides on the different parts of its domain allowed us to observe numerically new periodic patterns. We also used the underlying piecewise linear structure of $\mathfrak{P}$ in order to study rigorously two different periodic orbits. In the first case, we have shown that $132312$, which corresponds to the following sequence of collisions $\mathfrak{abcbacbcbabacb}$, can indeed be realized in a stable manner. We emphasize that Theorem \ref{THEORStabi____132312} establishes the existence of stable periodic orbits for the $\mathfrak{b}$-to-$\mathfrak{b}$ mapping $\widehat{\mathfrak{P}}$ for (any) restitution coefficients larger than $r_{\text{crit},132^J} \simeq 0.2201$, which represents a substantial improvement with respect to the stable patterns discussed in \cite{CDKK999}, that can be realized for restitution coefficients at most equal to $3-2\sqrt{2} \simeq 0.1716$. In the second case, we proved that the pattern $13223122$ can never be realized in a stable manner, which is consistent with the fact that such a pattern was never detected in our numerical simulations.\\
The study of the patterns $132312$ and $13223122$ conducted respectively in Sections \ref{SSECTPatte__132312__} and \ref{SSECTPatter_13223122} can clearly be repeated to any other periodic pattern. The approach consists only in identifying the dominating eigenvalue of a certain collision matrix, and to verify if the associated eigenvectors fulfill a finite number of conditions. It remains to investigate if all the possible periodic patterns can be characterized.\\
Finally, we emphasize that we focused in the present article on the $\mathfrak{b}$-to-$\mathfrak{b}$ mapping $\mathfrak{P}$, which encodes all the possible collision patterns of the original four one-dimensional particle system. It remains to verify if the patterns that we detected are associated to non-trivial sets of initial configurations of the original particle system, leading to a ``physical'' inelastic collapse. Answering such a question is out of the scope of the present article, but the tools developed in \cite{CDKK999} allow certainly to address this question.

\section{Perspectives: the projective dynamics and the statistical properties}
\label{SECTIperspectives}

\noindent
The explicit formulae of Theorem~\ref{THEOR__P__FukngLinea} show that the $\mathfrak{b}$--to--$\mathfrak{b}$ map $\mathfrak{P}_r:\mathbb{R}^3\to\mathbb{R}^3$ is piecewise linear and, for every fixed $r\in(0,1)$, is in fact a strict contraction (in each branch) in the Euclidean metric. Hence the piecewise linear map $\mathfrak{P}_r$ itself has trivial global dynamics: all trajectories converge exponentially to the origin, which is the unique fixed point.\\
\noindent
The nontrivial behaviour observed in Section~\ref{SSECTSpherReduc} arises only after a projective renormalization. Namely, one considers the map
\[
\widehat{\mathfrak{P}}_r \;:\;
X \subset \mathbb{S}^2 \longrightarrow X,\qquad
\widehat{\mathfrak{P}}_r(x)
   = \frac{\mathfrak{P}_r(x)}{\|\mathfrak{P}_r(x)\|},
\]
where $X$ is the invariant region of \eqref{EQUATDefinDomai_X_P_}
corresponding to the physically admissible signs of the variables
$u_1,u_3$. This is the map that determines the evolution of the projective direction of the plane $\mathcal{P}(k)$, and therefore the symbolic sequence of collisions.\\
\newline
\noindent
Unlike $\mathfrak{P}_r$, the normalized map
$\widehat{\mathfrak{P}}_r$ is not linear, not globally Lipschitz,
and not obviously covered by any classical framework for piecewise-smooth contractions. In particular, on each branch $X_i\subset X$ the map has the form
\[
\widehat{\mathfrak{P}}_r(x) = \frac{P_i(r)x}{\|P_i(r)x\|},
\]
which is smooth inside each region but has discontinuities along the boundaries $\partial X_i$ where the normalization factor varies
non-smoothly ; even though $P_i(r)$ is strictly contracting on $\mathbb{R}^3$, the quotient map $x\mapsto P_i(r)x/\|P_i(r)x\|$ need not be contracting on $X$ ; indeed, numerically $\widehat{\mathfrak{P}}_r$ appears to exhibit non-uniform hyperbolicity for several values of $r$, that is, there are regions in which the derivative of the map has norm equal to $1$ ; the invariant domain $X$ is forward-invariant, but standard projective-dynamics theorems
(e.g.\ classical Perron–Frobenius theory for positive matrices,
Birkhoff–Hopf theorems) do not apply.\\
\newline
\noindent
Moreover, near the boundaries of the partition the normalization introduces strong nonlinearities, and the local Jacobian $D\widehat{\mathfrak{P}}_r(x)$ can have eigenvalues larger than or equal to $1$ in modulus for some $x$, suggesting intermittent or non-uniformly hyperbolic behaviour. The numerical plots of Section~\ref{SECTISeconNumerSimul} are consistent with the picture that, depending on $r$, the system displays an alternation between
parameter intervals where all sampled trajectories converge to a finite periodic orbit, and intermediate intervals where orbits appear to accumulate on more complicated sets (for instance, apparently quasi-periodic invariant curves) and exhibit sensitive dependence on initial conditions within their basin.\\
\newline
\noindent
From the statistical point of view, the rigorous analysis of the mapping $\mathfrak{P}$, when its orbits remain in the domain of the matrix $P_2$ on the one hand (see Section \ref{SSECTSpectStudyMatri_P_i_}), and the study of the stable periodic orbits on the other hand (see Theorem \ref{THEORStabi____132312}), implies that the simplest scenario of a single global attractor with a unique physical measure is impossible in certain parameter ranges of $r$. The simulations in the previous section, and in particular the two-dimensional plots at fixed restitution coefficient (see, e.g., Figure~\ref{FIGUR10Orb_r=.2}), show that for certain values of $r$ the dynamics of $\widehat{\mathfrak{P}}_r$ displays a genuine coexistence of different types of attracting objects: several periodic sinks, together with
a whole family of quasi-periodic invariant curves (quasi-periodic in the sense that, for any positive number $\varepsilon > 0$ and any point $\mathfrak{P}^n(x_0)$ of an orbit on such invariant curves, there exists $k = k(n,\varepsilon)$ such that $\vert \mathfrak{P}^n(x_0) - \mathfrak{P}^{n+k}(x_0) \vert \leq \varepsilon$). Varying the initial condition, one obtains different invariant ellipses for the same parameter $r$, and the union of these curves appears to fill a region of positive two-dimensional Lebesgue measure in the $(w_1,w_2)$-representation.\\
\newline
\noindent
Each individual invariant curve is a one-dimensional set and therefore has zero Lebesgue measure in the ambient section, but their union occupies a positive-measure region. This picture is consistent with the existence of an {uncountable} family of invariant curves $\{\Lambda_t\}_{t\in T}$, possibly forming a lamination of this region. On each curve $\Lambda_t$, typical time-averages converge to an ergodic invariant probability $\mu_t$ supported on that curve (often equivalent to arc-length), while periodic sinks support atomic invariant measures. Thus, for a fixed $r$, one should expect \emph{infinitely many} invariant probability measures for $\widehat{\mathfrak{P}}_r$, parametrized at least by the label $t$ of the invariant curve. None of these measures is absolutely continuous with respect to the two-dimensional Lebesgue measure, and a global absolutely continuous invariant probability is unlikely to exist in such regimes (see the next section for more details).\\
\newline
\noindent
Nevertheless, the statistical behaviour of typical orbits within the positive-measure region can still be described: for Lebesgue-almost every $x$ in that region, we expect that:
\[
\frac{1}{N}\sum_{k=0}^{N-1}
\varphi\!\left(\widehat{\mathfrak{P}}_r^{\,k}(x)\right)
\;\xrightarrow[N\to\infty]{}\;
\int \varphi \, d\mu_{t(x)},
\]
where $t(x)$ is the label of the unique invariant curve or periodic orbit containing the tail of the orbit of $x$, and $\mu_{t(x)}$ is its associated ergodic invariant measure. Thus typical orbits still have well-defined statistical behaviour, but this behaviour depends sensitively on the leaf of the invariant lamination on which the orbit lands.\\
\newline
\noindent
A natural long-term goal is therefore to place the family
$\{\widehat{\mathfrak{P}}_r\}_{r\in(0,1)}$ within a rigorous dynamical-systems framework. One possible approach is to study the transfer operator associated to $\widehat{\mathfrak{P}}_r$, formally defined by
\[
(\mathcal{L}_r f)(x)
   = \sum_{y\in \widehat{\mathfrak{P}}_r^{-1}(x)}
      \frac{f(y)}{|\det D\widehat{\mathfrak{P}}_r(y)|}.
\]
However, the Jacobian factors involve derivatives of the normalization
map, so that $|\det D\widehat{\mathfrak{P}}_r|$ is not constant along symbolic cylinders\footnote{Given a piecewise-defined map $F : X \to X$ with a finite partition
$X = \bigsqcup_{i\in \mathcal{A}} X_i$ of its domain into regions on which $F$ is smooth, the {itinerary} of a point $x$ is the sequence
$(i_0,i_1,i_2,\ldots)$ such that $F^k(x)\in X_{i_k}$. For a finite word $(i_0,i_1,\ldots,i_{n-1})$ of symbols in $\mathcal{A}$,
the corresponding {symbolic cylinder} is the set
\[
[i_0 i_1\ldots i_{n-1}]
 \;=\; \{\, x\in X : F^k(x)\in X_{i_k} \text{ for } 0\le k<n \,\}.
\]
Thus a cylinder is the set of points sharing the same first $n$ steps of their symbolic dynamics.  In the present setting, cylinders correspond to trajectories experiencing the same sequence of collisions for $n$ iterates of $\widehat{\mathfrak{P}}_r$.} (unlike for $\mathfrak{P}_r$), and the usual theories for uniformly expanding or contracting maps do not apply.  To the best of our knowledge, there is currently no ready-made theory that directly covers piecewise-smooth maps on a two-dimensional manifold that arise as projectivizations of contracting linear maps with this kind of discontinuity structure. Furthermore, the coexistence of a continuum of invariant curves shows that a spectral gap for $\mathcal{L}_r$, uniformly in $r$, on any space of densities with respect to the Lebesgue measure is not realistic: such a gap would imply finiteness (or at least strong regularity) of ergodic components for each $r$, contradicting the observed lamination. Nevertheless we conjecture that, for Lebesgue-almost every $r$ inside the windows of stability of the periodic patterns $(\mathfrak{ab})^n(\mathfrak{cb})^n$ (that is, in particular for $r < 3-2\sqrt{2}$), there are finitely many attractors, suggesting a possible quasi-compactness scenario for the associated transfer operator. On the other hand, for $r > 3-2\sqrt{2}$, the existence of an uncountable family of invariant curves prevents this scenario to happen.\\
In general, a more plausible programme, compatible with the numerical evidence, is to study \emph{induced} transfer operators (see e.g. \cite{Aaro997}) or geometric structures on each invariant leaf, or on suitable transverse sections. For parameters in which $\widehat{\mathfrak{P}}_r$ exhibits a laminated region, one might aim to:
\begin{itemize}
  \item[(i)] construct an invariant foliation or lamination of the positive-measure region by invariant curves,
        and estimate the probability of landing on such a curve, that is, understanding the basins of attraction\footnote{That is the set: \[
\left\{x: \frac 1n \sum_{k=0}^{n-1}\delta_{\hat{\mathfrak{P}}^k_r(x)} \quad \text{converges weakly to} \quad \mu_t  \right\}.
\]} of the disintegrated invariant measures $\{ \mu_t \}$;
  \item[(ii)] understand the disintegration of Lebesgue measure along this lamination and the associated family
        of invariant measures $\{\mu_t\}$;
  \item[(iii)] investigate whether certain transverse or induced maps admit quasi-compact transfer operators on anisotropic Banach spaces,\footnote{Roughly speaking, an anisotropic Banach space is a functional space in which different directions of the dynamics are measured with different strengths: smoothness is enforced along expanding directions, while more irregular behaviour is allowed along contracting or neutral directions, so that the transfer operator becomes well behaved even when ordinary function spaces fail. See e.g. \cite{babook} for a more detailed discussion on the argument.} even if the global operator $\mathcal{L}_r$ does not;
  \item[(iv)] study $\mu_t$ and the geometry of the leaves when the restitution parameter $r$ changes.
\end{itemize}

\subsection{Coexistence of periodic sinks and quasi-periodic attractors}

\noindent
Our results show that for some values of $r$ the dynamics of $\widehat{\mathfrak{P}}_r$ exhibits coexistence of several types of attractors. For a given value of $r$, different initial data can converge to  attracting periodic orbits (finite sets of points in the $(w_1,w_2)$-representation) or quasi-periodic invariant curves (ellipses) surrounding the rotation centre associated with $P_2$ in the regime where it has a pair of complex conjugate eigenvalues.\\
\newline
\noindent
In the plots at $r=0.2$, some of the ten randomly chosen initial conditions produce orbits that accumulate on a finite set of points, with a relatively small cardinal (periodic orbits), while others produce orbits that lie on smooth closed curves (Figure~\ref{FIGUR10Orb_r=.2})%
\footnote{We refer here to the planar representation on the strip $\{x-z=1\}$ in coordinates $(w_1,w_2)$ defined by \eqref{EQUATFormuRepre_2_d_}. Each invariant curve in that representation corresponds to a closed curve on the projective sphere, obtained as the intersection of the sphere with a two-dimensional invariant manifold of the projective dynamics of $P_2$.}.
When one explores a grid of initial conditions for the same value of $r$, the union of such invariant curves that are actually visited by the dynamics turns out to form a region of positive two-dimensional Lebesgue measure in the $(w_1,w_2)$-plane, so this picture is consistent with a foliation (or lamination) of a positive-measure region by invariant one-dimensional curves.\\
\newline
\noindent
From the point of view of invariant measures, this suggests the following scenario. For a fixed $r$ in such a regime there are: invariant probability measures supported on attracting periodic orbits (Dirac measures); for each invariant ellipse, at least one natural invariant probability measure obtained by time-averaging along orbits on that curve, typically equivalent to arc-length on the curve;  convex combinations and weak limits of the measures above, giving rise to a rich family of invariant measures. Thus, at a fixed $r$ one should expect \emph{infinitely many} invariant probability measures for $\widehat{\mathfrak{P}}_r$, and there is no obvious reason to expect finiteness of ergodic components: the quasi-periodic region is more naturally described as a continuum of invariant curves, each
supporting its own ergodic measure. With respect to the two-dimensional Lebesgue measure on the section, these invariant measures are singular (as these measures are supported on one-dimensional sets), so there is no indication of absolutely continuous invariant probabilities in the usual sense. Nevertheless, for Lebesgue-almost
every initial condition in that region, the time-averages converge to the invariant measure supported on the specific curve (or periodic orbit) that contains the orbit, so that a typical orbit still has a well-defined statistical behaviour, which depends on its leaf of the invariant lamination. In particular, 
\[
\text{for Lebesgue-a.e.\ } x \in R(r), \qquad
\frac{1}{N}\sum_{k=0}^{N-1}
\varphi\!\left(\widehat{\mathfrak{P}}_r^{\,k}(x)\right)
\;\xrightarrow[N\to\infty]{}\;
\int \varphi \, d\mu_{t(x)},
\]
{where:}
\begin{itemize}
    \item[-] $R(r)\subset X$ is the region (of positive two-dimensional Lebesgue measure) filled by invariant quasi-periodic curves and attracting periodic orbits at the fixed parameter $r$;
    \item[-] $\Lambda_{t}$ denotes the invariant curve or periodic orbit labelled by the parameter $t$;
    \item[-] $\mu_{t}$ is the ergodic invariant probability measure supported on $\Lambda_{t}$ (arc-length measure on an invariant curve, or the uniform measure on a periodic orbit);
    \item[-] $t(x)$ is the label of the unique invariant curve or periodic orbit such that $\widehat{\mathfrak{P}}_r^{\,n}(x)\in \Lambda_{t(x)}$ for all sufficiently large $n$.
\end{itemize}

\noindent
The numerical data already support a concrete research plan to investigate the statistical properties of $\widehat{\mathfrak{P}}_r$ at fixed $r$. We briefly outline several steps that can be implemented in future works to support the possible theoretical framework outlined in the previous section.

\paragraph{(1) Improve what we know: identification and classification of attractors.} For a fixed $r$, one can sample a large grid of initial conditions, iterate $\widehat{\mathfrak{P}}_r$ for a long time, and classify the limiting
behaviour of the trajectories
computing the relative frequency (with respect to the grid of initial conditions) of each basin.

\paragraph{(2) Approximation of physical measures via time averages.}
For each detected attractor $A_j$ and for several initial conditions in its basin, one can compute empirical measures
\[
\mu_{N,x}
  \;=\; \frac{1}{N}\sum_{k=0}^{N-1} \delta_{\widehat{\mathfrak{P}}_r^k(x)}
\]
and compare them for different $x$ in the same basin. If these empirical measures converge and are close for all $x$ in the basin of $A_j$, this provides numerical evidence for the existence and uniqueness of (a natural notion of) a physical measure\footnote{That is, a measure $\mu_j$ which is invariant for the map and such that its basin of attraction has a positive Lebesgue measure.} $\mu_j$ supported on $A_j$. For quasi-periodic ellipses,
one expects $\mu_j$ to be equivalent to arc-length measure on the curve, while for periodic sinks $\mu_j$ is simply the uniform measure on the finitely many periodic points. For each physical measure $\mu_j$ one can select a representative orbit and compute time correlations of observables along the orbit, e.g., for suitable observable functions $\varphi, \psi$,
\[
C_N(\varphi,\psi;n)
   \;=\; \frac{1}{N}\sum_{k=0}^{N-1}
          \Big(\varphi(\widehat{\mathfrak{P}}_r^k(x))
               \,\psi(\widehat{\mathfrak{P}}_r^{k+n}(x))\Big)
          \;-\;
          \Big(\frac{1}{N}\sum_{k=0}^{N-1}\varphi(\widehat{\mathfrak{P}}_r^k(x))\Big)
          \Big(\frac{1}{N}\sum_{k=0}^{N-1}\psi(\widehat{\mathfrak{P}}_r^k(x))\Big).
\]
The asymptotic behaviour of $C_N(\varphi,\psi;n)$ as $n\to\infty$ (for fixed large $N$) can help distinguish quasi-periodic regimes (no decay, oscillatory behaviour) from mixing regimes (decay of
correlations at some rate), even if a rigorous theory is not yet available.

\paragraph{(3) Lyapunov exponents and (non-)hyperbolicity.} Finite-time Lyapunov exponents along long trajectories can be computed numerically. For the projective collision map 
\[
x_{k+1} = \widehat{\mathfrak{P}}_r(x_k), \qquad x_0 \in X\subset \mathbb{S}^2,
\]
the derivative 
\[
D\widehat{\mathfrak{P}}_r(x): T_x\mathbb{S}^2 \longrightarrow 
T_{\widehat{\mathfrak{P}}_r(x)}\mathbb{S}^2
\]
is well-defined and smooth on each branch of the partition.  
The linearised dynamics along an orbit 
\[
x,\; \widehat{\mathfrak{P}}_r(x),\; \widehat{\mathfrak{P}}_r^2(x),\;\ldots
\]
is given by the derivative cocycle
\[
D\widehat{\mathfrak{P}}_r^n(x)
   = D\widehat{\mathfrak{P}}_r\big(\widehat{\mathfrak{P}}_r^{\,n-1}(x)\big)
     \cdots
     D\widehat{\mathfrak{P}}_r\big(\widehat{\mathfrak{P}}_r(x)\big)
     D\widehat{\mathfrak{P}}_r(x).
\]
Let $v \in T_x\mathbb{S}^2$ be a unit tangent vector. The \emph{finite-time Lyapunov exponent} at $x$ over $n$ iterates is
\[
\lambda_n(x,v)
   = \frac{1}{n}\log 
     \left\| D\widehat{\mathfrak{P}}_r^n(x)\, v \right\|.
\]
The maximal exponent at time $n$ is
\[
\lambda_n(x)
   = \frac{1}{n}\log 
     \left\| D\widehat{\mathfrak{P}}_r^n(x) \right\|,
\]
and the (maximal) Lyapunov exponent is the limit
\[
\lambda(x) = \lim_{n\to\infty} \lambda_n(x),
\]
whenever the limit exists. If $x$ belongs to an attracting periodic orbit of period $p$, then all eigenvalues of $D\widehat{\mathfrak{P}}_r^p(x)$ have modulus $<1$, hence $\lambda(x) < 0$. If the orbit of $x$ lies on an invariant ellipse $\Lambda_t$, then $\lambda_{\mathrm{tan}}(x) = 0$ and $\lambda_{\perp}(x) < 0$, corresponding respectively to the tangent direction along the ellipse
and the transverse contracting direction. A positive exponent, $\lambda(x) > 0$, would indicate sensitive dependence on initial conditions and chaotic behaviour.\\
Along a numerically computed trajectory $x_k=\widehat{\mathfrak{P}}_r^{\,k}(x)$, set
\[
A_k = D\widehat{\mathfrak{P}}_r(x_k),
\]
choose a unit vector $v_0\in T_x\mathbb{S}^2$, and iterate
\[
v_{k+1} = \frac{A_k v_k}{\|A_k v_k\|}.
\]
The finite-time Lyapunov exponent is then approximated by
\[
\lambda_n(x)
   = \frac{1}{n}\sum_{k=0}^{n-1} 
        \log \|A_k v_k\|.
\]
For large $n$, $\lambda_n(x)$ provides a reliable numerical estimate of the maximal Lyapunov exponent at $x$.\\
Such computations provide additional information on the type of
statistical behaviour to be expected (e.g.\ quasi-periodic vs chaotic) and on the regularity class of potential invariant densities.

\subsection{Period-Adding Cascades and Border-Collision Structure}

\noindent
Collapse patterns are much richer than expected. There are infinitely many stable patterns indexed by \(n\), each appearing in narrow parameter windows, with highly irregular behavior between them.  
Unlike the three-particle case, the four-particle system shows strong dependence on both the restitution coefficient and the initial configuration. Statistical observables do not converge to deterministic limits in the irregular regimes.\\ Therefore, the statistical properties depend discontinuously on the restitution coefficient as well: crossing a root of \(Q_n(r)\) (see \eqref{EQUATPolynLowerBound}) causes sudden changes in the attractor and its statistical behavior, showing the appearance of infinitely many stability windows, separated by parameter intervals where the dynamics appears irregular. The sequence of stability windows observed in our numerical simulations, each associated to a periodic collision pattern of the form \((\mathfrak{ab})^n(\mathfrak{cb})^n\), is strongly reminiscent of the period–adding structure that arises in piecewise-linear dynamical systems. In classical border-collision models one encounters a sequence of parameter intervals in which the attracting orbit has symbolic itinerary \(L^nR^n\), with \(n=1,2,3,\dots\), separated by parameter regions where the dynamics becomes irregular or chaotic (see \cite[Sections 4.4–4.7]{Bro}). The correspondence in our setting is direct, with \(L\) playing the role of the collision pair \(\mathfrak{ab}\) and \(R\) that of \(\mathfrak{cb}\), so that the stable periodic patterns \((\mathfrak{ab})^n(\mathfrak{cb})^n\) constitute a period–adding cascade. The appearance of infinitely many such windows, separated by intervals where no periodic stabilization is observed, suggests a bifurcation structure reminiscent of mode-locking and devil’s staircase phenomena well known in piecewise-smooth systems (\cite{Bro}), thereby motivating further investigation in this direction.
\newpage

\begin{appendices}
\section{Algorithms used in the numerical simulations}

\subsection{Computation of $\mathfrak{P}$ introduced in \cite{DoHR025}}
\label{SSECTAppenAlgoTrigo}

\noindent
In \cite{DoHR025}, we rewrote the discrete dynamical system introduced in \eqref{EQUATDiscrDynamSyste}. More precisely, we took advantage of the spherical reduction discussed in Section \ref{SSECTSpherReduc}, and we wrote the explicit action of a collision of type $\mathfrak{a}$, $\mathfrak{b}$ or $\mathfrak{c}$ on the two vectors $p$ and $q$, that were parametrized only in terms of the angles $\theta$ and $\varphi$ (the so-called \emph{trigonometric representation} in \cite{DoHR025}), as it is done for instance in \eqref{EQUATParamInitiaDataAngle} in the case when the pair of particles \footnotesize{\circled{2}}\normalsize{}-\footnotesize{\circled{3}}\normalsize{} are colliding. In the case when one of the two other pairs is colliding, the expression of the position and velocity vectors is of course different from \eqref{EQUATParamInitiaDataAngle}, but remains very similar (for instance, see (4.8) and (4.11) in \cite{DoHR025} when the collision involves the pair \footnotesize{\circled{1}}\normalsize{}-\footnotesize{\circled{2}}\normalsize{}).\\
We provide here the code, written in MATLAB, of such an algorithm. The inputs of the program consist in three data, $\theta_0$ (denoted by \texttt{theta0} in the algorithm below), $\varphi_0$ (denoted by \texttt{phi0}) and \texttt{NN}, where $\theta_0 \in\ ]0,\pi/2[$ and $\varphi_0\ \in\ ]0,\pi[\backslash\{\pi/2\}$ are used to parametrize the initial configuration \texttt{X}, \texttt{V} of the system, assuming that the pair of particles \footnotesize{\circled{2}}\normalsize{}-\footnotesize{\circled{3}}\normalsize{} are initially in contact (\texttt{X} being the vector of initial relative positions, and \texttt{V} the vector of initial relative velocities) as:
\begin{align}
\texttt{X} = \begin{pmatrix} \sin\theta_0 \\ 0 \\ \cos\theta_0 \end{pmatrix},\hspace{5mm} \texttt{V} = \cos(\varphi_0) \begin{pmatrix} \cos\theta_0 \\ 0 \\ -\sin\theta_0 \end{pmatrix} + \sin\varphi_0 \begin{pmatrix} 0 \\ 1 \\ 0 \end{pmatrix}.
\end{align}
\texttt{NN} is the number of collisions that will be computed. One has to be careful here: one iteration of the algorithm below \emph{does not} describe one iteration of the $\mathfrak{b}$-to-$\mathfrak{b}$ mapping, but computes instead the evolution of the particle between two consecutive collisions. In other words, the algorithm computes the iterates of the discrete dynamical system \eqref{EQUATDiscrDynamSyste}, with orbits represented in the variables of the spherical reduction. Therefore, in order to compute a certain number \texttt{N} of iterations of the $\mathfrak{b}$-to-$\mathfrak{b}$ mapping, it is necessary (and sufficient) to choose \texttt{NN} larger than 3\texttt{N}.\\
The output of the program consists in two matrices \texttt{Trajectory} and \texttt{SphericalTrajectory}, of respective sizes $7 \times$(\texttt{NN}+1) and $3 \times$(\texttt{NN}+1). The first matrix contains the normalized position and velocity vectors corresponding to the trigonometric representation of the spherical billiard (respectively \texttt{Trajectory(1:3,j)} and \texttt{Trajectory(4:6,j)}, both with 3 entries) at the \texttt{j}-th collision (where \texttt{j}=1 corresponds to the initial configuration). The second matrix contains the angles $\theta$ (\texttt{SphericalTrajectory(1,j)}) and $\varphi$ (\texttt{SphericalTrajectory(2,j)}) of the trigonometric representation of the spherical reduction mapping. Finally, the last lines \texttt{Trajectory(7,j)} and \texttt{SphericalTrajectory(2,j)} of the two matrices encode which pair of particles is in contact at the time of the \texttt{j}-th collision.\\
In the code, the matrices \texttt{A}, \texttt{B} and \texttt{C} correspond respectively to $A$, $B$ and $C$ given by \eqref{EQUATCollision_Matri}.\\
\newline
\texttt{Trajectory = zeros(7,NN+1)}\\
\texttt{SphericalTrajectory = zeros(3,NN+1)}\\
\texttt{X = [sin(theta0);0;cos(theta0)]}\\
\texttt{V = cos(phi0)*[cos(theta0);0;-sin(theta0)] + sin(phi0)*[0;1;0]}\\
\texttt{Contact\_pair = 2}\\
\texttt{Trajectory(:,1) = [X;V;Contact\_pair]}\\
\texttt{SphericalTrajectory(:,1) = [theta0;phi0;Contact\_pair]}\\
\newline
\texttt{j = 2}\\
\texttt{\textcolor{blue}{while} j < NN+2}\\
\newline
\textcolor{white}{bla}\texttt{\textcolor{blue}{if} Contact\_pair == 1}\\
\textcolor{white}{blabla}\texttt{theta = acos(X(2,1));}\\
\textcolor{white}{blabla}\texttt{CosPhi = V(3,1)/cos(theta);}\\
\textcolor{white}{blabla}\texttt{phi = acos(CosPhi);}\\
\textcolor{white}{blabla}\texttt{\textcolor{blue}{if} phi < pi/2}\\
\textcolor{white}{blablabla}\texttt{X = [cos(theta)*sin(phi);0;cos(phi)];}\\
\textcolor{white}{blablabla}\texttt{X = X/norm(X);}\\
\textcolor{white}{blablabla}\texttt{V = B*V;}\\
\textcolor{white}{blablabla}\texttt{V = V - dot(V,X)*X;}\\
\textcolor{white}{blablabla}\texttt{V = V/norm(V);}\\
\textcolor{white}{blablabla}\texttt{Contact\_pair = 2;}\\
\textcolor{white}{blabla}\texttt{\textcolor{blue}{elseif} phi > pi/2}\\ 
\textcolor{white}{blablabla}\texttt{X = [sin(theta)*sin(phi);-cos(phi);0];}\\
\textcolor{white}{blablabla}\texttt{X = X/norm(X);}\\
\textcolor{white}{blablabla}\texttt{V = C*V;}\\
\textcolor{white}{blablabla}\texttt{V = V - dot(V,X)*X;}\\
\textcolor{white}{blablabla}\texttt{V = V/norm(V);}\\
\textcolor{white}{blablabla}\texttt{Contact\_pair = 3;}\\
\textcolor{white}{blabla}\texttt{\textcolor{blue}{end}}\\
\newline
\textcolor{white}{bla}\texttt{\textcolor{blue}{elseif} Contact\_pair == 2}\\
\textcolor{white}{blabla}\texttt{theta = acos(X(3,1));}\\
\textcolor{white}{blabla}\texttt{CosPhi = V(1,1)/cos(theta);}\\
\textcolor{white}{blabla}\texttt{phi = acos(CosPhi);}\\
\textcolor{white}{blabla}\texttt{\textcolor{blue}{if} phi < pi/2}\\
\textcolor{white}{blablabla}\texttt{X = [cos(phi);cos(theta)*sin(phi);0];}\\
\textcolor{white}{blablabla}\texttt{X = X/norm(X);}\\
\textcolor{white}{blablabla}\texttt{V = C*V;}\\
\textcolor{white}{blablabla}\texttt{V = V - dot(V,X)*X;}\\
\textcolor{white}{blablabla}\texttt{V = V/norm(V);}\\
\textcolor{white}{blablabla}\texttt{Contact\_pair = 3;}\\
\textcolor{white}{blabla}\texttt{\textcolor{blue}{elseif} phi > pi/2}\\\textcolor{white}{blablabla}\texttt{X = [0;sin(theta)*sin(phi);-cos(phi)];}\\
\textcolor{white}{blablabla}\texttt{X = X/norm(X);}\\
\textcolor{white}{blablabla}\texttt{V = A*V;}\\
\textcolor{white}{blablabla}\texttt{V = V - dot(V,X)*X;}\\
\textcolor{white}{blablabla}\texttt{V = V/norm(V);}\\
\textcolor{white}{blablabla}\texttt{Contact\_pair = 1;}\\
\textcolor{white}{blabla}\texttt{\textcolor{blue}{end}}\\
\newline
\textcolor{white}{bla}\texttt{\textcolor{blue}{elseif} Contact\_pair == 3}\\
\textcolor{white}{blabla}\texttt{theta = acos(X(1,1));}\\
\textcolor{white}{blabla}\texttt{CosPhi = V(2,1)/cos(theta);}\\
\textcolor{white}{blabla}\texttt{phi = acos(CosPhi);}\\
\textcolor{white}{blabla}\texttt{\textcolor{blue}{if} phi < pi/2}\\
\textcolor{white}{blablabla}\texttt{X = [0;cos(phi);cos(theta)*sin(phi)];}\\
\textcolor{white}{blablabla}\texttt{X = X/norm(X);}\\
\textcolor{white}{blablabla}\texttt{V = A*V;}\\
\textcolor{white}{blablabla}\texttt{V = V - dot(V,X)*X;}\\
\textcolor{white}{blablabla}\texttt{V = V/norm(V);}\\
\textcolor{white}{blablabla}\texttt{Contact\_pair = 1;}\\
\textcolor{white}{blabla}\texttt{\textcolor{blue}{elseif} phi > pi/2}\\ \textcolor{white}{blablabla}\texttt{X = [-cos(phi);0;sin(theta)*sin(phi)];}\\
\textcolor{white}{blablabla}\texttt{X = X/norm(X);}\\
\textcolor{white}{blablabla}\texttt{V = B*V;}\\
\textcolor{white}{blablabla}\texttt{V = V - dot(V,X)*X;}\\
\textcolor{white}{blablabla}\texttt{V = V/norm(V);}\\
\textcolor{white}{blablabla}\texttt{Contact\_pair = 2;}\\
\textcolor{white}{blabla}\texttt{\textcolor{blue}{end}}\\
\newline
\textcolor{white}{bla}\texttt{\textcolor{blue}{end}}\\
\newline
\textcolor{white}{bla}\texttt{Trajectory(:,j) = [X; V; Contact\_pair];}\\
\textcolor{white}{bla}\texttt{SphericalTrajectory(:,j) = [theta; phi; Contact\_pair]];}\\
\textcolor{white}{bla}\texttt{j=j+1;}\\
\newline
\texttt{\textcolor{blue}{end}}

\newpage
\subsection{Computation of $\mathfrak{P}$ relying on its piecewise linear expression}
\label{SSECTAppenAlgoLinea}

\noindent
Relying on the result of Theorem \ref{THEOR__P__FukngLinea}, the algorithm to compute the orbits of the mapping $\mathfrak{P}$ is the following, written in MATLAB.\\
The inputs of the program are \texttt{X} and \texttt{N}. \texttt{X} is a $3 \times 1$ matrix such that \texttt{X(1,1)X(3,1)} $< 0$, corresponding to the initial configuration $X = \hspace{0.5mm}^t\hspace{-0.25mm} (x,y,z)$ on which $\mathfrak{P}$ is acting according to the expressions given by Theorem \ref{THEOR__P__FukngLinea}. \texttt{N} $\in \mathbb{N}$ is the number of iterations of $\mathfrak{P}$ that will be computed.\\
The program will return the \texttt{3 x N} matrix called \texttt{Trajectory}, storing the initial configuration \texttt{X} as well as its \texttt{N} iterates obtained by applying recursively $\mathfrak{P}$.\\
In the code, the matrices \texttt{P1}, \texttt{P2} and \texttt{P3} correspond respectively to $P_1$, $P_2$ and $P_3$ given by \eqref{EQUATCollisionMatric_P_i_}.\\
\newline
\texttt{Trajectory = zeros(3,N+1)};\\
\newline
\texttt{\textcolor{blue}{if} X(1,1) < 0}\\ 
\textcolor{white}{bla}\texttt{X = -X};\\
\texttt{\textcolor{blue}{end}}\\
\newline
\texttt{Trajectory(:,1)} = \texttt{X/norm(X)};\\
\newline
\texttt{\textcolor{blue}{for} j = 1:1:N}\\
\newline
\textcolor{white}{bla}\texttt{Next\_coll\_cond2 = a*X(2,1) - X(1,1);}\\
\textcolor{white}{bla}\texttt{Next\_coll\_cond3 = a*X(2,1) - X(3,1);}\\
\newline
\textcolor{white}{bla}\texttt{\textcolor{blue}{if} X(2,1) > 0 \&\& Next\_coll\_cond2 > 0}\\
\textcolor{white}{blabla}\texttt{X = P1*X;}\\
\textcolor{white}{bla}\texttt{\textcolor{blue}{elseif} X(2,1) < 0 \&\& Next\_coll\_cond3 < 0}\\
\textcolor{white}{blabla}\texttt{X = P3*X;}\\
\textcolor{white}{bla}\texttt{\textcolor{blue}{else}}\\
\textcolor{white}{blabla}\texttt{X = P2*X;}\\
\textcolor{white}{bla}\texttt{\textcolor{blue}{end}}\\
\newline
\textcolor{white}{bla}\texttt{X=X/norm(X);}\\
\textcolor{white}{bla}\texttt{Trajectory(:,j+1) = X/norm(X);}\\
\newline
\texttt{\textcolor{blue}{end}}

\vspace{1cm}
\noindent
The simplicity of the previous algorithm makes it much more efficient than the code presented in Section \ref{SSECTAppenAlgoTrigo} and previously used in \cite{DoHR025}. In particular, besides the renormalization, only multiplications with fixed matrices are used.
\vspace{20mm}

\section{Remarkable values and polynomials}

\subsection{Lower and upper bounds of the windows of stability}
\label{SSECTAppenLowerUpperBoundWindoStabi}

\noindent
In this section, we provide the numerical approximations of the lower and upper bounds of the windows of stability of the periodic patterns $(\mathfrak{ab})^n(\mathfrak{cb})^n$ with $12$ decimals, for $1 \leq n \leq 100$.

\begin{table}
\vspace{-1cm}
\hspace{-3cm}
\begin{center}
  \begin{tabular}{l|cc||l|cc}
    \toprule
    \multirow{2}{*}{$n$} & {Lower bounds} & {Upper bounds}  &  \multirow{2}{*}{$n$} & {Lower bounds} & {Upper bounds} \\
    &  \multicolumn{2}{c||}{(12 decimals)} & &  \multicolumn{2}{c}{(12 decimals)} \\
    \toprule
    1 & 0.1715\ 7287\ 5254 & {not defined} & 51 & 0.0718\ 7459\ 2447 & 0.0718\ 7547\ 6957 \\
    2 & 0.1275\ 4409\ 7592 & 0.1715\ 7287\ 5254 & 52 & 0.0718\ 7164\ 2069 & 0.0718\ 7247\ 6597 \\
    3 & 0.0945\ 1940\ 7247 & 0.1010\ 2051\ 4434 & 53 & 0.0718\ 6885\ 6382 & 0.0718\ 6964\ 4629 \\
    4 & 0.0841\ 8702\ 7764 & 0.0864\ 2723\ 3726 & 54 & 0.0718\ 6622\ 3352 & 0.0718\ 6696\ 8677 \\
    5 & 0.0796\ 4890\ 9753 & 0.0807\ 0090\ 3149 & 55 & 0.0718\ 6373\ 2022 & 0.0718\ 6443\ 7486 \\
    \midrule
    6 & 0.0772\ 3623\ 3508 & 0.0778\ 1868\ 2882 & 56 & 0.0718\ 6137\ 2401 & 0.0718\ 6204\ 0798 \\
    7 & 0.0757\ 9451\ 9545 & 0.0761\ 5218\ 9750 & 57 & 0.0718\ 5913\ 5365 & 0.0718\ 5976\ 9248 \\
    8 & 0.0748\ 6202\ 3787 & 0.0750\ 9797\ 4967 & 58 & 0.0718\ 5701\ 2566 & 0.0718\ 5761\ 4270 \\
    9 & 0.0742\ 2322\ 3822 & 0.0743\ 8730\ 7127 & 59 & 0.0718\ 5499\ 6355 & 0.0718\ 5556\ 8022 \\
    10 & 0.0737\ 6606\ 0903 & 0.0738\ 8488\ 9127 & 60 & 0.0718\ 5307\ 9711 & 0.0718\ 5362\ 3309 \\
    \midrule
    11 & 0.0734\ 2741\ 7001 & 0.0735\ 1628\ 9652 & 61 & 0.0718\ 5125\ 6184 & 0.0718\ 5177\ 3522 \\
    12 & 0.0731\ 6946\ 8752 & 0.0732\ 3770\ 5419 & 62 & 0.0718\ 4951\ 9838 & 0.0718\ 5001\ 2580 \\
    13 & 0.0729\ 6840\ 0785 & 0.0730\ 2194\ 9003 & 63 & 0.0718\ 4786\ 5201 & 0.0718\ 4833\ 4882 \\
    14 & 0.0728\ 0859\ 7801 & 0.0728\ 5140\ 1470 & 64 & 0.0718\ 4628\ 7226 & 0.0718\ 4673\ 5263 \\
    15 & 0.0726\ 7946\ 8887 & 0.0727\ 1422\ 8709 & 65 & 0.0718\ 4478\ 1247 & 0.0718\ 4520\ 8951 \\
    \midrule
    16 & 0.0725\ 7362\ 1759 & 0.0726\ 0223\ 8447 & 66 & 0.0718\ 4334\ 2950 & 0.0718\ 4375\ 1532 \\
    17 & 0.0724\ 8576\ 7872 & 0.0725\ 0961\ 1075 & 67 & 0.0718\ 4196\ 8339 & 0.0718\ 4235\ 8923 \\
    18 & 0.0724\ 1204\ 1721 & 0.0724\ 3211\ 8927 & 68 & 0.0718\ 4065\ 3708 & 0.0718\ 4102\ 7336 \\
    19 & 0.0723\ 4956\ 4256 & 0.0723\ 6663\ 0089 & 69 &  0.0718\ 3939\ 5617 & 0.0718\ 3975\ 3256 \\
    20 & 0.0722\ 9615\ 5120 & 0.0723\ 1078\ 3959 & 70 & 0.0718\ 3819\ 0869 & 0.0718\ 3853\ 3419 \\
    \midrule
    21 & 0.0722\ 5013\ 8390 &  0.0722\ 6277\ 3735 & 71 & 0.0718\ 3703\ 6491 & 0.0718\ 3736\ 4789 \\
    22 & 0.0722\ 1020\ 8587 &  0.0722\ 2119\ 7327 & 72 & 0.0718\ 3592\ 9712 & 0.0718\ 3624\ 4538 \\
    23 & 0.0721\ 7533\ 6495 & 0.0721\ 8495\ 3132 & 73 & 0.0718\ 3486\ 7951 & 0.0718\ 3517\ 0032 \\
    24 & 0.0721\ 4470\ 1792 & 0.0721\ 5316\ 5860 & 74 & 0.0718\ 3384\ 8799 & 0.0718\ 3413\ 8814 \\
    25 & 0.0721\ 1764\ 4119 & 0.0721\ 2513\ 2894 & 75 & 0.0718\ 3287\ 0006 & 0.0718\ 3314\ 8589 \\
    \midrule
    26 & 0.0720\ 9362\ 7029 & 0.0721\ 0028\ 4938 & 76 & 0.0718\ 3192\ 9468 & 0.0718\ 3219\ 7212 \\
    27 & 0.0720\ 7221\ 1072 & 0.0720\ 7815\ 6740 & 77 & 0.0718\ 3102\ 5217 & 0.0718\ 3128\ 2678 \\
    28 & 0.0720\ 5303\ 3474 & 0.0720\ 5836\ 5070 & 78 & 0.0718\ 3015\ 5411 & 0.0718\ 3040\ 3108 \\
    29 & 0.0720\ 3579\ 2610 & 0.0720\ 4059\ 1941 & 79 & 0.0718\ 2931\ 8321 & 0.0718\ 2955\ 6743 \\
    30 & 0.0720\ 2023\ 6037 & 0.0720\ 2457\ 1718 & 80 & 0.0718\ 2851\ 2327 & 0.0718\ 2874\ 1930 \\
    \midrule
    31 & 0.0720\ 0615\ 1165 & 0.0720\ 1008\ 1092 & 81 & 0.0718\ 2773\ 5907 & 0.0718\ 2795\ 7122 \\
    32 & 0.0719\ 9335\ 7937 & 0.0719\ 9693\ 1238 & 82 & 0.0718\ 2698\ 7633 & 0.0718\ 2720\ 0862 \\
    33 & 0.0719\ 8170\ 3029 & 0.0719\ 8496\ 1610 & 83 & 0.0718\ 2626\ 6159 & 0.0718\ 2647\ 1783 \\
    34 & 0.0719\ 7105\ 5220 & 0.0719\ 7403\ 5005 & 84 & 0.0718\ 2557\ 0221 & 0.0718\ 2576\ 8597 \\
    35 & 0.0719\ 6130\ 1677 & 0.0719\ 6403\ 3605 & 85 & 0.0718\ 2489\ 8628 & 0.0718\ 2509\ 0093 \\
    \midrule
    36 & 0.0719\ 5234\ 4946 & 0.0719\ 5485\ 5776 & 86 & 0.0718\ 2425\ 0256 & 0.0718\ 2443\ 5128 \\
    37 & 0.0719\ 4410\ 0499 & 0.0719\ 4641\ 3474 & 87 & 0.0718\ 2362\ 4049 & 0.0718\ 2380\ 2626 \\
    38 & 0.0719\ 3649\ 4734 & 0.0719\ 3863\ 0118 & 88 & 0.0718\ 2301\ 9006 & 0.0718\ 2319\ 1572 \\
    39 & 0.0719\ 2946\ 3316 & 0.0719\ 3143\ 8845 & 89 & 0.0718\ 2243\ 4185 & 0.0718\ 2260\ 1006 \\
    40 & 0.0719\ 2294\ 9819 & 0.0719\ 2478\ 1064 & 90 & 0.0718\ 2186\ 8695 & 0.0718\ 2203\ 0023 \\
    \midrule
    41 & 0.0719\ 1690\ 4578 & 0.0719\ 1860\ 5262 & 91 & 0.0718\ 2132\ 1692 & 0.0718\ 2147\ 7766 \\
    42 & 0.0719\ 1128\ 3746 & 0.0719\ 1286\ 5997 & 92 & 0.0718\ 2079\ 2379 & 0.0718\ 2094\ 3425 \\
    43 & 0.0719\ 0604\ 8493 & 0.0719\ 0752\ 3062 & 93 & 0.0718\ 2028\ 0002 & 0.0718\ 2042\ 6233 \\
    44 & 0.0719\ 0116\ 4330 & 0.0719\ 0254\ 0774 & 94 & 0.0718\ 1978\ 3846 & 0.0718\ 1992\ 5465 \\
    45 & 0.0718\ 9660\ 0542 & 0.0718\ 9788\ 7381 & 95 & 0.0718\ 1930\ 3233 & 0.0718\ 1944\ 0431 \\
    \midrule
    46 & 0.0718\ 9232\ 9701 & 0.0718\ 9353\ 4549 & 96 & 0.0718\ 1883\ 7521 & 0.0718\ 1897\ 0481 \\
    47 & 0.0718\ 8832\ 7254 & 0.0718\ 8945\ 6934 & 97 & 0.0718\ 1838\ 6099 & 0.0718\ 1851\ 4994 \\
    48 & 0.0718\ 8457\ 1168 & 0.0718\ 8563\ 1808 & 98 & 0.0718\ 1794\ 8390 & 0.0718\ 1807\ 3383 \\
    49 & 0.0718\ 8104\ 1625 & 0.0718\ 8203\ 8740 & 99 & 0.0718\ 1752\ 3843 & 0.0718\ 1764\ 5091 \\
    50 & 0.0718\ 7772\ 0764 & 0.0718\ 7865\ 9328 & 100 & 0.0718\ 1711\ 1936 & 0.0718\ 1722\ 9586 \\
    \bottomrule
  \end{tabular}
    \end{center}
  \caption{Numerical approximation ($12$ first decimals) of the lower and upper bounds of the $100$ first windows of stability of the patterns $(\mathfrak{ab})^n(\mathfrak{cb})^n$.}
  \label{TABLEDecimals_LowerUpperBounds}
\end{table}

\newpage

\subsection{Table of the first polynomials $r^{2n}P_n(r)P_n(1/r)-r^{2n}$}
\label{SSECTAppenPolynLowerBoundStabi}

\noindent
We provide here the explicit expressions of the polynomials $r^{2n}P_n(r)P_n(1/r)-r^{2n}$, with $P_n = \text{Tr}\big(J(BA)^n\big)$, whose one of the roots of each is the lower bound of the interval of stability of the patterns $(\mathfrak{ab})^n(\mathfrak{cb})^n$.

\begin{table}[h!]
    \hspace{-2.5cm}
    \renewcommand{\arraystretch}{2} 
    \setlength{\tabcolsep}{10pt} 
    \begin{tabular}{|c|m{8.5cm}|m{8.5cm}|} 
        \hline
        \footnotesize $n$ & \footnotesize $Q_n$ (developed version) & \footnotesize $Q_n$ (factorized version) \\ \hline
        \footnotesize 1 & \vspace{-0.2cm}\tiny $\frac{1}{16}r^4 - \frac{1}{4}r^3 - \frac{5}{8}r^2 - \frac{1}{4}r + \frac{1}{16}$ & \vspace{-0.2cm}\tiny $\frac{1}{16}(r+1)^2(r^2-6r+1)$ \\ \hline
        \footnotesize 2 & \vspace{-0.2cm}\tiny $\frac{5}{256}r^8 - \frac{5}{32}r^7 + \frac{3}{64}r^6 - \frac{3}{32}r^5 + \frac{81}{128}r^4 - \frac{3}{32}r^3 + \frac{3}{64}r^2 - \frac{5}{32}r + \frac{5}{256}$  & \vspace{-0.2cm}\tiny $\frac{1}{256}(r+1)^2(5r^6-50r^5+107r^4-188r^3+107r^2-50r+5)$ \\ \hline
        \footnotesize 3 & \vspace{-0.2cm}\tiny $\frac{21}{4096}r^{12} -\frac{73}{1024}r^{11} +\frac{357}{2048}r^{10} +\frac{59}{1024}r^9 +\frac{283}{4096}r^8 +\frac{7}{512}r^7 -\frac{509}{1024}r^6 +\frac{7}{512}r^5 +\frac{283}{4096}r^4 + \frac{59}{1024}r^3 + \frac{357}{2048}r^2 - \frac{73}{1024}r + \frac{21}{4096}$ & \vspace{-0.2cm}\tiny $\frac{1}{4096}(r^2-1)^2(21r^8 - 292 r^7 + 756 r^6 - 348 r^5 + 1774 r^4 - 348 r^3 + 756 r^2 - 292 r + 21)$ \\ \hline
        \footnotesize 4 & \vspace{-0.2cm}\tiny $\frac{85}{65536}r^{16} - \frac{53}{2048}r^{15} + \frac{1095}{8192}r^{14} - \frac{211}{2048}r^{13} - \frac{1277}{16384}r^{12} - \frac{137}{2048}r^{11} - \frac{303}{8192}r^{10} - \frac{111}{2048}r^9 - \frac{17697}{32768}r^8 - \frac{111}{2048}r^7 - \frac{303}{8192}r^6 - \frac{137}{2048}r^5 - \frac{1277}{16384}r^4 - \frac{211}{2048}r^3 + \frac{1095}{8192}r^2 - \frac{53}{2048}r + \frac{85}{65536}$ & \vspace{-0.2cm}\tiny $\frac{1}{65536}(r + 1)^2(85 r^{14} - 1866 r^{13} + 12407 r^{12} - 29700 r^{11} + 41885 r^{10} - 58454 r^9 + 72599 r^8 - 90296 r^7 + 72599 r^6 - 58454 r^5 + 41885 r^4 - 29700 r^3 + 12407 r^2 - 1866 r + 85)$ \\ \hline
        \footnotesize 5 & \vspace{-0.2cm}\tiny $\frac{341}{1048576}r^{20} - \frac{2215}{262144}r^{19} + \frac{35547}{524288}r^{18} - \frac{42543}{262144}r^{17} - \frac{10319}{1048576}r^{16} + \frac{1325}{65536}r^{15} + \frac{2977}{131072}r^{14} - \frac{159}{65536}r^{13} - \frac{17155}{524288}r^{12} - \frac{12721}{131072}r^{11} - \frac{156383}{262144}r^{10} - \frac{12721}{131072}r^9 - \frac{17155}{524288}r^8 - \frac{159}{65536}r^7 + \frac{2977}{131072}r^6 + \frac{1325}{65536}r^5 - \frac{10319}{1048576}r^4 - \frac{42543}{262144}r^3 + \frac{35547}{524288}r^2 - \frac{2215}{262144}r + \frac{341}{1048576}$ & \vspace{-0.2cm}\tiny $\frac{1}{1048576} (r + 1)^2(341r^{18} - 9542r^{17} + 89837r^{16} - 340304r^{15} + 580452r^{14} - 799400r^{13} + 1042164r^{12} - 1287472r^{11} + 1498470r^{10} - 1811236r^9 + 1498470r^8 - 1287472r^7 + 1042164r^6 - 799400r^5 + 580452r^4 - 340304r^3 + 89837r^2 - 9542r + 341)$ \\ \hline
        \footnotesize 6 & \vspace{-0.2cm}\tiny $\frac{1365}{16777216}r^{24} - \frac{5459}{2097152}r^{23} + \frac{119361}{4194304}r^{22} - \frac{249245}{2097152}r^{21} + \frac{1046837}{8388608}r^{20} + \frac{178783}{2097152}r^{19} + \frac{198397}{4194304}r^{18} + \frac{41649}{2097152}r^{17} + \frac{324955}{16777216}r^{16} + \frac{32985}{1048576}r^{15} + \frac{74977}{2097152}r^{14} - \frac{15849}{1048576}r^{13} - \frac{2145421}{4194304}r^{12} - \frac{15849}{1048576}r^{11} + \frac{74977}{2097152}r^{10} + \frac{32985}{1048576}r^9 + \frac{324955}{16777216}r^8 + \frac{41649}{2097152}r^7 + \frac{198397}{4194304}r^6 + \frac{178783}{2097152}r^5 + \frac{1046837}{8388608}r^4 - \frac{249245}{2097152}r^3 + \frac{119361}{4194304}r^2 - \frac{5459}{2097152}r + \frac{1365}{16777216}$ & \vspace{-0.2cm}\tiny $\frac{1}{16777216}(r^2 - 1)^2(1365r^{20} - 43672r^{19} + 480174r^{18} - 2081304r^{17} + 3052657r^{16} - 2688672r^{15} + 6418728r^{14} - 2962848r^{13} + 10109754r^{12} - 2709264r^{11} + 14400596r^{10} - 2709264r^9 + 10109754r^8 - 2962848r^7 + 6418728r^6 - 2688672r^5 + 3052657r^4 - 2081304r^3 + 480174r^2 - 43672r + 1365)$ \\ \hline
        \footnotesize 7 & \vspace{-0.2cm}\tiny $\frac{5461}{268435456}r^{28} - \frac{51877}{67108864}r^{27} + \frac{1438737}{134217728}r^{26} - \frac{4339241}{67108864}r^{25} + \frac{40022743}{268435456}r^{24} - \frac{1129031}{33554432}r^{23} - \frac{5192403}{67108864}r^{22} - \frac{2073279}{33554432}r^{21} - \frac{6301571}{268435456}r^{20} - \frac{655231}{67108864}r^{19} - \frac{2823425}{134217728}r^{18} - \frac{1836819}{67108864}r^{17} - \frac{4944041}{268435456}r^{16} - \frac{872357}{16777216}r^{15} - \frac{18088117}{33554432}r^{14} - \frac{872357}{16777216}r^{13} - \frac{4944041}{268435456}r^{12} - \frac{1836819}{67108864}r^{11} - \frac{2823425}{134217728}r^{10} - \frac{655231}{67108864}r^9 - \frac{6301571}{268435456}r^8 - \frac{2073279}{33554432}r^7 - \frac{5192403}{67108864}r^6 - \frac{1129031}{33554432}r^5 + \frac{40022743}{268435456}r^4 - \frac{4339241}{67108864}r^3 + \frac{1438737}{134217728}r^2 - \frac{51877}{67108864}r + \frac{5461}{268435456}$ & \vspace{-0.2cm}\tiny $\frac{1}{268435456}(r + 1)^2(5461r^{26} - 218430r^{25} + 3308873r^{24} - 23756280r^{23} + 84226430r^{22} - 153728828r^{21} + 202461614r^{20} - 267780632r^{19} + 326798079r^{18} - 388436450r^{17} + 444427971r^{16} - 507766768r^{15} + 566161524r^{14} - 638513992r^{13} + 566161524r^{12} - 507766768r^{11} + 444427971r^{10} - 388436450r^9 + 326798079r^8 - 267780632r^7 + 202461614r^6 - 153728828r^5 + 84226430r^4 - 23756280r^3 + 3308873r^2 - 218430r + 5461)$ \\ \hline
        \footnotesize 8 & \vspace{-0.2cm}\tiny $\frac{21845}{4294967296}r^{32} - \frac{60073}{268435456}r^{31} + \frac{1010289}{268435456}r^{30} - \frac{8010651}{268435456}r^{29} + \frac{58278227}{536870912}r^{28} - \frac{34802629}{268435456}r^{27} - \frac{15076669}{268435456}r^{26} + \frac{2386713}{268435456}r^{25} + \frac{18759667}{1073741824}r^{24} - \frac{3118757}{268435456}r^{23} - \frac{3881071}{268435456}r^{22} + \frac{3019329}{268435456}r^{21} + \frac{9472389}{536870912}r^{20} - \frac{1744673}{268435456}r^{19} - \frac{8984773}{268435456}r^{18} - \frac{24778123}{268435456}r^{17} - \frac{1259891681}{2147483648}r^{16} - \frac{24778123}{268435456}r^{15} - \frac{8984773}{268435456}r^{14} - \frac{1744673}{268435456}r^{13} + \frac{9472389}{536870912}r^{12} + \frac{3019329}{268435456}r^{11} - \frac{3881071}{268435456}r^{10} - \frac{3118757}{268435456}r^9 + \frac{18759667}{1073741824}r^8 + \frac{2386713}{268435456}r^7 - \frac{15076669}{268435456}r^6 - \frac{34802629}{268435456}r^5 + \frac{58278227}{536870912}r^4 - \frac{8010651}{268435456}r^3 + \frac{1010289}{268435456}r^2 - \frac{60073}{268435456}r + \frac{21845}{4294967296}$ & \vspace{-0.2cm}\tiny $\frac{1}{4294967296}(r + 1)^2(21845r^{30} - 1004858r^{29} + 18152495r^{28} - 163470548r^{27} + 775014417r^{26} - 1943400350r^{25} + 2870559579r^{24} - 3759531400r^{23} + 4723541889r^{22} - 5737452490r^{21} + 6689265955r^{20} - 7592770156r^{19} + 8572053469r^{18} - 9579251550r^{17} + 10442693263r^{16} - 11702584944r^{15} + 10442693263r^{14} - 9579251550r^{13} + 8572053469r^{12} - 7592770156r^{11} + 6689265955r^{10} - 5737452490r^9 + 4723541889r^8 - 3759531400r^7 + 2870559579r^6 - 1943400350r^5 + 775014417r^4 - 163470548r^3 + 18152495r^2 - 1004858r + 21845)$ \\ \hline
        \footnotesize 9 & \vspace{-0.2cm}\tiny $\frac{87381}{68719476736}r^{36} - \frac{1092259}{17179869184}r^{35} + \frac{43209447}{34359738368}r^{34} - \frac{212853619}{17179869184}r^{33} + \frac{4269713485}{68719476736}r^{32} - \frac{296590343}{2147483648}r^{31} + \frac{252213133}{4294967296}r^{30} + \frac{188040453}{2147483648}r^{29} + \frac{905801569}{17179869184}r^{28} + \frac{159965871}{4294967296}r^{27} + \frac{416297761}{8589934592}r^{26} + \frac{116931159}{4294967296}r^{25} - \frac{243450859}{17179869184}r^{24} - \frac{30883381}{2147483648}r^{23} + \frac{81855995}{4294967296}r^{22} + \frac{79821671}{2147483648}r^{21} + \frac{1070869059}{34359738368}r^{20} - \frac{208374721}{8589934592}r^{19} - \frac{8911424427}{17179869184}r^{18} - \frac{208374721}{8589934592}r^{17} + \frac{1070869059}{34359738368}r^{16} + \frac{79821671}{2147483648}r^{15} + \frac{81855995}{4294967296}r^{14} - \frac{30883381}{2147483648}r^{13} - \frac{243450859}{17179869184}r^{12} + \frac{116931159}{4294967296}r^{11} + \frac{416297761}{8589934592}r^{10} + \frac{159965871}{4294967296}r^9 + \frac{905801569}{17179869184}r^8 + \frac{188040453}{2147483648}r^7 + \frac{252213133}{4294967296}r^6 - \frac{296590343}{2147483648}r^5 + \frac{4269713485}{68719476736}r^4 - \frac{212853619}{17179869184}r^3 + \frac{43209447}{34359738368}r^2 - \frac{1092259}{17179869184}r + \frac{87381}{68719476736}$ & \vspace{-0.2cm}\tiny $\frac{1}{68719476736}(r^2 - 1)^2(87381r^{32} - 4369036r^{31} + 86593656r^{30} - 860152548r^{29} + 4442813416r^{28} - 11206827036r^{27} + 12834443304r^{26} - 15536207028r^{25} + 24849279468r^{24} - 17306133084r^{23} + 40194497720r^{22} - 17205160596r^{21} + 54565912536r^{20} - 18092456300r^{19} + 70247023272r^{18} - 16425458532r^{17} + 88069872126r^{16} - 16425458532r^{15} + 70247023272r^{14} - 18092456300r^{13} + 54565912536r^{12} - 17205160596r^{11} + 40194497720r^{10} - 17306133084r^9 + 24849279468r^8 - 15536207028r^7 + 12834443304r^6 - 11206827036r^5 + 4442813416r^4 - 860152548r^3 + 86593656r^2 - 4369036r + 87381)$ \\ \hline
        \footnotesize 10 & \vspace{-0.2cm}\tiny $\frac{349525}{1099511627776}r^{40} - \frac{2446673}{137438953472}r^{39} + \frac{111323415}{274877906944}r^{38} - \frac{656403031}{137438953472}r^{37} + \frac{16860878795}{549755813888}r^{36} - \frac{13880529647}{137438953472}r^{35} + \frac{35431771167}{274877906944}r^{34} + \frac{3945563655}{137438953472}r^{33} - \frac{51933796943}{1099511627776}r^{32} - \frac{1902386801}{34359738368}r^{31} - \frac{3442320253}{68719476736}r^{30} - \frac{1723679415}{34359738368}r^{29} - \frac{1719740815}{137438953472}r^{28} + \frac{1000292645}{34359738368}r^{27} + \frac{834339215}{68719476736}r^{26} - \frac{947288877}{34359738368}r^{25} - \frac{18459801603}{549755813888}r^{24} - \frac{1260871583}{68719476736}r^{23} - \frac{1215137471}{137438953472}r^{22} - \frac{3475964857}{68719476736}r^{21} - \frac{148356134671}{274877906944}r^{20} - \frac{3475964857}{68719476736}r^{19} - \frac{1215137471}{137438953472}r^{18} - \frac{1260871583}{68719476736}r^{17} - \frac{18459801603}{549755813888}r^{16} - \frac{947288877}{34359738368}r^{15} + \frac{834339215}{68719476736}r^{14} + \frac{1000292645}{34359738368}r^{13} - \frac{1719740815}{137438953472}r^{12} - \frac{1723679415}{34359738368}r^{11} - \frac{3442320253}{68719476736}r^{10} - \frac{1902386801}{34359738368}r^9 - \frac{51933796943}{1099511627776}r^8 + \frac{3945563655}{137438953472}r^7 + \frac{35431771167}{274877906944}r^6 - \frac{13880529647}{137438953472}r^5 + \frac{16860878795}{549755813888}r^4 - \frac{656403031}{137438953472}r^3 + \frac{111323415}{274877906944}r^2 - \frac{2446673}{137438953472}r + \frac{349525}{1099511627776}$ & \vspace{-0.2cm}\tiny $\frac{1}{1099511627776}(r + 1)^2(349525r^{38} - 20272434r^{37} + 485489003r^{36} - 6201929820r^{35} + 45640128227r^{34} - 196122563810r^{33} + 488332084061r^{32} - 748977095072r^{31} + 957688309140r^{30} - 1227275900840r^{29} + 1441786368492r^{28} - 1711454577424r^{27} + 1967364859836r^{26} - 2191265777608r^{25} + 2428516122820r^{24} - 2696079712096r^{23} + 2926723698166r^{22} - 3177541629564r^{21} + 3418638461194r^{20} - 3715350730536r^{19} + 3418638461194r^{18} - 3177541629564r^{17} + 2926723698166r^{16} - 2696079712096r^{15} + 2428516122820r^{14} - 2191265777608r^{13} + 1967364859836r^{12} - 1711454577424r^{11} + 1441786368492r^{10} - 1227275900840r^9 + 957688309140r^8 - 748977095072r^7 + 488332084061r^6 - 196122563810r^5 + 45640128227r^4 - 6201929820r^3 + 485489003r^2 - 20272434r + 349525)$ \\ \hline
    \end{tabular}
    \caption{Expression of the $10$ first polynomials $r^{2n}P_n(r)P_n(1/r) - r^{2n}$, where $P_n$ is the trace of $J(BA)^n$.}
    \label{TABLEPolynCritiValueLower}
\end{table}

\end{appendices}

\newpage

\noindent\textbf{Acknowledgements.} 
The authors gratefully acknowledge the support of the UMI Group DinAmicI (www.dinamici.org), the INdAM group GNFM, and the project PRIN 2022 (Research Projects of National Relevance) - Project code 202277WX43.\\
The first author acknowledges the hospitality of the Gran Sasso Science Institute, where part of this work was carried out, as well as the support of the project PRIN 2022 directed by Alessia Nota (Research Projects of National Relevance) - Project code 202277WX43. The second author acknowledges the hospitality of the Scuola Normale Superiore di Pisa, where another part of this work was carried out, as well as the support of the projects ``Dynamics and Information
Research Institute - Quantum Information (Teoria dell'Informazione), Quantum Technologies'' directed by Stefano Marmi.

\def\adresse{
\begin{description}

\vspace{0.5cm}

\item[R.~Castorrini] {Classe di Scienze,\\
Scuola Normale Superiore di Pisa\\
and\\
Dipartimento di Economia, Ingegneria,\\
Società e Impresa (DEIM),\\
Università della Tuscia,\\
01100, Viterbo, Italy\\
E-mail: \texttt{roberto.castorrini@gmail.com}}

\vspace{0.5cm}

\item[T.~Dolmaire]
{Dipartimento di Ingegneria e Scienze\\ dell'Informazione e Matematica (DISIM),\\ Università degli Studi dell'Aquila, \\ 67100  L'Aquila, Italy \\  E-mail: \texttt{theophile.dolmaire@univaq.it}} 

\end{description}
}

\vspace{5mm}
\adresse

\end{document}